%% file: Article.tex
\documentclass{article}

\input{preamble}

\title{\textbf{Schauder estimates for degenerate stable \\
Kolmogorov equations}}
\author{\textbf{Lorenzo Marino}\footnote{Laboratoire de Mod\'elisation Math\'ematique d'Evry (LaMME), Universit\'e d'Evry Val d'Essonne, $23$
Boulevard de France $91037$ Evry, France and Dipartimento di Matematica, Universit\`a di Pavia, Via Adolfo Ferrata $5$, $27100$ Pavia, Italy.\newline
\emph{E-mail Adress:} lorenzo.marino@univ-evry.fr}\phantom{,} \!\footnote{This work was supported by a public grant
as part of the FMJH project.}}

\begin{document}
\maketitle

\begin{abstract}
We provide here global Schauder-type estimates for a chain of integro-partial differential equations (IPDE) driven by a degenerate
stable Ornstein-Uhlenbeck operator possibly perturbed by a deterministic drift, when the coefficients lie in some suitable anisotropic
H\"older spaces. Our approach mainly relies on a perturbative method based on forward parametrix expansions and, due to the low
regularizing properties on the degenerate variables and to some integrability constraints linked to the stability index, it also exploits
duality results between appropriate Besov Spaces. In particular, our method also applies in some super-critical cases. Thanks to these
estimates, we show in addition the well-posedness of the considered IPDE in a suitable functional space.
\end{abstract}


{\small{\textbf{MSC:} Primary: $35$K$65$, $35$R$09$, $35$B$45$; Secondary: $60$H$30$.}}

\section{Introduction}
For a fixed time horizon $T>0$ and two integers $n,d$ in $\N$, we are interested in proving global Schauder estimates for the following
parabolic integro-partial differential equation (IPDE):
\begin{equation}
\label{Degenerate_Stable_PDE}
\begin{cases}
   \partial_t u(t,\bm{x}) + \langle A \bm{x} + \bm{F}(t,\bm{x}), D_{\bm{x}}u(t,\bm{x})\rangle +
L_\alpha u(t,\bm{x}) \, = \, -f(t,\bm{x}), & \mbox{on } [0,T]\times \R^{nd} \\
    u(T,\bm{x}) \, = \, g(\bm{x}) & \mbox{on }\R^{nd}.
  \end{cases}
\end{equation}
where $\bm{x}=(\bm{x}_1,\dots,\bm{x}_n)$ is in $\R^{nd}$ with each $\bm{x}_i$ in $\R^d$ and $\langle \cdot,\cdot \rangle$ represents the
inner product on  $\R^{nd}$. We consider a symmetric non-local $\alpha$-stable operator $L_{\alpha}$ acting non-degenerately only on the
first $d$ variables and a matrix $A$ in $\R^{{nd}} \otimes\R^{{nd}}$ with the following sub-diagonal structure:
\begin{equation}\label{eq:def_matrix_A}
A \, := \, \begin{pmatrix}
               0_{d\times d} & \dots         & \dots         & \dots     & 0_{d\times d} \\
               A_{2,1}       & 0_{d\times d} & \dots         & \dots     & 0_{d\times d} \\
               0_{d\times d} & A_{3,2}       & 0_{d\times d} & \dots     & 0_{d\times d} \\
               \vdots        & \ddots        & \ddots        & \vdots    & \vdots        \\
               0_{d\times d} & \dots         & 0_{d\times d} & A_{n,n-1} & 0_{d\times d}
             \end{pmatrix}.
\end{equation}
We will assume moreover that it satisfies a H\"ormander-like condition, allowing the smoothing effect of $L_\alpha$ to propagate into the system. \newline
Above, the source $f\colon [0,T]\times \R^{nd} \to \R$ and the terminal condition $g\colon \R^{nd}\to \R$ are assumed to be bounded and to
belong to some suitable anisotropic H\"older space. \newline
The additional drift term $\bm{F}(t,\bm{x})=\bigl(F_1(t,\bm{x}),\dots,F_n(t,\bm{x})\bigr)$ can be seen as a perturbation of the
 Ornstein-Ulhenbeck operator $L_\alpha+\langle A\bm{x}, D_{\bm{x}}\rangle$ and it has  structure "compatible" with $A$, i.e.\ at level $i$, it depends only on the super diagonal entries:
\[F_i(t,\bm{x}) \, := \, F_i(t,\bm{x}_i,\dots,\bm{x}_n).\]
It may be unbounded but we assume it to be H\"older continuous with an index depending on the level of the chain.

\paragraph{Related Results.}
A large literature on the topic of Schauder estimates in the $\alpha$-stable non-local framework has been developed in the recent years (\textcolor{black}{see e.g.\ Lunardi and R\"ockner \cite{Lunardi:Rockner19} for an overview of the field}), mainly
in the non-degenerate setting and assuming that $\alpha\ge 1$, the so called sub-critical case. We mention for instance the stable-like setting, corresponding to time-inhomogeneous operators of the form
\begin{equation}\label{eq_Mik_Prag}
L_t\phi(\bm{x}) \, = \, \int_{\R^{nd}}\bigl[\phi(\bm{x}+\bm{y})-\phi(\bm{x}) - \mathds{1}_{1\le \alpha <2}\langle \bm{y}, D_{\bm{x}} \rangle
\bigr] m(t,\bm{x},\bm{y}) \frac{d\bm{y}}{\vert \bm{y} \vert^{d+\alpha}} + \mathds{1}_{1\le \alpha <2}\langle \bm{F}(t,\bm{x}), D_{\bm{x}}u(t,
\bm{x})\rangle
\end{equation}
where the diffusion coefficient $m$ is bounded from above and below, H\"older continuous in the spatial variable $\bm{x}$ and even in $\bm{y}$ if $\alpha=1$.
Under these conditions and assuming the drift $\bm{F}$ to be bounded and H\"older continuous in space, Mikulevicius and Pragarauskas in \cite{Mikulevicius:Pragarauskas14} obtained parabolic Schauder type bounds on the whole space and derived
from those estimates the well-posedness of the corresponding martingale problem. We notice however that for the super-critical case
(when $\alpha<1$), the drift term in \eqref{eq_Mik_Prag} is set to zero. This is mainly due to the fact that in the super-critical case,
$L_\alpha$ is of order $\alpha$ (in the Fourier space) and does not dominate the drift term $\bm{F}$ which is roughly speaking of order one.
\newline
In the non-degenerate, driftless framework (i.e.\ when $A\bm{x} +\bm{F}=0$ and $n=1$ in \eqref{Degenerate_Stable_PDE}), Bass \cite{Bass09} was the first to derive
elliptic Schauder estimates for stable like operators. We can refer as well to the recent work of Imbert and collaborators
\cite{Imbert:Jin:Shvydkoy18} concerning Schauder estimates for stable-like operator \eqref{eq_Mik_Prag} with $\alpha=1$ and some related applications to non-local
Burgers equations. Eventually, still in the driftless case, Ros-Oton and Serra worked in \cite{Ros-Oton:Serra16} for interior and boundary
elliptic-regularity in a general, symmetric $\alpha$-stable setting, assuming that the L\'evy measure $\nu_\alpha$ associated with $L_\alpha$
writes in polar coordinates $y=\rho s$, $(\rho,s)\in [0,\infty)\times \mathbb{S}^{d-1}$ as
\[\nu_\alpha(dy) \, = \, \tilde{\mu}(ds)\frac{d\rho}{\rho^{1+\alpha}}\]
where $\tilde{\mu}$ is a non-degenerate, symmetric measure on the sphere $\mathbb{S}^{d-1}$. Related to the above, we can mention also the
associated work of Fernandez-Real and Ros-Oton \cite{Fernanadez:Ros-Oton17} for parabolic equations.

In the elliptic setting, when $\alpha\in [1,2)$ and $L_\alpha$ is a non-degenerate, symmetric $\alpha$-stable operator and for bounded H\"older
drifts, global Schauder estimates were obtained by Priola in \cite{Priola12} or in \cite{priola18} for respective applications to
the strong well-posedness and Davie's uniqueness for the corresponding SDE. We notice furthermore that in the
sub-critical case, elliptic Schauder estimates can be proven for more general, translation invariant, L\'evy-type generators for
 following \cite{priola18} (see Section $6$, and Remark $5$ therein).

In the super-critical case, parabolic Schauder estimates were established by Chaudru de Raynal, Menozzi and Priola in
\cite{Chaudru:Menozzi:Priola19} under similar assumptions to \cite{Ros-Oton:Serra16}. An existence result is also provided therein. \newline
We  mention as well the work of Zhang and Zhao \cite{Zhang:Zhao18} who address through probabilistic arguments the parabolic
Dirichlet problem for stable-like operators of the form \eqref{eq_Mik_Prag} with a non-trivial bounded drift, i.e.\ getting rid of the
indicator function for the drift. They also obtain interior Schauder estimates and some boundary decay estimates (see e.g.\ Theorem $1.5$
therein).

As we have seen, most of the literature is focused on the non-degenerate case. In the degenerate diffusive setting,  Lunardi
\cite{Lunardi97} was the first one to prove Schauder estimates for linear Kolmogorov equations under weak H\"ormander assumptions,
exploiting anisotropic H\"older spaces (where the H\"older index depends on the variable considered), in order exactly to control the
multiple scales appearing in the different directions, due to the degeneracy of the system.\newline
After, in \cite{Lorenzi05} and \cite{Priola09}, the authors established Schauder-like estimates for hypoelliptic  Kolmogorov equations
driven by partially nonlinear smooth drifts.
On the other hand, let us also mention \cite{Chaudru:Honore:Menozzi18_Sharp} where the authors first establish Schauder estimates
for nonlinear Kolmogorov equations under some weak H\"ormander-type assumption. Their method is based on a perturbative approach through
proxies that we here adapt and exploit.
In the degenerate, stable setting, we have to refer also to a recent work of Zhang and collaborators \cite{Hao:Wu:Zhang19} who show Schauder
estimates for the degenerate kinetic  dynamics ($n=2$ above) extending a method based on Littlewood-Paley decompositions already used in
other \textcolor{black}{works by Zhang (see e.g. \cite{Zhang:Zhao18}),} to the degenerate, multi-scaled framework.  Even with different approaches and frameworks, we consider here a
generic $d$-level chain and we exploit thermic characterizations of Besov norms, our and their works bring to the same results in the
intersecting cases, at least to the best of our knowledge.
About a different but correlated argument, we mention that the L$^p$-maximal regularity for degenerate non-local Kolmogorov equations with
constant coefficients was also obtained in \cite{Chen:Zhang18} for the kinetic dynamics ($n=2$ above) and in \cite{Huang:Menozzi:Priola19}
for the general $n$-levels chain.

In the diffusive setting, Equation \eqref{Degenerate_Stable_PDE} appears naturally as a microscopic model for heat diffusion phenomena (see
\cite{Rey-Bellet:Thomas00}) or, in the kinetic case ($n=2$), it can be naturally associated with speed/position (or Hamiltonian)
dynamics where the speed component is noisy. It can be found in many fields of application from physics to finance, see for example
\cite{Herau:Nier04} or \cite{Barucci:Polidoro:Vespri01}. When noised by stable processes, it can be used to model the appearance
 of turbulence (cf.\cite{Cushman:Park:Kleinfelter:Moroni05}) or some abnormal diffusion phenomena. \newline
Moreover, the Schauder estimates will be a fundamental first step in order to study the weak and strong well-posedness for the following
stochastic differential equation (SDE):
\begin{equation}\label{SDE_Associated}
\begin{cases}
  d\bm{X}^1_t \, = \, \bm{F}_1(t,\bm{X}^1_t,\dots,\bm{X}^n_t)dt +dZ_t \\
  d\bm{X}^2_t \, = \, A_{2,1}\bm{X}^1_t+\bm{F}_2(t,\bm{X}^2_t,\dots,\bm{X}^n_t)dt \\
 \vdots \\
  d\bm{X}^n_t \, = \, A_{n,n-1}\bm{X}^{n-1}_t+\bm{F}_n(t,\bm{X}^n_t)dt
\end{cases}
\end{equation}
where $Z_t$ is a symmetric, $\R^d$-valued $\alpha$-stable process with non-degenerate L\'evy measure $\nu_\alpha$ on some filtered
probability space $(\Omega,(\mathcal{F}_t)_{t\ge0},\mathbb{P})$. The complete operator $L_\alpha + \langle A \bm{x} + \bm{F}(t,\bm{x}),
D_{\bm{x}}\rangle$ then corresponds to the infinitesimal generator of the process $\bigl( \bm{X}\bigr)_{t\ge 0}$, solution of Equation
\eqref{SDE_Associated}.

\paragraph{Mathematical Outline.} In this work, we will establish global Schauder estimates for the solution of the IPDE
\eqref{Degenerate_Stable_PDE} exploiting the perturbative approach firstly introduced in \cite{Chaudru:Honore:Menozzi18_Sharp} to derive
such estimates for degenerate Kolmogorov equations. Roughly speaking, the idea is to perform a first order parametrix expansion, such as a
Duhamel-type representation, to a solution of the IPDE \eqref{Degenerate_Stable_PDE} around a suitable proxy. The main idea behind consists
in exploiting this easier framework in order to
subsequently obtain a tractable control on the error expansion. When applying such a strategy, we basically have two ways to proceed.\newline
On the one hand, one can adopt a backward parametrix approach, as introduced by McKean and Singer \cite{Mckean:Singer67} in the
non-degenerate, diffusive setting. This technique has been extended to the degenerate Brownian case involving unbounded
perturbation, and successfully exploited for handling the corresponding martingale problem in \cite{Chaudru:Menozzi17_Weak}.
Anyway, this approach does not seem very adapted to our framework especially because it does not allow to deal easily with point-wise
gradient estimates which will, at least along the non-degenerate variable $\bm{x}_1$,  be fundamental to establish our result. \newline
On the other hand, the so-called forward parametrix approach has been successfully used by Friedman \cite{book:Friedman08} or Il'in et al.\ \cite{ilcprimein:Kalavsnikov:Oleuinik62} in the non-degenerate, diffusive
setting to obtain point-wise bounds on the fundamental solution and its derivatives for the corresponding heat-type equation or in
\cite{Chaudru17} to derive strong uniqueness for the associated SDE \eqref{SDE_Associated} (i.e. $n=2$ with the previous
notations). Especially, this approach is better tailored to exploit cancellation techniques that are crucial when derivatives come in,
 as opposed to the backward one.

The main difficulties to overcome in order to prove Schauder estimates in our framework will be linked to the degeneracy of the operator
$L_\alpha$ that acts only on the first $d$ variables, as well as the unboundedness of the perturbation $\bm{F}$. Concerning this second
issue, let us also mention that Schauder estimates for unbounded non-linear drift coefficients in the non-degenerate diffusive setting were
obtained under mild smoothness assumptions by Krylov and Priola \cite{Krylov:Priola10} who heavily used an auxiliary, deterministic flow associated with the
transport term in
\eqref{Degenerate_Stable_PDE}, i.e.\ for a fixed couple $(t,x)$,
\begin{equation}\label{eq:_INtro_def_flusso}
\begin{cases}
\partial_s\bm{\theta}_s(\bm{x}) =A\bm{\theta}_s(\bm{x})+\bm{F}(s,\bm{\theta}_s(\bm{x})); & \mbox{if } s>t \\
  \bm{\theta}_t(\bm{x}) \, = \, \bm{x},
\end{cases}
 \quad
\end{equation}
to precisely get rid of the unbounded terms.

%The perturbative approach is not usual to establish Schauder type estimates. The standard way is to proceed through a priori estimates to
%obtain for a given solution the expected bound. Existence and uniqueness issues, in the considered
%functional space, for the solution of the equation are addressed only in a second time. We can refer to \cite{krylov_Holder} for a clear
%presentation of this approach and to \cite{krylov_priola_10} for an extension of this method to non-degenerate operators with unbounded
%drift coefficients.
%
The drawback of this approach is that we will need at first to establish Schauder estimates in a small time interval.
This seems quite intuitive since the expansion along the chosen proxy on which the method relies is precisely designed for small times
because it requires that the original operator and the proxy are "close" enough in a suitable sense. To obtain the result for an arbitrary
but finite time, we will then iterate the reasoning, which is quite natural since Schauder estimates provide a sort of stability in the
considered functional space. We are therefore far from the optimal constants for the Schauder estimates established in the non-degenerate,
diffusive setting for time dependent coefficients by Krylov and Priola \cite{Krylov:Priola17}.

On the other hand, we want to establish the Schauder estimates in the sharpest possible H\"older setting for the coefficients of the IPDE
\eqref{Degenerate_Stable_PDE}. To do so, we will need to establish some subtle controls, in particular we have no true derivatives of the
coefficients. This is the reason why we will heavily rely on duality results on Besov spaces (see Section $4.1$ below, Chapter $3$ in
\cite{book:Lemarie-Rieusset02} or \cite{book:Triebel83} for a more complete survey of the argument). However, in contrast with the
non-degenerate case (cf. \cite{Chaudru:Menozzi:Priola19}), we will need to ask
for the perturbation $\bm{F}$ some additional regularity, represented by parameter $\gamma_i$ in assumption (\textbf{R}) below, on the
degenerate entries $\bm{F}_i$ $(i>1)$. This assumption seems quite natural if we think that, due to the  degenerate structure of the system
(cf. Section $2.2$ below), the more we descend on the chain, the lower the smoothing effect of $L_\alpha$ will be.
The additional smoothness on $\bm{F}$ can be then seen as the "price" to pay to re-equilibrate \textcolor{black}{the increasing time
singularities appearing along the chain.}

%%%%%%%%%%%%%%%%%%%%%%%%%%%%%%%%%%%%%%%%%%%%%%%%%%%%%%%%%%%%%%%%%%%%%%%%
\paragraph{Organization of the paper.} The article is organized as follows. We state our precise framework and give our main results in the
following Section $2$. Section $3$ is then dedicated to the perturbative approach which is the central argument to derive our estimates. In
particular, we obtain therein some Schauder estimates for drifted operators along the inhomogeneous flow $\bm{\theta}_{t,s}$ defined above
in \eqref{eq:_INtro_def_flusso}, as well as the key Duhamel representation for solutions. Since the arguments to show the Schauder estimates
will be quite long and involved, we postpone the proofs of these results in the next Sections $4$ and $5$. The existence results are then
established in Section $6$. In the last Section $7$, we are going to explain briefly how the perturbative approach presented before could be
applied with slight modifications to prove Schauder-type estimates for a class of completely non-linear, locally H\"older continuous drifts with an additional "diffusion" coefficient. \newline
Finally, the proof of some technical results concerning the stability properties of H\"older flows are postponed to the Appendix.

\setcounter{equation}{0}
\section{Setting and Main Results}
\subsection{Considered Operators}
The operator $L_\alpha$ we consider is the generator of a non-degenerate, symmetric, stable process and it acts only on the first $d$
coordinates of the system. More precisely, $L_\alpha$ can be represented for any sufficiently regular $\phi\colon [0,T]\times\R^{nd}\to \R$
as
\[L_\alpha\phi(t,\bm{x})\, := \, \text{p.v.}\int_{\R^d}\bigl[\phi(t,\bm{x}+B\bm{y})-\phi(t,\bm{x}) \bigr] \,\nu_\alpha(d\bm{y}) \,\, \text{ where }
\,\, B \, := \,
    \begin{bmatrix}
           I_{d\times d} \\
            0_{d\times d}\\
            \vdots\\
           0_{d\times d}
    \end{bmatrix}\]
and $\nu_\alpha$ is a symmetric, stable L\'evy measure on $\R^d$ of order $\alpha$ that we assume to be non-degenerate in a sense that we
are going to specify below. \newline
Passing to polar coordinates $y=\rho s$ where $(\rho,s) \in [0,\infty)\times \mathbb{S}^{d-1}$, it is well-known (see for example Chapter $3$ in \cite{book:Sato99}) that the stable L\'evy measure $\nu_\alpha$ can be decomposed as
\begin{equation}\label{eq:decomposition_measure}
\nu_\alpha(dy) \, :=\, \frac{d\rho\tilde{\mu}(ds)}{\rho^{1+\alpha}}
\end{equation}
where $\tilde{\mu}$ is a symmetric measure on $\mathbb{S}^{d-1}$ which represents the spherical part of $\nu_\alpha$.\newline
We remember now that the L\'evy symbol associated with $L_\alpha$ is defined through the Levy-Khitchine formula (see, for instance
\cite{book:Jacob05}) as:
\[\Psi(p) \, := \, \int_{\R^d}\bigl[e^{i p\cdot y}-1\bigr]\, \nu_\alpha(dy) \quad \text{ for any $p$ in $\R^d$},\]
where $"\cdot"$ represents the inner product on the smaller space $\R^d$. In the current symmetric setting, it can be rewritten (cf.
Theorem $14.10$ in \cite{book:Sato99}) as
\begin{equation}\label{def:Levy_Symbol_stable}
\Psi(p) \, = \, -\int_{\mathbb{S}^{d-1}}\vert p\cdot s \vert^\alpha \, \mu(ds)
\end{equation}
where $\mu=C_{\alpha,d}\tilde{\mu}$ is usually called the spherical measure associated with $\nu_\alpha$ . Following \cite{Kolokoltsov00}, we then say that
$\nu_\alpha$ is non-degenerate if the associated L\'evy symbol $\Psi$ is equivalent, up to some multiplicative constant, to $\vert p \vert^\alpha$.
More precisely, we suppose that $\mu$ is non-degenerate if
\begin{description}
  \item[(ND)] there exists a constant $\eta\ge 1$ such that for any $p$ in $\R^d$.
\begin{equation}\label{non-degeneracy_of_measure}
\eta^{-1}\vert p \vert^\alpha \, \le \, \int_{\mathbb{S}^{d-1}}\vert p\cdot s \vert^\alpha \, \mu(ds) \, \le\,\eta
\vert p \vert^\alpha
\end{equation}
\end{description}
It is important to remark that such a condition does not restrict our model too much. Indeed, there are many
different kind of spherical measures $\mu$ that are non-degenerate in the above sense, from the stable-like case, i.e.\ measures that are
absolutely continuous with respect to the Lebesgue measure on $\mathbb{S}^{d-1}$, to very singular ones such that the spherical measure induced by the sum of Dirac masses along the canonical directions:
\[\sum_{i=1}^{d}(\partial^2_{x_k})^{\alpha/2}.\]
We can introduce now the complete Ornstein-Uhlenbeck operator $L_{ou}$, defined for any sufficiently regular $\phi\colon \R^{nd}\to \R$ as
\begin{equation}\label{eq:def_of_OU_operator}
L^{ou}\phi(\bm{x}) \, := \, \langle A\bm{x}, D_{\bm{x}}\phi(\bm{x})\rangle +L_\alpha\phi(\bm{x})
\end{equation}
where $A$ is the matrix in $\R^{nd}\times\R^{nd}$ defined in Equation \eqref{eq:def_matrix_A}. We assume
that $A$ satisfies the following H\"ormander-like condition of non-degeneracy:
\begin{description}
  \item[(H)] $A_{i,i-1}$ is non-degenerate (i.e.\ it has full rank $d$) for any $i$ in $\llbracket2,n \rrbracket$.
\end{description}
Above, $\llbracket2,n \rrbracket$ denotes the set of all the integers in the interval. It is well known (see for example \cite{book:Sato99}) that under these assumptions, the operator $L^{ou}$ generates a convolution Markov
semigroup $\bigl(P^{ou}_t\bigr)_{t\ge 0}$ on $B_b(\R^{nd})$, the family of all the bounded and Borel measurable functions on $\R^{nd}$,
defined by
\[
\begin{cases}
  P^{ou}_t\phi(\bm{x}) \, = \, \int_{\R^{nd}}\phi(\bm{x}+\bm{y}) \, \mu_t(d\bm{y}) \, \, \text{ for }t>0, \\
  P^{ou}_0\phi(\bm{x}) \, = \, \phi(\bm{x}).
\end{cases}
\]
where $\bigl(\mu_t\bigr)_{t>0}$ is a family of Borel probability measures on $\R^{nd}$. In particular, the function
$P^{ou}_t\phi(x)$ provides the classical solution to the Cauchy problem

\begin{equation}\label{PDE_of_OU}
\begin{cases}
  \partial_tu(t,\bm{x})+L_\alpha u(t,\bm{x})+\langle A\bm{x}, D_{\bm{x}}u(t,\bm{x})\rangle \, = \, 0 \,\, \mbox{ on } (0,\infty)\times \R^{nd}, \\
 u(0,\bm{x}) \, = \, \phi(\bm{x}) \,\,  \mbox{ on } \R^{nd}.
\end{cases}
\end{equation}

Moving to the stochastic counterpart if necessary, it is readily derived from \cite{Priola:Zabczyk09} that the semigroup $(P^{ou}_t)_{t\ge0}$
admits a smooth density $p^{ou}(t,\cdot)$ with respect to the Lebesgue measure on $\R^{nd}$. Moreover,such a density $p^{ou}$ has the
following useful representation:
\begin{equation}\label{eq:Representation_of_p_ou}
  p^{ou}(t,\bm{x},\bm{y}) \, = \, \frac{1}{\det \mathbb{M}_t}p_S(t,\mathbb{M}^{-1}_t\bigl(e^{At}\bm{x}-\bm{y})\bigr)
\end{equation}
where $p_S$ is the density of $\bigl(S_t\bigr)_{t\ge 0}$, a stable process in $\R^{nd}$ whose L\'evy measure satisfies the assumption
(\textbf{ND}) above on $\R^{nd}$ and $\mathbb{M}_t$ is a diagonal matrix on $\R^{nd}\times\R^{nd}$ given by
\begin{equation}\label{eq:def_of_Mt}
\bigl[\mathbb{M}_t\bigr]_{i,j} \, := \,
\begin{cases}
  t^{i-1}I_{d\times d}, & \mbox{if } i=j \\
  0_{d\times d}, & \mbox{otherwise}.
\end{cases}
\end{equation}
We remark already that the appearance of the matrix $\mathbb{M}_t$ in Equation \eqref{eq:Representation_of_p_ou} and its particular
structure reflect the multi-scaled structure of the dynamics considered (cf. Paragraph below for a more precise explanation).
\newline
Moreover, the density $p_S$ shows a useful property we will call the smoothing effect since it will be fundamental to reduce the
singularities appearing when working with time integrals. Fixed $\gamma$ in $[0,\alpha)$, there exists a constant
$C:=C(\gamma)$ such that for any $l$ in $\llbracket0,3\rrbracket$,
\begin{equation}\label{Smoothing_effect_of_S}
\int_{\R^{nd}}\vert \bm{y} \vert^\gamma \vert D^l_{\bm{y}}p_S(t,\bm{y}) \vert \, d\bm{y} \, \le \, Ct^{\frac{\gamma-l}{\alpha}} \, \, \text{
for any }t>0.
\end{equation}
These results can be proven following the arguments of Proposition $2.3$ and Lemma $4.3$ in \cite{Huang:Menozzi:Priola19}. We will provide
however a complete proof in the Appendix for the sake of completeness.

\subsection{Intrinsic Time Scale and Associated H\"older spaces}
In this section, we are going to choose which is the most suitable functional space in which to state our Schauder
estimates.\newline
To answer this question, we need firstly to understand how the system typically behaves. We focus for the moment on the Ornstein-Uhlenbeck
case:
\[\bigl(\partial_t  + L^{ou}\bigr)u(t,\bm{x}) \, = \, -f(t,\bm{x}) \,\, \text{ on } (0,\infty)\times \R^{nd} \]
and search for a dilation operator $\delta_\lambda\colon (0,\infty)\times \R^{nd} \to (0,\infty)\times \R^{nd}$ that is invariant for the
considered dynamics, i.e.\ a dilation that transforms solutions of the above equation into other solutions of the same equation.\newline
Due to the structure of $A$ and the $\alpha$-stability of $\nu$, we can consider for any fixed $\lambda>0$, the following
\[ \delta_\lambda(t,\bm{x}) := (\lambda^\alpha t,\lambda\bm{x}_1,\lambda^{1+\alpha}\bm{x}_2,\dots,\lambda^{1+\alpha(n-1)}\bm{x}_n),\]
i.e.\ with a slight abuse of notation, $\bigl(\delta_\lambda(t,\bm{x})\bigr)_0:=\lambda^\alpha t$ and for any $i$ in $\llbracket
1,n\rrbracket$, $\bigl(\delta_\lambda(t,\bm{x})\bigr)_i := \lambda^{1+\alpha(i-1)}\bm{x}_i$. It then holds that
\[\bigl(\partial_t +L^{ou}\bigr) u = 0 \, \Longrightarrow \bigl(\partial_t +L^{ou} \bigr)(u \circ \delta_\lambda) = 0.\]
The previous reasoning suggests us to introduce a parabolic distance $d_P$ that is homogenous with respect to the dilation $\delta_\lambda$,
so that $d_P\bigl(\delta_\lambda(t,\bm{x});\delta_\lambda(s,\bm{x}')\bigr) = \lambda d_P\bigl((t,\bm{x});(s,\bm{x}')\bigr)$. Precisely,
following the notations in \cite{Huang:Menozzi:Priola19}, we set for any $s,t$ in $[0,T]$ and any $\bm{x},\bm{x}'$ in $\R^{nd}$,
\begin{equation}\label{Definition_distance_d_P}
d_P\bigl((t,\bm{x}),(s,\bm{x}')\bigr)  \, := \, \vert s-t\vert^\frac{1}{\alpha}+\sum_{j=1}^{n} \vert(\bm{x}-\bm{x}')_j\vert^{
\frac{1}{1+\alpha(j-1)}}.
\end{equation}
The idea of a dilation $\delta_\lambda$ that summarizes the multi-scaled behaviour of the dynamics was firstly introduced by
Lanconelli and Polidoro in \cite{Lanconelli:Polidoro94} for degenerate Kolmogorov equations in the diffusive setting. Since then, it has
become a "standard" tool in the analysis of degenerate equations (see for example \cite{Lunardi97}, \cite{Huang:Menozzi:Priola19}
or \cite{Hao:Wu:Zhang19}). \newline
Since we will quite always use only the spatial part of the distance $d_P$, we denote for simplicity
\begin{equation}\label{Definition_distance}
d(\bm{x},\bm{y}) \, = \, \sum_{j=1}^{n} \vert(\bm{x}-\bm{x}')_j\vert^{\frac{1}{1+\alpha(j-1)}}.
\end{equation}
Technically speaking, $d_P$ (and thus, $d$) does not however induce a norm on $[0,T]\times \R^{nd}$ in the usual sense since it lacks of linear homogeneity. \textcolor{black}{We remark anyhow again that for any $\lambda>0$, it precisely holds that $d\bigl(\delta_\lambda(t,\bm{x});\delta_\lambda(s,\bm{x}')\bigr) = \lambda d\bigl((t,\bm{x});(s,\bm{x}')\bigr)$.}
As it can be seen, $d_P$ is an extension of the standard parabolic distance in the stable case, adapted to respect the multi-scaled nature
of our dynamics. Indeed, the exponents appearing in \eqref{Definition_distance_d_P} are those which make each space component homogeneous to
the characteristic time scale $t^{1/\alpha}$.\newline
 The appearance of this kind of phenomena is due essentially by the particular structure of the
matrix $A$ (cf. Equation \eqref{Degenerate_Stable_PDE}) that allows the smoothing effect of $L_\alpha$, acting only on the first variable,
to propagate in the system, as it can be seen in the following lemma:
\begin{lemma}[Scaling Lemma]\label{lemma:Scaling_Lemma}
Let $i$ be in $\llbracket 1,n \rrbracket$. Then, there exist $\{C_j\}_{j \in \llbracket 1,n\rrbracket}$ positive constants, depending only
from $A$ and $i$, such that
\[D_{\bm{x}_i}p^{ou}(t,\bm{x},\bm{y}) \, = \, -\sum_{j=i}^{n}C_jt^{j-i}D_{\bm{y}_j}p^{ou}(t,\bm{x},\bm{y})\]
for any $t>0$ and any $\bm{x}$, $\bm{y}$ in $\R^{nd}$.
\end{lemma}
\begin{proof}
Recalling the representation of $p^{ou}$ in Equation \eqref{eq:Representation_of_p_ou}, it is easy to see that
\[D_{\bm{x}_i}p^{ou}(t,\bm{x},\bm{y}) \,=\, \frac{1}{\det \mathbb{M}_t}D_{\bm{z}}p_S(t,\cdot)\bigl(\mathbb{M}^{-1}_t(e^{At}\bm{x}- \bm{y})
\bigr)\mathbb{M}^{-1}_tD_{\bm{x}_i}\bigl[e^{At}\bm{x}-\bm{y}\bigr].\]
Hence, in order to conclude, we need to show that
\begin{equation}\label{Proof:Scaling_Lemma}
D_{\bm{x}_i}\bigl[e^{At}\bm{x}-\bm{y}\bigr] \, = \, -\sum_{j=i}^{n}C_jt^{j-i}D_{\bm{y}_j}\bigl[e^{At}\bm{x}-\bm{y}\bigr].
\end{equation}
To prove the above equality, we need to analyze more in depth the structure of the resolvent $e^{At}$. Recalling from Equation
\eqref{eq:def_matrix_A} that $A$ has a sub-diagonal structure, we notice that for any $i,j$ in $\llbracket 1,n\rrbracket$,
\begin{equation}\label{Proof:Scaling_Lemma1}
\Bigl[e^{At}\Bigr]_{i,j} \, = \,
    \begin{cases}
     C_{i,j} t^{j-i}, & \mbox{if } j\ge i;\\
     0, & \mbox{otherwise},
    \end{cases}
\end{equation}
for a family of constants $\{C_{i,j}\}_{i,j \in \llbracket 1,n\rrbracket}$ depending only from $A$. It then follows that for any
$\bm{x},\bm{y}$ in $\R^{nd}$, it holds that
\begin{equation}\label{Proof:Scaling_Lemma2}
\Bigl[e^{At}\bm{x}-\bm{y}\Bigr]_{i} \, = \, \sum_{k=1}^{i}C_{i,k}t^{i-k}\bm{x}_k-\bm{y}_i.
\end{equation}
Equation \eqref{Proof:Scaling_Lemma} then follows immediately. \textcolor{black}{For a more detailed proof of this result, see also
\cite{Huang:Menozzi15} or \cite{Huang:Menozzi:Priola19}.}

\end{proof}

We finally remark the link with the stochastic counterpart of equation \eqref{Degenerate_Stable_PDE}. From a more
probabilistic point of view, the exponents in equation \eqref{Definition_distance_d_P}, can be related to the characteristic
time scales of the iterated integrals of an $\alpha$-stable process.

We are now ready to define the suitable H\"older spaces for our estimates. We start recalling some useful notations we will need below.
Fixed $k$ in $\N\cup \{0\}$ and $\beta$ in $(0,1)$, we follow Krylov \cite{book:Krylov96_Holder}, denoting the usual \emph{homogeneous}
H\"older space $C^{k+\beta}(\R^d)$ as the family of functions $\phi\colon \R^d \to \R$ such that
\[\Vert \phi \Vert_{C^{k+\beta}} \,:=\, \sum_{i=1}^{k}\sup_{\vert\vartheta\vert = i}\Vert D^\vartheta\phi
\Vert_{L^\infty}+\sup_{\vert\vartheta\vert=k}[D^{\vartheta}\phi]_\beta \, < \, \infty\]
where
\[[D^{\vartheta}\phi]_\beta \, := \, \sup_{\bm{x}\neq \bm{y}}\frac{\vert D^{\vartheta}\phi(\bm{x}) -
D^{\vartheta}\phi(\bm{y}) \vert}{\vert\bm{x} - \bm{y}\vert^{\beta}}.\]
Additionally, we are going to need the associated subspace $C^{k+\beta}_b(\R^d)$ of bounded functions in $C^{k+\beta}(\R^d)$, equipped with
the norm
\[\Vert \cdot \Vert_{C^{k+\beta}_b} \, = \, \Vert \cdot \Vert_{L^\infty} + \Vert \cdot \Vert_{C^{k+\beta}}.\]
We can now define the anisotropic H\"older space with multi-index of regularity associated with the distance $d$.
For sake of brevity and readability, we firstly define for a function $\phi \colon \R^{nd} \to \R$, a point $z$ in $\R^{d(n-1)}$ and $i$ in $\llbracket 1,n\rrbracket$, the
function
\[\Pi_z^i \phi \colon x \in \R^d \to \phi(z_1,\dots,z_{i-1},x,z_{i+1},z_n)\]
with the obvious modifications if $i=1$ or $i=n$.
Intuitively speaking, the function $\Pi_z^i\phi$ is the restriction of $\phi$ on its $i$-th $d$-dimensional variable while fixing all the other coordinates
in $z$. The space $C^{k+\beta}_d(\R^{nd})$ is then defined as the family of all the function $\phi \colon \R^{nd} \to \R$ such that
\[\Vert \phi \Vert_{C^{k+\beta}_d} \, := \, \sum_{i=1}^{n}\sup_{z\in \R^{d(n-1)}}\Vert \Pi_z^i\phi(x)
\Vert_{C^{\frac{k+\beta}{1+\alpha(i-1)}}}.\]
The modification to the bounded subspace $C^{k+\beta}_{b,d}$ is straightforward.\newline
Roughly speaking, the anisotropic norm works component-wise, i.e.\ we firstly fix a coordinate and then calculate the standard H\"older norm along that particular direction, but with
index scaled according to the dilation of the system in that direction, uniformly over time and the other space components. We conclude
summing the contributions associated with each component.
\newline We highlight however that it is possible to recover the expected joint regularity for the partial derivatives, when
they exist. In such a case, they actually turn out to be H\"older continuous in the pseudo-metric $d$ with order one less than the function.
(cf. Lemma \ref{lemma:Reverse_Taylor_Expansion} in the Appendix for the case $i=1$).

Since we are working with evolution equations, the functions we consider will quite often depend on time, too. For this reason, we denote by
$L^\infty(0,T,C^{k+\beta}_d(\R^{nd}))$ (respectively, $L^\infty(0,T,C^{k+\beta}_{b,d}(\R^{nd}))$) the family of functions $\psi \colon
[0,T]\times \R^{nd} \to \R$ with finite $C^{k+\beta}_d$-norm (respectively, $C^{k+\beta}_{b,d}$-norm), uniformly in time.

\subsection{Assumptions and Main Results}

From this point further, we consider two fixed numbers $\alpha$ in $(0,2)$ and $\beta$ in $(0,1)$ such that $\alpha$ will represent the
index of stability of the operator $L_\alpha$ while $\beta$ will stand for the index of H\"older regularity of the coefficients.

From this point further, we assume the following:
\begin{description}
  \item[(S)] assumptions (\textbf{ND}) and (\textbf{H}) are satisfied and the drift $\bm{F}=(\bm{F}_1,\dots,\bm{F}_n)$ is such that for
      any $i$ in $\llbracket 1,n\rrbracket$, $\bm{F}_i$ depends only on time and on the last $n-(i-1)$ components, i.e.\
      $\bm{F}_i(t,\bm{x}_i,\dots,\bm{x}_n)$;
  \item[(P)] $\alpha$ is a number in $(0,2)$, $\beta$ is in $(0,1)$ such that $\alpha+\beta\in (1,2)$ and if $\alpha<1$ (super-critical
      case),
  \[\beta<\alpha, \quad  1-\alpha <\frac{\alpha-\beta}{1+\alpha(n-1)};\]
  \item[(R)] Recalling the notations in Section $2.2$, the source $f$ is in $L^\infty(0,T;C^{\beta}_{b,d}(\R^{nd}))$, the terminal
      condition $g$ is in $C^{\alpha+\beta}_{b,d}(\R^{nd})$ and for any $i$ in $\llbracket 1,n\rrbracket$, $\bm{F}_i$ belongs to
      $L^\infty(0,T;C^{\gamma_i+\beta}_{d}(\R^{nd}))$ where
    \begin{equation}\label{Drift_assumptions}
    \gamma_i \,:= \,
    \begin{cases}
        1+ \alpha(i-2), & \mbox{if } i>1; \\
        0, & \mbox{if } i=1.
   \end{cases}
   \end{equation}
\end{description}
From now on, we will say that assumption (\textbf{A}) holds when the above conditions (\textbf{S}), (\textbf{P}) and
(\textbf{R}) are in force.

\begin{remark}[About the Assumptions]
We remark that the constraints (\textbf{P}) we are imposing in the super-critical case ($\alpha<1$) seem quite natural for our system. The
condition $\beta<\alpha$ reflects essentially the low integrability properties of the stable density $p_S$ (cf. Equation
\eqref{Smoothing_effect_of_S}). Even if one is interested only on the fractional Laplacian case, i.e. $L_\alpha=\Delta^{\alpha/2}$, such a
condition cannot be dropped in general, since it does not refer to the integrability property of $p_\alpha$ and its derivatives but instead
\textcolor{black}{to those of }its "projection" $p_S$ on the bigger space $\R^{nd}$ (cf. Equation
\eqref{eq:Representation_of_p_ou}).\newline
About the second condition $\alpha+\beta>1$, it is necessary to give a point-wise definition of the gradient of a solution $u$ with respect
to the non-degenerate variable $x_1$. Moreover, there is a famous counterexample of Tanaka and his collaborators
\cite{Tanaka:Tsuchiya:Watanabe74_counterexample} that shows that even in the scalar case, weak uniqueness (a direct consequence of Schauder
estimates) may fail for the associated SDE if $\alpha+\beta$ is smaller than one. \newline
The last assumption is indeed a technical constraint and it is necessary to work properly with the perturbation $\bm{F}$ at any level
$i=1,\dots,n$. In particular, it seems the minimal threshold that allows us to exploit the smoothing effect of the density (see for example
Equation \eqref{Proof:Second_Besov_Control_I_2} in the proof of Lemma \ref{lemma:Second_Besov_COntrols} for more details). We conclude
highlighting that these assumptions are always fulfilled if $\alpha\ge 1$ (sub-critical case).
\end{remark}

At this stage, it should be clear that under our assumptions (\textbf{A}), the IPDE \eqref{Degenerate_Stable_PDE} will be understood in a
\emph{distributional} sense. Indeed, we cannot hope to find a "classical" solution for \eqref{Degenerate_Stable_PDE}, since for such a
function $u$ in $L^\infty(0,T;C^{\alpha+\beta}_{b,d}(\R^{nd}))$, the total gradient $D_{\bm{x}}u$ is not defined point-wise.\newline
Let us denote for any function $\phi\colon [0,T]\to\R^{nd}$ regular enough, the complete operator $\mathcal{L}_\alpha$ as
\begin{equation}\label{eq:Complete_Operator}
\mathcal{L}_\alpha\phi(t,\bm{x}) \, := \, \langle A \bm{x} + \bm{F}(t,\bm{x}), D_{\bm{x}}u(t,\bm{x})\rangle + L_\alpha u(t,\bm{x}).
\end{equation}
We will say that a function $u$ in $L^\infty(0,T;C^{\alpha+\beta}_{b,d}(\R^{nd}))$ is a distributional (or weak) solution of the IPDE
\eqref{Degenerate_Stable_PDE} if for any $\phi$ in $C^\infty_0((0,T]\times\R^{nd})$, it holds that
\begin{equation}\label{eq:Def_weak_sol}
\int_{0}^{T}\int_{\R^{nd}}\Bigl(-\partial_t+\mathcal{L}_\alpha^*\Bigr)\phi(t,\bm{y})u(t,\bm{y}) \, d\bm{y} +
\int_{\R^{nd}}g(\bm{y})\phi(T,\bm{y}) \, d\bm{y} \, = \, - \int_{0}^{T}\int_{\R^{nd}}\phi(t,\bm{y})f(t,\bm{y}) \, d\bm{y}
\end{equation}
where $\mathcal{L}^*_\alpha$ denotes the formal adjoint of $\mathcal{L}_\alpha$.
\textcolor{black}{On the other hand, denoting from now on,
\begin{equation}\label{eq:norm_H_for_F}
\Vert \bm{F}\Vert_H \, := \sup_{i\in\llbracket 1,n \rrbracket}\Vert \bm{F}_i\Vert_{L^\infty(C^{\gamma_i+\beta}_{d})},
\end{equation}
we will quite often use the following other notion of solution:}
\begin{definition}
\label{definition:mild_sol}
A function $u$ is a mild solution in $L^\infty\bigl(0,T;C^{\alpha+\beta}_{b,d}(\R^{nd})\bigr)$ of Equation \eqref{Degenerate_Stable_PDE} if
for any triple of sequences $\{f_m\}_{m\in \N}$, $\{g_m\}_{m\in \N}$ and $\{\bm{F}_m\}_{m\in \N}$ such that
\begin{itemize}
  \item $\{f_m\}_{m\in \N}$ is in $C^\infty_b((0,T)\times\R^{nd})$ and $f_m$ converges to $f$ in $L^\infty\bigl(0,T;C^\beta_{b,d}(\R^{nd})\bigr)$;
  \item $\{g_m\}_{m\in \N}$ is in $C^\infty_b(\R^{nd})$ and $g_m$ converges to $g$ in $C^{\alpha+\beta}_{b,d}(\R^{nd})$;
   \item $\{\bm{F}_m\}_{m\in \N}$ is in $C^\infty_b((0,T)\times\R^{nd};\R^{nd})$ and $\Vert\bm{F}_m-\bm{F}\Vert_H$ converges to $0$,
\end{itemize}
there exists a sub-sequence $\{u_{m}\}_{m\in \N}$ in $C^\infty_b\bigl((0,T)\times\R^{nd}\bigr)$ such that
\begin{itemize}
  \item $u_{m}$ converges to $u$ in $L^\infty\bigl(0,T;C^{\alpha+\beta}_{b,d}(\R^{nd})\bigr)$;
  \item for any fixed $m$ in $\N$, $u_{m}$ is a classical solution of the following "regularized" IPDE:
  \begin{equation}\label{Regularizied_PDE}
\begin{cases}
   \partial_t u_m(t,\bm{x})+L_\alpha u_m(t,\bm{x}) + \langle A \bm{x} + \bm{F}_m(t,\bm{x}), D_{\bm{x}}
   u_m(t,\bm{x})\rangle  \, = \, -f_m(t,\bm{x}) &\mbox{on }   (0,T)\times \R^{nd}, \\
    u_m(T,\bm{x}) \, = \, g_m(\bm{x}) & \mbox{on }\R^{nd}.
  \end{cases}
\end{equation}
\end{itemize}
\end{definition}

We can now state our main result:
\begin{theorem}(Schauder Estimates)
\label{theorem:Schauder_Estimates}
Let $u$ be a mild solution in $L^{\infty}\bigl(0,T;C^{\alpha+\beta}_{b,d}(\R^{nd})\bigr)$ of Equation \eqref{Degenerate_Stable_PDE}. Under
(\textbf{A}), there exists a constant $C:=C(T,(\bm{A}))$ such that
\begin{equation}\label{equation:Schauder_Estimates}
\Vert u \Vert_{L^\infty(C^{\alpha+\beta}_d)} \, \le \, C \bigl[\Vert f \Vert_{L^\infty(C^{\beta}_{b,d})} + \Vert g
\Vert_{C^{\alpha+\beta}_{b,d}}\bigr].
\end{equation}
\end{theorem}

Associated with an existence result we will exhibit in Section $6$, we will eventually derive the well-posedness for Equation
\eqref{Degenerate_Stable_PDE}.

\begin{theorem}
Under (\textbf{A}), there exists a unique mild solution $u$ in $L^{\infty}\bigl(0,T;C^{\alpha+\beta}_{b,d}(\R^{nd})\bigr)$ of the IPDE
\eqref{Degenerate_Stable_PDE}. Moreover, such a function $u$ is a weak solution, too.
\end{theorem}

In the following, we will denote for sake of brevity
\begin{equation}\label{eq:def_alpha_i_and_beta_i}
\alpha_i \, := \, \frac{\alpha}{1+\alpha(i-1)} \, \, \text{ and } \, \, \beta_i \, := \, \frac{\beta}{1+\alpha(i-1)}\,\,  \text{ for any
$i$ in }\llbracket 1,n\rrbracket.
\end{equation}
Clearly, these quantities were introduced to reflect exactly the relative scale of the system at every considered level $i$ (cf. Section
$2.2$ above).\newline
\textcolor{black}{In the following, as well as in Theorem \ref{theorem:Schauder_Estimates} above,} $C$ denotes a generic constant that may
change from line to line but depending only on the parameters in assumption (\textbf{A}). Other dependencies that may occur are explicitly
specified.

\setcounter{equation}{0}
\section{Proof through Perturbative Approach}
As already said in the Introduction, our method of proof relies on a perturbative approach introduced in
\cite{Chaudru:Honore:Menozzi18_Sharp} for the degenerate, Kolmogorov, diffusive setting.\newline
Roughly speaking, we will firstly choose a suitable proxy for the equation of interest, i.e.\ an operator whose associated semigroup and
density are known and that is close enough to the original one:
\[L_\alpha+\langle A \bm{x} +\bm{F}(t,\bm{x}), D_{\bm{x}}\rangle.\]
Furthermore, we will exhibit suitable regularization properties for the proxy and in particular, we will show that it satisfies the
Schauder estimates \eqref{equation:Schauder_Estimates}. This will be the purpose of the Sub-section $3.1$.\newline
In Sub-section $3.2$ below, we will then expand a solution $u$ of the IPDE \eqref{Degenerate_Stable_PDE} along the chosen proxy through a
Duhamel-type formula and eventually show that the expansion error only brings a negligible contribution so that the Schauder estimates still
holds for $u$. Due to our choice of method, this will be possible only adding some more assumptions on the system. Namely, we will assume in
addition to be in a small time interval, so that the proxy and the original operator do not differ too much. \newline
The last Sub-section $3.3$ will finally show how to remove the additional assumption in order to prove the Schauder
estimates (Theorem \ref{theorem:Schauder_Estimates}) through a scaling argument.

\subsection{Frozen Semigroup}
The crucial element in our approach consists in choosing wisely a suitable proxy operator along which to expand a solution $u$ in
$L^\infty(0,T;C^{\alpha+\beta}_{b,d}(\R^{nd}))$ of Equation \eqref{Degenerate_Stable_PDE}. In order to deal with potentially unbounded
perturbations $\bm{F}$, it is natural to use a proxy involving a non-zero first order term  associated with a flow
representing the dynamics driven by $A\bm{x}+\bm{F}$, the transport part of Equation \eqref{Degenerate_Stable_PDE} (see e.g.
\cite{Krylov:Priola10} or \cite{Chaudru:Menozzi:Priola19}).\newline
Remembering that we assume $\bm{F}$ to be H\"older continuous, we know that there exists a solution of
\[
\begin{cases}
  d \bm{\theta}_{\tau,s}(\bm{\xi}) \, = \, \bigl[A\bm{\theta}_{\tau,s}(\bm{\xi})+\bm{F}(s,\bm{\theta}_{\tau,s}(\bm{\xi}))\bigr]ds \mbox{ on
  } [\tau,T],\\
  \bm{\theta}_{\tau,\tau}(\bm{\xi}) \, = \, \bm{\xi},
\end{cases}
\]
even if it may be not unique. For this reason, we are going to choose one particular flow, denoted by
$\bm{\theta}_{\tau,s}(\bm{\xi})$, and consider it fixed throughout the work.\newline
More precisely, given a freezing couple $(\tau,\bm{\xi})$ in $[0,T]\times\R^{nd}$, the flow will be defined on $[\tau,T]$ as
\begin{equation}\label{Flow}
  \bm{\theta}_{\tau,s}(\bm{\xi}) \, = \, \bm{\xi} + \int_{\tau}^{s}\bigl[A\bm{\theta}_{\tau,v}(\bm{\xi})+\bm{F}(v,\bm{\theta}_{\tau,v}(\bm{\xi}))\bigr] \, dv.
\end{equation}
We can now introduce the "frozen" IPDE associated with the chosen proxy:
\begin{equation}\label{Frozen_PDE}
\begin{cases}
   \partial_t \tilde{u}^{\tau,\bm{\xi}}(t,\bm{x})+L_\alpha \tilde{u}^{\tau,\bm{\xi}}(t,\bm{x}) + \langle A \bm{x} +
   \bm{F}(t,\bm{\theta}_{\tau,t}(\bm{\xi})), D_{\bm{x}} \tilde{u}^{\tau,\bm{\xi}}(t,\bm{x})\rangle  \, = \, -f(t,\bm{x}) &
   \mbox{on }   (0,T)\times \R^{nd}, \\
    \tilde{u}^{\tau,\bm{\xi}}(T,\bm{x}) \, = \, g(\bm{x}) & \mbox{on }\R^{nd}.
  \end{cases}
\end{equation}
Remarking that the proxy operator $L_\alpha+\langle A \bm{x} + \bm{F}(t,\bm{\theta}_{\tau,t}(\bm{\xi})), D_{\bm{x}}\rangle$ can be seen as
an Ornstein-Ulhenbeck operator with an additional time-dependent component $\bm{F}(t,\bm{\theta}_{\tau,t}(\bm{\xi}))$,  it is clear that
under assumption (\textbf{A}), it generates a two parameters semigroups we will denote by $\bigl(\tilde{P}^{\tau,\bm{\xi}}_{t,s}\bigr)_{t\le
s}$. Moreover, it admits a density given by
\begin{equation}\label{eq:definition_tilde_p}
\tilde{p}^{\tau,\bm{\xi}}(t,s,x,y)\,  = \, \frac{1}{\det (\mathbb{M}_{s-t})} p_S\bigl( s-t,\mathbb{M}^{-1}_{s-t} (\bm{y} -
\tilde{\bm{m}}^{ \tau,\bm{\xi}}_{t,s}(\bm{x}))\bigr),
\end{equation}
remembering Equation \eqref{eq:Representation_of_p_ou} for the definition of $p_S$ and with the following notation for the "frozen shift"
$\tilde{\bm{m}}^{\tau,\bm{\xi}}_{t,s}(\bm{x})$:
\begin{equation}\label{eq:def_tilde_m}
\tilde{\bm{m}}^{\tau,\bm{\xi}}_{t,s}(\bm{x}) \, = \, e^{A(s-t)}\bm{x} + \int_{t}^{s}e^{A(s-v)} \bm{F}(v, \bm{\theta}_{\tau,v}
(\bm{\xi})) \, dv.
\end{equation}

We point out already the following important property of the shift $\tilde{\bm{m}}^{\tau,\bm{\xi}}_{t,s}(\bm{x})$:

\begin{lemma}%identifications theta m
\label{lemma:identification_theta_m}
Let $t<s$ in $[0,T]$ and $\bm{x}$ a point in $\R^{nd}$. Then,
\begin{equation}\label{eq:identification_theta_m}
\tilde{\bm{m}}^{\tau,\bm{\xi}}_{t,s}(\bm{x}) \, = \, \bm{\theta}_{\tau,s}(\bm{\xi}),
\end{equation}
taking $\tau=t$ and $\bm{\xi}=\bm{x}$.
\end{lemma}
\begin{proof}
We start noticing that by construction, $\tilde{\bm{m}}^{\tau,\bm{\xi}}_{t,s}(\bm{x})$ satisfies
\[\tilde{\bm{m}}^{\tau,\bm{\xi}}_{t,s}(\bm{x}) \, = \, \bm{x} +\int_{t}^{s} \bigl[A\tilde{\bm{m}}^{t,\bm{x}}_{t,v}(\bm{x})+\bm{F}(v, \bm{\theta}_{\tau,v} (\bm{\xi}))\bigr] \, dv.\]
It then holds that
\[\vert \tilde{\bm{m}}^{t,\bm{x}}_{t,s}(\bm{x}) - \bm{\theta}_{t,s}(\bm{x}) \vert \, \le \, \int_{t}^{s}A\vert
\tilde{\bm{m}}^{t,\bm{x}}_{t,v}(\bm{x})-\bm{\theta}_{t,v}(\bm{x}) \vert \, dv.\]
The above Equation \eqref{eq:identification_theta_m} then follows immediately applying the Gr\"onwall lemma.
\end{proof}

Moreover, we can extend the smoothing effect \eqref{Smoothing_effect_of_S} of $p_S$ to the frozen density $\tilde{p}^{\tau,\bm{\xi}}$
through the representation \eqref{eq:definition_tilde_p}:

\begin{lemma}[Smoothing effects of the frozen density]
\label{lemma:Smoothing_effect_frozen}
Let $\vartheta,\varrho$ be two multi-indexes in $\N^n$ such that $\vert \varrho+\vartheta\vert \le 3$ and $\gamma$ in
$[0,\alpha)$. Under (\textbf{A}), there exists a constant $C:=C(\vartheta,\varrho,\gamma)$ such that
\begin{equation}\label{eq:Smoothing_effects_of_tilde_p}
\int_{\R^{nd}} \vert D^\varrho_{\bm{y}}D^\vartheta_{\bm{x}}
\tilde{p}^{ \tau,\bm{\xi}} (t,s,\bm{x},\bm{y}) \vert d^\gamma\bigl(\bm{y},\tilde{\bm{m}}^{\tau,\bm{\xi}}_{\tau,s}(\bm{x})\bigr) \,
d\bm{y} \, \le \, C (s-t)^{\frac{\gamma}{\alpha}-\sum_{i=k}^{n} \frac{\vartheta_k+\varrho_k}{\alpha_k}}
\end{equation}
for any $t<s$ in $[0,T]$, any $\bm{x}$ in $\R^{nd}$ and any frozen couple $(\tau,\bm{\xi})$ in $[0,T]\times\R^{nd}$.
In particular, if $\vert \vartheta \vert \neq 0$, it holds for any $\phi$ in $C^\gamma_d(\R^{nd})$ that
\begin{equation}\label{eq:Control_of_semigroup}
\bigl{\vert} D^\vartheta_{\bm{x}}\tilde{P}^{\tau,\bm{\xi}}_{t,s}\phi(\bm{x}) \bigr{\vert} \, \le \, C \Vert\phi\Vert_{C^\gamma_d}
(s-t)^{\frac{\gamma}{\alpha}-\sum_{k=1}^{n}\frac{\vartheta_k}{\alpha_k}}.
\end{equation}
\end{lemma}
\begin{proof}
Let us start assuming that $\vert\vartheta\vert=1$ and $\vert\varrho\vert=1$ . The other cases can be treated in a similar way. \newline
Since $p_S$ is the density of an $\alpha$-stable process, we remember that the following $\alpha$-scaling property
\begin{equation}\label{eq:alpha_scaling_p_S}
p_S(t,\bm{y}) \, = \, t^{-\frac{nd}{\alpha}}p_S(1,t^{-\frac{1}{\alpha}}\bm{y})
\end{equation}
holds for any $t>0$ and any $\bm{y}$ in $\R^{nd}$. Fixed $i$ in $\llbracket1,n\rrbracket$, we then denote for simplicity
\[\mathbb{T}_{s-t}\, := \, (s-t)^{\frac{1}{\alpha}}\mathbb{M}_{s-t}\]
and we calculate the derivative of $\tilde{p}^{\tau,\bm{\xi}}$ with respect to $\bm{x}_i$ through
\begin{multline*}
\vert D_{\bm{x}_i}\tilde{p}^{\tau,\bm{\xi}}(t,s,\bm{x},\bm{y}) \vert \, = \, \Bigl{\vert} \frac{1}{\det (\mathbb{M}_{s-t})} D_{\bm{x}_i} \bigl[p_S\bigl(s-t,\mathbb{M}^{-1}_{s-t} (\tilde{\bm{m}}^{\tau,\bm{\xi}}_{t,s}(\bm{x})-\bm{y})\bigr)] \Bigr{\vert} \\
= \, \Bigl{\vert} \frac{1}{\det(\mathbb{T}_{s-t})}D_{\bm{x}_i}\bigl[p_S\bigl(1,\mathbb{T}^{-1}_{s-t}(\tilde{\bm{m}}^{\tau,\bm{\xi}}_{t,s}
(\bm{x})-\bm{y})\bigr)] \Bigr{\vert} \\
= \,\Bigl{\vert} \frac{1}{\det (\mathbb{T}_{s-t})} D_{\bm{z}}p_S\bigl(1,\cdot\bigr)(\mathbb{T}^{-1}_{s-t}
(\tilde{\bm{m}}^{\tau,\bm{\xi}}_{t,s}(\bm{x})-\bm{y}))\mathbb{T}^{-1}_{s-t}D_{\bm{x}_i}(\tilde{\bm{m}}^{\tau,\bm{\xi}}_{t,s}
(\bm{x})) \Bigr{\vert}.
\end{multline*}
where in the second equality we exploited the $\alpha$-scaling property \eqref{eq:alpha_scaling_p_S}. From Equation \eqref{Proof:Scaling_Lemma1} in the Scaling Lemma \ref{lemma:Scaling_Lemma}, we notice now that
\[\bigl{\vert} \mathbb{T}^{-1}_{s-t}D_{\bm{x}_i}(\tilde{\bm{m}}^{\tau,\bm{\xi}}_{t,s} (\bm{x})) \bigr{\vert} \, = \, \bigl{\vert}
\mathbb{T}^{-1}_{s-t}D_{\bm{x}_i}\bigl(e^{A(t-s)}(\bm{x})\bigr) \bigr{\vert} \, = \,
(s-t)^{-\frac{1}{\alpha}}\sum_{k=i}^{n}C_k(s-t)^{-(k-1)}(s-t)^{k-i} \, \le \,  C(s-t)^{-\frac{1+\alpha(i-1)}{\alpha}}\]
and we use it to show that
\[\vert D_{\bm{x}_i}\tilde{p}^{\tau,\bm{\xi}}(t,s,\bm{x},\bm{y}) \vert \, \le \, C(s-t)^{-\frac{1+\alpha(i-1)}{\alpha}}\frac{1}{\det
(\mathbb{T}_{s-t})} \bigl{\vert}D_{\bm{z}}p_S\bigl(1,\cdot)(\mathbb{T}^{-1}_{s-t}
(\tilde{\bm{m}}^{\tau,\bm{\xi}}_{t,s}(\bm{x})-\bm{y})\bigr)\bigr{\vert}.\]
Similarly, if we fix $j$ in $\llbracket1,n\rrbracket$, it holds that
\[\vert D_{\bm{y}_j}D_{\bm{x}_i}\tilde{p}^{\tau,\bm{\xi}}(t,s,\bm{x},\bm{y}) \vert \, \le \, C(s-t)^{-\frac{1}{\alpha_i}-\frac{1}{\alpha_j}}
\frac{1}{\det (\mathbb{T}_{s-t})} \bigl{\vert}D^2_{\bm{z}}p_S\bigl(1,\cdot)(\mathbb{T}^{-1}_{s-t}(\tilde{\bm{m}}^{\tau,\bm{\xi}}_{t,s}
(\bm{x})-\bm{y})\bigr)\bigr{\vert}.\]
It is then easy to show by iteration of the same argument that
\begin{equation}\label{eq:derivative_frozen_density}
\vert D^\varrho_{\bm{y}}D^\vartheta_{\bm{x}}\tilde{p}^{\tau,\bm{\xi}}(t,s,\bm{x},\bm{y}) \vert \, \le \,
C(s-t)^{-\sum_{k=1}^{n}\frac{\varrho_k+\vartheta_k}{\alpha_k}}\frac{1}{\det (\mathbb{T}_{s-t})}
\bigl{\vert}D^{\vert \varrho+\vartheta\vert}_{\bm{z}}p_S\bigl(1,\cdot)(\mathbb{T}^{-1}_{s-t} (\tilde{\bm{m}}^{\tau,\bm{\xi}}_{t,s}
(\bm{x})-\bm{y})\bigr)\bigr{\vert}.
\end{equation}
Control \eqref{eq:Smoothing_effects_of_tilde_p} follows immediately from the analogous smoothing effect for $p_S$ (cf.\ Equation
\eqref{Smoothing_effect_of_S}) and the change of variables $\bm{z}=\mathbb{T}^{-1}_{s-t} (\tilde{\bm{m}}^{\tau,
\bm{\xi}}_{t,s} (\bm{x})-\bm{y})$. Indeed,
\begin{multline*}
\int_{\R^{nd}} \vert D^\varrho_{\bm{y}}D^\vartheta_{\bm{x}}\tilde{p}^{ \tau,\bm{\xi}} (t,s,\bm{x},\bm{y}) \vert
d^\gamma\bigl(\bm{y},\tilde{\bm{m}}^{\tau,\bm{\xi}}_{\tau,s}(\bm{x})\bigr) \,
d\bm{y} \\
\le \, C(s-t)^{-\sum_{k=1}^{n}\frac{\varrho_k+\vartheta_k}{\alpha_k}}\int_{\R^{nd}}\frac{1}{\det(\mathbb{T}_{s-t})}\bigl{\vert}D^{\vert
\varrho+\vartheta\vert}_{\bm{z}}p_S\bigl(1,\cdot\bigr)(\mathbb{T}^{-1}_{s-t}(\tilde{\bm{m}}^{\tau,\bm{\xi}}_{t,s} (\bm{x})-\bm{y}))
\bigr{\vert} d^\gamma\bigl(\bm{y},\tilde{\bm{m}}^{\tau,\bm{\xi}}_{t,s}(\bm{x})\bigr) \, d\bm{y} \\
= \,(s-t)^{-\sum_{k=1}^{n}\frac{\varrho_k+\vartheta_k}{\alpha_k}}\int_{\R^{nd}}\bigl{\vert}D^{\vert \varrho+\vartheta\vert}_{\bm{z}}p_S(1,
\bm{z})\bigr{\vert}d^\gamma\bigl(\mathbb{T}_{s-t}(\bm{z})+\tilde{\bm{m}}^{\tau,\bm{\xi}}_{t,s}(\bm{x}),\tilde{\bm{m}}^{\tau,
\bm{\xi}}_{t,s}(\bm{x})\bigr) \,d\bm{y}
\end{multline*}
To conclude, we notice that
\[d^\gamma\bigl(\mathbb{T}_{s-t}(\bm{z})+\tilde{\bm{m}}^{\tau,\bm{\xi}}_{t,s}(\bm{x}),\tilde{\bm{m}}^{\tau,\bm{\xi}}_{t,s}(\bm{x})\bigr) \,
\le \, C\sum_{i=1}^{n}\vert (s-t)^{\frac{1+\alpha(i-1)}{\alpha}}\bm{z}_i\vert^{\frac{\gamma}{1+\alpha(i-1)}} \, = \,
(s-t)^{\frac{\gamma}{\alpha}}\sum_{i=1}^{n}\vert \bm{z}_i\vert^{\frac{\gamma}{1+\alpha(i-1)}}\]
and use it to write that
\begin{multline*}
\int_{\R^{nd}} \bigl{\vert} D^\varrho_{\bm{y}}D^\vartheta_{\bm{x}}\tilde{p}^{\tau,\bm{\xi}} (t,s,\bm{x},\bm{y}) \bigr{\vert}
d^\gamma\bigl(\bm{y},\tilde{\bm{m}}^{\tau,\bm{\xi}}_{\tau,s}(\bm{x})\bigr) \, d\bm{y} \\
\le \,C(s-t)^{\frac{\gamma}{\alpha}-\sum_{k=1}^{n}\frac{\varrho_k+\vartheta_k}{\alpha_k}}\sum_{i=1}^{n}\int_{\R^{nd}}\bigl{\vert}D^{\vert
\varrho+\vartheta\vert}_{\bm{z}}p_S(1,z)\bigr{\vert}\vert \bm{z}_i\vert^{\frac{\gamma}{1+\alpha(i-1)}} \, d\bm{y}\, \le \, C(s-t)^{\frac{
\gamma}{\alpha}- \sum_{k=1}^{n}\frac{\varrho_k+\vartheta_k}{\alpha_k}}
\end{multline*}
where in the last passage we used the smoothing effect for $p_S$ (Equation \eqref{Smoothing_effect_of_S}), recalling that for any $i$ in
$\llbracket 1,n\rrbracket$, it holds that
\[\frac{\gamma}{1+\alpha(i-1)} \, \le \, \gamma \, <\, \alpha\]
and we have thus the required integrability.

To prove instead the second inequality \eqref{eq:Control_of_semigroup}, we use a cancellation argument to write
\begin{multline*}
\bigl{\vert} D^\vartheta_{\bm{x}}\tilde{P}^{\tau,\bm{\xi}}_{t,s}\phi(\bm{x}) \bigr{\vert} \, = \, \Bigl{\vert}\int_{\R^{nd}}
D^\vartheta_{\bm{x}}\tilde{p}^{\tau,\bm{\xi}}(t,s,\bm{x},\bm{y}) \bigl[\phi(\bm{y})- \phi(\tilde{\bm{m}}^{\tau,\bm{\xi}}_{t,s}
(\bm{x}))\bigr] \, d\bm{y}\Bigr{\vert} \\
\le \, \int_{\R^{nd}} \vert D^\vartheta_{\bm{x}}\tilde{p}^{\tau,\bm{\xi}}(t,s,\bm{x},\bm{y})\vert\, \vert
\phi(\bm{y})- \phi(\tilde{\bm{m}}^{\tau,\bm{\xi}}_{t,s} (\bm{x}))\vert \, d\bm{y}.
\end{multline*}
But since we assume $\phi$ to be in $C^\gamma_d(\R^{nd})$, we can control the last expression as
\[\bigl{\vert} D^\vartheta_{\bm{x}}\tilde{P}^{\tau,\bm{\xi}}_{t,s}\phi(\bm{x}) \bigr{\vert} \, \le \, \Vert
\phi\Vert_{C^\gamma_d}\int_{\R^{nd}} d^\gamma\bigl(\bm{y}, \tilde{\bm{m}}^{\tau,\bm{\xi}}_{t,s} (\bm{x})\bigr) \vert
D^\vartheta_{\bm{x}}\tilde{p}^{\tau,\bm{\xi}}(t,s,\bm{x},\bm{y})\vert  \, d\bm{y} \, \le \, C \Vert\phi\Vert_{C^\gamma_d}
(s-t)^{\frac{\gamma}{\alpha}-\sum_{k=1}^{n}\frac{\vartheta_k}{\alpha_k}}\]
where in the last passage we used Equation \eqref{eq:Smoothing_effects_of_tilde_p}.
\end{proof}

We can define now our candidate to be the mild solution of the "frozen" IPDE. If it exists and it is
smooth enough, such a candidate appears to be the representation of the solution of \eqref{Frozen_PDE} obtained through the Duhamel
principle. For this reason, the following expression:
\begin{equation}\label{Duhamel_representation_of_proxy}
\tilde{u}^{\tau,\bm{\xi}}(t,x) \, := \, \tilde{P}^{\tau,\bm{\xi}}_{t,T}g(\bm{x}) + \int_{t}^{T}\tilde{P}^{\tau,\bm{\xi}}_{t,s}f(s,\bm{x}) \,
ds \quad \text{ for any $(t,\bm{x})$ in }[0,T]\times \R^{nd},
\end{equation}
will be called the Duhamel representation of the proxy. As it seems, under our assumption (\textbf{A}) such a representation is robust
enough to satisfy Schauder estimates similar to \eqref{equation:Schauder_Estimates}. Since the proof of this result is quite long, we will
postpone it to Section $4.2$ for clarity.

\begin{prop}(Schauder Estimates for the Proxy)
\label{prop:Schauder_Estimates_for_proxy}
Under (\textbf{A}), there exists a constant $C:=C(T)$ such that
\begin{equation}\label{equation:Schauder_Estimates_for_proxy}
\Vert \tilde{u}^{\tau,\bm{\xi}} \Vert_{L^\infty(C^{\alpha+\beta}_{b,d})} \, \le \, C\bigl[\Vert g \Vert_{C^{\alpha+\beta}_{b,d}} + \Vert f
\Vert_{L^\infty(C^{\beta}_{b,d})}\bigr]
\end{equation}
for any freezing couple $(\tau,\bm{\xi})$ in $[0,T]\times\R^{nd}$.
\end{prop}

We conclude this section showing that $\tilde{u}^{\tau,\bm{\xi}}$ is indeed a mild solution in
$L^\infty(0,T;C^{\alpha+\beta}_{b,d}(\R^{nd}))$ of the "frozen" IPDE \eqref{Frozen_PDE}.
Moreover, the converse statement is also true. If regular enough, any solution of \eqref{Frozen_PDE} corresponds to the Duhamel
representation \eqref{Duhamel_representation_of_proxy}.

\begin{prop}
\label{prop:frozen_Duhamel_Formula}
Let us assume to be under assumption (\textbf{A}). Then,
\begin{itemize}
  \item the function $\tilde{u}^{\tau,\bm{\xi}}$ defined in \eqref{Duhamel_representation_of_proxy} is a mild solution
      in $L^\infty(0,T;C^{\alpha+\beta}_{b,d}(\R^{nd}))$ of the "frozen" IPDE \eqref{Frozen_PDE} for any freezing couple $(\tau,\bm{\xi})$
      in $[0,T]\times \R^{nd}$;
  \item Fixed a freezing couple $(\tau,\bm{\xi})$ in $[0,T]\times \R^{nd}$, let $\tilde{v}^{\tau,\bm{\xi}}$ be a mild solution in
      $L^\infty(0,T;C^{\alpha+\beta}_{b,d}(\R^{nd}))$ of the IPDE \eqref{Frozen_PDE}. Then,
    \[\tilde{v}^{\tau,\bm{\xi}}(t,x) \, = \, \tilde{P}^{\tau,\bm{\xi}}_{t,T}g(\bm{x}) +  \int_{t}^{T}\tilde{P}^{\tau,\bm{\xi}}_{t,s}f
    (s,\bm{x}) \, ds.\]
\end{itemize}
\end{prop}
\begin{proof}
The first assertion is quite straightforward. Let $\{f_m\}_{m\in \N}$, $\{g_m\}_{m\in \N}$ and $\{\bm{F}_m\}_{m\in \N}$ be three sequences of
smooth and bounded coefficients such that $f_m \to f$ in $L^\infty\bigl(0,T;C^\beta_{b,d}(\R^{nd})\bigr)$, $g_m \to g$ in
$C^{\alpha+\beta}_{b,d}(\R^{nd})$ and $\Vert \bm{F}_m-\bm{F}\Vert_H \to 0$. Denoting now by $\Bigl(\tilde{P}^{m,\tau, \bm{\xi}}_{t,s}
\Bigr)_{t\le s}$ the semigroup associated with the "regularized" operator
\[L_\alpha +\langle A\bm{x}+\bm{F}_m(t,\bm{\theta}_{\tau,t}(\bm{\xi})),D_{\bm{x}}\rangle,\]
it is not difficult to show that for any fixed $m$ in $\N$, the following
\[\tilde{u}^{\tau,\bm{\xi}}_m \, := \, \tilde{P}^{m,\tau,\bm{\xi}}_{t,T}g_m(\bm{x})+\int_{t}^{T}\tilde{P}^{m,\tau,\bm{\xi}}_{t,s}f_m
(s,\bm{x} )\, ds\]
is a classical solution of the "frozen" IPDE \eqref{Frozen_PDE} with regularized coefficients $f_m,g_m$ and $\bm{F}_m$.
A detailed guide of this result can be found, even if in the diffusive setting, in Lemma $3.3$ in \cite{Krylov:Priola10}. Using now the
Schauder Estimates \eqref{equation:Schauder_Estimates_for_proxy} for the regularized solutions $\tilde{u}^{\tau,\bm{\xi}}_m$, it follows
immediately that $\tilde{u}^{\tau,\bm{\xi}}_m \to \tilde{u}^{\tau,\bm{\xi}}$ in $L^\infty\bigl(0,T;C^{\alpha+\beta}_{b,d}(\R^{nd})\bigr)$
and thus, that
$\tilde{u}^{\tau,\bm{\xi}}$ is a mild solution of \eqref{Frozen_PDE} in $L^\infty\bigl(0,T;C^{\alpha+\beta}_{b,d}(\R^{nd})\bigr)$.

To prove the second statement, we start fixing a freezing couple $(\tau,\bm{\xi})$ in $[0,T]\times \R^{nd}$ and consider three sequences
$\{f_m\}_{m\in \N}$, $\{g_m\}_{m\in \N}$ and $\{\bm{F}_m\}_{m\in \N}$ of bounded and smooth coefficients such that $f_m \to f$ in
$L^\infty\bigl(0,T;C^\beta_{b,d}(\R^{nd})\bigr)$, $g_m \to g$ in $C^{\alpha+\beta}_{b,d}(\R^{nd})\bigr)$ and $\Vert \bm{F}_m-\bm{F}\Vert_H
\to 0$. They can be constructed through mollification.\newline
Since $\tilde{v}^{\tau,\bm{\xi}}$ is a mild solution of the "frozen" IPDE \eqref{Frozen_PDE}, we
know that there exists a sequence $\{\tilde{v}^{\tau,\bm{\xi}}_m\}_{m \in \N}$ of classical solutions of the "regularized frozen" IPDE
\eqref{Frozen_PDE} with coefficients $f_m,g_m$ and $\bm{F}_m$ such that $\tilde{v}^{\tau,\bm{\xi}}_m \to \tilde{v}^{\tau,\bm{\xi}}$ in $L^\infty\bigl(0,T;C^{\alpha+\beta}_{b,d}(\R^{nd})\bigr)$. Fixed
$m$ in $\N$, we then denote
\[h_m(t,\bm{x}) \, : = \, \tilde{v}^{\tau,\bm{\xi}}_m\bigl(t,\bm{x}-\int_{t}^{T}e^{A(t-s)}\bm{F}_m(s, \bm{\theta}_{\tau,s}(\bm{\xi}))\,
ds\bigr)\]
for any $t$ in $[0,T]$ and any $x$ in $\R^{nd}$. Direct calculations imply that
\begin{gather*}
D_{\bm{x}} h_m(t,\bm{x}) \, = \, D_{\bm{x}} \tilde{v}^{\tau,\bm{\xi}}_m\bigl(t,\bm{x}-\int_{t}^{T}e^{A(t-s)}\bm{F}_m(s,
\bm{\theta}_{\tau,s}(\bm{\xi})) \,
ds\bigr) \\
L_\alpha h_m(t,\bm{x}) \, = \, L_\alpha \tilde{v}^{\tau,\bm{\xi}}_m\bigl(t,\bm{x}-\int_{t}^{T}e^{A(t-s)}\bm{F}_m(s,
\bm{\theta}_{\tau,s}(\bm{\xi})) \, ds\bigr)
\end{gather*}
and
\begin{multline*}
  \partial_t h_m(t,\bm{x}) \, = \, \partial_t \tilde{v}^{\tau,\bm{\xi}}_m\bigl(t,\bm{x}-\int_{t}^{T}e^{A(t-s)}\bm{F}_m(s,
  \bm{\theta}_{\tau,s}(\bm{\xi})) \, ds\bigr)\\
  + \bigl{\langle} \bm{F}_m(t,\bm{\theta}_{\tau,t}(\bm{\xi}))-A\int_{t}^{T}e^{A(t-s)}\bm{F}_m(s, \bm{\theta}_{\tau,s}(\bm{\xi})) \, ds,
  D_{\bm{x}}\tilde{v}^{\tau,\bm{\xi}}_m\bigl(t,\bm{x}-\int_{t}^{T}e^{A(t-s)}\bm{F}_m(s,\bm{\theta}_{\tau,s}(\bm{\xi}))\,ds\bigr)\bigr{\rangle}.
\end{multline*}
Remembering that $\tilde{v}^{\tau,\bm{\xi}}_m$ is a classical solution of Equation \eqref{Frozen_PDE} replacing therein $f$, $g$ and
$\bm{F}$ with coefficients $f_m,g_m$ and $\bm{F}_m$, it follows immediately that the function $h_m$ solves for any $m$ in $\N$ the following:
\begin{equation}\label{Frozen_PDE_No_Perturb}
\begin{cases}
   \partial_t h_m(t,\bm{x})+L_\alpha h_m(t,\bm{x}) + \langle A \bm{x}, D_{\bm{x}} h_m(t,\bm{x})\rangle  \, = \, -l_m(t,\bm{x}),\\
    h_m(T,\bm{x}) \, = \, g_m(\bm{x})
  \end{cases}
\end{equation}
where $l_m(t,\bm{x}):=f_m\bigl(t,\bm{x}-\int_{t}^{T}e^{A(t-s)}\bm{F}_m(s, \bm{\theta}_{\tau,s}(\bm{\xi})) \, ds\bigr)$.\newline
Since we are going to exploit reasonings in Fourier spaces, we need however to have
integrability properties on the solution $h_m$. For this reason, we introduce now a family $\{\rho_R\}_{R>0}$ of smooth functions such that
any $\rho_R$ is equal to $1$ in
$B(0,R)$ and vanishes outside $B(0,R+1)$. We then denote for any $R>0$,
\[h_{m,R}(t,\bm{x}) \, := \,h_m(t,\bm{x})\rho_R(\bm{x}).\]
It is then straightforward that $h_{m,R}$ solves
\begin{equation}\label{Frozen_PDE_No_Perturb2}
\begin{cases}
   \partial_t h_{m,R}(t,\bm{x})+L_\alpha h_{m,R}(t,\bm{x}) + \langle A \bm{x}, D_{\bm{x}} h_{m,R}(t,\bm{x})\rangle  \, = \, -\tilde{l}_{m,R}(t,\bm{x}),\\
    h_{m,R}(T,\bm{x}) \, = \, g_{m,R}(\bm{x})
  \end{cases}
\end{equation}
where $g_{m,R}(\bm{x})=g_m(\bm{x})\rho_R(\bm{x})$ and
\[\tilde{l}_{m,R}(t,\bm{x}) \, = \, \rho_R(\bm{x})l_m(t,\bm{x})+h_m(t,\bm{x})L_\alpha\rho_R(\bm{x})+ \int_{\R^d}\bigl[h_m(t,\bm{x}+By)-h_m(t,\bm{x})\bigr]
\bigl[\rho_R(\bm{x}+By)-\rho_R(\bm{x}) \bigr]\, \nu_\alpha(dy).\]
Noticing now that $\tilde{l}_{m,R}$ is integrable with integrable Fourier transform, we can  apply the Fourier transform in space to
equation \eqref{Frozen_PDE_No_Perturb2} in order to write that
\[
\begin{cases}
   \partial_t \widehat{h}_{m,R}(t,\bm{p})+\mathcal{F}_x \bigl(\bigl[L_\alpha + \langle A \bm{x}, D_{\bm{x}}\rangle\bigr] h_{m,R}\bigr)(t,\bm{p})  \,
   =
   \, -\widehat{\tilde{l}}_{m,R}(t,\bm{p}), \\
    \widehat{h}_{m,R}(T,\bm{p}) \, = \, \widehat{g}_{m,R}(\bm{p}).
  \end{cases}
\]
We remember in particular that the above operator $L_\alpha + \langle A \bm{x}, D_{\bm{x}}\rangle$ has an associated L\'evy symbol
$\Psi^{ou}(\bm{p})$ and, following Section $3.3.2$ in \cite{book:Applebaum09}, it holds that
\[\mathcal{F}_x \bigl(\bigl[L_\alpha + \langle A \bm{x}, D_{\bm{x}}\rangle\bigr] h_{m,R}\bigr)(t,\bm{p}) \, = \,
\Psi^{ou}(\bm{p})\widehat{h}_{m,R}(t,\bm{p}).\]
We can then use it to show that $\widehat{h}_{m,R}$ is a classical solution of the following equation:
\[
\begin{cases}
   \partial_t \widehat{h}_{m,R}(t,\bm{p})+ \Psi^{ou}(\bm{p})\widehat{h}_{m,R}(t,\bm{p})  \, = \, -\widehat{\tilde{l}}_{m,R}(t,\bm{p}), \\
    \widehat{h}_{m,R}(T,\bm{p}) \, = \, \widehat{g}_{m,R}(\bm{p}).
  \end{cases}
\]
The above equation can be easily solved by integration in time, giving the following representation of $\widehat{h}_{m,R}(t,\bm{p})$:
\[\widehat{h}_{m,R}(t,\bm{p}) \, = \, e^{(T-t)\Psi^{ou}(\bm{p})}\widehat{g}_{m,R}(\bm{p}) + \int_{t}^{T} e^{(s-t)\Psi^{ou}(\bm{p})} \widehat{\tilde{l}}_{m,R} (s,\bm{p})\, ds.\]
In order to go back to $\tilde{v}^{\tau,\bm{\xi}}_m$, we apply now the inverse Fourier transform to write that
\[h_{m,R}(t,\bm{x}) \, = \, P^{ou}_{T-t}g_{m,R}(\bm{x}) + \int_{t}^{T}P^{ou}_{s-t}\tilde{l}_{m,R}(s,\bm{x}) \, ds,\]
remembering that $\bigl(P^{ou}_t\bigr)_{t\ge 0}$ is the convolution Markov semigroup associated with the Ornstein-Uhlenbeck operator
$L_\alpha + \langle A \bm{x}, D_{\bm{x}}\rangle $. Letting $m$ go to $\infty$, it then follows immediately that $g_{m,R}\to g_m$,
$h_{m,R}\to h_{m}$ and $\tilde{l}_{m,R} \to l_m$. A change of variable allows us to show the Duhamel representation, at least in the
regularized setting:
\[\tilde{v}^{\tau,\bm{\xi}}_m(t,\bm{y})  = \, P^{ou}_{T-t}g_m\Bigl(\bm{y}+\int_{t}^{T}e^{A(t-s)}\bm{F}_m(s,\bm{\theta}_{\tau,s}(\bm{\xi}))\,
ds\Bigr) + \int_{t}^{T}P^{ou}_{s-t}f_m\Bigl(s,\bm{y}+\int_{t}^{s}e^{A(t-u)}\bm{F}_m(u, \bm{\theta}_{\tau,u}(\bm{\xi}))\, du\Bigr) \, ds.\]
Letting $m$ goes to zero and remembering that $\tilde{v}^{\tau,\bm{\xi}}_m \to \tilde{v}^{\tau,\bm{\xi}}$, $f_m \to f$ , $g_m \to g$
and $\bm{F}_m \to \bm{F}$ in the right functional spaces, we can conclude that
$\tilde{v}^{\tau,\bm{\xi}}=\tilde{u}^{\tau,\bm{\xi}}$.
\end{proof}

\subsection{Expansion along the Proxy}

We are going to use now the "frozen" IPDE \eqref{Frozen_PDE} in order to derive appropriate quantitative
controls of a solution $u$ of Equation \eqref{Degenerate_Stable_PDE}. Up to now, the freezing parameters $(\tau,\bm{\xi})$ were set free but
they will be later chosen appropriately depending on the control we aim to establish.\newline
The main idea is to exploit the Duhamel formula (Proposition \ref{prop:frozen_Duhamel_Formula}) for the proxy to expand any solution $u$ of
the original IPDE \eqref{Degenerate_Stable_PDE} along the proxy.
To make things more precise, let $u$ be a mild solution in $L\bigl(0,T;C^{\alpha+\beta}_{b,d}(\R^{nd})\bigr)$ of the IPDE
\eqref{Degenerate_Stable_PDE}. Mollifying if necessary, it is possible to construct three sequences $\{f_m\}_{m\in \N}$, $\{g_m\}_{m\in \N}$
and $\{\bm{F}_m\}_{m\in \N}$ of bounded and smooth functions with bounded derivatives such that $f_m \to f$ in
$L^\infty\bigl(0,T;C^\beta_{b,d}(\R^{nd})\bigr)$, $g_m \to g$ in $C^{\alpha+\beta}_{b,d}(\R^{nd})$ and
$\Vert \bm{F}_m-\bm{F}\Vert_H \to 0$. Since $u$ is a mild solution of \eqref{Degenerate_Stable_PDE}, we know that there exists a smooth
sequence $\{u_m\}_{m\in \N}$ \textcolor{black}{converging to} $u$ in $L\bigl(0,T;C^{\alpha+\beta}_{b,d}(\R^{nd})\bigr)$ \textcolor{black}{and such that for any fixed $m$ in $\N$,} $u_m$ solves in a classical sense the "regularized" IPDE \eqref{Regularizied_PDE}.\newline
Exploiting now that $\bm{F}_m$ is bounded and smooth, we can define the "regularized" flow $\bm{\theta}^m_{\tau,\cdot}
(\bm{\xi})$ as the \emph{unique} flow satisfying
\begin{equation}\label{eq:def_reg_flow}
\bm{\theta}^m_{\tau,t}(\bm{\xi}) \, = \, \bm{\xi} +
\int_{\tau}^{t}\bigl[A\bm{\theta}^m_{\tau,s}(\bm{\xi})+\bm{F}_m(s,\bm{\theta}^m_{\tau,s}(\bm{\xi}))\bigr] \, ds, \quad t \, \in \, [\tau,T].
\end{equation}
It is then easy to notice that $u_m$ is also a classical solution in $L\bigl(0,T;C^{\alpha+\beta}_{b,d}(\R^{nd})\bigr)$ of
\[\partial_t u_m(t,\bm{x}) +L_\alpha u_m(t,\bm{x})+ \langle A \bm{x}  +\bm{F}_m(t,\bm{\theta}^m_{\tau,t}(\bm{\xi})),
D_{\bm{x}}u_m(t,\bm{x})\rangle \, = \, -\bigl[f_m(t,\bm{x})+R_m^{\tau,\bm{\xi}}(s,\bm{x})\bigr] \]
on $(0,T)\times \R^{nd}$ with terminal condition $g_m$. Above, we have denoted
\begin{equation}\label{eq:def_remainder_regul}
R_m^{\tau,\bm{\xi}}(t,\bm{x})\,:=\, \langle \bm{F}_m(t,\bm{x})-\bm{F}_m(t,\bm{\theta}^m_{\tau,t}(\bm{\xi})),D_{\bm{x}}u_m(t,\bm{x})
\rangle.
\end{equation}

\textcolor{black}{Since clearly, $R^{\tau,\bm{\xi}}_m$ is in $L^\infty\bigl(0,T;C^{\alpha+\beta}_{b,d}(\R^{nd})\bigr)$,} we can use the Duhamel Formula (Proposition \ref{prop:frozen_Duhamel_Formula}) for the proxy to write that
\[
u_m(t,\bm{x}) \, = \,  \tilde{P}^{m,\tau,\bm{\xi}}_{t,T}g_m(\bm{x}) +  \int_{t}^{T}\tilde{P}^{m,\tau,\bm{\xi}}_{t,s}\bigl[f_m(s,\bm{x})+
R_m^{\tau,\bm{\xi}}(s,\bm{x})\bigr] \, ds, \quad (t,\bm{x}) \, \in \, (0,T)\times\R^{nd}
\]
where $\Bigl(\tilde{P}^{m,\tau,\bm{\xi}}_{t,s}\Bigr)_{t\le s}$ is the semigroup associated with the operator $L_\alpha+\langle
A\bm{x}+\bm{F}_m(t,\bm{\theta}^m_{\tau,t}(\bm{\xi})),D_{\bm{x}} \rangle$.

The reasoning above is summarized in the following Duhamel-type formula that allows to expand any classical solution $u_m$ of the
"regularized" IPDE \eqref{Regularizied_PDE} along the "regularized frozen" proxy.

\begin{prop}[Duhamel Type Formula]
\label{prop:Expansion_along_proxy}
Let $(\tau,\bm{\xi})$ a freezing couple in $[0,T]\times\R^{nd}$. Under (\textbf{A}), any classical solution $u_m$ of the "regularized" IPDE \eqref{Regularizied_PDE} can be represented  as
\begin{equation}\label{eq:Expansion_along_proxy}
u_m(t,\bm{x}) \, = \, \tilde{u}^{\tau,\bm{\xi}}_m(t,\bm{x}) + \int_{t}^{T}  \tilde{P}^{m,\tau,\bm{\xi}}_{t,s}R^{m,\tau,\bm{\xi}}(s,\bm{x})\,
ds, \quad (t,\bm{x}) \, \in \, (0,T)\times\R^{nd}
\end{equation}
where $R_m^{\tau,\bm{\xi}}$ is as in \eqref{eq:def_remainder_regul} and $\tilde{u}^{\tau,\bm{\xi}}_m$ is defined through the Duhamel
representation \eqref{Duhamel_representation_of_proxy} with the "regularized" coefficients $f_m$, $g_m$.
\end{prop}

Thanks to the above representation (Equation \eqref{eq:Expansion_along_proxy}), we know that, since we have already shown the suitable control for the frozen solution $u^{\tau,\bm{\xi}}_m$ (namely, Proposition \ref{prop:Schauder_Estimates_for_proxy} with $f_m,g_m$), the main term which remains to be investigated in order to show the Schauder Estimates (Theorem \ref{theorem:Schauder_Estimates}) is the remainder
\begin{equation}\label{eq:qqq}
  \int_{t}^{T}\tilde{P}^{m,\tau,\bm{\xi}}_{t,s}R_m^{\tau,\bm{\xi}}(s,\bm{x}) \, ds,
\end{equation}
that represents exactly the error in the expansion along the proxy.

To be precise, we could have passed to the limit in Equation \eqref{eq:Expansion_along_proxy} in order to obtain a similar Duhamel-type
formula for a mild solution $u$ in $L^\infty\bigl([0,T];C^{\alpha+\beta}_{b,d}(\R^{nd})\bigr)$. However, a problem appears when trying to
give a precise meaning at the limit for the remainder contribution \eqref{eq:qqq}. We already know that the limit exists point-wise by
difference, but for our approach to work, we need to establish precise quantitative controls on this term. Such estimates could be obtained
through duality techniques in Besov spaces (cf. Section $5.1$) but only at the expense of fixing
already the freezing couple as $(\tau,\bm{\xi})=(t,\bm{x})$. The drawback of this method is that it does not allow to differentiate Equation
\eqref{eq:Expansion_along_proxy}, which is needed to estimate $D_{\bm{x}_1}u$.

In order to show the suitable estimates for \eqref{eq:qqq}, we will need at first an additional constraint on the behaviour of
the system. In particular, we will say to be under assumption (\textbf{A'}) when assumption (\textbf{A}) is considered and if moreover,
\begin{description}
  \item[(ST)] we assume to be in a small time interval, i.e.\ $T\le 1$.
\end{description}

Under these stronger assumptions, we will then be  able to show in Section $5$ below that the following control holds:

\begin{prop}[A Priori Estimates]
\label{prop:A_Priori_Estimates}
Let $u$ be a mild solution in $L^\infty\bigl([0,T];C^{\alpha+\beta}_{b,d}(\R^{nd})\bigr)$ of IPDE \eqref{Degenerate_Stable_PDE}. Under
(\textbf{A'}), there exists a constant $C\ge 1$ such that
\begin{equation}\label{eq:A_Priori_Estimates}
 \Vert u \Vert_{L^\infty(C^{\alpha+\beta}_{b,d})} \, \le \, Cc_0^{\frac{\beta-\gamma_n}{\alpha}}\bigl[\Vert g
\Vert_{C^{\alpha+\beta}_{b,d}}+\Vert f \Vert_{L^\infty(C^{\beta}_{b,d})}\bigr]+C\bigl(c_0^{\frac{\beta-\gamma_n}{\alpha}}\Vert
\bm{F}\Vert_H + c_0^{\frac{\alpha+\beta-1}{1+\alpha(n-1)}}\bigr)\Vert u \Vert_{L^\infty(C^{\alpha+\beta}_{b,d})}
\end{equation}
where $c_0 \in (0,1)$ is assumed to be fixed but chosen later.
\end{prop}

\textcolor{black}{
We remark already that in the above control, the constants multiplying $\Vert u \Vert_{L^\infty(C^{\alpha+\beta}_{b,d})}$ have to be small
if one wants to derive the expected Schauder estimates. If $c_0$ is small enough, then $Cc_0^{\frac{\alpha+\beta-1}{1+\alpha(n-1)}}$ can be
made smaller than $1/4$. Anyhow, for this chosen small $c_0$, the quantity $c_0^{\frac{\beta-\gamma_n}{\alpha}}$ becomes large and therefore, it needs to be balanced with $C\Vert \bm{F}\Vert_H$. Namely, we can conclude if for instance, $Cc_0^{\frac{\beta-\gamma_n}{\alpha}}\Vert \bm{F}\Vert_H<1/4$ that implies in particular that $\Vert \bm{F}\Vert_H$ has to be small with respect to $c_0$.}bbbbbbbbb

\subsection{Conclusion of Proof}
In the first part of this section, we prove the Schauder estimates (Theorem \ref{theorem:Schauder_Estimates}) from the A Priori estimates (Proposition \ref{prop:A_Priori_Estimates}) through a suitable scaling procedure. Roughly speaking, the idea is to start from a general dynamics and then use the scaling procedure to make the H\"older norm $\Vert \bm{F} \Vert_H$ small enough in order to make a \emph{circular} argument work. Again, if $c_0$ and $\Vert \bm{F}\Vert_H$ are small enough in \eqref{eq:A_Priori_Estimates}, the
$L^\infty\bigl(0,T;C^{\alpha+\beta}_{b,d}(\R^{nd})\bigr)$-norm of $u$ on the right-hand side of \eqref{eq:A_Priori_Estimates} can be absorbed by the left-hand one. Once the Schauder estimates \eqref{equation:Schauder_Estimates} holds in the scaled dynamics, we will conclude going back to the original IPDE through the inverse scaling procedure, even if for a small final time horizon $T$. \newline
The second part of the section focuses on showing how to drop the additional assumption (\textbf{A'}). The key point here is to proceed through iteration up to an arbitrary, but finite, given time $T$ thanks to the stability of a solution $u$ in the space $L^{\infty}([0,T],C^{\alpha+\beta}_{b,d}(\R^{nd}))$.
\subsubsection{Scaling Argument}
Under (\textbf{A}), we start considering a mild solution $u$ of the IPDE \eqref{Degenerate_Stable_PDE} on $[0,T]$ for some final time $T\le
1$ to be fixed later. For a scaling parameter $\lambda$ in $(0,1]$ to be chosen later, we would like to analyze the IPDE
\eqref{Degenerate_Stable_PDE} under the change of variables
\begin{equation}\label{eq:change_of_variable}
(t,\bm{x}) \, \mapsto \, (\lambda t,\mathbb{T}_\lambda\bm{x})
\end{equation}
where $\mathbb{T}_\lambda:=\lambda^{1/\alpha}\mathbb{M}_\lambda$. Again, the scaling is performed accordingly to the homogeneity induced by
the distance $d_P$ in \eqref{Definition_distance_d_P}.\newline
To this purpose, we firstly introduce the scaled solution $u_\lambda$ defined by
\[u_\lambda(t,\bm{x}) \,:= \, u(\lambda t,\mathbb{T}_\lambda\bm{x}).\]
It then follows immediately that this function $u_\lambda$ is a mild solution of
\[\begin{cases}
    \lambda^{-1}\partial_tu_\lambda(t,\bm{x}) +\lambda^{-1}L_\alpha u_\lambda+\bigl{\langle} A \mathbb{T}_\lambda \bm{x} +\bm{F}(\lambda
    t,\mathbb{T}_\lambda\bm{x}), \mathbb{T}^{-1}_\lambda
    D_{\bm{x}}u_\lambda(t,\bm{x})\bigr{\rangle}  \, = \, -f(\lambda t,\mathbb{T}_\lambda\bm{x}), &
    \mbox{on } (0,T_\lambda)\times \R^{nd}, \\
    u_\lambda(T_\lambda,\bm{x}) \, = \, g(\mathbb{T}_\lambda\bm{x})\ & \mbox{on }\R^{nd},
  \end{cases}\]
where $T_\lambda:=T/\lambda$. Since we want the scaled dynamics to satisfy assumption $(\bm{A}')$, we choose now $T$ so that
$T_\lambda \le 1$. It is important to notice that this is possible since we assumed $\lambda$ to be fixed, even if we have not chosen it
yet. Denoting now
\begin{align*}
  f_\lambda(t,\bm{x}) &:= \, \lambda f(\lambda t,\mathbb{T}_\lambda \bm{x}); \\
  g_\lambda(\bm{x}) &:= \, g(\mathbb{T}_\lambda \bm{x});\\
  A_\lambda &:= \, \lambda\mathbb{T}^{-1}_\lambda A\mathbb{T}_\lambda; \\
  \bm{F}_\lambda(t,\bm{x}) &:= \lambda\mathbb{T}^{-1}_\lambda \bm{F}(\lambda t,\mathbb{T}_\lambda \bm{x}),
\end{align*}
we can rewrite the scaled dynamics as:
\begin{equation}
\label{Scaled_Degenerate_Stable_PDE}
\begin{cases}
    \partial_tu_\lambda(t\bm{x}) +\bigl{\langle} A_\lambda \bm{x} +\bm{F}_\lambda(t\bm{x}), D_{\bm{x}}u_\lambda(t\bm{x}) \bigr{\rangle} +
    L_\alpha u_\lambda(t\bm{x}) \, =
    \, -f_\lambda(t\bm{x}), &
    \mbox{on } (0,T_\lambda)\times \R^{nd}, \\
    u_\lambda(T_\lambda,\bm{x}) \, = \, g_\lambda(\bm{x}) & \mbox{on }\R^{nd}.
  \end{cases}
\end{equation}

To continue, we need now the following lemma that exploits how the scaling procedure reflects on the norms of the
coefficients. Recalling Equation \eqref{eq:norm_H_for_F} for the definition of $\Vert\cdot \Vert_H$, a direct calculation on the norms leads
to the following result:
\begin{lemma}[Scaling Homogeneity of Norms]
\label{lemma:Scaling_Homogeneity_Norms}
Under (\textbf{A}), it holds that
\begin{align} \label{eq:Scaling_Homogeneity_Norms}
  \Vert \bm{F}_{\lambda}\Vert_H \, &= \, \lambda^{\beta/\alpha}\Vert \bm{F}\Vert_H; \notag\\
  \lambda^{\frac{\alpha+\beta}{\alpha}}\Vert f\Vert_{L^\infty(C^\beta_{b,d})}\, \le \, \Vert &f_\lambda\Vert_{L^\infty(C^\beta_{b,d})} \, \le
  \, \Vert f\Vert_{L^\infty(C^\beta_{b,d})}\\
  \lambda^{\frac{\alpha+\beta}{\alpha}}\Vert g\Vert_{C^{\alpha+\beta}_{b,d}} \, \le \, \Vert &g_\lambda\Vert_{C^{\alpha+\beta}_{b,d}} \, \le
  \, \Vert g\Vert_{C^{\alpha+\beta}_{b,d}};\notag \\
  \lambda^{\frac{\alpha+\beta}{\alpha}}\Vert u\Vert_{L^\infty(C^{\alpha+\beta}_{b,d})}\, \le \, \Vert
  &u_\lambda\Vert_{L^\infty(C^{\alpha+\beta}_{b,d})} \, \le \, \Vert u\Vert_{L^\infty(C^{\alpha+\beta}_{b,d})} \notag
\end{align}
\end{lemma}
Since the scaled dynamics \eqref{Scaled_Degenerate_Stable_PDE} satisfies assumption (\textbf{A'}), we know from Proposition
\ref{prop:A_Priori_Estimates} that the scaled solution $u_\lambda$ satisfies the A Priori Estimates (Equation \eqref{eq:A_Priori_Estimates}):
\begin{equation}\label{eq:A_priori_Estimate_Scaled}
\Vert u_\lambda \Vert_{L^\infty(C^{\alpha+\beta}_{b,d})} \, \le \, Cc_0^{\frac{\beta-\gamma_n}{\alpha}}\bigl[\Vert g_\lambda
\Vert_{C^{\alpha+\beta}_{b,d}}+\Vert f_\lambda \Vert_{L^\infty(C^{\beta}_{b,d})}\bigr]+C\bigl(c_0^{\frac{\beta-\gamma_n}{\alpha}}\Vert
\bm{F}_\lambda\Vert_H + c_0^{\frac{\alpha+\beta-1}{1+\alpha(n-1)}}\bigr)\Vert u_\lambda \Vert_{L^\infty(C^{\alpha+\beta}_{b,d})}
\end{equation}
for some constant $c_0$ in $(0,1]$ to be chosen later.\newline
We would like now to exploit a circular argument in order to bring to the left-hand side of \eqref{eq:A_priori_Estimate_Scaled} the term involving $u_\lambda$ on the right-hand one.
To do that, we need to choose properly $\lambda$ and $c_0$ in order to have
\[C\bigr(c_0^{\frac{\beta-\gamma_n}{\alpha}}\Vert\bm{F}_\lambda\Vert_H + c_0^{\frac{\alpha+\beta-1}{1+\alpha(n-1)}}\bigr)\, < \, 1.\]
This is true if for example we choose firstly $c_0$ such that
\[Cc_0^{\frac{\alpha+\beta-1}{1+\alpha(n-1)}} \, = \, \frac{1}{4}\]
and fixed $c_0$, we choose $\lambda$ so that
\[Cc_0^{\frac{\beta-\gamma_n}{\alpha}}\lambda^{\beta/\alpha}\Vert\bm{F}\Vert_H  \, = \,
Cc_0^{\frac{\beta-\gamma_n}{\alpha}}\Vert\bm{F}_\lambda\Vert_H \, = \, \frac{1}{4}.\]
With this choice, it thus follows from \eqref{eq:A_priori_Estimate_Scaled} that
\[\Vert u_\lambda \Vert_{L^\infty(C^{\alpha+\beta}_{b,d})} \, \le \, 2Cc_0^{\frac{\beta-\gamma_n}{\alpha}}\bigl[\Vert g_\lambda
\Vert_{C^{\alpha+\beta}_{b,d}}+\Vert f_\lambda \Vert_{L^\infty(C^{\beta}_{b,d})}\bigr].\]
We can finally conclude using Lemma \ref{lemma:Scaling_Homogeneity_Norms} to go back to the original dynamics and write that
\[\Vert u \Vert_{L^\infty(C^{\alpha+\beta}_{b,d})} \, \le \,\lambda^{-\frac{\alpha+\beta}{\alpha}}\Vert u_\lambda
\Vert_{L^\infty(C^{\alpha+\beta}_{b,d})} \, \le \, \overline{C}\bigl[\Vert g\Vert_{C^{\alpha+\beta}_{b,d}}+\Vert f
\Vert_{L^\infty(C^{\beta}_{b,d})}\bigr]\]
for some constant $\overline{C}>0$ defined by
\[\overline{C} \, := \,  2\lambda^{-\frac{\alpha+\beta}{\alpha}}Cc_0^{\frac{\beta-\gamma_n}{\alpha}}.\]
\subsubsection{Schauder Estimates for General Time}
Up to this point, we have assumed to be in a small enough final time horizon (i.e. $T\le 1$) to let our procedure work. We are going now to
extend the Schauder estimates (Equation \eqref{equation:Schauder_Estimates}) to an arbitrary but fixed final time $T_0>0$. Our proof
will consist essentially in a backward iterative procedure through a chain of identical differential dynamics on different, small
enough, time intervals. We recall indeed that the Schauder estimates precisely provide a stability result in the chosen functional space.

\begin{prop}
Under (\textbf{A}), let $T_0>T$ and $u$ a mild solution in $L^{\infty}(0,T_0,C^{\alpha+\beta}_{b,d}(\R^{nd}))$ of the IPDE
\eqref{Degenerate_Stable_PDE} on $[0,T_0]$ that satisfies the Schauder Estimates (Equation \eqref{equation:Schauder_Estimates}) on $[0,T]$. Then, there
exists a constant $C_0:=C_0(T_0)$ such that
\[\Vert u \Vert_{L^{\infty}(0,T_0;C^{\alpha+\beta}_{b,d})} \, \le \, C_0\Bigl[\Vert f\Vert_{L^{\infty}( 0,T_0;C^{\beta}_{b,d})} +
\Vert g \Vert_{C^{\alpha+\beta}_{b,d}}\Bigr].\]
\end{prop}
\begin{proof}
Fixed $N=\lceil\frac{T_0}{T}\rceil$, we are going to consider a system of $N$ Cauchy problems:
\[
\begin{cases}
    \partial_tu_k(t,\bm{x}) +\bigl{\langle}A \bm{x} + \bm{F}(t,\bm{x}), D_{\bm{x}}u_k(t,\bm{x})\bigr{\rangle} +
L_\alpha u_k(t,\bm{x}) \, = \, -f(t,\bm{x}), & \mbox{on } ((1-\frac{k}{N})T_0,(1-\frac{k-1}{N})T_0)\times \R^{nd} \\
    u_k((1-\frac{k-1}{N})T_0,\bm{x}) \, = \, u_{k-1}((1-\frac{k-1}{N})T_0,\bm{x}) & \mbox{on }\R^{nd}.
  \end{cases}
\]
for $k=1,\dots,N$ with the notation that $u_0(T_0,\bm{x})=g(\bm{x})$. Reasoning iteratively, we find that any mild solution of the IPDE
\eqref{Degenerate_Stable_PDE} on $[0,T_0]$ is also a mild solution of any of the equations of the system. Moreover, since any solution $u_k$ is
defined on $[(1-\frac{k}{N})T_0,(1-\frac{k-1}{N})T_0]$ and
\[(1-\frac{k-1}{N})T_0-(1-\frac{k}N)T_0 \, = \, \frac{k}{N}T_0 - \frac{k-1}{N}T_0 \, = \, \frac{1}{N}T_0 \, \le \, T,\]
the Schauder Estimates (Equation \eqref{theorem:Schauder_Estimates}) hold for any $u_k$ with terminal condition
$u_{k-1}((1-\frac{k-1}{N})T_0,\cdot)$. In particular,
\begin{multline*}
\Vert u_k \Vert_{L^{\infty}((1-\frac{k}{N})T_0,(1-\frac{k-1}{N})T_0;C^{\alpha+\beta}_{b,d})} \, \le \, C\Bigl[\Vert f
\Vert_{L^{\infty}((1-\frac{k}{N})T_0,(1-\frac{k-1}{N})T_0;C^{\beta}_{b,d})} + \Vert u_{k-1}((1-\frac{k-1}{N})T_0,\cdot)
\Vert_{C^{\alpha+\beta}_{b,d}}\Bigr] \\
\le \, C^2\Bigl[\Vert f \Vert_{L^{\infty}((1-\frac{k}{N})T_0,(1-\frac{k-1}{N})T_0;C^{\beta}_{b,d})} + \Vert f \Vert_{ L^{\infty}((1-
\frac{k-1}{N})T_0,(1-\frac{k-2}{N})T_0;C^{\beta}_{b,d})}+ \Vert u_{k-2}((1-\frac{k-2}{N})T_0,\cdot)
\Vert_{C^{\alpha+\beta}_{b,d}}\Bigr]\\
\le \, C^2\Bigl[\Vert f \Vert_{L^{\infty}((1-\frac{k}{N})T_0,(1-\frac{k-2}{N})T_0;C^{\beta}_{b,d})}+\Vert
u_{k-2}((1-\frac{k-2}{N})T_0, \cdot) \Vert_{C^{\alpha+\beta}_{b,d}}\Bigr]
\end{multline*}
since $u_{k-1}$ satisfies the Schauder Estimates with terminal condition $u_{k-2}((1-\frac{k-2}{N})T_0,\cdot)$. Applying the same procedure
recursively, we finally find that
\[\Vert u_k \Vert_{L^{\infty}((1-\frac{k}{N})T_0,(1-\frac{k-1}{N})T_0;C^{\alpha+\beta}_{b,d})} \, \le \, C^k\Bigl[\Vert f
\Vert_{L^{\infty}((1-\frac{k}{N})T_0,T_0;C^{\beta}_{b,d})} + \Vert g \Vert_{C^{\alpha+\beta}_{b,d}}\Bigr].\]
Hence,
\[\Vert u \Vert_{L^{\infty}(0,T_0;C^{\alpha+\beta}_{b,d})} \, \le \, C^N\Bigl[\Vert f\Vert_{L^{\infty}(0,T_0;C^{\beta}_{b,d})} +
\Vert g\Vert_{C^{\alpha+\beta}_{b,d}}\Bigr]\]
and we have concluded.
\end{proof}

\setcounter{equation}{0}

\section{Schauder Estimates for the Proxy}
The aim of this section is to show how to properly control a solution $\tilde{u}^{\tau,\bm{\xi}}$ of the "frozen" IPDE \eqref{Frozen_PDE} in order to prove the Schauder estimates (Proposition \ref{prop:Schauder_Estimates_for_proxy}) for the proxy.
We recall the definition of $\tilde{u}^{\tau,\bm{\xi}}$ through the Duhamel representation \eqref{Duhamel_representation_of_proxy}. Namely, for any freezing couple $(\tau,\bm{\xi})$ in $[0,T]\times\R^{nd}$, it holds that
\begin{equation}\label{align:Representation}
  \tilde{u}^{\tau,\bm{\xi}}(t,\bm{x}) \, = \, \tilde{P}^{\tau,\bm{\xi}}_{t,T}g(\bm{x}) + \tilde{G}^{\tau,\bm{\xi}}_{t,T}f(t,\bm{x})
\end{equation}
where we have denoted for simplicity with $\bigl(\tilde{G}^{\tau,\bm{\xi}}_{v,r}\bigr)_{t>v\ge0}$ the family of Green kernels associated with
the frozen density $\tilde{p}^{\tau,\bm{\xi}}$. Namely, for any $v<r$ in $[0,T]$,
\begin{equation}\label{eq:def_Green_Kernel}
\tilde{G}^{\tau,\bm{\xi}}_{v,r}f(t,x) \, := \, \int_{v}^{r}\int_{\R^{nd}}\tilde{p}^{\tau,\bm{\xi}}(t,s,\bm{x},\bm{y}) f(s,\bm{y}) \,
d\bm{y}\, ds.
\end{equation}
We can then differentiate the above equation with respect to $\bm{x}_1$ so that to obtain an analogous Duhamel type representation for the derivative $D_{\bm{x}_1}\tilde{u}^{\tau,\bm{\xi}}$:
\begin{equation}\label{align:Representation_deriv}
 D_{\bm{x}_1}\tilde{u}^{\tau,\bm{\xi}}(t,\bm{x})\, =
  \,D_{\bm{x}_1}\tilde{P}^{\tau,\bm{\xi}}_{t,T}g(\bm{x})+D_{\bm{x}_1}\tilde{G}^{\tau,\bm{\xi}}_{t,T}f(t,\bm{x})
 \end{equation}
It is then clear that in order to control $\tilde{u}^{\tau,\bm{\xi}}(t,\bm{x})$ in the norm $\Vert \cdot \Vert_{L^\infty(C^{\alpha+\beta}_{b,d})}$, we can analyze separately the contributions appearing from the frozen semigroup $\tilde{P}^{\tau,\bm{\xi}}g(\bm{x})$ and those from the frozen Green kernel $\tilde{G}^{\tau,\bm{\xi}}_{t,T}f(t,\bm{x})$.

\subsection{First Besov Control}
We focus for the moment on the contribution in the Duhamel representation \eqref{align:Representation} associated with the source $g$ that, as it will be seen, is the more delicate to treat. In the non-degenerate setting (i.e. with respect to $\bm{x}_1$), it precisely write:
\[D_{\bm{x}_1}\tilde{P}^{\tau,\bm{\xi}}g(\bm{x}) \, = \, \int_{\R^{nd}} D_{\bm{x}_1}\tilde{p}^{\tau,\bm{\xi}}(t,T,\bm{x},\bm{y})
g(\bm{y}) \, d\bm{y}.\]
Looking at the particular structure of $\tilde{p}^{\tau,\bm{\xi}}$ (cf.\ Equation \eqref{eq:definition_tilde_p}), it can be seen from Lemma
\ref{lemma:Scaling_Lemma} that
\begin{lemma}\label{lemma:link_derivative_density}
Let $i$ in $\llbracket 1,n \rrbracket$. Then, there exist constants $\{C_j\}_{j\in \llbracket i,n \rrbracket}$ such that
\begin{equation}\label{eq:link_derivative_density}
D_{\bm{x}_i}\tilde{p}^{\tau,\bm{\xi}}(t,s,\bm{x},\bm{y}) \, = \, \sum_{j=i}^{n}C_j(s-t)^{j-i} D_{\bm{y}_j}
\tilde{p}^{\tau,\bm{\xi}}(t,s,\bm{x},\bm{y})
\end{equation}
for any $t< s$ in $[0,T]$, any $\bm{x},\bm{y}$ in $\R^{nd}$ and any freezing couple $(\tau,\bm{\xi})$ in $[0,T]\times \R^{nd}$.
\end{lemma}

We can now use equation \eqref{eq:link_derivative_density} to rewrite $D_{\bm{x}_1}\tilde{P}^{\tau,\bm{\xi}}g(\bm{x})$ as
\begin{equation}\label{Besov:eq_Introduction}
\bigl{\vert}D_{\bm{x}_1}\tilde{P}^{\tau,\bm{\xi}}g(\bm{x})\bigr{\vert} \, = \, \Bigl{\vert}\int_{\R^{nd}}
D_{\bm{x}_1}\tilde{p}^{\tau,\bm{\xi}}(t,T,\bm{x},\bm{y})
g(\bm{y}) \, d\bm{y}\Bigr{\vert}\, \le \, C\sum_{j=1}^{n}(s-t)^{j-1}\Bigl{\vert}
\int_{\R^{nd}}D_{\bm{y}_j}\tilde{p}^{\tau,\bm{\xi}}(t,T,\bm{x},\bm{y})
g(\bm{y}) \,d\bm{y} \Bigr{\vert}.
\end{equation}
Remembering that $g$ is in $C^{\alpha+\beta}_{b,d}(\R^{nd})$ for $\alpha+\beta>1$ by hypothesis,  we know that it is differentiable with respect to
the first (non-degenerate) variable $\bm{x}_1$. Then, the above expression can be controlled easily for $j=1$ as
\[\Bigl{\vert}\int_{\R^{nd}}D_{\bm{y}_1}\tilde{p}^{\tau,\bm{\xi}}(t,T,\bm{x},\bm{y})g(\bm{y}) \, d\bm{y} \Bigr{\vert} \, = \,
\Bigl{\vert}\int_{\R^{nd}}\tilde{p}^{\tau,\bm{\xi}}(t,T,\bm{x},\bm{y})D_{\bm{y}_1}g(\bm{y}) \, d\bm{y} \Bigr{\vert} \, \le \,
\Vert D_{\bm{y}_1}g \Vert_{L^\infty} \, \le \, \Vert g \Vert_{C^{\alpha+\beta}_{b,d}}\]
using integration by parts formula. We can then focus on the degenerate components in \eqref{Besov:eq_Introduction}, i.e.
\begin{equation}\label{Besov:eq_Introduction2}
\Bigl{\vert}\int_{\R^{nd}}D_{\bm{y}_j}\tilde{p}^{\tau,\bm{\xi}}(t,T,\bm{x},\bm{y})g(\bm{y}) \, d\bm{y} \Bigr{\vert}
\end{equation}
for some $j>1$. Since $g$ is not differentiable with respect to $\bm{y}_j$ if $j>1$, we cannot apply the same reasoning above but we will
need a more subtle control. Our main idea will be to use the duality in Besov spaces to derive bounds for expression
\eqref{Besov:eq_Introduction2}. Namely, we introduce for a given $\bm{y}$ in $\R^d$,
\[\bm{y}_{\smallsetminus j} \, :=  \, (\bm{y}_1,\dots,\bm{y}_{j-1},\bm{y}_{j+1},\dots,\bm{y}_n) \, \in \, \R^{(n-1)d}.\]
With this definition at hand, we then denote for any function $\phi$ on $\R^{nd}$, the function $\phi(\bm{y}_{\smallsetminus j},\cdot)$ on $\R^d$ with a slight abuse of notation as
\begin{equation}\label{eq:notation_smallsetminus}
\phi(\bm{y}_{\smallsetminus j},z) \, := \, \phi(\bm{y}_1,\dots,\bm{y}_{j-1},z,\bm{y}_{j+1},\dots,\bm{y}_n).
\end{equation}
The key point now is to control the H\"older modulus of $g(\bm{y}_{\smallsetminus j},\cdot)$ on $\R^d$, uniformly in $\bm{y}_{\smallsetminus j}
\in \R^{(n-1)d}$. To do so, we will need the identification $C^{\alpha_j+\beta_j}_b(\R^d) = B^{\alpha_j+\beta_j}_{\infty,\infty}
(\R^d)$ with the usual notations for the Besov spaces.

We recall now some useful definitions/characterizations about Besov spaces $B^{\tilde{\gamma}}_{p,q}(\R^d)$. For a more detailed
analysis of this argument, we suggest the reader to see Section $2.6.4$ of Triebel \cite{book:Triebel83}. For $\tilde{\gamma}$ in $(0,1)$,
$q,p$ in $(0,+\infty]$, we define the Besov space of indexes $(\tilde{\gamma},p,q)$ on $\R^d$ as:
\[B^{\tilde{\gamma}}_{p,q}(\R^d):= \{f \in \mathcal{S}'(\R^d)\colon \Vert f \Vert_{\mathcal{H}^{\tilde{\gamma}}_{p,q}} \, < + \infty\}\]
where $\mathcal{S}(\R^d)$ denotes the Schwartz class on $\R^d$ and
\begin{equation}\label{alpha-thermic_Characterization}
\Vert f \Vert_{\mathcal{H}^{\tilde{\gamma}}_{p,q}} \, := \, \Vert (\phi_0\hat{f})^\vee \Vert_{L^p}+ \Bigl(\int_{0}^{1}
v^{-\frac{\tilde{\gamma}}{\alpha}}\Vert \partial_vp_h(v,\cdot)\ast f \Vert^q_{L^p} \, dv\Bigr)^{\frac{1}{q}}
\end{equation}
with $\phi_0$ a function in $C^\infty_0(\R^d)$ such that $\phi_0(0) \neq 0$ and $p_h$ the isotropic $\alpha$-stable heat kernel on
$\R^d$, i.e.\ the stable density on $\R^d$ whose L\'evy symbol is equivalent to $\vert \lambda \vert^\alpha$. \newline
We point out that the quantities in \eqref{alpha-thermic_Characterization} are well-defined for any $q\neq +\infty$. The modifications for
$q=+\infty$ are obvious and can be written passing to the limit. The previous definition of $B^{\tilde{\gamma}}_{p,q}(\R^d)$ is known as the
stable thermic characterization of Besov spaces and it is particularly adapted to our framework. By a little abuse of notation, we will
write $\Vert f \Vert_{B^{\tilde{\gamma}}_{p,q}}:=\Vert f \Vert_{\mathcal{H}^{\tilde{\gamma}}_{p,q}}$ when this quantity is finite.

For the heat-kernel $p_h$, it is possible to show an improvement of the smoothing effect (cf. equation \eqref{Smoothing_effect_of_S}),
due essentially to its better decay at infinity. Namely, we are no more bounded to the condition $\gamma<\alpha$ but we can integrate up to
an order $\gamma$ strictly smaller than $1+\alpha$.

\begin{lemma}[Smoothing Effect of the Isotropic Stable Heat-Kernel]
\label{lemma:Smoothing_effect_of_Heat_Kern}
Let $l$ be in $\{1,2\}$ and $\gamma$ in $[0,1+\alpha)$. Then, there exists a positive constant $C:=C(\gamma)$ such that
\begin{equation}\label{Smoothing_effect_of_Heat_Kern}
\int_{\R^d}\vert y \vert^\gamma \vert \partial_vD^l_yp_h(v,y) \vert \, dy \, \le \, Ct^{\frac{\gamma-l}{\alpha}-1}.
\end{equation}
\end{lemma}
A proof of the above result can be derived using the estimates of Kolokoltsov \cite{Kolokoltsov00} (see also \cite{Bogdan:Jakubowski07}).

As already indicated before, it can be seen from the $\alpha-$thermic characterization \eqref{alpha-thermic_Characterization} that
\begin{equation}\label{Besov:ident_Holder_Besov}
C^{\tilde{\gamma}}_b(\R^d) \, = \, B^{\tilde{\gamma}}_{\infty,\infty}(\R^d).
\end{equation}
Moreover, it is well known (see for example Proposition $3.6$ in \cite{book:Lemarie-Rieusset02}) that
$B^{\tilde{\gamma}}_{\infty,\infty}(\R^d)$
and $B^{-\tilde{\gamma}}_{1,1}(\R^d)$ are in duality. Namely, it holds
\begin{equation}\label{Besov:duality_in_Besov}
\bigl{\vert}\int_{\R^d} fg \, dx \bigr{\vert}\, \le \, C\Vert f \Vert_{B^{\tilde{\gamma}}_{\infty,\infty}}\Vert g
\Vert_{B^{-\tilde{\gamma}}_{1,1}}.
\end{equation}
for any $f$ in $B^{\tilde{\gamma}}_{\infty,\infty}(\R^d)$ and any function $g$ in $B^{-\tilde{\gamma}}_{1,1}(\R^d)$.

With these definitions and properties at hand, we can now go back at expression \eqref{Besov:eq_Introduction2} to write that
\begin{multline*}
\Bigl{\vert}\int_{\R^{nd}}D_{\bm{y}_j}\tilde{p}^{\tau,\bm{\xi}}(t,T,\bm{x},\bm{y})g(\bm{y}) \, d\bm{y} \Bigr{\vert} \, \le \,
\int_{\R^{(n-1)d}}\Bigl{\vert}D_{\bm{y}_j}\tilde{p}^{\tau,\bm{\xi}}(t,T,\bm{x},\bm{y})g(\bm{y})d\bm{y}_j\Bigr{\vert} d\bm{y}_{\smallsetminus
j} \\
\le \,\int_{\R^{(n-1)d}}\Bigl{\Vert}D_{\bm{y}_j}\tilde{p}^{\tau,\bm{\xi}} (t,T,\bm{x},\bm{y}_{\smallsetminus j},\cdot) \Bigr{\Vert}_{
B^{-(\alpha_j+\beta_j)}_{1,1}}\Bigl{\Vert}g(\bm{y}_{\smallsetminus j},\cdot)\Bigr{\Vert}_{B^{\alpha_j+\beta_j}_{\infty,\infty}}
d\bm{y}_{\smallsetminus j} \\
\le \, \Vert g \Vert_{C^{\alpha+\beta}_{b,d}}\int_{\R^{(n-1)d}}\Bigl{\Vert}D_{\bm{y}_j}\tilde{p}^{\tau,\bm{\xi}}(t,T,\bm{x},
\bm{y}_{\smallsetminus j},\cdot) \Bigr{\Vert}_{B^{-(\alpha_j+\beta_j)}_{1,1}} \, d\bm{y}_{\smallsetminus j}.
\end{multline*}

In order to control the above quantities, we will then need a control on the integral of the Besov norms of the derivatives of the proxy.
Since however an additional derivative with respect to $\bm{x}_1$ will often appear, for example in
Equation \eqref{Proof:H\"older_Frozen_Semigroup_Idj} below, we state the following result in a more general way.

\begin{lemma}[First Besov Control]
\label{lemma:First_Besov_COntrols}
Let $j$ be  in $\llbracket 2,n\rrbracket$ and $l\in \{0,1\}$. Under (\textbf{A}), there exists a constant $C:=C(j,l)$ such that
\[\int_{\R^{(n-1)d}}\Bigl{\Vert}D_{\bm{y}_j}D^{l}_{\bm{x}_1}\tilde{p}^{\tau,\bm{\xi}}(t,s,\bm{x},\bm{y}_{\smallsetminus
j},\cdot)\Bigr{\Vert}_{B^{-(\alpha_j+\beta_j
)}_{1,1}} \, d\bm{y}_{\smallsetminus j} \, \le \, C(s-t)^{\frac{\alpha+\beta}{\alpha}-\frac{1}{\alpha_j}-\frac{l}{\alpha}}\]
for any $t<s$ in $[0,T]$, any $\bm{x}$ in $\R^{nd}$ and any frozen couple $(\tau,\bm{\xi})$ in $[0,T]\times \R^{nd}$.
\end{lemma}
\begin{proof}
To control the Besov norm in $B^{-(\alpha_j+\beta_j)}_{1,1}(\R^d)$, we are going to use the stable thermic characterization
\eqref{alpha-thermic_Characterization} with $\tilde{\gamma}=-(\alpha_j+\beta_j)$.
We start considering the second term in the characterization, i.e.
\[\int_{0}^{1}v^{\frac{\alpha_j+\beta_j}{\alpha}}\int_{\R^d}\Bigl{\vert}\int_{\R^d}\partial_vp_h(v,z-\bm{y}_j)D_{\bm{y}_j}
D^{l}_{\bm{x}_1}\tilde{p}^{\tau,\bm{\xi}}(t, s,\bm{x},\bm{y})\,d\bm{y}_j  \Bigr{\vert} \, dzdv.\]
Fixed a constant $\delta_j\ge1$ to be chosen later, we split the integral with respect to $v$ in two components:
\begin{multline*}
\Vert D_{\bm{y}_j}D^{l}_{\bm{x}_1}\tilde{p}^{\tau,\bm{\xi}}(t,s,\bm{x},\bm{y}_{\smallsetminus j},\cdot)
\Vert_{B^{-(\alpha_j+\beta_j)}_{1,1}} \\
= \,\int_{0}^{(s-t)^{\delta_j}}v^{\frac{\alpha_j+\beta_j}{\alpha}}\int_{\R^d}\Bigl{\vert}\int_{\R^d} \partial_v p_h(v,z-\bm{y}_j)
D_{\bm{y}_j} D^{l}_{\bm{x}_1}\tilde{p}^{\tau,\bm{\xi}}(t,s,\bm{x},\bm{y}) \,d\bm{y}_j  \Bigr{\vert} \, dzdv \\
+\int_{(s-t)^{\delta_j}}^{1}v^{\frac{\alpha_j+\beta_j}{\alpha}}\int_{\R^d}\Bigl{\vert}\int_{\R^d}\partial_vp_h(v,z-\bm{y}_j)
D_{\bm{y}_j}D^{l}_{\bm{x}_1}\tilde{p}^{\tau,\bm{\xi}}(t,s,\bm{x},\bm{y})\,d\bm{y}_j  \Bigr{\vert} \, dzdv \, =: \,
\bigl(I_1 +I_2\bigr)(\bm{y}_{\smallsetminus j}).
\end{multline*}
The second component $I_2$ has no time-singularity and can be easily controlled by
\[I_2(\bm{y}_{\smallsetminus j}) \, = \,\int_{(s-t)^{\delta_j}}^{1}v^{\frac{\alpha_j+\beta_j}{\alpha}}\int_{\R^d} \Bigl{\vert}
\int_{\R^d}D_z\partial_vp_h(v,z-\bm{y}_j)\otimes D^{l}_{\bm{x}_1}\tilde{p}^{\tau,\bm{\xi}}(t,s,\bm{x}, \bm{y})\,d\bm{y}_j
\Bigr{\vert} \, dzdv\]
using integration by parts formula and noticing that $D_{\bm{y}_j}p_h(v,z-\bm{y}_j)=-D_zp_h(v,z-\bm{y}_j)$. Then,
\[I_2(\bm{y}_{\smallsetminus j}) \, \le \, \int_{(s-t)^{\delta_j}}^{1}v^{\frac{\alpha_j+\beta_j}{\alpha}}\int_{
\R^d}\int_{\R^d}\vert D_z\partial_vp_h(v,z-\bm{y}_j)\vert \, \vert D^{l}_{\bm{x}_1}\tilde{p}^{\tau,\bm{\xi}}(t,s,\bm{x}, \bm{y}) \vert
\,d\bm{y}_j  \, dzdv.\]
We can then use Fubini theorem to separate the integrals and apply the smoothing effect of the heat-kernel $p_h$ (Lemma
\ref{lemma:Smoothing_effect_of_Heat_Kern}) to show that
\begin{multline*}
I_2(\bm{y}_{\smallsetminus j}) \, \le \, \int_{(s-t)^{\delta_j}}^{1}v^{\frac{\alpha_j+\beta_j}{\alpha}}\int_{\R^d}\Bigl(\int_{
\R^d}\vert D_z\partial_vp_h(v,z-\bm{y}_j)\vert \, dz\Bigr)  \vert D^{l}_{\bm{x}_1}\tilde{p}^{\tau,\bm{\xi}}(t,s,\bm{x},
\bm{y}) \vert \, d\bm{y}_jdv \\
\le \, C\Bigl(\int_{(s-t)^{\delta_j}}^{1}v^{\frac{\alpha_j+\beta_j-1}{\alpha}-1} \, dv\Bigr)\Bigl(\int_{\R^d}\vert D^{l}_{\bm{x}_1}
\tilde{p}^{\tau,\bm{\xi}}(t,s,\bm{x},\bm{y})\vert \,d\bm{y}_j\Bigr) \\
\le \, C(s-t)^{\frac{\delta_j(\alpha_j+\beta_j-1)}{\alpha}}\int_{\R^d}\vert
D^{l}_{\bm{x}_1}\tilde{p}^{\tau,\bm{\xi}}(t,s,\bm{x},\bm{y})\vert \,d\bm{y}_j.
\end{multline*}
Using the smoothing effect (Equation \eqref{eq:Smoothing_effects_of_tilde_p}) of the frozen density $\tilde{p}^{\tau,\bm{\xi}}$, we have
thus found that
\begin{equation}\label{Proof:First_Besov_COntrol_I2}
\int_{\R^{(n-1)d}}I_2(\bm{y}_{\smallsetminus j}) \, d\bm{y}_{\smallsetminus j} \, \le \,
(s-t)^{\frac{\delta_j(\alpha_j+\beta_j-1)}{\alpha}}
\int_{\R^{nd}}\vert D^{l}_{\bm{x}_1}\tilde{p}^{\tau,\bm{\xi}}(t,s,\bm{x},\bm{y})\vert\,d\bm{y} \, \le \,
C(s-t)^{\frac{\delta_j(\alpha_j+\beta_j-1)-l}{\alpha}}.
\end{equation}
%%%%%%%%%%%%%%%%%%%%%%%%%%%%%%%%%%%%%%%%%%%%%%%%%%%%%%%%%%%%%%%%%%%%%%%%%
On the other hand, the term $I_1$ needs a more delicate treatment in order to avoid time-integrability problems. We start using a
cancellation argument with respect to the derivative $\partial_vp_h$ of the heat-kernel to rewrite $I_1$ as
\begin{multline*}
I_1(\bm{y}_{\smallsetminus j}) \, = \\
\int_{0}^{(s-t)^{\delta_j}}\!\!v^{\frac{\alpha_j+\beta_j}{\alpha}}\int_{\R^d}\Bigl{\vert}\int_{\R^d}\partial_vp_h(v,z-\bm{y}_j)
\bigl[D_{\bm{y}_j}D^l_{\bm{x}_1}\tilde{p}^{\tau,\bm{\xi}}(t,s,\bm{x},\bm{y})-D_{\bm{y}_j}D^l_{\bm{x}_1}\tilde{p}^{\tau,\bm{\xi}}(t,
s,\bm{x},\bm{y}_{\smallsetminus j},z)\bigr]\,d\bm{y}_j  \Bigr{\vert} \,dzdv \\
= \, \int_{0}^{(s-t)^{\delta_j}}\!\! v^{\frac{\alpha_j+\beta_j}{\alpha}}\int_{\R^d}\Bigl{\vert}\int_{\R^d}D_z\partial_vp_h(v,z-\bm{y}_j)
\otimes\bigl[D^l_{\bm{x}_1}\tilde{p}^{\tau,\bm{\xi}}(t,s,\bm{x},\bm{y})-D^l_{\bm{x}_1}\tilde{p}^{\tau,\bm{\xi}}(t,s,\bm{x},
\bm{y}_{\smallsetminus j},z)\bigr]\,d\bm{y}_j  \Bigr{\vert} \,dzdv
\end{multline*}
where in the second passage we used again integration by parts formula to move the derivative to $p_h$ and the equality $D_{\bm{y}_j}p_h(v,
z-\bm{y}_j)=-D_zp_h(v,z-\bm{y}_j)$. We can then apply a Taylor expansion with respect to variable $\bm{y}_j$ to write that
\begin{multline*}
I_1(\bm{y}_{\smallsetminus j}) \, = \\
\int_{0}^{(s-t)^{\delta_i}}\!\!v^{\frac{\alpha_j+\beta_j}{\alpha}}\int_{\R^d}\Bigl{\vert}\int_{\R^d}D_z\partial_vp_h(v,z-\bm{y}_j)
\int_{0}^{1}\!D_{\bm{y}_j}D^l_{\bm{x}_1}\tilde{p}^{\tau,\bm{\xi}}(t,s,\bm{x},\bm{y}_{\smallsetminus
j},\bm{y}_j+\lambda(z-\bm{y}_j))\cdot(z-\bm{y}_j) \,  d\mu d\bm{y}_j \Bigr{\vert} \, dzdv \\
\le \, \int_{0}^{(s-t)^{\delta_i}}\!\!v^{\frac{\alpha_j+\beta_j}{\alpha}}\int_{\R^d}\int_{\R^d}\int_{0}^{1}\vert
D_z\partial_vp_h(v,z-\bm{y}_j)\vert\, \vert D_{\bm{y}_j}D^l_{\bm{x}_1}\tilde{p}^{\tau,\bm{\xi}}(t,s,\bm{x},
\bm{y}_{\smallsetminus j},\bm{y}_j+\lambda(z-\bm{y}_j))\vert\, \vert z-\bm{y}_j\vert \,  d\lambda d\bm{y}_jdzdv
\end{multline*}
We can then use Fubini theorem and changes of variables $\tilde{z}=z-\bm{y}_j$ (fixed $\bm{y}_j$) and
$\tilde{\bm{y}}_j=\bm{y}_j+\lambda\tilde{z}$ (considering $\tilde{z}$ and $\lambda$ fixed) to separate the integrals so that
\[I_1(\bm{y}_{\smallsetminus j}) \, \le \, \int_{0}^{(s-t)^{\delta_i}}v^{\frac{\alpha_j+\beta_j}{\alpha}} \Bigl(\int_{\R^d} \vert D_z\partial_vp_h(v,\tilde{z})\vert \,
\vert \tilde{z} \vert \, dz\Bigr) \Bigl(\int_{\R^d} \vert D_{\bm{y}_j}D^l_{\bm{x}_1}\tilde{p}^{\tau,\bm{\xi}}(t,s,\bm{x}, \bm{y
}_{\smallsetminus j },\tilde{\bm{y}}_j) \vert\,d\bm{y}_j\Bigr)\, dv.
\]
The smoothing effect of the heat-kernel $p_h$ (Lemma \ref{lemma:Smoothing_effect_of_Heat_Kern}) allows now to control the first term:
\begin{multline*}
I_1(\bm{y}_{\smallsetminus j}) \, \le \, C\Bigl(\int_{0}^{(s-t)^{\delta_i}}v^{\frac{\alpha_j+\beta_j-1}{\alpha}}\, dv\Bigr)
\Bigl(\int_{\R^d} \vert D_{\bm{y}_j}D^l_{\bm{x}_1}\tilde{p}^{\tau,\bm{\xi}}(t,s,\bm{x}, \bm{y}_{\smallsetminus j },z+\lambda(\bm{y}_j-z))
\vert\,d\bm{y}_j\Bigr) \\
\le \, C(s-t)^{\delta_j\frac{\alpha_j+\beta_j}{\alpha}} \int_{\R^d}\vert D_{\bm{y}_j}D^l_{\bm{x}_1}\tilde{p}^{\tau,\bm{\xi}}
(t,s,\bm{x}, \bm{y}_{\smallsetminus j },z+\lambda(\bm{y}_j-z)) \vert \,d\bm{y}_j.
\end{multline*}
It then follows using the smoothing effect of the frozen semigroup (Lemma \ref{lemma:Smoothing_effect_frozen}) that
\begin{multline}\label{Proof:First_Besov_COntrol_I1}
\int_{\R^{(n-1)d}}I_1(\bm{y}_{\smallsetminus j}) \, d\bm{y}_{\smallsetminus j} \, \le \,
C(s-t)^{\delta_j\frac{\alpha_j+\beta_j}{\alpha}}\int_{\R^{nd}}\vert D_{\bm{y}_j}D^l_{\bm{x}_1}\tilde{p}^{\tau,\bm{\xi}}
(t,s,\bm{x}, \bm{y}_{\smallsetminus j },z+\lambda(\bm{y}_j-z)) \vert \,d\bm{y}  \\
\le \, C(s-t)^{\delta_j\frac{\alpha_j+\beta_j}{\alpha}-\frac{l}{\alpha} -\frac{l}{\alpha_j}}.
\end{multline}
Going back to equations \eqref{Proof:First_Besov_COntrol_I2} and \eqref{Proof:First_Besov_COntrol_I1}, we notice that we need $\delta_j$ to
be such that
\[\delta_j\bigl[\frac{\alpha_j +\beta_j}{\alpha}\bigr] \, = \, \frac{\alpha+\beta}{\alpha} \,\, \text{ and } \,\,
\delta_j\bigl[\frac{\alpha_j+\beta_j-1}{\alpha}\bigr] \, = \, \frac{\alpha+\beta}{\alpha}-\frac{1}{\alpha_j}.\]
Recalling Equation \eqref{eq:def_alpha_i_and_beta_i} for the relative definitions, we can thus conclude choosing
$\delta_j=(\alpha+\beta)/(\alpha_j+\beta_j)=1+\alpha(j-1)$. \newline
Reproducing the previous computations, we can also write for a test function in $\phi_0$ in $C^\infty_0(\R^d)$,
\begin{multline*}
\int_{\R^{(n-1)d}}\Bigl{\Vert} \bigl(\phi_0\bigl(D_{\bm{y}_j}D^{l}_{\bm{x}_1}\tilde{p}^{\tau,\bm{\xi}}(t,s,\bm{x},\bm{y}_{\smallsetminus
j},\cdot)\bigr)\hat{\phantom{A}}\bigr)^\vee \Bigr{\Vert}_{L^1} \, d\bm{y}_{\smallsetminus j} \\
= \, \int_{\R^{(n-1)d}}\int_{\R^d}\Bigl{\vert}\int_{\R^d}D_{\bm{y}_j}\hat{\phi_0}(z-\bm{y}_j)\cdot
D^l_{\bm{x}_1}\tilde{p}^{\tau,\bm{\xi}}(t,s,\bm{x},\bm{y})\, d\bm{y}_j \Bigr{\vert} \, dzd\bm{y}_{\smallsetminus j} \\
\le \, C \int_{\R^{nd}}\vert D^l_{\bm{x}_1}\tilde{p}^{\tau,\bm{\xi}}(t,s,\bm{x},\bm{y}) \vert \, d\bm{y} \, \le \,
C(s-t)^{-\frac{l}{\alpha}}.
\end{multline*}
\end{proof}

\subsection{Proof of Proposition \ref{prop:Schauder_Estimates_for_proxy}}

Thanks to the First Besov Control (Lemma \ref{lemma:First_Besov_COntrols}), we are now ready to prove the Schauder Estimates for the proxy (Proposition \ref{prop:Schauder_Estimates_for_proxy}). Such a proof will be divided in three parts: the estimates for the supremum norms of the solution and its non-degenerate gradient are stated in Lemma \ref{lemma:supremum_norm_proxy} while the controls of the H\"older moduli of the solution and its gradient with respect to the non-degenerate variable are given in Lemmas \ref{lemma:Holder_modulus_proxy_Non-Deg} and \ref{lemma:Holder_modulus_proxy_Deg}, respectively.

\begin{lemma}(Controls on Supremum Norm)
\label{lemma:supremum_norm_proxy}
Under (\textbf{A}), there exists a constant $C:=C(T)\ge 1$ such that for any freezing couple $(\tau,\bm{\xi})$ in
$[0,T]\times\R^{nd}$, any $t$ in $[0,T]$ and any $\bm{x}$ in $\R^{nd}$,
\[\vert \tilde{u}^{\tau,\bm{\xi}}(t,\bm{x}) \vert + \vert D_{\bm{x}_1}\tilde{u}^{\tau,\bm{\xi}}(t,\bm{x})\vert \, \le \,
C\Bigl[\Vert g \Vert_{C^{\alpha+\beta}_{b,d}} + \Vert f \Vert_{L^\infty(C^\beta_{b,d})} \Bigr].\]
\end{lemma}
\begin{proof}
We start noticing that $\tilde{P}^{\tau,\bm{\xi}}_{t,T}g(\bm{x})$ and $\tilde{G}^{\tau,\bm{\xi}}_{t,T}f(t,\bm{x})$ can be easily
bounded using the supremum norm of $f$ and $g$, respectively. \newline
Moreover, we can use the controls on the frozen semigroup (Equation \eqref{eq:Control_of_semigroup}) to control
$D_{\bm{x}_1}\tilde{G}^{\tau,\bm{\xi}}_{t,T}f(t,\bm{x})$. Indeed,
\[\bigl{\vert} D_{\bm{x}_1}\tilde{G}^{\tau,\bm{\xi}}_{t,T}f(t,\bm{x}) \bigr{\vert} \, \le \, \int_{t}^{T}
\bigl{\vert}D_{\bm{x}_1}\tilde{P}^{\tau,\bm{\xi}}_{t,s}
f(s,\bm{x}) \bigr{\vert} \, ds \, \le \, C (T-t)^{\frac{\alpha+\beta-1}{\alpha}}\Vert f\Vert_{L^\infty(C^\beta)} \, \le \, C
T^{\frac{\alpha+\beta-1}{\alpha}}\Vert f\Vert_{L^\infty C^\beta)}\]
remembering in the last inequality that $\alpha + \beta -1 >0$ by hypothesis (\textbf{P}). \newline
It remains to control $D_{\bm{x}_1}\tilde{P}^{\tau,\bm{\xi}}_{t,T}g(\bm{x})$. As shown the previous Sub-section $4.1$, we start using
the scaling lemma \ref{lemma:link_derivative_density} to write that
\begin{multline*}
\bigl{\vert} D_{\bm{x}_1}\tilde{P}^{\tau,\bm{\xi}}_{t,T}g(\bm{x}) \bigr{\vert}\, = \,
\Bigl{\vert}\int_{\R^{nd}}D_{\bm{x}_1}\tilde{p}^{\tau,\bm{\xi}}(t,T,\bm{x},\bm{y})g(\bm{y})
\,d\bm{y}\Bigr{\vert} \\
\le \, C\sum_{j=1}^{n}(T-t)^{j-1} \Bigl{\vert} \int_{\R^{nd}} D_{\bm{y}_j} \tilde{p}^{\tau,\bm{\xi}} (t,T,\bm{x},\bm{y})
g(\bm{y}) \, d\bm{y} \Bigr{\vert}\, =: \, C\sum_{j=1}^{n}(T-t)^{j-1}J_j.
\end{multline*}
Since $g$ is differentiable in the first, non-degenerate variable $\bm{x}_1$, the contribution $J_1$ can be easily bounded using integration
by parts formula:
\begin{equation}\label{Proof:Supremum_for_Frozen_Semigroup_J1}
J_1\, = \,\Bigl{\vert} \int_{\R^{nd}}\tilde{p}^{\tau,\bm{\xi}}(t,T,\bm{x},\bm{y})D_{\bm{y}_1}g(\bm{y}) \, d\bm{y}
\Bigr{\vert} \, \le \,\Vert D_{\bm{y}_1}g\Vert_{L^{\infty}} \, \le \,\Vert g\Vert_{C^{\alpha+\beta}_{b,d}}.
\end{equation}
To control the other terms $J_j$ for $j>1$, we use instead the duality in Besov spaces \eqref{Besov:duality_in_Besov} and the identification
\eqref{Besov:ident_Holder_Besov}, so that
\begin{equation}\label{Proof:Supremum_for_Frozen_Semigroup_Jj}
J_j \, \le \, C\Vert g \Vert_{C^{\alpha+\beta}_{b,d}}\int_{\R^{(n-1)d}}  \Vert D_{\bm{y}_j}\tilde{p}^{\tau,\bm{\xi}}(t,T,\bm{x},
\bm{y}_{\smallsetminus j},\cdot)\Vert_{B^{-(\alpha_j+\beta_j)}_{1,1}} \, d\bm{y}_{\smallsetminus j}\, \le \, C\Vert g
\Vert_{C^{\alpha+\beta}_{b,d}}(T-t)^{\frac{\alpha+\beta}{\alpha}-\frac{1}{\alpha_j}}
\end{equation}
where in the last inequality we applied the first Besov Control (Lemma \ref{lemma:First_Besov_COntrols}).\newline
Looking back at equations \eqref{Proof:Supremum_for_Frozen_Semigroup_J1} and \eqref{Proof:Supremum_for_Frozen_Semigroup_Jj}, it finally
holds that
\[\bigl{\vert} D_{\bm{x}_1}\tilde{P}^{\tau,\bm{\xi}}_{t,T}g(\bm{x}) \bigr{\vert} \, \le \, C\Vert g\Vert_{C^{\alpha+\beta}_{b,d}}
\bigl(1+\sum_{j=2}^{n}(T-t)^{j-1}(T-t)^{\frac{\alpha+\beta}{\alpha}-\frac{1}{\alpha_j}}\bigr) \, \le \,
C\bigl(1+T^{\frac{\alpha+\beta-1}{\alpha}}\bigr)\Vert g \Vert_{C^{\alpha+\beta}_{b,d}}\]
where in the last passage we used again that $\alpha+\beta -1 >0$ by hypothesis (\textbf{P}).
\end{proof}

Before starting with the calculations on the H\"older modulus. For fixed $(t,\bm{x},\bm{x}')$ in $[0,T]\times \R^{2nd}$, we will need to
distinguish two cases. We will say that the \emph{off-diagonal regime} holds if $T-t \le c_0d^\alpha(\bm{x},\bm{x}')$ for a constant $c_0$
to be specified but meant to be smaller than $1$. This means in particular that the spatial distance is larger than the characteristic
time-scale up to the prescribed constant $c_0$ which will be useful further on in the computations for a circular argument. \newline
On the other hand, we will say that the global diagonal regime is in force when $T-t \ge c_0d^\alpha(\bm{x},\bm{x}')$ and the spatial points
are instead closer than the typical time-scale magnitude. In particular, when a time integration is involved (for example in the control of
the frozen Green kernel), the same two regime appears even if in a local base. Considering a
variable $s$ in $[t,T]$, there are again a local off-diagonal regime if $s-t \le c_0d^\alpha(\bm{x},\bm{x}')$ and a local diagonal
regime when $s-t \ge c_0d^\alpha(\bm{x},\bm{x}')$. In particular we will denote with $t_0$  the critical time at
which a change of regime occurs in the globally diagonal regime. Namely,
\begin{equation}\label{eq:def_t0}
t_0 \, := \, \bigl(t+c_0d^\alpha(\bm{x},\bm{x}')\bigr)\wedge T.
\end{equation}
We highlight however that this approach was already used in \cite{Chaudru:Honore:Menozzi18_Sharp} to obtain Schauder estimates for
degenerate Kolmogorov equations and can be adapted in the current setting.

Moreover, it is important to notice that the norm $\Vert \cdot\Vert_{C^{\alpha+\beta}_d}$ is essentially defined as the sum of the norms
$\Vert \cdot \Vert_{C^{\frac{\alpha+\beta}{1+\alpha(i-1)}}}$ with respect to the $i$-th variable and uniformly on the other components.
Thus, there is a big difference between the case $i=1$ where
$\alpha+\beta$ is in $(1,2)$ and we have to deal with a proper derivative and the other situations ($i>1$) where instead
$(\alpha+\beta)/(1+\alpha(i-1))<1$ and the norm is calculated directly on the function. For this reason, we are going to analyze the two
cases separately. Lemma \ref{lemma:Holder_modulus_proxy_Non-Deg} will work on the non-degenerate setting ($i=1$) while Lemma
\ref{lemma:Holder_modulus_proxy_Deg} will concern the degenerate one ($i>1$).

\begin{lemma}[Controls on H\"older Moduli: Non-Degenerate]
\label{lemma:Holder_modulus_proxy_Non-Deg}
Let $\bm{x},\bm{x}'$ be in $\R^{nd}$ such that $\bm{x}_j=\bm{x}'_j$ for any $j\neq 1$. Under (\textbf{A}), there exists a constant $C\ge 1$
such that for any $t$ in $[0,T]$ and any freezing couple $(\tau,\bm{\xi})$ in $[0,T]\times \R^{nd}$, it holds that
\[\bigl{\vert} D_{\bm{x}_1}\tilde{u}^{\tau,\bm{\xi}}(t,\bm{x})- D_{\bm{x}_1}\tilde{u}^{\tau,\bm{\xi}}(t,\bm{x}')\vert\\
\le \, Cc^{\frac{\alpha+\beta-2}{\alpha}}_0\bigl(\Vert g\Vert_{C^{\alpha+\beta}_{b,d}} + \Vert f \Vert_{L^\infty(C^\beta_{b,d})}\bigr)
d^{\alpha+\beta-1}(\bm{x},\bm{x}').\]
\end{lemma}

Before proving the above result, we point out the control on the H\"older modulus of $\tilde{u}^{\tau,\bm{\xi}}$ with respect to the
degenerate variables ($i>1$):

\begin{lemma}[Controls on H\"older Moduli: Degenerate]
\label{lemma:Holder_modulus_proxy_Deg}
Let $i$ be in $\llbracket 2,n\rrbracket$ and $\bm{x},\bm{x}'$ in $\R^{nd}$ such that $\bm{x}_j=\bm{x}'_j$ for any $j\neq i$. Under
(\textbf{A}), there exists a constant $C:=C(i)$ such that for any $t$ in $[0,T]$ and any freezing couple $(\tau,\bm{\xi})$ in $[0,T]\times
\R^{nd}$, it holds that
\[\bigl{\vert} \tilde{u}^{\tau,\bm{\xi}}(t,\bm{x})- \tilde{u}^{\tau,\bm{\xi}}(t,\bm{x}') \vert\, \le \,
Cc^{\frac{\beta-\gamma_i}{\alpha}}_0\bigl(\Vert g\Vert_{C^{\alpha+\beta}_{b,d}} + \Vert f \Vert_{L^\infty(C^\beta_{b,d})}\bigr)d^{\alpha+
\beta}(\bm{x},\bm{x}').\]
\end{lemma}

\paragraph{Proof of Lemma \ref{lemma:Holder_modulus_proxy_Non-Deg}}
\emph{Controls on frozen semigroup.} Let us consider firstly the off-diagonal regime, i.e.\ the case $T-t \le c_0d^\alpha(\bm{x},\bm{x}')$.
Using the scaling lemma
\ref{lemma:link_derivative_density}, it holds that
\[D_{\bm{x}_1}\tilde{P}^{\tau,\bm{\xi}}_{t,T}g(\bm{x}) \, = \,
\int_{\R^{nd}}D_{\bm{x}_1}\tilde{p}^{\tau,\bm{\xi}}(t,T,\bm{x},\bm{y})g(\bm{y}) \, d\bm{y} \, = \,
\sum_{j=1}^{n}C_j(T-t)^{j-1}\int_{\R^{nd}}D_{\bm{y}_j}\tilde{p}^{\tau,\bm{\xi}}(t,T,\bm{x},\bm{y})g(\bm{y}) \, d\bm{y}.\]
It then follows that
\begin{multline}\label{Proof:H\"older_for_Frozen_Semigroup_Decomp}
\bigl{\vert}D_{\bm{x}_1}\tilde{P}^{\tau,\bm{\xi}}_{t,T}g(\bm{x}) - D_{\bm{x}_1}\tilde{P}^{\tau,\bm{\xi}}_{t,T}g(\bm{x}')
\bigr{\vert} \, \le \,C\sum_{j=1}^n(T-t)^{j-1}\Bigl{\vert}\int_{\R^{nd}}\bigl[D_{\bm{y}_j}\tilde{p}^{\tau, \bm{\xi}}
(t,T,\bm{x},\bm{y})-D_{\bm{y}_j}\tilde{p}^{\tau,\bm{\xi}}(t,T,\bm{x}',\bm{y}) \bigr] g(\bm{y}) \,
d\bm{y}\Bigr{\vert}\\
=:\, C\sum_{j=1}^n (T-t)^{j-1} I^{od}_j.
\end{multline}
We are going to treat separately the cases $j=1$ and $j>1$ for the \emph{off-diagonal} contributions $\bigl(I^{od}_j\bigr)_{j\in \llbracket
1,n\rrbracket}$. Indeed, the function $g$ is differentiable only with respect to the first
component $\bm{y}_1$. In this first case, we can apply integration by parts formula to move the derivative on $g$, so that
\[I^{od}_1 \, = \, \Bigl{\vert}\int_{\R^{nd}} \bigl[\tilde{p}^{\tau,\bm{\xi}}(t,T,\bm{x},\bm{y}) -\tilde{p}^{\tau,\bm{\xi}}
(t,T,\bm{x}',\bm{y})\bigr] D_{\bm{y}_1}g(\bm{y}) \,d\bm{y}\Bigr{\vert}.\]
Noticing that $D_{\bm{y}_1}g$ is in $C^{\alpha+\beta-1}_{b,d}(\R^{nd})$ thanks to the reverse Taylor expansion (Lemma
\ref{lemma:Reverse_Taylor_Expansion}), the last expression can be then rewritten as
\begin{multline}\label{zz1}
I^{od}_1\le \, \Bigl{\vert}\int_{\R^{nd}} \tilde{p}^{\tau,\bm{\xi}}(t,T,\bm{x},\bm{y})\bigl[D_{\bm{y}_1}g(\bm{y})\pm D_{
\bm{y}_1}g(\tilde{\bm{m}}^{\tau,\bm{\xi}}_{t,T}(\bm{x}))\bigr] -\tilde{p}^{\tau,\bm{\xi}}(t,T,\bm{x}',\bm{y})\bigl[D_{
\bm{y}_1}g(\bm{y})\pm D_{\bm{y}_1}g(\tilde{\bm{m}}^{\tau,\bm{\xi}}_{t,T}(\bm{x}')) \bigr] \, d\bm{y}\Bigr{\vert} \\
\le \, C\Vert g \Vert_{C^{\alpha+\beta}_{b,d}}\Bigl{\{}\int_{\R^{nd}}
\bigl[\tilde{p}^{\tau,\bm{\xi}}(t,T,\bm{x},\bm{y})d^{\alpha+\beta-1}(\bm{y}, \tilde{\bm{m}}^{\tau,\bm{\xi}}_{t,T}(\bm{x}))
+\tilde{p}^{\tau,\bm{\xi}}(t,T,\bm{x}',\bm{y}) d^{\alpha+\beta-1}(\bm{y},\tilde{\bm{m}}^{\tau,\bm{\xi}}_{t,T}(\bm{x}'))\bigr]
\, d\bm{y} \\
+ d^{\alpha+\beta-1}(\tilde{\bm{m}}^{\tau,\bm{\xi}}_{t,T}(\bm{x}), \tilde{\bm{m}}^{\tau,\bm{\xi}}_{t,T}(\bm{x}'))\Bigr{\}}
\end{multline}
Now, we use the smoothing effect of $\tilde{p}^{\tau,\bm{\xi}}$ (Equation \eqref{eq:Smoothing_effects_of_tilde_p}) to control the two
integrals in the last expression, so that
\[I^{od}_1\, \le \, C\Vert g \Vert_{C^{\alpha+\beta}_{b,d}}\bigl[(T-t)^{\frac{\alpha+\beta-1}{\alpha}} +
d^{\alpha+\beta-1}(\tilde{\bm{m}}^{\tau,\bm{\xi}}_{t,T}(\bm{x}),\tilde{\bm{m}}^{\tau,\bm{\xi}}_{t,T}(\bm{x}'))\bigr].\]
We can then conclude the case $j=1$ recalling that the mapping $\bm{x}\to \tilde{\bm{m}}^{\tau,\bm{\xi}}_{t,T}(\bm{x})$ is affine (see
Equation \eqref{eq:def_tilde_m} for definition of $\tilde{\bm{m}}^{\tau,\bm{\xi}}_{t,T}(\bm{x})$) in
order to show that
\begin{equation}\label{Proof:H\"older_for_Frozen_Semigroup_Iod1}
I^{od}_1 \, \le \, C\Vert g \Vert_{C^{\alpha+\beta}_{b,d}}\bigl[(T-t)^{\frac{\alpha + \beta -1}{\alpha}} +
d^{\alpha+\beta-1}(\bm{x},\bm{x}')\bigr].
\end{equation}
Let us consider now the case $j>1$. Using the duality in Besov spaces (Equation \eqref{Besov:duality_in_Besov}) and the identification
\eqref{Besov:ident_Holder_Besov}, we can write from Equation \eqref{Proof:H\"older_for_Frozen_Semigroup_Decomp} that
\begin{multline}\label{Proof:H\"older_for_Frozen_Semigroup_Iodj}
I^{od}_j \, \le \, C\Vert g \Vert_{C^{\alpha+\beta}_{b,d}}\int_{\R^{(n-1)d}}  \Vert D_{\bm{y}_j}\tilde{p}^{\tau,\bm{\xi}}(t,T,\bm{x},
\bm{y}_{\smallsetminus j},\cdot)-D_{\bm{y}_j}\tilde{p}^{\tau,\bm{\xi}} (t,T,\bm{x}',\bm{y}_{\smallsetminus j},\cdot)
\Vert_{B^{-(\alpha_j+\beta_j)}_{1,1}} \, d\bm{y}_{\smallsetminus j}\\
\le \, C\Vert g \Vert_{C^{\alpha+\beta}_{b,d}}\int_{\R^{(n-1)d}}  \Vert D_{\bm{y}_j}\tilde{p}^{\tau,\bm{\xi}}(t,T,\bm{x},
\bm{y}_{\smallsetminus j},\cdot)\Vert_{B^{-(\alpha_j+\beta_j)}_{1,1}} + \Vert D_{\bm{y}_j}\tilde{p}^{\tau,\bm{\xi}}
(t,T,\bm{x}',\bm{y}_{\smallsetminus j},\cdot) \Vert_{B^{-(\alpha_j+\beta_j)}_{1,1}} \, d\bm{y}_{\smallsetminus j} \\
\le \, C\Vert g \Vert_{C^{\alpha+\beta}_{b,d}}(T-t)^{\frac{\alpha+\beta}{\alpha}-\frac{1}{\alpha_j}}
\end{multline}
where in the last inequality we applied the first Besov Control (Lemma \ref{lemma:First_Besov_COntrols}). Going back at equations
\eqref{Proof:H\"older_for_Frozen_Semigroup_Iod1} and \eqref{Proof:H\"older_for_Frozen_Semigroup_Iodj}, we finally conclude
that
\begin{multline}\label{a2}
\bigl{\vert}D_{\bm{x}_1}\tilde{P}^{\tau,\bm{\xi}}_{t,T}g(\bm{x}) - D_{\bm{x}_1}\tilde{P}^{\tau,\bm{\xi}}_{t,T}g(\bm{x}')
\bigr{\vert} \, \le \, C\Vert g\Vert_{C^{\alpha+\beta}_{b,d}}\bigl[(T-t)^{\frac{\alpha + \beta-1}{\alpha}}+d^{\alpha+\beta-1}(\bm{x},\bm{x}')
+\sum_{j=2}^n(T-t)^{j-1}(T-t)^{\frac{\alpha+\beta}{\alpha}-\frac{1}{\alpha_j}}\bigr] \\
\le \, C\Vert g\Vert_{C^{\alpha+\beta}_{b,d}} \bigl[(T-t)^{\frac{\alpha + \beta -1}{\alpha}}+d^{\alpha+\beta-1}(\bm{x},\bm{x}')
\bigr] \, \le \, C\Vert g\Vert_{C^{\alpha+\beta}_{b,d}}d^{\alpha+\beta-1}(\bm{x},\bm{x}')
\end{multline}
where in the last passage we used that $T-t \le c_0 d^\alpha(\bm{x},\bm{x}')$ for some $c_0\le1$.\newline
We focus now on the diagonal regime, i.e.\ when $T-t > c_0d^\alpha(\bm{x},\bm{x}')$. Remembering that we assumed that
$\bm{x}_j=\bm{x}'_j$ for any $j$ in $\llbracket 2,n\rrbracket$, we start using a Taylor expansion on
the density $\tilde{p}^{\tau,\bm{\xi}}$ with respect to the first, non-degenerate variable $\bm{x}_1$. Namely,
\begin{multline*}
D_{\bm{x}_1}\tilde{P}^{\tau,\bm{\xi}}_{t,T}g(\bm{x}) - D_{\bm{x}_1}\tilde{P}^{\tau,\bm{\xi}}_{t,T}g(\bm{x}') \, = \,
\int_{\R^{nd}}\bigl[D_{\bm{x}_1}\tilde{p}^{\tau,\bm{\xi}}(t,T,\bm{x},\bm{y})-D_{\bm{x}_1}\tilde{p}^{\tau,\bm{\xi}}
(t,T,\bm{x}',\bm{y})\bigr]g(\bm{y}) \, d\bm{y} \\
= \, \int_{\R^{nd}}\int_{0}^{1}D^2_{\bm{x}_1}\tilde{p}^{\tau,\bm{\xi}}\bigl(t,T,\bm{x}'+\lambda(\bm{x}-\bm{x}'),\bm{y}\bigr)
(\bm{x}-\bm{x}')_1 g(\bm{y}) \, d\lambda d\bm{y}.
\end{multline*}
Moreover, from the Scaling Lemma \ref{lemma:link_derivative_density}, it holds that
\[D^2_{\bm{x}_1}\tilde{p}^{\tau,\bm{\xi}}\bigl(t,T,\bm{x}'+\lambda(\bm{x}-\bm{x}'),\bm{y}\bigr) \, = \,
\sum_{j=1}^{n}C_j(T-t)^{j-1}D_{\bm{y}_j}D_{\bm{x}_1}\tilde{p}^{\tau,\bm{\xi}}\bigl(t,T,\bm{x}'+\lambda(\bm{x}-\bm{x}'),\bm{y}\bigr)\]
and we can use it to write
\begin{multline}\label{Proof:H\"older_Frozen_Semigroup_Decomp2}
\bigl{\vert}D_{\bm{x}_1}\tilde{P}^{\tau,\bm{\xi}}_{t,T}g(\bm{x}) - D_{\bm{x}_1}\tilde{P}^{\tau,\bm{\xi}}_{t,T}g(\bm{x}')
\bigr{\vert} \\
\le \,C\vert (\bm{x}-\bm{x}')_1\vert\sum_{j=1}^n(T-t)^{j-1}\Bigl{\vert}\int_{0}^{1}\int_{\R^{nd}}D_{\bm{y}_j}D_{\bm{x}_1} \tilde{p}^{\tau,\bm{\xi}}\bigl(t,T,\bm{x}'+\lambda(\bm{x}-\bm{x}'),\bm{y}\bigr)g(\bm{y})\, d\bm{y}d\lambda\Bigr{\vert} \\
=:\, C\vert(\bm{x}-\bm{x}')_1 \vert\sum_{j=1}^n (T-t)^{j-1} I^d_j.
\end{multline}
Similarly to the off-diagonal regime, we are going to treat separately the cases $j=1$ and $j>1$ for the \emph{diagonal} contributions
$\bigl(I^{d}_j\bigr)_{j\in \llbracket 1,n\rrbracket}$. In the first case, we can apply
integration by parts formula to show that
\[I^d_1 \, = \, \Bigl{\vert}\int_{0}^{1} \int_{\R^{nd}}
D_{\bm{x}_1}\tilde{p}^{\tau,\bm{\xi}}\bigl(t,T,\bm{x}'+\lambda(\bm{x}-\bm{x}'),\bm{y}\bigr)\otimes D_{\bm{y}_1}g(\bm{y}) \,
d\bm{y}d\lambda \Bigr{\vert}.\]
A cancellation argument with respect to $D_{\bm{x}_1}\tilde{p}^{\tau,\bm{\xi}}$ then leads to
\begin{multline*}
I^d_1 \, = \, \Bigl{\vert}\int_{0}^{1} \int_{\R^{nd}}D_{\bm{x}_1}\tilde{p}^{\tau,\bm{\xi}}(t,T,\bm{x}'+\lambda(\bm{x}-\bm{x}'),\bm{y}) \otimes\bigl[D_{\bm{y}_1}g(\bm{y})-D_{\bm{y}_1} g(\tilde{\bm{m}}^{\tau,\bm{\xi}}_{t,T}
(\bm{x}'+\lambda(\bm{x}-\bm{x}')))\bigr] \, d\bm{y}d\lambda\Bigr{\vert}\\
\le \, C\Vert g \Vert_{C^{\alpha+\beta}_{b,d}}\int_{0}^{1} \int_{\R^{nd}} \vert D_{\bm{x}_1}\tilde{p}^{\tau,\bm{\xi}}(t,T,\bm{x}'+\lambda(\bm{x}-\bm{x}'),\bm{y}) \vert d^{\alpha+\beta-1}\bigl(\bm{y},\tilde{\bm{m}}^{\tau,\bm{\xi}}_{t,T}
(\bm{x}'+\lambda(\bm{x}-\bm{x}'))\bigr) \, d\bm{y}d\lambda.
\end{multline*}
Since $\alpha + \beta -1 < \alpha$ by hypothesis (\textbf{P}), we can conclude using the smoothing effect
of $\tilde{p}^{\tau,\bm{\xi}}$ (Lemma \ref{lemma:Smoothing_effect_frozen}) to show that
\begin{equation}\label{Proof:H\"older_Frozen_Semigroup_Id1}
I^d_1 \, \le \, C\Vert g \Vert_{C^{\alpha+\beta}_{b,d}}(T-t)^{\frac{\alpha+\beta-2}{\alpha}}.
\end{equation}
For the case $j>1$, we use instead the duality in Besov spaces \eqref{Besov:duality_in_Besov} and the identification
\eqref{Besov:ident_Holder_Besov} to write
\begin{multline}\label{Proof:H\"older_Frozen_Semigroup_Idj}
I^d_j \, \le \, \int_{0}^{1} \int_{\R^{(n-1)d}} \Vert
D_{\bm{y}_j}D_{\bm{x}_1}\tilde{p}^{\bm{\xi}}(t,T,\bm{x}'+\lambda(\bm{x}-\bm{x}'),\bm{y}_{\smallsetminus j},\cdot)
\Vert_{B^{-(\alpha_j+\beta_j)}_{1,1}} \, d\bm{y}_{\smallsetminus j}d\lambda \\ \le \, C\Vert g \Vert_{C^{\alpha+\beta}_{b,d}}
(T-t)^{\frac{\alpha+\beta}{\alpha}-\frac{1}{\alpha_j}-\frac{1}{\alpha}}
\end{multline}
where in the last passage we applied the First Besov control (Lemma \ref{lemma:First_Besov_COntrols}). From equations
\eqref{Proof:H\"older_Frozen_Semigroup_Decomp2}, \eqref{Proof:H\"older_Frozen_Semigroup_Id1} and
\eqref{Proof:H\"older_Frozen_Semigroup_Idj}, it is possible to conclude that
\begin{multline*}
\bigl{\vert}D_{\bm{x}_1}\tilde{P}^{\bm{\xi}}_{t,T}g(\bm{x}) - D_{\bm{x}_1}\tilde{P}^{\bm{\xi}}_{t,T}g(\bm{x}') \bigr{\vert}
\, \le \, C\Vert g\Vert_{ C^{\alpha+\beta}_{b,d}}\vert (\bm{x}-\bm{x}')_1 \vert\sum_{j=1}^n (T-t)^{j-1}(T-t)^{\frac{\alpha+\beta
}{\alpha}- \frac{1}{\alpha_j}- \frac{1}{\alpha}}\\
\le \, C\Vert g \Vert_{C^{\alpha+\beta}_{b,d}}\vert (\bm{x}-\bm{x}')_1 \vert(T-t)^{\frac{\alpha+\beta-2}{\alpha}} \, \le \,
Cc^{\frac{\alpha+\beta-2}{\alpha}}_0\Vert g\Vert_{C^{\alpha+\beta}_{b,d}}d^{\alpha + \beta-1}(\bm{x},\bm{x}')
\end{multline*}
where in the last passage we used that $\vert(\bm{x}-\bm{x}')_1\vert=d(\bm{x},\bm{x}')$ and, since
$\frac{\alpha+\beta-2}{\alpha}<0$, that
\[\vert(\bm{x}-\bm{x}')_1\vert (T-t)^{\frac{\alpha+\beta-2}{\alpha}} \, \le \, c^{\frac{\alpha+\beta-2}{\alpha}}_0d^{\alpha +
\beta-1}(\bm{x},\bm{x}').\]
Remembering that $c_0$ is considered fixed and bigger then zero, the searched control follows immediately.
%%%%%%%%%%%%%%%%%%%%%%%%%%%%%%%%%%%%%%%%%%%%%%%%%%%%%%%%%%%%%%%%%%%%%%%%%%%%%%%%%%%%%%%%%%%%%%%%

\emph{Controls on frozen Green kernel.} In order to preserve the previous terminology of off-diagonal/diagonal regime for the frozen
semigroup, we have introduced the transition time $t_0$, defined in \eqref{eq:def_t0}. Then, while integrating in $s$ from $t$ to $T$, we
will say that the "local" off-diagonal regime holds for $\tilde{G}^{\tau,\bm{\xi}}$ if $s$ is in $[t,t_0]$ and that the "local" diagonal
regime holds if $s$ is in $[t_0,T]$. With the notations of \eqref{eq:def_Green_Kernel}, it seems quite natural now to decompose the
derivative of the frozen Green kernel with respect to $t_0$, i.e.
\[D_{\bm{x}_1}\tilde{G}^{\tau,\bm{\xi}}_{t,T}f(t,\bm{x}) \, = \, D_{\bm{x}_1}\tilde{G}^{\tau,\bm{\xi}}_{t,t_0}f(t,\bm{x})
+D_{\bm{x}_1}\tilde{G}^{\tau,\bm{\xi}}_{t_0,T}f(t,\bm{x}).\]
We remark however that the globally off-diagonal regime is considered in the above decomposition, too. Indeed, when $T-t\le
c_0d^\alpha(\bm{x},\bm{x}')$, $t_0$ coincides with $T$ and the second term on the right-hand side vanishes.\newline
We start considering the off-diagonal regime represented by $\bigl{\vert} D_{\bm{x}_1} \tilde{G}^{\tau,\bm{\xi}}_{t,t_0} f(t,\bm{x}) -
D_{\bm{x}_1}\tilde{G}^{\tau,\bm{\xi}}_{t,t_0}f(t,\bm{x}') \bigr{\vert}$. \newline
It holds that
\[\bigl{\vert} D_{\bm{x}_1}\tilde{G}^{\tau,\bm{\xi}}_{t,t_0}f(t,\bm{x})-D_{\bm{x}_1}\tilde{G}^{\tau,\bm{\xi}}_{t,t_0}
f(t,\bm{x}')\bigr{\vert} \, \le \, \int_{t}^{t_0}\Bigl[\bigl{\vert} D_{\bm{x}_1}\tilde{P}^{\tau,\bm{\xi}}_{t,s}f(s,\bm{x})\bigr{\vert}
+ \bigl{\vert} D_{\bm{x}_1} \tilde{P}^{\tau,\bm{\xi}}_{t,s}f(s,\bm{x}') \bigr{\vert}\Bigr] \, ds\]
We then use the control on the frozen semigroup (Equation \eqref{eq:Control_of_semigroup}) to find that
\[\bigl{\vert} D_{\bm{x}_1}\tilde{G}^{\tau,\bm{\xi}}_{t,t_0}f(t,\bm{x})-D_{\bm{x}_1}\tilde{G}^{\tau,\bm{\xi}}_{t,t_0}
f(t,\bm{x}') \bigr{\vert} \, \le \, C\Vert f \Vert_{L^{\infty}(C^\beta_{b,d})}\int_{t}^{t_0}(s-t)^{\frac{\beta -1}{\alpha}} \, ds  \,
\le \, C\Vert f \Vert_{L^{\infty}(C^\beta_{b,d})}(t_0-t)^{\frac{\beta + \alpha -1}{\alpha}}.\]
Our choice of $t_0$ (cf. Equation \eqref{eq:def_t0}) allows then to conclude that
\[\bigl{\vert}
D_{\bm{x}_1}\tilde{G}^{\tau,\bm{\xi}}_{t,t_0}f(t,\bm{x})-D_{\bm{x}_1}\tilde{G}^{\tau,\bm{\xi}}_{t,t_0}f(t,\bm{x}')
\bigr{\vert} \, \le \, C\Vert f \Vert_{L^{\infty}(C^\beta_{b,d})}d^{\beta + \alpha -1}(\bm{x},\bm{x}') \]
remembering that, by assumption, $c_0\le 1$. \newline
We can focus now on the diagonal regime represented by $\bigl{\vert} D_{\bm{x}_1} \tilde{G}^{\tau,\bm{\xi}}_{t_0,T}f(t,\bm{x}) -
D_{\bm{x}_1}\tilde{G}^{\tau,\tilde{\bm{\xi}}}_{t_0,T}f(t,\bm{x}')\bigr{\vert}$.\newline
We start applying a Taylor expansion on the derivative of the semigroup $\tilde{P}^{\tau,\bm{\xi}}f(t,\bm{x})$ so that
\begin{multline*}
\bigl{\vert}
D_{\bm{x}_1}\tilde{G}^{\tau,\bm{\xi}}_{t_0,T}f(t,\bm{x})-D_{\bm{x}_1}\tilde{G}^{\tau,\bm{\xi}}_{t_0,T}f(t,\bm{x}')
\bigr{\vert}\, = \, \Bigl{\vert}\int_{t_0}^{T} \Bigl[ D_{\bm{x}_1}\tilde{P}^{\tau,\bm{\xi}}_{t,s}f(s,\bm{x})
-D_{\bm{x}_1}\tilde{P}^{\tau, \bm{\xi}}_{t,s}f(s,\bm{x}')\Bigr] \,  ds \Bigr{\vert} \\
= \, \Bigl{\vert} \int_{t_0}^{T} \int_{0}^{1}D^2_{\bm{x}_1}\tilde{P}^{\tau,\bm{\xi}}_{t,s}f(s,\bm{x}+\lambda(\bm{x}'-\bm{x}))
(\bm{x}'-\bm{x})_1 \,d\lambda ds \Bigr{\vert}.
\end{multline*}
Then, the Fubini theorem and the control on the frozen semigroup (Equation \eqref{eq:Control_of_semigroup}) allow us to write that
\begin{multline*}
\bigl{\vert} D_{\bm{x}_1}\tilde{G}^{\tau,\bm{\xi}}_{t_0,T}f(t,\bm{x})-D_{\bm{x}_1}\tilde{G}^{\tau,\bm{\xi}}_{t_0,T}
f(t,\bm{x}') \bigr{\vert} \, \le \, C\Vert f \Vert_{L^{\infty}(C^\beta_{b,d})}\vert (\bm{x}-\bm{x}')_1 \vert
\int_{t_0}^{T}(s-t)^{\frac{\beta -2}{\alpha}} \, ds  \\
\le \, C\Vert f\Vert_{L^{\infty}(C^\beta_{b,d})}\vert (\bm{x}-\bm{x}')_1\vert\bigl[(s-t)^{\frac{\alpha+\beta-2}{\alpha}}\bigr]_{
t_0}^T.
\end{multline*}
Since by hypothesis (\textbf{P}) in assumption (\textbf{A}), it holds that $\alpha+\beta -2<0$, it follows that
\[\bigl{\vert}D_{\bm{x}_1}\tilde{G}^{\tau,\bm{\xi}}_{t_0,T}f(t,\bm{x})-D_{\bm{x}_1}\tilde{G}^{\tau,\bm{\xi}}_{t_0,T}
f(t,\bm{x}')\bigr{\vert} \, \le \, C\Vert f\Vert_{L^{\infty}(C^\beta_{b,d})}\vert (\bm{x}-\bm{x}')_1\vert(t_0-t)^{\frac{
\alpha +\beta -2}{\alpha}}.\]
Using that $\vert (\bm{x}-\bm{x}')_1\vert =d(\bm{x},\bm{x}')$ and remembering our choice of $t_0$ in \eqref{eq:def_t0}, we can then conclude
that
\[\bigl{\vert} D_{\bm{x}_1}\tilde{G}^{\tau,\bm{\xi}}_{t_0,T}f(t,\bm{x})-D_{\bm{x}_1} \tilde{G}^{\tau,\bm{\xi}}_{t_0,T}
f(t,\bm{x}') \bigr{\vert} \, \le \, Cc_0^{\frac{\alpha+\beta-2}{\alpha}}\Vert f\Vert_{L^{\infty}(C^\beta_{b,d})}d^{\alpha+\beta-1}
(\bm{x},\bm{x}').\]

\paragraph{Proof of Lemma \ref{lemma:Holder_modulus_proxy_Deg}}
\emph{Controls on frozen semigroup.} Using the change of variables $\bm{z} =
\tilde{\bm{m}}^{\tau,\bm{\xi}}_{t,T}(\bm{x})-\bm{y}$, we can rewrite $\tilde{P}^{\tau,\bm{\xi}}_{t,T} g(\bm{x})$ as
\begin{multline*}
\tilde{P}^{\tau,\bm{\xi}}_{t,T}g(\bm{x}) \, = \, \int_{\R^{nd}}\tilde{p}^{\tau,\bm{\xi}}(t,T,\bm{x},\bm{y})g(\bm{y}) \,
d\bm{y} \, = \, \int_{\R^{nd}}\frac{1}{\det\bigl(\mathbb{M}_{T-t}\bigr)}p_S(T-t,\mathbb{M}^{-1}_{T-t}\bigl(
\tilde{\bm{m}}^{\tau,\bm{\xi}}_{t,T}(\bm{x})-\bm{y}\bigr)g(\bm{y})\,d\bm{y} \\
=\,\int_{\R^{nd}}\frac{1}{\det\bigl(\mathbb{M}_{T-t}\bigr)}p_S(T-t,\mathbb{M}^{-1}_{T-t}\bm{z}\bigr)g(\tilde{\bm{m}}^{\tau,\bm{\xi}}_{t,T}
(\bm{x})-\bm{z}) \, d\bm{z}.
\end{multline*}
It then follows that
\[\bigl{\vert}\tilde{P}^{\tau,\bm{\xi}}_{t,T}g(\bm{x}) - \tilde{P}^{\tau,\bm{\xi}}_{t,T}g(\bm{x}')\bigr{\vert} \, =\,
\Bigl{\vert}\int_{\R^{nd}}\frac{1}{\det\bigl(\mathbb{M}_{T-t}\bigr)}p_S\bigl(T-t,\mathbb{M}^{-1}_{T-t}\bm{z}\bigr)
\bigl[g(\tilde{\bm{m}}^{\tau,\bm{\xi}}_{t,T}(\bm{x})-\bm{z})-g(\tilde{\bm{m}}^{\tau,\bm{\xi}}_{t,T}(\bm{x}')-\bm{z})\bigr]
\,dz\Bigr{\vert}.\]
We observe now that the function $\bm{x}\to \tilde{\bm{m}}^{\tau,\bm{\xi}}_{t,T}(\bm{x})$ is affine (cf. Equation \eqref{eq:def_tilde_m})
and thus, that
\[\bigl(\tilde{\bm{m}}^{\tau,\bm{\xi}}_{t,T}(\bm{x})-\bm{z}\bigr)_1 \, = \,
\bigl(\tilde{\bm{m}}^{\tau,\bm{\xi}}_{t,T}(\bm{x}')-\bm{z}\bigr)_1\]
since $\bm{x}_1=\bm{x}'_1$. It then holds that
\[\bigl{\vert} g(\tilde{\bm{m}}^{\tau,\bm{\xi}}_{t,T}(\bm{x})-\bm{z})-g(\tilde{\bm{m}}^{\tau,\bm{\xi}}_{t,T}(\bm{x}')-\bm{z})\bigr{\vert} \,
\le \, C\Vert g \Vert_{C^{\alpha+\beta}_{b,d}}d^{\alpha+\beta}\bigl(\tilde{\bm{m}}^{\tau,\bm{\xi}}_{t,T}(\bm{x}),
\tilde{\bm{m}}^{\tau,\bm{\xi}}_{t,T}(\bm{x}')\bigr) \, \le\, C\Vert g \Vert_{C^{\alpha+\beta}_{b,d}}d^{\alpha+\beta}(\bm{x}, \bm{x}').\]
Hence, we can conclude using it to write
\begin{multline*}
\bigl{\vert}\tilde{P}^{\tau,\bm{\xi}}_{t,T}g(\bm{x}) - \tilde{P}^{\tau,\bm{\xi}}_{t,T}g(\bm{x}')\bigr{\vert} \, \le \, C\Vert g
\Vert_{C^{\alpha+\beta}_{b,d}}d^{\alpha+\beta}(\bm{x}, \bm{x}')\int_{\R^{nd}} \frac{1}{\det\bigl(\mathbb{M}_{T-t}\bigr)}
p_S\bigl(T-t,\mathbb{M}^{-1}_{T-t} z\bigr) \,dz\\
\le \, C\Vert g \Vert_{C^{\alpha+\beta}_{b,d}}d^{\alpha+\beta}(\bm{x}, \bm{x}').
\end{multline*}
%%%%%%%%%%%%%%%%%%%%%%%%%%%%%%%%%%%%%%%%%%%%%%%%%%%%%%%%%%%%%%%%%%%%%%%%%%%%%%%%%%%%%%%%%%%%

\emph{Controls on frozen Green kernel.} We will assume the same notations appeared in the previous lemma for the frozen Green Kernel. In
particular, we decompose the frozen Green Kernel as
\[\tilde{G}^{\tau,\bm{\xi}}_{t,T}f(t,\bm{x}) \, = \, \tilde{G}^{\tau,\bm{\xi}}_{t,t_0}f(t,\bm{x})+
\tilde{G}^{\tau,\bm{\xi}}_{t_0,T}f(t,\bm{x})\]
with $t_0$ defined in \eqref{eq:def_t0}. \newline
We start rewriting the off-diagonal regime contribution as
\begin{multline*}
\bigl{\vert} \tilde{G}^{\tau,\bm{\xi}}_{t,t_0}f(t,\bm{x})-\tilde{G}^{\tau,\bm{\xi}}_{t,t_0}f(t,\bm{x}') \bigr{\vert} \\
= \,\Bigl{\vert}\int_{t}^{t_0}\int_{\R^{nd}}\tilde{p}^{\tau,\bm{\xi}}(t,s,\bm{x},\bm{y})\bigl[f(s,\bm{y})\pm
f(s,\tilde{\bm{m}}^{\bm{\xi}}_{t,s}(\bm{x}))\bigr] -
\tilde{p}^{\tau,\bm{\xi}}(t,s,\bm{x}',\bm{y}) \bigl[f(s,\bm{y})\pm f(s,\tilde{\bm{m}}^{\tau,\bm{\xi}}_{t,s}(\bm{x}'))\bigr]\,
d\bm{y}ds\Bigr{\vert} \\
\le\,\Bigl{\vert}\int_{t}^{t_0}\int_{\R^{nd}}\tilde{p}^{\tau,\bm{\xi}}(t,s,\bm{x},\bm{y})\bigl[f(s,\bm{y})-
f(s,\tilde{\bm{m}}^{\tau,\bm{\xi}}_{t,s}(\bm{x}))\bigr] \,d\bm{y}ds
-\tilde{p}^{\tau,\bm{\xi}}(t,s,\bm{x}',\bm{y})\bigl[f(s,\bm{y})- f(s,\tilde{\bm{m}}^{\tau,\bm{\xi}}_{t,s}(\bm{x}'))\bigr]\,
d\bm{y}ds\Bigr{\vert} \\
+ \Bigl{\vert}\int_{t}^{t_0} f(s,\tilde{\bm{m}}^{\tau,\bm{\xi}}_{t,s}(\bm{x}))-f(s,\tilde{\bm{m}}^{\tau,\bm{\xi}}_{t,s}(\bm{x}')) \, ds
\Bigr{\vert}.
\end{multline*}
We can then use the smoothing effect for $\tilde{p}^{\tau,\bm{\xi}}$ (Equation \eqref{eq:Smoothing_effects_of_tilde_p} in
Lemma \ref{lemma:Smoothing_effect_frozen}) to show that
\begin{equation}\label{zz2}
\bigl{\vert} \tilde{G}^{\tau,\bm{\xi}}_{t,t_0}f(t,\bm{x})-\tilde{G}^{\tau,\bm{\xi}}_{t,t_0}f(t,\bm{x}') \bigr{\vert} \, \le \,
C\Vert f \Vert_{L^{\infty}(C^\beta_{b,d})}\int_{t}^{t_0}\bigl[(s-t)^{\beta/\alpha} +d^\beta(\tilde{\bm{m}}^{\tau,\bm{\xi}}_{t,s}
(\bm{x}),\tilde{\bm{m}}^{\tau,\bm{\xi}}_{t,s}(\bm{x}'))\bigr] \, ds.
\end{equation}
Recalling from Equation \eqref{eq:def_tilde_m} that $\bm{x}\to \tilde{\bm{m}}^{\tau,\bm{\xi}}_{t,s}(\bm{x})$ is affine, it follows that
\begin{multline*}
\bigl{\vert} \tilde{G}^{\tau,\bm{\xi}}_{t,t_0}f(t,\bm{x})-\tilde{G}^{\tau,\bm{\xi}}_{t,t_0}f(t,\bm{x}') \bigr{\vert} \, \le \,
C\Vert f\Vert_{L^{\infty}(C^\beta_{b,d})}\int_{t}^{t_0}\bigl[(s-t)^{\beta/\alpha} +d^\beta(\bm{x},\bm{x}')\bigr] \, ds \\
\le \, C\Vert f \Vert_{L^{\infty}(C^\beta_{b,d})}\bigl[(t_0-t)d^\beta(\bm{x},\bm{x}') + (t_0-t)^{\frac{\beta +\alpha}{\alpha}}\bigr].
\end{multline*}
Using that $t_0-t\le c_0d^\alpha(\bm{x},\bm{x}')$ for some $c_0\le 1$, we can finally conclude that
\[\bigl{\vert} \tilde{G}^{\tau,\bm{\xi}}_{t,t_0}f(t,\bm{x})-\tilde{G}^{\tau,\bm{\xi}}_{t,t_0}f(t,\bm{x}') \bigr{\vert} \, \le \,
C\Vert f\Vert_{L^{\infty}(C^\beta_{b,d})}d^{\beta + \alpha}(\bm{x},\bm{x}').\]
Now, we can focus our analysis to the diagonal regime contribution, i.e. $\bigl{\vert} \tilde{G}^{\tau,\bm{\xi}}_{t_0,T} f(t,\bm{x}) -
\tilde{G}^{\tau,\bm{\xi}}_{t_0,T}f(t,\bm{x}')\bigr{\vert}$.\newline
We start applying a Taylor expansion on the frozen semigroup $\tilde{P}^{\tau,\bm{\xi}}_{t,s}f$ with respect to the $i$-th variable
$\bm{x}_i$, which is, by hypothesis, the only one for which the entries of $\bm{x}$ and $\bm{x}'$ differ.
Namely,
\begin{multline*}
\bigl{\vert}\tilde{G}^{\tau,\bm{\xi}}_{t_0,T}f(t,\bm{x})-\tilde{G}^{\tau,\bm{\xi}}_{t_0,T}f(t,\bm{x}') \bigr{\vert} \, = \,
\bigl{\vert} \int_{t_0}^{T} \tilde{P}^{\tau,\bm{\xi}}_{t,s}f(s,\bm{x}) - \tilde{P}^{\tau,\bm{\xi}}_{t,s}f(s,\bm{x}') \,  ds
\bigr{\vert} \\
= \, \bigl{\vert} \int_{t_0}^{T} \int_{0}^{1}D_{\bm{x}_i}\tilde{P}^{\tau,\bm{\xi}}_{t,s}f(s,\bm{x}+\lambda(\bm{x}'-\bm{x}))
\cdot(\bm{x}'-\bm{x})_i \, d\lambda ds \bigr{\vert}.
\end{multline*}
The control on the frozen semigroup (Equation \eqref{eq:Control_of_semigroup}) then implies that
\begin{equation}\label{a3}
\bigl{\vert}\tilde{G}^{\tau,\bm{\xi}}_{t_0,T}f(t,\bm{x})-\tilde{G}^{\tau,\bm{\xi}}_{t_0,T}f(t,\bm{x}') \bigr{\vert} \, \le \,
C\Vert f\Vert_{L^{\infty}(C^\beta_{b,d})}\vert (\bm{x}-\bm{x}')_i\vert\int_{t_0}^{T}(s-t)^{\frac{\beta}{\alpha}-\frac{1}{\alpha_i}}
\, ds.
\end{equation}
Noticing from assumption (\textbf{P}) that $\beta+\alpha-1-\alpha(i-1) <0$ for $i\ge 2$, it holds that
\[\int_{t_0}^{T}(s-t)^{\frac{\beta}{\alpha}-\frac{1}{\alpha_i}}\, ds \, = \, \int_{t_0}^{T}(s-t)^{\frac{\beta-[1+\alpha(i-1)]}{\alpha}}\, ds
\, \le \, C\Bigl[-(s-t)^{\frac{\beta+\alpha-1-\alpha(i-1)}{\alpha}}\Bigr]_{t_0}^T\, \le \,C(t_0-t)^{\frac{\beta-1-\alpha(i-2)}{\alpha}}.\]
Using that $\vert (\bm{x}-\bm{x}')_i\vert=d^{1+\alpha(i-1)}(\bm{x},\bm{x}')$ and our choice of $t_0$ (cf. Equation \eqref{eq:def_t0}), we
can then conclude from \eqref{a3} that
\[\bigl{\vert} \tilde{G}^{\tau,\bm{\xi}}_{t_0,T}f(t,\bm{x})-\tilde{G}^{\tau,\bm{\xi}}_{t_0,T}f(t,\bm{x}') \bigr{\vert} \, \le  \,
Cc^{\frac{\beta-1-\alpha(i-2)}{\alpha}}_0\Vert f \Vert_{L^{\infty}(C^\beta_{b,d})}d^{\alpha+\beta}(\bm{x},\bm{x}') \, \le \, Cc^{\frac{\beta-\gamma_i}{\alpha}}_0\Vert f \Vert_{L^{\infty}(C^\beta_{b,d})}d^{\alpha+\beta}(\bm{x},\bm{x}'),\]
remembering the definition of $\gamma_i$ in \eqref{Drift_assumptions}.

\section{A Priori Estimates}

Since the aim of this section is to prove Proposition \ref{prop:A_Priori_Estimates}, we will assume tacitly from this point further
that assumption (\textbf{A'}) holds. \textcolor{black}{Moreover, we recall here that we are throughout this section considering the regularized framework of Section $3.2$.}

\textbf{WARNING:} For notational simplicity, we drop here the sub-scripts and the superscripts in $m$ associated with the
regularization. For any fixed $(\tau,\bm{\xi})$ in $[0,T]\times\R^{nd}$, we rewrite, with some abuse in notations, the Duhamel expansion
(Equation \eqref{eq:Expansion_along_proxy}) as:
\begin{equation}\label{align:Representation2}
u(t,\bm{x}) \, = \, \tilde{u}^{\tau,\bm{\xi}}(t,\bm{x}) + \int_{t}^{T}\tilde{P}^{\tau,\bm{\xi}}_{t,s}R^{\tau,\bm{\xi}}(s,\bm{x})\, ds,
\end{equation}
where $\tilde{u}^{\tau,\bm{\xi}}$ is defined through the Duhamel representation \eqref{Duhamel_representation_of_proxy} and
\[R^{\tau,\bm{\xi}}(t,\bm{x}) \, = \, \bigl{\langle}\bm{F}(t,\bm{x})-\bm{F}(t,\bm{\theta}_{\tau,t}(\bm{\xi})), D_{\bm{x}}u(t,\bm{x})
\bigr{\rangle}, \quad (t,\bm{x}) \, \in \, (0,T)\times\R^{nd}.\]
It is however important to keep in mind that $f$, $g$, $\bm{F}$ are now smooth and bounded functions so that all the terms above are clearly
defined. \textcolor{black}{ We recall however that we aim at obtaining controls in the $L^\infty(C^{\alpha+\beta}_{b,d})$-norm, uniformly with respect to the regularization parameter.} \newline
From the expansion above, we know moreover that for any $(t,\bm{\xi})$ in $[0,T]\times \R^{nd}$, it holds  that
\begin{equation}\label{align:Representation2_deriv}
D_{\bm{x}_1}u(t,\bm{x}) \, =\,
D_{\bm{x}_1}\tilde{u}^{\tau,\bm{\xi}}(t,\bm{x})+\int_{t}^{T}D_{\bm{x}_1}\tilde{P}^{\tau,\bm{\xi}}_{t,s}R^{\tau,\bm{\xi}}(s,\bm{x})\, ds.
\end{equation}
As seen in the previous section, these decompositions will allow us to control $u$ in norm $L^\infty(0,T;C^{\alpha+\beta}_{b,d}(\R^{nd}))$
analyzing separately the contributions from the Duhamel representation $\tilde{u}^{\tau,\bm{\xi}}$ and those from the expansion error $R^{\tau,\bm{\xi}}(t,\bm{x})$ for suitable choices of freezing parameters $(\tau,\bm{\xi})$.

\subsection{Second Besov Control}
This sub-section focuses on the contribution associated with the remainder term $R^{m,\tau,\bm{\xi}}$ appearing in the
Duhamel-type expansion \eqref{align:Representation2}. We recall that we aim at controlling it with the $L^\infty(C^{\alpha+\beta}_{b,d})$-norm of the coefficients, uniformly in the regularization parameter.  Let us start decomposing it through
\[\Bigl{\vert}\int_{t}^{T}  \tilde{P}^{\tau,\bm{\xi}}_{t,s}R^{\tau,\bm{\xi}}(s,\bm{x})\, ds\Bigr{\vert}\, = \, \Bigl{\vert}
\sum_{j=1}^{n}\int_{t}^{T}\int_{\R^{nd}}\tilde{p}^{\tau,\bm{\xi}}(t,s,\bm{x},\bm{y})\bigl[\bm{F}_j(s,\bm{y})-\bm{F}_j(s,\bm{\theta}_{\tau,s}
(\bm{\xi}))\bigr]\cdot D_{\bm{y}_j}u(s,\bm{y}) \, d\bm{y}ds\Bigl{\vert}.\]
We then notice that the non-degenerate contribution in the sum (corresponding to the index $j=1$) can be treated easily remembering that
$u$ is differentiable with respect to the first component with a bounded derivative.
Indeed, using the smoothing effect for the frozen density $\tilde{p}^{\tau,\bm{\xi}}$ (Equation \eqref{eq:Smoothing_effects_of_tilde_p}), it
holds that
\begin{multline*}
\Bigl{\vert}\int_{t}^{T}\int_{\R^{nd}}\tilde{p}^{\tau,\bm{\xi}}(t,s,\bm{x},\bm{y})\bigl[\bm{F}_1(s,\bm{y})-\bm{F}_1(s,\bm{\theta}_{\tau,s}
(\bm{\xi}))\bigr]\cdot D_{\bm{y}_1}u(s,\bm{y}) \, d\bm{y}ds\Bigl{\vert} \\
\le \, C\Vert D_{\bm{y}_1}u(s,\bm{y})\Vert_{l^\infty(L^\infty)}\Vert \bm{F}\Vert_H \int_{t}^{T}\int_{\R^{nd}}\tilde{p}^{\tau,\bm{\xi}}
(t,s,\bm{x},\bm{y})d^{\alpha+\beta}\bigl(\bm{y},\bm{\theta}_{\tau,s}(\bm{\xi})\bigr) \, d\bm{y}ds \\
\le \, C\Vert D_{\bm{y}_1}u(s,\bm{y})\Vert_{l^\infty(L^\infty)}\Vert \bm{F}\Vert_H \int_{t}^{T} (s-t)^{\frac{\beta}{\alpha}} \, ds \, \le \,
C\Vert D_{\bm{y}_1}u(s,\bm{y})\Vert_{l^\infty(L^\infty)}\Vert \bm{F}\Vert_H(T-t)^{\frac{\alpha+\beta}{\alpha}}.
\end{multline*}
In order to deal with the degenerate indexes, we will use, similarly to the previous subsection, a reasoning in Besov spaces. Since $u$ is
not differentiable with respect to $\bm{y}_j$ if $j>1$, we move the
derivative to the other terms using integration by parts formula:
\[\Bigl{\vert}\int_{t}^{T}\int_{\R^{nd}}D_{\bm{y}_j}\cdot\Bigl{\{}\tilde{p}^{\tau,\bm{\xi}}(t,s,\bm{x},\bm{y})\bigl[\bm{F}_j(s,\bm{y})-\bm{F}_j
(s,\bm{\theta}_{\tau,s}(\bm{\xi}))\bigr]\Bigr{\}} u(s,\bm{y}) \, d\bm{y}ds\Bigl{\vert}.\]
In order to rely again on the duality in Besov spaces \eqref{Besov:duality_in_Besov}, we rewrite the above expression as
\begin{multline*}
  \Bigl{\vert}\int_{t}^{T}\int_{\R^{nd}}D_{\bm{y}_j}\cdot\Bigl{\{}\tilde{p}^{\tau,\bm{\xi}}(t,s,\bm{x},\bm{y})\bigl[\bm{F}_j(s,\bm{y})-\bm{F}_j
(s,\bm{\theta}_{\tau,s}(\bm{\xi}))\bigr]\Bigr{\}} u(s,\bm{y}) \, d\bm{y}ds\Bigl{\vert} \,\le \\
\int_{t}^{T}\int_{\R^{(n-1)d}}\Bigl{\Vert}D_{\bm{y}_j}\cdot\Bigl{\{}\tilde{p}^{\tau,\bm{\xi}}(t,s,\bm{x},\bm{y}_{\smallsetminus j},
\cdot)\bigl[\bm{F}_j (s,\bm{y}_{\smallsetminus j},\cdot)-\bm{F}_j(s,\bm{\theta}_{\tau,s}(\bm{\xi})) \bigr]\Bigr{\}} \Bigr{\Vert
}_{B^{-(\alpha_j+\beta_j)}_{1,1}}\Vert u(s,\bm{y}_{\smallsetminus j},\cdot) \Vert_{B^{\alpha_j+\beta_j}_{\infty,\infty}} \,
d\bm{y}_{\smallsetminus j}ds.
\end{multline*}
Remembering identification \eqref{Besov:ident_Holder_Besov}, it holds now that
\begin{multline*}
\Bigl{\vert}\int_{t}^{T}\int_{\R^{nd}}D_{\bm{y}_j}\cdot\Bigl{\{}\tilde{p}^{\tau,\bm{\xi}}(t,s,\bm{x},\bm{y})\bigl[\bm{F}_j(s,\bm{y})-\bm{F}_j
(s,\bm{\theta}_{\tau,s}(\bm{\xi}))\bigr]\Bigr{\}} u(s,\bm{y}) \, d\bm{y}ds\Bigl{\vert} \\
\le \,\Vert u\Vert_{L^\infty(C^{\alpha+\beta}_{b,d})} \int_{t}^{T}\int_{\R^{(n-1)d}}\Bigl{\Vert}D_{\bm{y}_j} \cdot\Bigl{\{}
\tilde{p}^{\tau,\bm{\xi}}(t,s,\bm{x}, \bm{y}_{\smallsetminus j}, \cdot)\bigl[\bm{F}_j (s,\bm{y}_{\smallsetminus
j},\cdot)-\bm{F}_j(s,\bm{\theta}_{ \tau,s} (\bm{\xi})) \bigr]\Bigr{\}} \Bigr{\Vert}_{B^{-(\alpha_j+\beta_j)}_{1,1}} \,
d\bm{y}_{\smallsetminus j}ds.
\end{multline*}
It then remains to control the integral of the Besov norm above. To do that, we will need a refinement of the smoothing effect
\eqref{eq:Smoothing_effects_of_tilde_p} that involves only partial differences of variables. For a fixed $i$ in $\llbracket2,n\rrbracket$,
we start denoting  by $d_{i:n}(\cdot,\cdot)$ the part of the anisotropic distance considering only the last $n-(i-1)$ variables. Namely,
\[d_{i:n}( \bm{x},\bm{x}') \, := \, \sum_{j=i}^{n} \vert(\bm{x}-\bm{x}')_j\vert^{\frac{1}{1+\alpha(j-1)}}.\]

\begin{lemma}[Partial Smoothing Effect]
\label{lemma:Partial_Smoothing_Effect}
Let $i$ be in $\llbracket2,n\rrbracket$, $\gamma$ in $(0,1\wedge \alpha(1+\alpha(i-1)))$ and $\vartheta$, $\varrho$ two $n$-multi-indexes
such that $\vert \vartheta +\varrho \vert\le 3$. Then, there exists a constant $C:=C(\vartheta,\varrho,\gamma)$ such that for any $t<s$ in
$[0,T]$, any $\bm{x}$ in $\R^{nd}$,
\begin{equation}\label{eq:Partial_Smoothing_Effect}
\int_{\R^{nd}} \vert D^\varrho_{\bm{y}}D^\vartheta_{\bm{x}} \tilde{p}^{ \tau,\bm{\xi}} (t,s,\bm{x},\bm{y}) \vert d^\gamma_{i:n}
\bigl(\bm{y},\bm{\theta}_{\tau,s}(\bm{\xi})\bigr) \, d\bm{y} \, \le \, C (s-t)^{\frac{\gamma}{\alpha}-\sum_{i=k}^{n}
\frac{\vartheta_k+\varrho_k}{\alpha_k}}
\end{equation}
taking $(\tau,\bm{\xi})=(t,x)$.
\end{lemma}

The above assumption on $\gamma$ should not appear to much strange. Indeed, in the partial distance $d^{\gamma}_{i:n}( \bm{x},\bm{x}')$, the
stronger term to be integrated  is at level $i$ with intensity of order $\gamma/(1+\alpha(i-1))$. Since by the smoothing effect (Equation
\eqref{eq:Smoothing_effects_of_tilde_p}) of the frozen density, we know we can integrate against $\tilde{p}^{\tau,\bm{\xi}}$ contributions
of order up to $\alpha$, so it appears the condition $\gamma<\alpha(1+\alpha(i-1))$.\newline
A proof of this result can be obtained mimicking with slightly modifications the proof in Lemma \ref{lemma:Smoothing_effect_frozen}.

As done above for the first Besov control, we will however state the result considering a possibly additional derivative with respect to
$\bm{x}_1$. Namely, we would like to control the following:
\[D_{\bm{y}_j}\cdot\Bigl{\{}D^{\vartheta}_{\bm{x}}\tilde{p}^{\tau,\bm{\xi}}(t,s,\bm{x},
\bm{y}_{\smallsetminus j},\cdot) \otimes \bigl[\bm{F}_j (s,\bm{y}_{\smallsetminus j},\cdot)-\bm{F}_j(s,\bm{\theta}_{\tau,s}(\bm{\xi}))
\bigr]\Bigr{\}}\]
where we have denoted as in \eqref{eq:notation_smallsetminus}, $\bm{F}_j (s,\bm{y}_{\smallsetminus j},\cdot) := \bm{F}_j
(s,\bm{y}_{1},\dots,\bm{y}_{j-1},\cdot,\bm{y}_{j+1},\dots,\bm{y}_n)$ and, with a slightly abuse of notation, by $D_{\bm{y}_j}\cdot$ an
extended form of the divergence over the $j$-th variable. In other words, this "enhanced" divergence form decreases by one the order of the
input tensor.

\begin{lemma}[Second Besov Control]
\label{lemma:Second_Besov_COntrols}
Let $j$ be in $\llbracket 2,n\rrbracket$ and $\vartheta$ a multi-index in $\N^n$ such that $\vert \vartheta \vert\le 2$. Under
(\textbf{A}'), there exists a constant $C:=C(j,\vartheta)$ such that for any $\bm{x}$ in $\R^{nd}$ and any $t<s$ in $[0,T]$
\begin{multline*}
\int_{\R^{(n-1)d}}\Bigl{\Vert}D_{\bm{y}_j}\cdot\Bigl{\{}D^{\vartheta}_{\bm{x}}\tilde{p}^{\tau,\bm{\xi}}(t,s,\bm{x},
\bm{y}_{\smallsetminus j},\cdot) \otimes \bigl[\bm{F}_j (s,\bm{y}_{\smallsetminus j},\cdot)-\bm{F}_j(s,\bm{\theta}_{\tau,s}(\bm{\xi}))
\bigr]\Bigr{\}} \Bigr{\Vert}_{B^{-(\alpha_j+\beta_j)}_{1,1}} \, d\bm{y}_{\smallsetminus j}\\
\le \, C\Vert \bm{F}
\Vert_H(s-t)^{\frac{\beta}{\alpha} -\sum_{k=1}^{n}\frac{\vartheta_k}{\alpha_k}}
\end{multline*}
taking $(\tau,\bm{\xi})=(t,\bm{x})$.
\end{lemma}
\begin{proof}
To control the Besov norm in $B^{-(\alpha_j+\beta_j)}_{1,1}(\R^d)$, we are going to use the stable thermic characterization
\eqref{alpha-thermic_Characterization} with $\tilde{\gamma}=-(\alpha_j+\beta_j)$. Since the first term can be controlled as in the First
Besov Control (Lemma \ref{lemma:First_Besov_COntrols}), we will
focus on the second one, i.e.
\[\int_{0}^{1}v^{\frac{\alpha_j+\beta_j}{\alpha}}\int_{\R^d}\Bigl{\vert}\int_{\R^d}\partial_vp_h(v,z-\bm{y}_j)D_{\bm{y}_j}\cdot
\Bigl{\{}D^\vartheta_{\bm{x}}\tilde{p}^{\tau,\bm{\xi}}(t,s,\bm{x},\bm{y})\otimes\bigl[\bm{F}_j(s,\bm{y})-\bm{F}_j(s,\bm{\theta}_{\tau,s}(\bm{
\xi}))\bigr]\Bigr{\}}
\,d\bm{y}_j \Bigr{\vert} \, dzdv.\]
We start applying integration by parts formula noticing that $D_{\bm{y}_j}p_h(v,z-\bm{y}_j)=-D_z p_h(v,z-\bm{y}_j)$, to write that
\[\int_{0}^{1}v^{\frac{\alpha_j+\beta_j}{\alpha}}\int_{\R^d}\Bigl{\vert}\int_{\R^d}D_z\partial_vp_h(v,z-\bm{y}_j)\cdot \Bigl{\{}D^\vartheta_{
\bm{x}} \tilde{p}^{\tau,\bm{\xi}}(t,s,\bm{x},\bm{y})\otimes\bigl[\bm{F}_j(s,\bm{y})-\bm{F}_j(s,\bm{\theta}_{\tau,s}(\bm{\xi}))\bigr]\Bigr{\}}
 \,d\bm{y}_j \Bigr{\vert} \,dzdv.\]
Fixed a constant $\delta_j\ge1$ to be chosen later, we then split the above integral with respect to $v$ into two components
\begin{multline*}
\int_{0}^{(s-t)^{\delta_j}}v^{\frac{\alpha_j+\beta_j}{\alpha}}\int_{\R^d}\Bigl{\vert}\int_{\R^d}D_{z}\partial_vp_h
(v,z-\bm{y}_j)\cdot\Bigl{\{}D^\vartheta_{\bm{x}}\tilde{p}^{\tau,\bm{\xi}}(t,s,\bm{x},\bm{y})\bigl[\bm{F}_j(s,\bm{y})-\bm{F}_j(s,\bm{
\theta}_{\tau,s}(\bm{\xi}))\bigr] \Bigr{\}}\,d\bm{y}_j  \Bigr{\vert} \, dzdv \\
+\int_{(s-t)^{\delta_j}}^{1}v^{\frac{\alpha_j+\beta_j}{\alpha}}\int_{\R^d}\Bigl{\vert}\int_{\R^d}D_z\partial_vp_h(v,z-\bm{y}_j)\cdot\Bigl{\{}
D^\vartheta_{\bm{x}} \tilde{p}^{\tau,\bm{\xi}}(t,s,\bm{x},\bm{y})\bigl[\bm{F}_j(s,\bm{y})-\bm{F}_j(s,\bm{\theta}_{\tau,s}(\bm{\xi}))\bigr]
\Bigl{\}} \,d\bm{y}_j \Bigr{\vert} \, dzdv \\
=: \, \bigl(I_1+I_2\bigr)(\bm{y}_{\smallsetminus j}).
\end{multline*}
The second component $I_2$ has no time-singularity and it can be easily controlled using Fubini theorem
\begin{multline*}
I_2(\bm{y}_{\smallsetminus j}) \, \le \\
C\Vert \bm{F}\Vert_{H}\int_{(s-t)^{\delta_j}}^1 v^{\frac{\alpha_j+\beta_j}{\alpha}}\int_{\R^d}\Bigl(\int_{\R^d}\vert
D_z\partial_vp_h(v,z-\bm{y}_j) \vert \, dz\Bigr) \vert D^\vartheta_{\bm{x}} \tilde{p}^{\tau,\bm{\xi}}(t,s,\bm{x},\bm{y})\vert
d^{1+\alpha(j-2)+\beta}_{j:n}(\bm{y}, \bm{\theta}_{\tau,s}(\bm{\xi})) \,d\bm{y}_jdv,
\end{multline*}
remembering that $\bm{F}_j(t,\cdot)$ depends only on the last $(n-j)$ variables and it is in $C^{1+\alpha(j-2)+\beta}_{b,d}(\R^{nd})$ by
assumption (\textbf{R}). We can then use the smoothing effect of the heat-kernel $p_h$ (Equation \eqref{Smoothing_effect_of_Heat_Kern}) and
the Fubini theorem again to write that
\begin{multline*}
I_2(\bm{y}_{\smallsetminus j}) \, \le \, C\Vert \bm{F}\Vert_{H}\int_{(s-t)^{\delta_j}}^1
\frac{v^{\frac{\alpha_j+\beta_j-1}{\alpha}}}{v}\int_{\R^d} \vert D^\vartheta_{\bm{x}} \tilde{p}^{\tau,\bm{\xi}}(t,s,\bm{x},\bm{y})\vert
d^{1+\alpha(j-2)+\beta}_{j:n}(\bm{y}, \bm{\theta}_{\tau,s}(\bm{\xi})) \,d\bm{y}_jdv \\
\le \, C\Vert \bm{F}\Vert_{H}\Bigl(\int_{(s-t)^{\delta_j}}^1\frac{v^{\frac{\alpha_j+\beta_j-1}{\alpha}}}{v} \, dv\Bigr)\Bigl(\int_{\R^d}
\vert D^\vartheta_{\bm{x}} \tilde{p}^{\tau,\bm{\xi}}(t,s,\bm{x},\bm{y})\vert d^{1+\alpha(j-2)+\beta}_{j:n}(\bm{y},
\bm{\theta}_{\tau,s}(\bm{\xi})) \,d\bm{y}_j\Bigr).
\end{multline*}
Noticing from \eqref{eq:def_alpha_i_and_beta_i} that $\alpha_j+\beta_j-1<0$, it holds now that
\[I_2(\bm{y}_{\smallsetminus j})\,\le\,C \Vert \bm{F}\Vert_{H}(s-t)^{\delta_j\frac{\alpha_j+\beta_j-1}{\alpha}}\int_{\R^d} \vert
D^\vartheta_{\bm{x}} \tilde{p}^{\tau,\bm{\xi}}(t,s,\bm{x},\bm{y})\vert d^{1+\alpha(j-2)+\beta}_{j:n}(\bm{y}, \bm{\theta}_{\tau,s}(\bm{\xi}))
\,d\bm{y}_j.\]
We can finally add the integral with respect to the other components $\bm{y}_{\smallsetminus j}$. In order to use now the partial smoothing
effect (Equation \eqref{eq:Partial_Smoothing_Effect}), we take $\tau=t$ and $\bm{\xi}=x$ and notice that
by assumption (\textbf{P}),
\begin{equation}\label{a1}
1+\alpha(j-2)+\beta \, = \, 1+\alpha(j-1)-(\alpha-\beta) \, < \, 1+\alpha(j-1)-(1-\alpha)\bigl(1+\alpha(j-1)\bigr) \, =\,
\alpha(1+\alpha(j-1)).
\end{equation}
It then holds that
\begin{multline}\label{Proof:Second_Besov_Control_I_2}
\int_{\R^{(n-1)d}}I_2(\bm{y}_{\smallsetminus j}) \, d\bm{y}_{\smallsetminus j}\,\le\, C \Vert \bm{F}
\Vert_{H}(s-t)^{\delta_j\frac{\alpha_j+\beta_j-1}{\alpha}} \int_{\R^{nd}} \vert
D^\vartheta_{\bm{x}} \tilde{p}^{\tau,\bm{\xi}}(t,s,\bm{x},\bm{y})\vert d^{1+\alpha(j-2)+\beta}_{j:n}(\bm{y}, \bm{\theta}_{\tau,s}(\bm{\xi}))
\,d\bm{y} \\
\le\, C\Vert \bm{F} \Vert_H(s-t)^{\delta_j \frac{\alpha_j+\beta_j-1}{\alpha}+\frac{1+\alpha(j-2)+\beta}{\alpha}-
\sum_{k=1}^{n}\frac{\vartheta_k}{\alpha_k}}.
\end{multline}
To control the other term $I_1$, we focus at first on the inner integral with respect to $\bm{y}_j$:
\[\int_{\R^d}D_z \partial_vp_h (v,z-\bm{y}_j)\cdot \Bigl{\{}D^\vartheta_{\bm{x}}\tilde{p}^{\tau,\bm{\xi}}(t,s,\bm{x},\bm{y})\otimes\bigl[
\bm{F}_j(s,\bm{y})-\bm{F}_j(s,\bm{\theta}_{\tau,s} (\bm{\xi}))\bigr]\Bigr{\}}\,d\bm{y}_j.\]
We start using a cancellation argument with respect to the density $p_h$ to divide it in
\begin{multline*}
\int_{\R^d}D_z \partial_vp_h (v,z-\bm{y}_j)\cdot\Bigl{\{}D^\vartheta_{\bm{x}}\tilde{p}^{\tau,\bm{\xi}}(t,s,\bm{x},\bm{y})\otimes
\bigl[\bm{F}_j(s,\bm{y}) -\bm{F}_j(s,\bm{\theta}_{\tau,s}(\bm{\xi}))\Bigr{\}} \\
- D^\vartheta_{\bm{x}}\tilde{p}^{\tau,\bm{\xi}}(t,s,\bm{x},\bm{y}_{\smallsetminus j},z)\otimes\bigl[\bm{F}_j(s,\bm{y}_{\smallsetminus
j},z)-\bm{F}_j(s,\bm{\theta}_{\tau,s}(\bm{\xi}))\bigr]\Bigr{\}}\,d\bm{y}_j  \\
= \, \int_{\R^d}D_z \partial_vp_h(v,z-\bm{y}_j)\cdot \Bigl{\{}D^\vartheta_{\bm{x}}\tilde{p}^{\tau,\bm{\xi}}(t,s,\bm{x},\bm{y})\otimes\bigl[
\bm{F}_j(s,\bm{y})-\bm{F}_j(s,\bm{y}_{\smallsetminus j},z)\bigr]\Bigr{\}} \,d\bm{y}_j \\
+\int_{\R^d}D_z \partial_vp_h(v,z-\bm{y}_j)\cdot\Bigl{\{}\bigl[D^\vartheta_{\bm{x}}\tilde{p}^{\tau,\bm{\xi}}(t,s,\bm{x},\bm{y})-
D^\vartheta_{\bm{x}} \tilde{p}^{\tau,\bm{\xi}}(t,s,\bm{x},\bm{y}_{\smallsetminus j},z)\bigr]\otimes\bigl[\bm{F}_j(s,\bm{y}_{\smallsetminus
j},z)-\bm{F}_j(s,\bm{\theta}_{\tau,s}(\bm{\xi})) \bigr]\Bigl{\}} \,d\bm{y}_j \\
 =:\, \bigl(J_1+J_2\bigr)(v,\bm{y}_{\smallsetminus j},z).
\end{multline*}
Remembering  notation \eqref{eq:notation_smallsetminus} for $\bm{F}_j(s,\bm{y}_{\smallsetminus j},z)$ and that $\bm{F}_j$ is
$\frac{1+\alpha(j-2)+\beta}{1+\alpha(j-1)}$-H\"older continuous with respect to its $j$-th variable by assumption (\textbf{R}), the first
component $J_1$ can be easily controlled using the Fubini theorem by
\begin{multline*}
\int_{\R^d}\Bigl{\vert}J_1(v,\bm{y}_{\smallsetminus j},z)\Bigr{\vert}\,dz\\
\le \,C\Vert\bm{F} \Vert_H \int_{\R^d}\Bigl(\int_{\R^d}\vert z-\bm{y}_j\vert^{\frac{1+\alpha(j-2)+\beta}{1+\alpha(j-1)}}\vert D_{z}
\partial_vp_h(v,z-\bm{y}_j)\vert \, dz\Bigr)\vert D^\vartheta_{\bm{x}}\tilde{p}^{\tau,\bm{\xi}}(t,s,\bm{x},\bm{y}) \vert \,d\bm{y}_j\\
\le \,C\Vert\bm{F} \Vert_H v^{\frac{1}{\alpha}\frac{1+\alpha(j-2)+\beta}{1+\alpha(j-1)}-\frac{1}{\alpha}-1} \int_{\R^d}
\vert D^\vartheta_{\bm{x}}\tilde{p}^{\tau,\bm{\xi}}(t,s,\bm{x},\bm{y}) \vert \,d\bm{y}_j
\end{multline*}
where in the last passage we used the smoothing effect of the heat-kernel $p_h$ (Equation \eqref{Smoothing_effect_of_Heat_Kern}), noticing
that
\[\frac{1+\alpha(j-2)+\beta}{1+\alpha(j-1)} \, = \, 1+\frac{\beta-\alpha}{1+\alpha(j-1)}\, < \, 1+\alpha,\]
since $\alpha>\beta$ by assumption (\textbf{P}).
Using now the identity
\begin{equation}\label{Proof:Second_Besov_Control}
\frac{\alpha_j+\beta_j}{\alpha}+\frac{1}{\alpha}\Bigl(\frac{1+\alpha(j-2)+\beta}{1+\alpha(j-1)}-1\Bigr) \, = \, \frac{2\beta_j}{\alpha},
\end{equation}
we add the integral with respect to $v$ and write that
\begin{multline*}
\int_{0}^{(s-t)^{\delta_j}}v^{\frac{\alpha_j+\beta_j}{\alpha}}\int_{\R^d}\Bigl{\vert}J_1(v,\bm{y}_{\smallsetminus j},z)\Bigr{\vert}\,dzdv\,
\le \,C\Vert\bm{F} \Vert_H \int_{0}^{(s-t)^{\delta_j}} \frac{v^{\frac{2\beta_j}{\alpha}}}{v}\int_{\R^d}
\vert D^\vartheta_{\bm{x}}\tilde{p}^{\tau,\bm{\xi}}(t,s,\bm{x},\bm{y}) \vert \,d\bm{y}_j dv \\
\le \, C\Vert\bm{F} \Vert_H (s-t)^{\delta_j\frac{2\beta_j}{\alpha}} \int_{\R^d} \vert
D^\vartheta_{\bm{x}}\tilde{p}^{\tau,\bm{\xi}}(t,s,\bm{x},\bm{y}) \vert \,d\bm{y}_j.
\end{multline*}
Adding the integral with respect to the other components $\bm{y}_{\smallsetminus j}$, we can finally conclude that
\begin{multline}\label{Proof:Second_Besov_Control_J1}
\int_{\R^{(n-1)d}}\int_{0}^{(s-t)^{\delta_j}}v^{\frac{\alpha_j+\beta_j}{\alpha}}\int_{\R^d}\Bigl{\vert}J_1(v,\bm{y}_{\smallsetminus j},z)
\Bigr{\vert} \, dzdv \,d\bm{y}_{\smallsetminus j}\, \le \,  C\Vert \bm{F}\Vert_H(s-t)^{\delta_j\frac{2\beta_j}{\alpha}} \int_{\R^{nd}}
\vert D^\vartheta_{\bm{x}}\tilde{p}^{\tau,\bm{\xi}}(t,s,\bm{x},\bm{y}) \vert\,d\bm{y}\\
\le \, C\Vert \bm{F} \Vert_H (s-t)^{\delta_j\frac{2\beta_j}{\alpha}-\sum_{k=1}^{n}\frac{\vartheta_k}{\alpha_k}}
\end{multline}
To control the second component $J_2$, we start applying a Taylor expansion on $\tilde{p}^{\tau,\bm{\xi}}$ with respect to $\bm{y}_j$:
\begin{multline}\label{eq:J_2_component_for_extension}
J_2(v,\bm{y}_{\smallsetminus j},z) \,= \, \int_{\R^d}D_z \partial_vp_h(v,z-\bm{y}_j)\cdot\Bigl{\{}\bigl[\bm{F}_j(s,\bm{y}_{\smallsetminus
j},z) -\bm{F}_j(s,\bm{\theta}_{\tau,s}(\bm{\xi}))\bigr]\\
\otimes\int_{0}^{1}D_{\bm{y}_j}D^\vartheta_{\bm{x}} \tilde{p}^{\tau,\bm{\xi}}(t,s,\bm{x},\bm{y}_{\smallsetminus j},\bm{y}_j+
\lambda(z-\bm{y}_j))\cdot(z)\Bigr{\}} \,d\lambda d\bm{y}_j.
\end{multline}
We then notice that for any fixed $\lambda$ in $[0,1]$, it holds that
\begin{multline*}
\vert \bm{F}_j(s,\bm{y}_{\smallsetminus j}, z)-\bm{F}_j(s,\bm{\theta}_{\tau,s}(\bm{\xi})) \vert \, \le \, C\Vert \bm{F}
\Vert_H\Bigl{\{}\bigl{\vert}\bigl(z-\bm{\theta}_{\tau,s}(\bm{\xi})\bigr)_j\bigr{\vert}^{\frac{1+\alpha(j-2)+\beta}{1
+\alpha(j-1)}}+\sum_{k=j+1}^{n}\bigl{\vert}\bigl(\bm{y}-\bm{\theta}_{\tau,s}(\bm{\xi})\bigr)_k
\bigr{\vert}^{\frac{1+\alpha(j-2)+\beta}{1+\alpha(k-1)}}\Bigr{\}} \\
\le \, C\Vert \bm{F} \Vert_H\Bigl{\{}\bigl{\vert} \lambda(\bm{y}_j-z) \bigr{\vert}^{\frac{1+\alpha(j-2)+\beta}{1 +\alpha(j-1)}}+\bigl{\vert}
\bigl(\bm{y}_j+\lambda(z-\bm{y}_j)-\bm{\theta}_{\tau,s}(\bm{\xi})\bigr)_j\bigr{\vert}^{\frac{1+ \alpha(j-2) +\beta}{1
+\alpha(j-1)}}+\sum_{k=j+1}^{n}\bigl{\vert}\bigl(\bm{y}-\bm{\theta}_{\tau,s}(\bm{\xi})\bigr)_k\bigr{\vert}^{\frac{1+\alpha(j-2)+\beta}{1+
\alpha(k-1)}}\Bigr{\}}\\
\le \, C\Vert \bm{F} \Vert_H\Bigl{\{}\bigl{\vert} z-\bm{y}_j \bigr{\vert}^{\frac{1+\alpha(j-2)+\beta}{1 +\alpha(j-1)}} +
d^{1+\alpha(j-2)+\beta}_{j:n}\bigl((\bm{y}_{\smallsetminus j},\bm{y}_j+\lambda(z-\bm{y}_j))\bigr),\bm{\theta}_{\tau,s}(\bm{\xi})
\bigr)\Bigr{\}}
\end{multline*}
where as in \eqref{eq:notation_smallsetminus}, we denoted $(\bm{y}_{\smallsetminus j},\bm{y}_j+\lambda(z-\bm{y}_j)):=(\bm{y}_1,\dots,
\bm{y}_{j-1},\bm{y}_1,\dots,\bm{y}_{j-1},\bm{y}_j+\lambda(z-\bm{y}_j),\bm{y}_{j+1}, \dots,\bm{y}_n)$.
We can thus split $J_2$ as
\begin{multline}\label{Proof:Second_Besov_Control_J'0}
\vert J_2(v,\bm{y}_{\smallsetminus j},z) \vert \, \le \\
C\Vert \bm{F} \Vert_H\int_{0}^{1}\Bigl{\{}\int_{\R^d}\vert z-\bm{y}_j \vert^{\frac{1+\alpha(j-2)+\beta}{1
+\alpha(j-1)}+1}\vert D_z \partial_vp_h(v,z-\bm{y}_j)\vert\, \vert D_{\bm{y}_j}D^\vartheta_{\bm{x}}
\tilde{p}^{\tau,\bm{\xi}}(t,s,\bm{x},\bm{y}_{\smallsetminus j},\bm{y}_j+\lambda(z-\bm{y}_j))\vert \,d\bm{y}_j \\
+ \int_{\R^d}\vert z-\bm{y}_j \vert\, \vert D_z \partial_vp_h(v,z-\bm{y}_j)\vert\, \vert D_{\bm{y}_j}D^\vartheta_{\bm{x}}
\tilde{p}^{\tau,\bm{\xi}}(t,s,\bm{x},\bm{y}_{\smallsetminus j},\bm{y}_j+\lambda(z-\bm{y}_j)) \vert\\
\times d^{1+\alpha(j-2)+\beta}_{j:n} ((\bm{y}_{\smallsetminus j},\bm{y}_j+\lambda(z-\bm{y}_j)),\bm{\theta}_{\tau,s}(\bm{\xi}))
d\bm{y}_j\Bigr{\}}\, d\lambda \, =: \,
C\Vert \bm{F} \Vert_H\int_{0}^{1}\bigl(J_{2,1}+J_{2,2}\bigr)(v,\bm{y}_{\smallsetminus j},z,\lambda)\, d\lambda
\end{multline}
Adding now the integral with respect to $z$, the first term $J_{2,1}$ can be rewritten as
\begin{multline*}
\int_{\R^d} J_{2,1}(v,\bm{y}_{\smallsetminus j},z,\lambda) \, dz \, \le \\
\int_{\R^d}\int_{\R^d}\vert z-\bm{y}_j \vert^{\frac{1+\alpha(j-2)+\beta}{1+\alpha(j-1)}+1}\vert D_z \partial_vp_h(v,z-\bm{y}_j)\vert\, \vert
D_{\bm{y}_j}D^\vartheta_{\bm{x}} \tilde{p}^{\tau,\bm{\xi}}(t,s,\bm{x},\bm{y}_{\smallsetminus j},\bm{y}_j+\lambda(z-\bm{y}_j))\vert \,
d\bm{y}_j dz.
\end{multline*}
The Fubini theorem and the change of variables $\tilde{z}=z-\bm{y}_j$ and $\tilde{\bm{y}}_j=\bm{y}_j+\lambda\tilde{z}$ allow then to
divide the integrals:
\begin{multline*}
\int_{\R^d} J_{2,1}(v,\bm{y}_{\smallsetminus j},z,\lambda) \, dz \, \le \, \Bigl(\int_{\R^d} \vert \tilde{z}
\vert^{\frac{1+\alpha(j-2)+\beta}{1+\alpha(j-1)}+1}\vert D_{\tilde{z}}\partial_vp_h(v,\tilde{z}) \vert \,
d\tilde{z}\Bigr)\Bigl(\int_{\R^d}\vert D_{\tilde{\bm{y}}_j} D^\vartheta_{\bm{x}} \tilde{p}^{\tau,\bm{\xi}}(t,s,\bm{x},
\bm{y}_{\smallsetminus j},\tilde{\bm{y}}_j)\vert \, d\tilde{\bm{y}}_j \Bigr)
\end{multline*}
Noticing now from assumption (\textbf{P}) that
\[\frac{1+\alpha(j-2)+\beta}{1+\alpha(j-1)}+1 \, =\, 1-\frac{\beta-\alpha}{1+\alpha(j-1)}+1 \,< \, 2-(1-\alpha) \, = \, 1+\alpha, \]
we can use the smoothing effect of the heat-kernel $p_h$ (Equation \eqref{Smoothing_effect_of_Heat_Kern}) to show that
\[\int_{\R^d} J_{2,1}(v,\bm{y}_{\smallsetminus j},z,\lambda) \, dz \,\le \,\frac{v^{\frac{1+\alpha(j-2)+\beta}{\alpha(1+\alpha(j-1))}}}{v}
\int_{\R^d}\vert D_{\tilde{\bm{y}}_j} D^\vartheta_{\bm{x}}\tilde{p}^{\tau,\bm{\xi}}(t,s,\bm{x},\bm{y}_{\smallsetminus j},
\tilde{\bm{y}}_j) \vert \, d\tilde{\bm{y}}_j.\]
Remembering equation \eqref{Proof:Second_Besov_Control}, we can add the in integral with respect to $v$ and show that
\[\int_{0}^{(s-t)^{\delta_j}}v^{\frac{\alpha_j+\beta_j}{\alpha}}\int_{\R^d} J_{2,1}(v,\bm{y}_{\smallsetminus j},z,\lambda) \, dz\, dv \, \le
\, (s-t)^{\delta_j\frac{2\beta_j+1}{\alpha}}\int_{\R^d}\vert D_{\tilde{\bm{y}}_j} D^\vartheta_{\bm{x}}
\tilde{p}^{\tau,\bm{\xi}}(t,s,\bm{x}, \bm{y}_{\smallsetminus j},\tilde{\bm{y}}_j)\vert \, d\tilde{\bm{y}}_j.\]
Adding the integral with respect to $\bm{y}_{\smallsetminus j}$, we can conclude with $J_{2,1}$ that
\begin{multline}\label{Proof:Second_Besov_Control_J_{2,1}}
\int_{\R^{(n-1)d}}\int_{0}^{(s-t)^{\delta_j}}v^{\frac{\alpha_j+\beta_j}{\alpha}}\int_{\R^d} J_{2,1}(v,\bm{y}_{\smallsetminus j},z,\lambda) \,
dz\, dv \, d\bm{y}_{\smallsetminus j} \\
\le \, C (s-t)^{\delta_j\frac{2\beta_j+1}{\alpha}}\int_{\R^{nd}}\vert D_{\bm{y}_j}D^\vartheta_{\bm{x}}\tilde{p}^{\tau,\bm{\xi}}(t,s,\bm{x},\bm{y})
\vert \, dy \, \le \, C (s-t)^{\delta_j\frac{2\beta_j+1}{\alpha}-\frac{1}{\alpha_j}-\sum_{k=1}^{n}\frac{\vartheta_k}{\alpha_k}}
\end{multline}
where, for simplicity, we have changed back the variable $\tilde{\bm{y}}_j$ with $\bm{y}_j$. \newline
To control instead the term $J_{2,2}$ (cf. Equation \eqref{Proof:Second_Besov_Control_J'0}), we can use again the Fubini theorem and the
changes of variables $\tilde{z}=z-\bm{y}_j$, $\tilde{\bm{y}}_j=\bm{y}_j+\lambda\tilde{z}$ to
divide the integrals and show that
\begin{multline*}
\int_{\R^d} J_{2,2}(v,\bm{y}_{\smallsetminus j},z,\lambda) \, dz  \\
\le \, \Bigl(\int_{\R^d}\vert\tilde{z}\vert \, \vert D_{\tilde{z}} \partial_v
p_h(v,\tilde{z})\vert\, d\tilde{z}\Bigr)\Bigl(\int_{\R^d}\vert D_{\tilde{\bm{y}}_j}D^\vartheta_{\bm{x}} \tilde{p}^{\tau,\bm{\xi}}
(t,s,\bm{x},\bm{y}_{\smallsetminus j},\tilde{\bm{y}}_j) \vert d^{1+\alpha(j-2)+\beta}_{j:n} ((\bm{y}_{\smallsetminus j},\tilde{\bm{y}}_j),
\bm{\theta}_{\tau,s}(\bm{\xi})) d\tilde{\bm{y}}_j\Bigr)\\
\le \, \frac{1}{v}\int_{\R^d}\vert D_{\bm{y}_j}D^\vartheta_{\bm{x}} \tilde{p}^{\tau,\bm{\xi}} (t,s,\bm{x},\bm{y}) \vert
d^{1+\alpha(j-2)+\beta}_{j:n} (\bm{y}, \bm{\theta}_{\tau,s}(\bm{\xi})) d\bm{y}_j dv
\end{multline*}
where in the second inequality we used the smoothing effect of the heat-kernel $p_h$ (Equation \eqref{Smoothing_effect_of_Heat_Kern}) and
changed back the variable $\tilde{\bm{y}}_j$ with $\bm{y}_j$ for simplicity. It then follows that
\begin{multline*}
\int_{0}^{(s-t)^{\delta_j}}v^{\frac{\alpha_j+\beta_j}{\alpha}}\int_{\R^d} J_{2,2}(v,z,\bm{y}_{\smallsetminus j}) \, dz dv \\
\le \, (s-t)^{\delta_j\frac{\alpha_j+\beta_j}{\alpha}}\int_{\R^d}\vert D_{\bm{y}_j}D^\vartheta_{\bm{x}} \tilde{p}^{\tau,\bm{\xi}}
(t,s,\bm{x},\bm{y}) \vert d^{1+\alpha(j-2)+\beta}_{j:n} (y, \bm{\theta}_{\tau,s}(\bm{\xi})) d\bm{y}_j.
\end{multline*}
Taking now $\tau=t$ and $\bm{\xi}=\bm{x}$, we conclude with $J_{2,2}$ applying the partial smoothing effect
\eqref{eq:Partial_Smoothing_Effect} of
$\tilde{p}^{\tau,\bm{\xi}}$ under the hypothesis $1+\alpha(j-2)+\beta \le  \alpha(1+\alpha(j-1))$ (see Equation \eqref{a1}) to write that
\begin{multline}\label{Proof:Second_Besov_Control_J'2}
\int_{\R^{(n-1)d}}\int_{0}^{(s-t)^{\delta_j}}v^{\frac{\alpha_j+\beta_j}{\alpha}}\int_{\R^d} J_{2,2}(v,z,\bm{y}_{\smallsetminus j}) \, dz \,
dv d\bm{y}_{\smallsetminus j} \\
\le \, (s-t)^{\delta_j\frac{\alpha_j+\beta_j}{\alpha}}\int_{\R^{nd}}\vert D_{\bm{y}_j}D^\vartheta_{\bm{x}} \tilde{p}^{\tau,\bm{\xi}}
(t,s,\bm{x},\bm{y}) \vert d^{1+\alpha(j-2)+\beta}_{j:n} (y, \bm{\theta}_{\tau,s}(\bm{\xi})) d\bm{y} \\
\le \, C(s-t)^{\delta_j\frac{\alpha_j+\beta_j}{\alpha}+\frac{1+\alpha(j-2)+\beta}{\alpha}- \frac{1}{\alpha_j}-\sum_{k=1}^{n}
\frac{\vartheta_k}{\alpha_k}}.
\end{multline}
Looking back to Equations \eqref{Proof:Second_Besov_Control_I_2}, \eqref{Proof:Second_Besov_Control_J1},
\eqref{Proof:Second_Besov_Control_J_{2,1}} and \eqref{Proof:Second_Besov_Control_J'2}, we can finally choose the right $\delta_j$. Since
$s-t\le T-t<1$ by hypothesis (\textbf{ST}), it is enough to take $\delta_j$ such that the following quantities
\[\delta_j\frac{\alpha_j+\beta_j-1}{\alpha}+\frac{1+\alpha(j-2)+\beta}{\alpha}, \quad  \delta_j\frac{2\beta_j}{\alpha
}, \quad \delta_j\frac{ 2\beta_j+1}{\alpha}-\frac{1}{\alpha_j} \,\, \text{ and
}\,\, \delta_j\frac{\alpha_j+\beta_j}{\alpha}+\frac{1+\alpha(j-2)+\beta}{\alpha}-\frac{1}{\alpha_j}\]
are bigger than $\beta/\alpha$. This is true if for example we choose
\[\delta_j \, = \, \frac{[1+\alpha(j-2)][1+\alpha(j-1)]}{1+\alpha(j-2)-\beta}.\qedhere\]
\end{proof}

\subsection{Proof of Proposition \ref{prop:A_Priori_Estimates}}
We have now all the tools necessary to prove the A Priori estimates (Proposition \ref{prop:A_Priori_Estimates}). In Lemma \ref{lemma:supremum_norm} below, we will state the estimates for the supremum norms of the solution and its non-degenerate gradient while the controls of the H\"older moduli of the solution and its gradient with respect to the non-degenerate variable are given in Lemmas \ref{lemma:Holder_modulus_Non-Deg} and \ref{lemma:Holder_modulus_Deg}, respectively.\newline
The Schauder estimates (Theorem \ref{theorem:Schauder_Estimates}) for a solution $u$ in $L^\infty(0,T;C^{\alpha+\beta}_{b,d}(\R^{nd}))$ of
equation \eqref{Degenerate_Stable_PDE} will then follows immediately.

\begin{lemma}[Supremum Estimates]
\label{lemma:supremum_norm}
Let $u$ be as in Equation \eqref{align:Representation2}. Then, there exists a constant $C\ge1$ such that for any $t$ in $[0,T]$ and any $\bm{x}$ in $\R^{nd}$,
\[\vert u(t,\bm{x}) \vert + \vert D_{\bm{x}_1}u(t,\bm{x})\vert \, \le \, C\Bigl[\Vert g \Vert_{C^{\alpha+\beta}_{b,d}} + \Vert f
\Vert_{L^\infty(C^\beta_{b,d})} + \Vert \bm{F}\Vert_H\Vert u \Vert_{L^\infty(C^{\alpha+\beta}_{b,d})}\Bigr].\]
\end{lemma}
\begin{proof}
As indicated above, we can control the supremum norm of $u$ and its gradient with respect to $\bm{x}_1$ analyzing separately the
contributions from the proxy $\tilde{u}^{\tau,\bm{\xi}}$, that have already been handled in Lemma
\ref{lemma:supremum_norm_proxy}, and those from the perturbative term $R^{\tau,\bm{\xi}}(s,\bm{x})$.
To control the contribution $\int_{t}^{T} D_{\bm{x}_1}\tilde{P}^{\tau,\bm{\xi}}_{t,s}R^{\tau,\bm{\xi}}(s,\bm{x})\, ds$, we start splitting
it up in the following way
\begin{multline}\label{Proof:Supremum_Schauder0}
\int_{t}^{T} D_{\bm{x}_1}\tilde{P}^{\tau,\bm{\xi}}_{t,s}R^{\tau,\bm{\xi}}(s,\bm{x})\, ds \, = \, \sum_{j=1}^{n}\int_{t}^{T} \int_{\R^{nd}}
D_{\bm{x}_1}\tilde{p}^{\tau,\bm{\xi}}(t,s,\bm{x},\bm{y})\bigl[\bm{F}_j(s,\bm{y}) -\bm{F}_j(s,\bm{\theta}_{\tau,s}
(\bm{\xi})) \bigr]\cdot D_{\bm{y}_j}u(s,\bm{y}) \, d\bm{y}ds \\
=:\sum_{j=1}^{n}I_j(t,\bm{x}).
\end{multline}
Since by hypothesis $u$ has a proper derivative with respect to the first variable $\bm{x}_1$, it is possible to bound $I_1$ through
\[\vert I_1(t,\bm{x})\vert \, \le \, C\Vert \bm{F}\Vert_H \Vert u\Vert_{L^{\infty}(C^{\alpha+\beta}_{b,d})}\int_{t}^{T}\int_{\R^{nd}}
\bigl{\vert}D_{\bm{x}_1}\tilde{p}^{\tau,\bm{\xi}}(t,s,\bm{x},\bm{y})\bigr{\vert} d^\beta(\bm{y},\bm{\theta}_{\tau,s}(\bm{\xi})) \, d\bm{y}
ds.\]
We take now $(\tau,\bm{\xi})=(t,\bm{x})$ so that $\bm{\theta}_{\tau,s}(\bm{\xi})=\tilde{\bm{m}}^{\tau,\bm{\xi}}_{t,s}$ (cf. Equation
\eqref{eq:identification_theta_m} in Lemma \ref{lemma:identification_theta_m}) and we then use the smoothing effect for the frozen density
$\tilde{p}^{\tau,\bm{\xi}}$ (Equation \eqref{eq:Smoothing_effects_of_tilde_p}) to show that
\begin{equation}\label{Proof:Supremum_Schauder1}
\vert I_1(t,\bm{x})\vert \, \le \,C\Vert \bm{F}\Vert_H \Vert u\Vert_{L^{\infty}( C^{\alpha+\beta}_{b,d})}
(T-t)^{\frac{\beta+\alpha-1}{\alpha}}.
\end{equation}
Hence, it holds that $\vert I_1(t,\bm{x})\vert \le  C\Vert \bm{F}\Vert_H \Vert u \Vert_{L^{\infty}(C^{\alpha+\beta}_{b,d})}$, since
$T\le1$ and $\alpha+\beta>1$ by assumptions (\textbf{ST}) and (\textbf{P}).\newline
The control for the terms $I_j$ with $j>1$ can be obtained easily from the second Besov control (Lemma \ref{lemma:Second_Besov_COntrols}).
For this reason, we start applying integration by parts formula to show that
\[\vert I_j(t,\bm{x}) \vert \, = \,
\Bigl{\vert}\int_{t}^{T}\int_{\R^{nd}}D_{\bm{y}_j}\cdot\Bigl{\{}D_{\bm{x}_1}\tilde{p}^{\tau,\bm{\xi}}(t,s,\bm{x},\bm{y})
\bigl[\bm{F}_j(s,\bm{y})-\bm{F}_j(s,\bm{\theta}_{t,s}(\bm{\xi}))\bigr]\Bigr{\}} u(s,\bm{y}) \, d\bm{y}ds\Bigr{\vert}\]
We can then use identification \eqref{Besov:ident_Holder_Besov} and duality in Besov spaces \eqref{Besov:duality_in_Besov} to write
that
\begin{multline*}
\vert I_j(t,\bm{x}) \vert \\
\le \, \Vert u \Vert_{L^\infty(C^{\alpha+\beta}_{b,d})}
\int_{\R^{(n-1)d}}\Bigl{\Vert}D_{\bm{y}_j}\cdot\Bigl{\{}D_{\bm{x}_1}\tilde{p}^{\tau,\bm{\xi}}(t,s,\bm{x},\bm{y}_{\smallsetminus j},\cdot)
\bigl[\bm{F}_j (s, \bm{y}_{\smallsetminus j},\cdot)-\bm{F}_j(s,\bm{\theta}_{\tau,s} (\bm{\xi}))\bigr]\Bigl{\}}
\Bigr{\Vert}_{B^{-(\alpha_j+\beta_j)}_{1,1}} \, d\bm{y}_{\smallsetminus j}.
\end{multline*}
Taking now $(\tau,\bm{\xi})=(t,\bm{x})$, the second Besov control (Lemma \ref{lemma:Second_Besov_COntrols}) can be applied to show that
\begin{equation}\label{Proof:Supremum_Schauder2}
\vert I_j(t,\bm{x})\vert \, \le \, C\Vert \bm{F}\Vert_H\Vert u \Vert_{L^\infty(C^{\alpha+\beta}_{b,d})}\int_{t}^{T}
(s-t)^{\frac{\beta-1}{\alpha}}\,ds  \,\le \, C\Vert \bm{F}\Vert_H\Vert u \Vert_{L^\infty(C^{\alpha+\beta}_{b,d})}(T-t)^{
\frac{\beta+\alpha-1}{\alpha}}.
\end{equation}
Since by assumption (\textbf{ST}) $T\le 1$, we can conclude that $\vert I_j(t,\bm{x})\vert \le C\Vert\bm{F}\Vert_H\Vert u
\Vert_{L^\infty(C^{\alpha+\beta}_{b,d})}$. \newline
The control on the pertubative term
\[\int_{t}^{T}  \tilde{P}^{\tau,\bm{\xi}}_{t,s}R^{\tau,\bm{\xi}}(s,\bm{x})\, ds\]
can be obtained in a similar way. Namely, the inequalities \eqref{Proof:Supremum_Schauder1} and \eqref{Proof:Supremum_Schauder2} hold again
with $(T-t)^{\frac{\beta+\alpha-1}{\alpha}}$ replaced by $(T-t)^{\frac{\beta+\alpha}{\alpha}}$.
\end{proof}

As already specified in the previous sub-section, there is a big difference between the non-degenerate case $i=1$ where
$\alpha+\beta$ is in $(1,2)$ and we have to deal with a proper derivative and the other degenerate situations ($i>1$) where instead
$(\alpha+\beta)/(1+\alpha(i-1))<1$ and the norm is calculated directly on the function. Again, we are going to analyze the two
cases separately. Lemma \ref{lemma:Holder_modulus_proxy_Non-Deg} will work on the non-degenerate setting ($i=1$) while lemma
\ref{lemma:Holder_modulus_proxy_Deg} will concerns the degenerate one ($i>1$).

Moreover, we will need to divide the proofs in two cases, depending on which regime we are considering. Since the global off-diagonal
regime, i.e.\ when $T-t\le c_0d^\alpha(\bm{x},\bm{x}')$, will work essentially as the already shown Schauder estimates
(Proposition \ref{prop:Schauder_Estimates_for_proxy}) for the proxy, the proof will be quite shorter.\newline
Instead, in the global diagonal case, such that $T-t \ge c_0d^\alpha(\bm{x},\bm{x}')$, when a time integration is involved (for example in
the control of the frozen Green kernel or the perturbative term), two different situations appear. There are again a local off-diagonal
regime if $s-t \le c_0d^\alpha(\bm{x},\bm{x}')$ and a local diagonal regime when $s-t \ge c_0d^\alpha(\bm{x},\bm{x}')$. In order to handle
these terms properly, the key tool is to be able to change the freezing points depending on which regime we are. It seems reasonable that,
when the spatial points are in a local diagonal regime, the auxiliary frozen
densities are considered for the same freezing parameter and conversely that in the local off-diagonal regime, the densities are frozen
along their own spatial argument. For this reason, we have postponed the relative proofs in two specific sub-sections.

Before presenting the main results of this section, we are going to state three auxiliary estimates associated with our proxy we will need
below. We refer to the Section A.$2$ for a precise proof of these results. \newline
The first one concerns the sensitivity of the H\"older flow $\bm{\theta}_{t,s}$ with respect to the initial point. Indeed,
\begin{lemma}[Controls on the Flows]
\label{lemma:Controls_on_Flow1}
Let $t<s$ be two points in $[0,T]$ and $\bm{x},\bm{x}'$ two points in $\R^{nd}$. Then, there exists a constant $C\ge 1$ such that
\[d(\bm{\theta}_{t,s}(\bm{x}), \bm{\theta}_{t,s}(\bm{x}')) \, \le \, C\Vert \bm{F}\Vert_H \bigl[d(\bm{x},\bm{x}')+
(s-t)^{1/\alpha}\bigr].\]
\end{lemma}

The second result is the following:

\begin{lemma}%controls on means1
\label{lemma:Controls_on_means1}
Let $t<s$ be two points in $[0,T]$ and $\bm{x},\bm{x}'$ two points in $\R^{nd}$ and $\bm{y},\bm{y}'$ two points in $\R^{nd}$ such that
$\bm{y}_1=\bm{y}'_1$. Then, there
exists a constant $C\ge1$ such that
\[\bigl{\vert}(\tilde{\bm{m}}^{t,\bm{x}}_{t,s}(\bm{y})-\tilde{\bm{m}}^{t,\bm{x}'}_{t,s}(\bm{y}'))_1 \bigr{\vert} \, \le
\, C\Vert \bm{F}\Vert_H\Bigl[(s-t)d^\beta(\bm{x},\bm{x}')+(s-t)^{\frac{\alpha+\beta}{\alpha}}\Bigr].\]
\end{lemma}

Finally, the impact of the freezing point in the linearization procedure is the argument of this last Lemma. Namely,

\begin{lemma}%controls on means
\label{lemma:Controls_on_means}
Let $t$ be in $[0,T]$ and $\bm{x},\bm{x}'$ two points in $\R^{nd}$. Then, there exists a constant $C\ge 1$
such that
\[d(\tilde{\bm{m}}^{t,\bm{x}}_{t,t_0}(\bm{x}'),\tilde{\bm{m}}^{t,\bm{x}'}_{t,t_0} (\bm{x}')) \, \le \, Cc_0^{\frac{1}{1+\alpha(n-1)}}\Vert
\bm{F}\Vert_Hd(\bm{x},\bm{x}')\]
where $t_0$ is the change of regime time defined in \eqref{eq:def_t0}.
\end{lemma}

Thanks to the above controls, we will eventually prove the following results:

\begin{lemma}[Controls on H\"older Moduli: Non-Degenerate]
\label{lemma:Holder_modulus_Non-Deg}
Let $\bm{x},\bm{x}'$ be in $\R^{nd}$ such that $\bm{x}_j=\bm{x}'_j$ for any $j\neq 1$ and $u$ as in
 Equation \eqref{align:Representation2}. Then, there exists a constant $C\ge 1$ such that for any $t$ in $[0,T]$,
\begin{multline*}
\bigl{\vert} D_{\bm{x}_1}u(t,\bm{x})- D_{\bm{x}_1}u(t,\bm{x}')\vert\\
\le \, C\Bigl{\{}c_0^{\frac{\alpha+\beta-2}{\alpha}}\bigl(\Vert g\Vert_{C^{\alpha+\beta}} + \Vert f \Vert_{L^\infty(C^\beta)}\bigr) +
\bigl(c_0^{\frac{\alpha+\beta-1}{1+\alpha(n-1)}} +c_0^{
\frac{\alpha+\beta-2 }{\alpha}}\Vert \bm{F}\Vert_{H}\bigr)\Vert u \Vert_{L^\infty(C^{\alpha+\beta}_{b,d})}\Bigr{\}} d^{\alpha+\beta-1}
(\bm{x},\bm{x}').
\end{multline*}
\end{lemma}

We can point out now the analogous result in the degenerate setting.

\begin{lemma}[Controls on H\"older Moduli: Degenerate]
\label{lemma:Holder_modulus_Deg}
Let $i$ be in $\llbracket 1,n\rrbracket$ and $\bm{x},\bm{x}'$ in $\R^{nd}$ such that $x_j=x'_j$ for any $j\neq i$ and  $u$ as in
Equation \eqref{align:Representation2}. Then, there exists a constant $C\ge1$ such that
for any $t$ in $[0,T]$,
\[\bigl{\vert} u(t,\bm{x})- u(t,\bm{x}') \vert\, \le \, C\Bigl{\{}c_0^{\frac{\beta-\gamma_i}{\alpha}}\bigl(\Vert g
\Vert_{C^{\alpha+\beta}} + \Vert f \Vert_{L^\infty(C^\beta)}\bigr) + \bigl(c_0^{\frac{\alpha+\beta}{1+\alpha(n-1)}}
+c_0^{\frac{\beta-\gamma_i }{\alpha}}\Vert \bm{F}\Vert_{H}\bigr)\Vert u \Vert_{L^\infty(C^{\alpha+\beta}_{b,d})}\Bigr{\}}d^{\alpha+\beta}
(\bm{x},\bm{x}').\]
\end{lemma}

\subsubsection{Off-Diagonal Regime}
We focus here on the proof of the Controls on the H\"older Moduli either in the non-degenerate setting (Proposition
\ref{lemma:Holder_modulus_Non-Deg}) and in the degenerate one (Proposition \ref{lemma:Holder_modulus_Deg}), when a
off-diagonal regime is assumed. For this reason, all the statements presented in this sub-section will tacitly assume that $T-t \le
c_0d^\alpha(\bm{x},\bm{x}')$ for some given $(t,\bm{x},\bm{x}')$ in $[0,T]\times \R^{2nd}$. \newline
To show these two controls, we will need to adapt the auxiliary estimates above to the off-diagonal regime case we consider here. Namely,
\begin{gather}
\label{eq:Sensitivity_on_flow}
d(\tilde{\bm{m}}^{t,\bm{x}}_{t,T} (\bm{x}),\tilde{\bm{m}}^{t,\bm{x}'}_{t,T}(\bm{x}')) \, = \, d(\bm{\theta}_{t,T}(\bm{x}),
\bm{\theta}_{t,T}(\bm{x}')) \, \le \, C\Vert \bm{F}\Vert_H d(\bm{x},\bm{x}'); \\
\label{eq:Sensitivity_on_flow1}
\text{if }\bm{x}_1 \, = \, \bm{x}_1', \quad \bigl{\vert} \bigl(\tilde{\bm{m}}^{t,\bm{x}}_{t,T} (\bm{x})-\tilde{\bm{m}
}^{t,\bm{x}'}_{t,T}(\bm{x}')\bigr)_1  \bigr{\vert} \, \le \, C\Vert \bm{F}\Vert_Hd^{\alpha+\beta}(\bm{x},\bm{x}')
\end{gather}
They can be obtained easily from Equation \eqref{eq:identification_theta_m} in Lemma \ref{lemma:identification_theta_m}
and the sensitivity controls (Lemmas \ref{lemma:Controls_on_Flow1} and \ref{lemma:Controls_on_means1}, respectively), taking $s=T$ and
$(\bm{y},\bm{y}')=(\bm{x},\bm{x}')$.

\paragraph{Proof of Proposition \ref{lemma:Holder_modulus_Non-Deg} in the Off-Diagonal Regime.}
From the Duhamel-type expansion \eqref{align:Representation2_deriv}, we can represent a mild solution $u$ of Equation for any fixed
$(\tau,\bm{\xi}),(\tau',\bm{\xi}')$ in $[0,T]\times \R^{nd}$ as
\begin{multline*}
\vert D_{\bm{x}_1}u(t,\bm{x})-D_{\bm{x}_1}u(t,\bm{x}')\vert \, \le \, \bigl{\vert}D_{\bm{x}_1}\tilde{P}^{\tau,\bm{\xi}}_{t,T}g(\bm{x}) -
D_{\bm{x}_1}\tilde{P}^{\tau'\!,\bm{\xi}'}_{t,T}\!g(\bm{x}')\bigr{\vert} + \bigl{\vert}D_{\bm{x}_1}\tilde{G}^{\tau,\bm{\xi}}_{t,T}f(t,\bm{x}) - D_{\bm{x}_1}\tilde{G}^{\tau'\!,\bm{\xi}'}_{t,T}\!f(t,\bm{x}') \bigr{\vert}\\
+\Bigl{\vert}\int_{t}^{T}D_{\bm{x}_1}\tilde{P}^{\tau,\bm{\xi}}_{t,s}R^{\tau,\bm{\xi}}(s,\bm{x}) -
D_{\bm{x}_1}\tilde{P}^{\tau'\!,\bm{\xi}'}_{t,s}\!R^{\tau'\!,\bm{\xi}'}\!(s,\bm{x}') \, ds\Bigr{\vert}.
\end{multline*}
After possible differentiations, we will choose $\tau=\tau'=t$, $\bm{\xi} = \bm{x}$ and $\bm{\xi}' = \bm{x}'$ in order to exploit the
sensitivity controls \eqref{eq:Sensitivity_on_flow1} and \eqref{eq:Sensitivity_on_flow}.

\emph{Control on the frozen semigroup.} It can be handled following the analogous part in the proof of the
H\"older control for the proxy (Lemma
\ref{lemma:Holder_modulus_proxy_Non-Deg}).  The only difference is that we cannot control
\[d(\tilde{\bm{m}}^{\tau,\bm{\xi}}_{t,T} (\bm{x}),\tilde{\bm{m}}^{\tau'\!,\bm{\xi}'}_{t,T}\!(\bm{x}'))\]
in Equation \eqref{zz1} using the affinity of the mapping $\bm{x} \to \tilde{\bm{m}}^{\tau,\bm{\xi}}_{t,T}(\bm{x})$, since the two freezing
point are now different. Instead, we can take $\tau = \tau' = t$, $\bm{\xi} = \bm{x}$ and $\bm{\xi}' = \bm{x}'$ and apply the sensitivity
control \eqref{eq:Sensitivity_on_flow} to write that
\[d(\tilde{\bm{m}}^{\tau,\bm{\xi}}_{t,T} (\bm{x}),\tilde{\bm{m}}^{\tau'\!,\bm{\xi}'}_{t,T}\!(\bm{x}')) \, = \,
d(\bm{\theta}_{t,T}(\bm{x}), \bm{\theta}_{t,T}(\bm{x}')) \, \le \, C\Vert \bm{F}\Vert_H d(\bm{x},\bm{x}').\]
%%%%%%%%%%%%%%%%%%%%%%%%%%%%%%%%%%%%%%%%%%%%%%%%%%%%%%%%%%%%%%%%%

\emph{Control on the Green kernel.} It follows immediately from the proof of the H\"older control (Lemma
\ref{lemma:Holder_modulus_proxy_Non-Deg}) for the proxy, noticing that $t_0=T$, since we are in the off-diagonal regime.
%%%%%%%%%%%%%%%%%%%%%%%%%%%%%%%%%%%%%%%%%%%%%%%%%%%%%%%

\emph{Control on the perturbative error.} Since we do not exploit the difference of the spatial points $(\bm{x},\bm{x}')$ in the
off-diagonal regime but instead we control the two contributions separately, we can rely  on the controls on the supremum norms we have
already shown in Lemma \ref{lemma:supremum_norm}. Namely, we start writing that
\begin{multline}\label{z1}
\Bigl{\vert}\int_{t}^{T}D_{\bm{x}_1}\tilde{P}^{\tau,\bm{\xi}}_{t,s}R^{\tau,\bm{\xi}}(s,\bm{x})- D_{\bm{x}_1}
\tilde{P}^{\tau'\!,\bm{\xi}'}_{t,s}\!R^{\tau'\!,\bm{\xi}'}\!(s,\bm{x}')\, ds\Bigr{\vert} \\
\le \, \Bigl{\vert}\int_{t}^{T}D_{\bm{x}_1}\tilde{P}^{\tau,\bm{\xi}}_{t,s}R^{\tau,\bm{\xi}}(s,\bm{x})\, ds\Bigr{\vert}+
\Bigl{\vert}\int_{t}^{T}D_{\bm{x}_1} \tilde{P}^{\tau'\!,\bm{\xi}'}_{t,s}\!R^{\tau'\!,\bm{\xi}'}\!(s,\bm{x}')\, ds\Bigr{\vert}
\end{multline}
Then, we can follow the same reasonings of Lemma \ref{lemma:supremum_norm} concerning the remainder term (cf. Equations
\eqref{Proof:Supremum_Schauder0}, \eqref{Proof:Supremum_Schauder1} and \eqref{Proof:Supremum_Schauder2}) to
show that
\begin{equation}\label{z2}
\Bigl{\vert}\int_{t}^{T}D_{\bm{x}_1}\tilde{P}^{\tau,\bm{\xi}}_{t,s}R^{\tau, \bm{\xi}}(s,\bm{x}) \, ds\Bigr{\vert} \, \le \, C\Vert
\bm{F}\Vert_H\Vert u \Vert_{L^{\infty}(C^{\alpha+\beta}_{b,d})} (T-t)^{\frac{\alpha+\beta-1}{\alpha}}.
\end{equation}
Using it in the above Equation \eqref{z1}, we can finally conclude that
\begin{equation}\label{z3}
\Bigl{\vert}\int_{t}^{T}D_{\bm{x}_1}\tilde{P}^{\tau,\bm{\xi}}_{t,s}R^{\tau,\bm{\xi}}(s,\bm{x})- D_{\bm{x}_1}
\tilde{P}^{\tau'\!,\bm{\xi}'}_{t,s}\!R^{\tau'\!,\bm{\xi}'}\!(s,\bm{x}')\, ds\Bigr{\vert} \, \le \,
C\Vert \bm{F}\Vert_H\Vert u \Vert_{L^{\infty}(C^{\alpha+\beta}_{b,d})} d^{\alpha+\beta-1}(\bm{x},\bm{x}')
\end{equation}
remembering that we assumed to be in the off-diagonal regime, i.e.\ $T-t\le c_0d^\alpha(\bm{x},\bm{x}')$ for some $c_0\le1$.

\paragraph{Proof of Proposition \ref{lemma:Holder_modulus_Deg} in the Off-Diagonal Regime.}
As done before, we are going to analyze separately the single terms appearing from the Duhamel-type representation
\eqref{align:Representation2} of a solution $u$:
\begin{multline*}
\vert u(t,\bm{x})-u(t,\bm{x}')\vert \, \le \, \bigl{\vert}\tilde{P}^{\tau,\bm{\xi}}_{t,T}g(\bm{x}) - \tilde{P}^{\tau'\!,\bm{\xi}'}_{t,T}\!g
(\bm{x}')\bigr{\vert} + \bigl{\vert}\tilde{G}^{\tau,\bm{\xi}}_{t,T}f(t,\bm{x}) - \tilde{G}^{\tau'\!,\bm{\xi}'}_{t,T}\!f(t,\bm{x}')\bigr{\vert}\\
+\Bigl{\vert}\int_{t}^{T}\tilde{P}^{\tau,\bm{\xi}}_{t,s}R^{\tau,\bm{\xi}}(s,\bm{x})-\tilde{P}^{\tau'\!,\bm{\xi}'}_{t,s}
\!R^{\tau'\!,\bm{\xi}'}\!(s,\bm{x}') \, ds\Bigr{\vert}
\end{multline*}
for some $(\tau,\bm{\xi}),(\tau',\bm{\xi}')$ in $[0,T]\times \R^{nd}$ fixed but to be chosen later as $\tau = \tau' = t$, $\bm{\xi} = \bm{x}$
and $\bm{\xi}' = \bm{x}'$.

\emph{Control on the frozen semigroup.} We can essentially follow the proof of the H\"older control (Lemma
\ref{lemma:Holder_modulus_proxy_Deg}) for the proxy. However, this time we cannot exploit the affinity of the mapping $\bm{x} \to
\tilde{\bm{m}}^{\tau,\bm{\xi}}_{t,T}(\bm{x})$ to control the difference
\[\bigl{\vert} g(\tilde{\bm{m}}^{\tau,\bm{\xi}}_{t,T}(\bm{x})-z)-g(\tilde{\bm{m}}^{\tau,\bm{\xi}}_{t,T}(\bm{x}')-z)\bigr{\vert}.\]
Instead, we notice now that we can bound it as
\[\bigl{\vert} g(\tilde{\bm{m}}^{\tau,\bm{\xi}}_{t,T}(\bm{x})-z)-g(\tilde{\bm{m}}^{\tau,\bm{\xi}}_{t,T}(\bm{x}')-z)\bigr{\vert} \,
\le \, C\Vert g \Vert_{C^{\alpha+\beta}_{b,d}}\Bigl[d^{\alpha+\beta}\bigl(\tilde{\bm{m}}^{\tau,\bm{\xi}}_{t,T}(\bm{x}),\tilde{
\bm{m}}^{\tau,\bm{\xi}}_{t,T}(\bm{x}')\bigr)+\bigl{\vert}\bigl(\tilde{\bm{m}}^{\tau,\bm{\xi}}_{t,T}(\bm{x})-\tilde{\bm{m}}^{
\tau,\bm{\xi}}_{t,T}(\bm{x}')\bigr)_1 \bigr{\vert}\Bigr]\]
since $g$ is differentiable and thus Lipschitz continuous, in the first non-degenerate variable.\newline
Taking now  $\tau = \tau' = t$, $\bm{\xi} = \bm{x}$ and $\bm{\xi}' = \bm{x}'$, we can use the sensitivity controls
\eqref{eq:Sensitivity_on_flow} and \eqref{eq:Sensitivity_on_flow1} (noticing that by assumption, $\bm{x}_1=\bm{x}'_1$) to write that
\[\bigl{\vert} g(\tilde{\bm{m}}^{\tau,\bm{\xi}}_{t,T}(\bm{x})-z)-g(\tilde{\bm{m}}^{\tau,\bm{\xi}}_{t,T}(\bm{x}')-z)\bigr{\vert} \, \le \,
C\Vert F\Vert_H\Vert g \Vert_{C^{\alpha+\beta}_{b,d}}d^{\alpha+\beta}(\bm{x},\bm{x}')\]

%%%%%%%%%%%%%%%%%%%%%%%%%%%%%%%%%%%%%%%%%%%%%%%%%%%%%%%%%%%%%%%%%%%%%%%%%%%%%%%%%%%%%%%%%%%%%%%%%%%%%

\emph{Control on the Green kernel.} It can be obtained following the analogous part in the proof of the H\"older control (Lemma
\ref{lemma:Holder_modulus_proxy_Deg}) for the proxy. Similarly to the paragraph "Control on the frozen semigroup" in the previous proof, we
need to take $(\tau,\bm{\xi})= (t,\bm{x})$, $(\tau,\bm{\xi}') = (t,\bm{x})$ and apply the sensitivity control
\eqref{eq:Sensitivity_on_flow} to control the term
\[d(\tilde{\bm{m}}^{\tau,\bm{\xi}}_{t,T} (\bm{x}),\tilde{\bm{m}}^{\tau'\!,\bm{\xi}'}_{t,T}\!(\bm{x}'))\]
appearing in Equation \eqref{zz2}.
%%%%%%%%%%%%%%%%%%%%%%%%%%%%%%%%%%%%%%%%%%%%%%%%%%%%%%%%%%%%%%%%%%%%%%%%%%%%%

\emph{Control on the perturbative error.} The proof of this estimate essentially matches the previous, analogous one in the
non-degenerate setting. Namely, Equations \eqref{z1}, \eqref{z2} and \eqref{z3} hold again with $(T-t)^{\frac{\beta+\alpha}{\alpha}}$
instead of $(T-t)^{\frac{\beta+\alpha-1}{\alpha}}$.

\subsubsection{Diagonal Regime}

Since the aim of this section is to prove Lemmas \ref{lemma:Holder_modulus_Non-Deg} and \ref{lemma:Holder_modulus_Deg} when a
diagonal regime is assumed, we will assume from this point further that $T-t \ge c_0d^\alpha(\bm{x},
\bm{x}')$ for some given $(t,\bm{x},\bm{x}')$ in $[0,T]\times \R^{2nd}$. \newline
As preannounced in the introduction of this section, we need here a modification of the Duhamel-type representation
\eqref{eq:Expansion_along_proxy} that allows to change the freezing points along the time integration variable. Remembering the previous
notations for $\tilde{G}^{\tau,\bm{\xi}}_{r,v}$ and $R^{\tau,\bm{\xi}}$ in \eqref{eq:def_Green_Kernel} and \eqref{eq:def_remainder_regul}
respectively, it holds that

\begin{lemma}[Change of Frozen Point] %Change of freezing point
Let $(\tau,\bm{\xi})$ be a freezing couple in $[0,T]\times\R^{nd}$ and $\tilde{\bm{\xi}}$ another freezing point in $\R^{nd}$. Then, any
classical solution $u$ in $L^\infty(0,T;C^{\alpha+\beta}_{b,d}(\R^{nd}))$ of Equation \eqref{Degenerate_Stable_PDE} can be represented for
any $(t,\bm{x})$ in $[0,T]\times\R^{nd}$ as
\begin{multline}\label{eq:Change_of_Freez_point_HOlder}
u(t,\bm{x}) \, = \, \tilde{P}^{\tau,\tilde{\bm{\xi}}}_{t,T}g(\bm{x}) + \tilde{G}^{\tau,\bm{\xi}}_{t,t_0}f(t,\bm{x}) +
\tilde{G}^{\tau,\tilde{\bm{\xi}}}_{t_0,T}f(t,\bm{x}) \\
+ \int_{t}^{t_0}  \tilde{P}^{\tau,\bm{\xi}}_{t,s}R^{\tau,\bm{\xi}} (s,\bm{x})\, ds +\int_{t_0}^{T} \tilde{P}^{\tau,\tilde{\bm{\xi}}}_{t,s}
R^{\tau,\tilde{\bm{\xi}}}(s,\bm{x})\, ds + \tilde{P}^{\tau,\bm{\xi}}_{t,t_0}u(t_0,\bm{x}) -\tilde{P}^{\tau,\tilde{\bm{\xi}}}_{t,
t_0}u(t_0,\bm{x})
\end{multline}
where $t_0$ is the change of regime time defined in \eqref{eq:def_t0}.
\end{lemma}
\begin{proof}
Fixed $t$ in $(0,T)$, we start considering another point $r$ in $(t,T)$. On $(0,r)$, it is clear that $u$ is again a mild
solution of equation \eqref{Degenerate_Stable_PDE} but with terminal condition $u(r,\bm{x})$. Then, Duhamel expansion
\eqref{eq:Expansion_along_proxy} can be applied with respect to the frozen couple $(\tau,\bm{\xi})$, allowing us to write that
\[u(t,\bm{x}) \, = \, \tilde{P}^{\tau,\bm{\xi}}_{t,r}g(\bm{x}) + \int_{t}^{r}\tilde{P}^{\tau,\bm{\xi}}_{t,s}f(s,\bm{x}) \, ds + \int_{t}^{r}
\tilde{P}^{\tau,\bm{\xi}}_{t,s}R^{\tau,\bm{\xi}} u(s,\bm{x})\, ds.\]
Noticing that $u$ is independent from $r$, it is possible now to differentiate the above equality with respect to $r$ in $(t,T)$ to show that
\begin{equation}\label{eq:change_frez_point0}
0 \, = \, \partial_r \bigl[\tilde{P}^{\tau,\bm{\xi}}_{t,r}u(r,\bm{x})\bigr] + \tilde{P}^{\tau,\bm{\xi}}_{t,r}f(r,\bm{x}) +
\tilde{P}^{\tau,\bm{\xi}}_{t,r}R^{\tau,\bm{\xi}}(r,\bm{x}).
\end{equation}
We highlight now that the above expression holds for any chosen frozen couple $(\tau,\bm{\xi})$ and any fixed time $r$. Thus, it is possible
to integrate it with respect to $r$ for a fixed $\bm{\xi}$ between $t$ and $t_0$ and for another frozen point
$\tilde{\bm{\xi}}$ between $t_0$ and $T$, leading to
\begin{multline*}
0 \, = \, \tilde{P}^{\tau,\bm{\xi}}_{t,t_0}u(t_0,\bm{x}) - \tilde{P}^{\tau,\bm{\xi}}_{t,t}u(t,\bm{x}) +\int_{t}^{t_0}\tilde{P}^{\tau,
\bm{\xi}}_{t,r}f(r,\bm{x}) \,  dr +  \int_{t}^{t_0} \tilde{P}^{\tau,\bm{\xi}}_{t,r}R^{\tau,\bm{\xi}}(r,\bm{x}) \, dr  \\
+ \tilde{P}^{\tau,\tilde{\bm{\xi}}}_{t,T}u(T,\bm{x}) - \tilde{P}^{\tau,\tilde{\bm{\xi}}}_{t,t_0}u(t_0,\bm{x}) + \int_{t_0}^{T}
\tilde{P}^{\tau,\tilde{\bm{\xi}}}_{t,r}f(r,\bm{x}) \,  dr +  \int_{t_0}^{T} \tilde{P}^{\tau,\tilde{\bm{\xi}}}_{t,r}
R^{\tau,\tilde{\bm{\xi}}}(r,\bm{x}) \,  dr.
\end{multline*}
With our previous notations, the above expression can be finally rewritten as
\begin{multline*}
0 \, = \, \tilde{P}^{\tau,\bm{\xi}}_{t,t_0}u(t_0,\bm{x}) - u(t,\bm{x}) + \tilde{G}^{\tau,\bm{\xi}}_{t,t_0}f(t,\bm{x}) +  \int_{t}^{t_0}
\tilde{P}^{\tau,\bm{\xi}}_{t,r}R^{\tau,\bm{\xi}}(r,\bm{x}) \, dr  \\
+ \tilde{P}^{\tau,\tilde{\bm{\xi}}}_{t,T}g(\bm{x}) - \tilde{P}^{\tau,\tilde{\bm{\xi}}}_{t,t_0}u(t_0,\bm{x}) +
\tilde{G}^{\tau,\tilde{\bm{\xi}}}_{t_0,T}f(t,\bm{x}) +  \int_{t_0}^{T}
\tilde{P}^{\tau,\tilde{\bm{\xi}}}_{t,r}R^{\tau,\tilde{\bm{\xi}}}(r,\bm{x}) \,  dr
\end{multline*}
and we have concluded.
\end{proof}

Similarly to the off-diagonal case, we are going to apply the auxiliary estimates associated with the proxy (Lemmas
\ref{lemma:Controls_on_means1} and \ref{lemma:Controls_on_means}) in the current diagonal regime. Namely, taking $s=t_0$ and
$(\bm{y},\bm{y}')=(\bm{x},\bm{x})$ in Lemma \ref{lemma:Controls_on_means1}, we know that there exists a constant
$C\ge1$ such that for any $t$ in $[0,T]$ and any $\bm{x},\bm{x}'$ in $\R^{nd}$,
\begin{equation}\label{eq:Sensitivity_on_mean1} \text{if }\bm{x}_1 \, = \, \bm{x}'_1, \quad
\bigl{\vert}(\tilde{\bm{m}}^{t,\bm{x}}_{t,t_0}(\bm{x})-
\tilde{\bm{m}}^{t,\bm{x}'}_{t,t_0}(\bm{x}))_1\bigr{\vert} \, \le \, C\Vert \bm{F}\Vert_Hd^{\alpha+\beta}(\bm{x},\bm{x}').
\end{equation}
Moreover, in order to control the perturbative term when a local diagonal regime appears, i.e.\ when the time integration variable $s$ is in
$[t_0,T]$, we will quite often use a Taylor expansion on the frozen density. To be able to exploit the already proven controls, such that
the smoothing effect for the frozen density (Equation \eqref{eq:Smoothing_effects_of_tilde_p}) or the Besov control (Lemma
\ref{lemma:Second_Besov_COntrols}), we will need the following:
\begin{equation}
\label{eq:translation_inv_for_density}
\text{if } s-t\ge c_0d^\alpha(\bm{x},\bm{x}'), \quad \bigl{\vert} D^\vartheta_{\bm{x}}\tilde{p}^{\tau,\bm{\xi}'} (t,s,\bm{x}+ \lambda(\bm{x}'-
\bm{x}),\bm{y})\bigr{\vert} \, \le \, C \bigl{\vert}D^\vartheta_{\bm{x}}\tilde{p}^{\tau,\bm{\xi}'} (t,s,\bm{x},\bm{y})\bigr{\vert}
\end{equation}
for any multi-index $\vartheta$ in $\N^d$ such that $\vert \vartheta\vert \le 2$ and any $\lambda$ in $[0,1]$.
The proof of these results can be found in Section A.$2$.

We are now ready to prove Propositions \ref{lemma:Holder_modulus_Non-Deg} and \ref{lemma:Holder_modulus_Deg} when a global
diagonal regime is considered.

\paragraph{Proof of Proposition \ref{lemma:Holder_modulus_Non-Deg} in the Diagonal Regime.}
We start recalling that in Lemma \ref{lemma:Holder_modulus_Non-Deg} we assumed fixed a time $t$ in $[0,T]$ and two spatial points
$\bm{x},\bm{x}'$ in $\R^{nd}$ such that $\bm{x}_j=\bm{x}'_j$ if $j \neq 1$.\newline
From the above representation \eqref{eq:Change_of_Freez_point_HOlder} and the Duhamel-type formula \eqref{eq:Expansion_along_proxy}, we know
that
\begin{multline*}
D_{\bm{x_1}}u(t,\bm{x}) - D_{\bm{x_1}}u(t,\bm{x}') \, = \, \Bigl(D_{\bm{x_1}}\tilde{P}^{\tau,\tilde{\bm{\xi}}}_{t,T}g(\bm{x}) -
D_{\bm{x_1}}\tilde{P}^{\tau'\!,\bm{\xi}'}_{t,T}\!g(\bm{x}')\Bigr) \\
+ \, \Bigl(D_{\bm{x_1}}\tilde{G}^{\tau,\bm{\xi}}_{t,t_0}f(t,\bm{x})+ D_{\bm{x_1}}\tilde{G}^{\tau,\tilde{\bm{\xi}}}_{t_0,T}f(t,\bm{x})
-D_{\bm{x_1}}\tilde{G}^{\tau'\!,\bm{\xi}'}_{t,T}\!f(t,\bm{x}')\Bigr) \\
+ \Bigl(\int_{t}^{t_0}  D_{\bm{x_1}}\tilde{P}^{\tau,\bm{\xi}}_{t,s}R^{\tau,\bm{\xi}}(s,\bm{x})\, ds +\int_{t_0}^{T}D_{\bm{x_1}}
\tilde{P}^{\tau,\tilde{\bm{\xi}}}_{t,s}R^{\tau,\tilde{\bm{\xi}}}(s,\bm{x})\, ds -\int_{t}^{T}  D_{\bm{x_1}}\tilde{P}^{\tau'\!,
\bm{\xi}'}_{t,s}\!R^{\tau'\!,\bm{\xi}'}\!(s,\bm{x}') \, ds\Bigr) \\
+ \Bigl(D_{\bm{x_1}}\tilde{P}^{\tau,\bm{\xi}}_{t,t_0}u(t_0,\bm{x})
-D_{\bm{x_1}}\tilde{P}^{\tau,\tilde{\bm{\xi}}}_{t,t_0}u(t_0,\bm{x})\Bigr)
\end{multline*}
for some freezing couples $(\tau,\bm{\xi}),(\tau,\tilde{\bm{\xi}}),(\tau',\bm{\xi}')$ in $[0,T]\times \R^{nd}$ fixed but to be chosen later.
To help the readability of the following, we assume from this point further $\tau=\tau'$ and $\tilde{\bm{\xi}}=\bm{\xi}'$.

\emph{Control on frozen semigroup.} We start focusing on the control of the frozen semigroup, i.e.
\[\bigl{\vert}D_{\bm{x_1}}\tilde{P}^{\tau,\bm{\xi}'}_{t,T}g(\bm{x}) - D_{\bm{x_1}}\tilde{P}^{\tau,\bm{\xi}'}_{t,T}g(\bm{x}')\bigr{\vert}.\]
Since the freezing couples coincide, the control on the frozen semigroup can be obtained following the proof of the H\"older control (Lemma
\ref{lemma:Holder_modulus_proxy_Non-Deg}) for the proxy.
%%%%%%%%%%%%%%%%%%%%%%%%%%%%%%%%%%%%%%%%%%%%%%%%%%%%%%%%%%%%%%%%%%%%%%%%%%%%%%%%%%%%%%%%%%%%%%%

\emph{Control on the Green kernel.} As done before, we split the analysis with respect to the change of regime time $t_0$. Namely, we write
\begin{multline*}
\bigl{\vert} D_{\bm{x}_1}\tilde{G}^{\tau,\bm{\xi}}_{t,t_0}f(t,\bm{x})+
D_{\bm{x}_1}\tilde{G}^{\tau,\tilde{\bm{\xi}}}_{t_0,T}f(t,\bm{x}) -D_{\bm{x}_1}\tilde{G}^{
\tau,\bm{\xi}'}_{t,T}f(t,\bm{x}')\bigl{\vert} \\
\le \, \bigl{\vert} D_{\bm{x}_1}\tilde{G}^{\tau,\bm{\xi}}_{t,t_0}f(t,\bm{x})-D_{\bm{x}_1}\tilde{G}^{\tau,\bm{\xi}'}_{t,t_0}
f(t,\bm{x}') \bigl{\vert}+ \bigl{\vert} D_{\bm{x}_1}\tilde{G}^{\tau,\tilde{\bm{\xi}}}_{t_0,T}f(t,\bm{x}) -
D_{\bm{x}_1}\tilde{G}^{\tau,\bm{\xi}'}_{t_0,T}f(t,\bm{x}')\bigl{\vert}.
\end{multline*}
While in the local off-diagonal regime, the first term in the r.h.s.\ of the above expression can be handled as in the
global off-diagonal regime, the local diagonal regime contribution represented by
\[\bigl{\vert} D_{\bm{x}_1} \tilde{G}^{\tau,\tilde{\bm{\xi}}}_{t_0,T}f(t,\bm{x}) - D_{\bm{x}_1}
\tilde{G}^{\tau,\bm{\xi}'}_{t_0,T}f(t,\bm{x}')\bigr{\vert} \, = \,\bigl{\vert} D_{\bm{x}_1} \tilde{G}^{\tau,\bm{\xi}'}_{t_0,T}f(t,\bm{x}) -
D_{\bm{x}_1} \tilde{G}^{\tau,\bm{\xi}'}_{t_0,T}f(t,\bm{x}')\bigr{\vert}\]
since $\tilde{\bm{\xi}}=\bm{\xi}'$, can be controlled following again the proof of the H\"older control (Lemma
\ref{lemma:Holder_modulus_proxy_Non-Deg}) for the proxy.
%%%%%%%%%%%%%%%%%%%%%%%%%%%%%%%%%%%%%%%%%%%%%%%%%%%%%%%%%%%%%%

\emph{Control on the discontinuity term.} We can now focus on the contribution
\[\bigl{\vert}D_{\bm{x_1}}\tilde{P}^{\tau,\bm{\xi}}_{t,t_0}u(t_0,\bm{x})
-D_{\bm{x_1}}\tilde{P}^{\tau,\tilde{\bm{\xi}}}_{t,t_0}u(t_0,\bm{x})\bigr{\vert},\]
arising from the change of freezing point in the representation \eqref{eq:Change_of_Freez_point_HOlder}.\newline
Since at fixed time $t_0$, the function $u$ shows the same spatial regularity of $g$, this control can be handled following the paragraph in
the proof of the H\"older control for the proxy (Lemma \ref{lemma:Holder_modulus_proxy_Non-Deg}) concerning the frozen semigroup in the
off-diagonal regime. The only main difference is in Equation \eqref{zz1} where, this time, we need to take $(\tau,\bm{\xi},\bm{\xi}') =
(t,\bm{x},\bm{x}')$ and exploit the sensitivity estimate (Lemma \ref{lemma:Controls_on_means}) to control the quantity
\[d(\tilde{\bm{m}}^{\tau,\bm{\xi}}_{t,t_0}(\bm{x}),\tilde{\bm{m}}^{\tau,\bm{\xi}'}_{t,t_0}(\bm{x})).\]
In the end, it is possible to show again (cf. Equation \eqref{a2}) that
\[\bigl{\vert}D_{\bm{x}_1}\tilde{P}^{\tau,\bm{\xi}}_{t,t_0}u(t_0,\bm{x})-D_{\bm{x_1}}\tilde{P}^{\tau,\tilde{\bm{\xi}}}_{t,t_0}
u(t_0,\bm{x})\bigr{\vert} \le \, C\Vert u\Vert_{L^\infty(C^{\alpha+\beta}_{b,d})}c_0^{\frac{\alpha + \beta
-1}{\alpha}}d^{\alpha+\beta-1}(\bm{x},\bm{x}').\]
%%%%%%%%%%%%%%%%%%%%%%%%%%%%%%%%%%%%%%%%%%%%%%%%%%%%%%%%%%%%%%%%%%%%%%%%%%%%%%%

\emph{Control on the perturbative term.} We start splitting the analysis into two cases with respect to the
critical time $t_0$ giving the change of regime. Namely, we write
\begin{multline*}
\Bigl{\vert}\int_{t}^{t_0}  D_{\bm{x}_1}\tilde{P}^{\tau,\bm{\xi}}_{t,s}R^{\tau,\bm{\xi}}(s,\bm{x})\, ds +\int_{t_0}^{T}
D_{\bm{x}_1}\tilde{P}^{\tau,\bm{\xi}'}_{t,s}R^{\tau,\bm{\xi}'}(s,\bm{x})\, ds -\int_{t}^{T}  D_{\bm{x}_1}\tilde{P}^{
\tau,\bm{\xi}'}_{t,s}R^{\tau,\bm{\xi}'}(s,\bm{x}') \, ds\Bigl{\vert} \\
\le \, \Bigl{\vert}\int_{t}^{t_0}D_{\bm{x}_1}\tilde{P}^{\tau,\bm{\xi}}_{t,s}R^{\tau,\bm{\xi}}(s,\bm{x})-D_{\bm{x}_1}
\tilde{P}^{\tau,\bm{\xi}'}_{t,s}R^{\tau,\bm{\xi}'}(s,\bm{x}')\, ds\Bigr{\vert} +\Bigl{\vert}\int_{t_0}^{T}D_{\bm{x}_1}\tilde{P}^{
\tau,\bm{\xi}'}_{t,s}R^{\tau,\bm{\xi}'}(s,\bm{x}) -D_{\bm{x}_1}\tilde{P}^{\tau,\bm{\xi}'}_{t,s}R^{\tau,\bm{\xi}'}(s,\bm{x}') \,
ds\Bigl{\vert}.
\end{multline*}
We then notice that the local off-diagonal regime represented by
\[\Bigl{\vert}\int_{t}^{t_0}D_{\bm{x}_1}\tilde{P}^{\tau,\bm{\xi}}_{t,s}R^{\tau,\bm{\xi}}(s,\bm{x})-D_{\bm{x}_1}
\tilde{P}^{\tau,\bm{\xi}'}_{t,s}R^{\tau,\bm{\xi}'}(s,\bm{x}')\, ds\Bigr{\vert}\]
can be handled following the proof in the global off-diagonal regime of Lemma \ref{lemma:Holder_modulus_Non-Deg}.\newline
We can then focus our attention on the local diagonal regime, i.e.
\[\Bigl{\vert}\int_{t_0}^{T}D_{\bm{x}_1}\tilde{P}^{\tau,\bm{\xi}'}_{t,s}R^{\tau,\bm{\xi}'}(s,\bm{x})
-D_{\bm{x}_1}\tilde{P}^{\tau,\bm{\xi}'}_{t,s}R^{\tau,\bm{\xi}'}(s,\bm{x}') \, ds\Bigl{\vert}.\]
Since the freezing couples coincide, we can use a Taylor expansion with respect to the first variable $\bm{x}_1$ and write that
\begin{multline*}
\Bigl{\vert}\int_{t_0}^{T}D_{\bm{x}_1}\tilde{P}^{\tau,\bm{\xi}'}_{t,s}R^{\tau,\bm{\xi}'}(s,\bm{x})-D_{\bm{x}_1}\tilde{P}^{
\tau,\bm{\xi}'}_{t,s}R^{\tau,\bm{\xi}'}(s,\bm{x}') \, ds\Bigl{\vert}\\
= \, \Bigl{\vert} \int_{t_0}^{T}\int_{\R^{nd}}\int_{0}^{1}D^2_{\bm{x}_1} \tilde{p}^{\tau,\bm{\xi}'} (t,s,\bm{x}+
\lambda(\bm{x}'-\bm{x}),\bm{y})(\bm{x}'-\bm{x})_1R^{\tau,\bm{\xi}'}(s,\bm{y}) \, d\bm{y} ds d\lambda \Bigr{\vert}.
\end{multline*}
Noticing that we are integrating from $t_0$ to $T$, equation \eqref{eq:translation_inv_for_density} can be rewritten as
\begin{multline}\label{Proof:H\"older_Perturbative_term_Id0}
\Bigl{\vert}\int_{t_0}^{T}D_{\bm{x}_1}\tilde{P}^{\tau,\bm{\xi}'}_{t,s}R^{\tau,\bm{\xi}'}(s,\bm{x})-D_{\bm{x}_1}\tilde{P}^{
\tau,\bm{\xi}'}_{t,s}R^{\tau,\bm{\xi}'}(s,\bm{x}') \, ds\Bigl{\vert}\, \le \, \vert(\bm{x}'-\bm{x})_1\vert\sum_{j=1}^{n}\int_{0}^{1}\int_{t_0
}^{T}\Bigl{\vert} \int_{\R^{nd}}D^2_{\bm{x}_1}\tilde{p}^{\tau,\bm{\xi}'} (t,s,\bm{x},\bm{y})\\
\Bigl{\{}\bigl[\bm{F}_j(s,\bm{y})-\bm{F}_j(s,\bm{\theta}_{t,s}(\bm{\xi}'))\bigr]\cdot D_{\bm{y}_j}u(s,\bm{y})\Bigr{\}} \, d\bm{y}
\Bigr{\vert} ds d\lambda \, =: \, \vert(\bm{x}-\bm{x}')_1\vert\sum_{j=1}^{n}\int_{t_0}^{T}I^d_j(s) ds
\end{multline}
As done before, we are going to treat separately the cases $j=1$ and $j>1$. In the first case, the term $I^{d}_1$ can be easily controlled by
\begin{multline}\label{Proof:H\"older_Perturbative_term_Id1}
I^{d}_1(s) \, \le \, \Vert D_{\bm{y}_1}u\Vert_{L^\infty(L^\infty)} \int_{\R^{nd}} \bigl{\vert}
D^2_{\bm{x}_1}\tilde{p}^{\tau,\bm{\xi}'}(t,s,\bm{x},\bm{y})\bigr{\vert}\,\bigl{\vert}\bm{F}_1(s,\bm{y})-
\bm{F}_1(s,\bm{\theta}_{t,s}(\bm{\xi}'))\bigr{\vert}\, d\bm{y} \\
\le \, C\Vert \bm{F} \Vert_H\Vert u \Vert_{L^{\infty}(C^{\alpha+\beta}_{b,d})}(s-t)^{\frac{\beta-2}{\alpha}}
\end{multline}
where in the last passage we used the smoothing effect for the frozen density $\tilde{p}^{\tau,\bm{\xi}}$ (Equation
\eqref{eq:Smoothing_effects_of_tilde_p}).\newline
On the other side, the case $j>1$ can be exploited using the second Besov control (Lemma \ref{lemma:Second_Besov_COntrols}). For this reason,
we start using integration by parts formula to show that
\[I^{d}_j(s) \, = \, \Bigl{\vert}\int_{\R^{nd}}D_{\bm{y}_j}\cdot\Bigl{\{}D^2_{\bm{x}_1}\tilde{p}^{\tau,\bm{\xi}'}(t,s,\bm{x},\bm{y})
\otimes\bigl[\bm{F}_j(s,\bm{y})-\bm{F}_j(s,\bm{\theta}_{t,s}(\bm{\xi}'))\bigr]\Bigr{\}}u(s,\bm{y})\, d\bm{y} \Bigr{\vert}.\]
Through the duality in Besov spaces \eqref{Besov:duality_in_Besov} and the identification \eqref{Besov:ident_Holder_Besov}, we
then write that
\begin{multline*}
I^{d}_j(s) \,\le \\
 C\Vert u \Vert_{L^\infty(C^{\alpha+\beta}_{b,d})}\int_{\R^{(n-1)d}} \Vert D_{\bm{y}_j}\cdot\Bigl{\{}D^2_{\bm{x}_1} \tilde{p}^{\tau,
\bm{\xi}'}(t,s,\bm{x},\bm{y}_{\smallsetminus j},\cdot)\otimes\bigl[\bm{F}_j(s,\bm{y}_{\smallsetminus j},\cdot) - \bm{F}_j(s,\bm{\theta}_{t,s}(
\bm{\xi}')) \bigr]\Bigr{\}}\Vert_{ B^{-(\alpha_j +\beta_j)}_{1,1}}\,d\bm{y}_{\smallsetminus j}.
\end{multline*}
We can now apply the Second Besov control (Lemma \ref{lemma:Second_Besov_COntrols}) to show that
\begin{equation}\label{Proof:H\"older_Perturbative_term_Idj}
I^d_j(s) \, \le \, C\Vert \bm{F}\Vert_H\Vert u \Vert_{L^\infty(C^{\alpha+\beta}_{b,d})}(s-t)^{\frac{\beta-2}{\alpha}}.
\end{equation}
Going back at Equations \eqref{Proof:H\"older_Perturbative_term_Id0} \eqref{Proof:H\"older_Perturbative_term_Id1} and
\eqref{Proof:H\"older_Perturbative_term_Idj}, we can write that
\begin{multline}\label{Proof:H\"older_Perturbative_term_Id2}
\Bigl{\vert}\int_{t_0}^{T}D_{\bm{x}_1}\tilde{P}^{\tau,\bm{\xi}'}_{t,s}\!R^{\tau,\bm{\xi}'}(s,\bm{x}) -D_{\bm{x}_1} \tilde{P}^{
\tau,\bm{\xi}'}_{t,s}R^{\tau,\bm{\xi}'}(s,\bm{x}') \, ds\Bigl{\vert} \, \le \, C\Vert \bm{F}\Vert_H\Vert u\Vert_{L^\infty(C^{\alpha+\beta}_{b,d})} \vert(\bm{x}-\bm{x}')_1\vert\int_{t_0}^{T}
(s-t)^{\frac{\beta-2}{\alpha}} ds \\
\le \, C\Vert \bm{F}\Vert_H\Vert u \Vert_{L^\infty(C^{\alpha+\beta}_{b,d})} \vert(\bm{x}-\bm{x}')_1\vert
(t_0-t)^{\frac{\alpha+\beta-2}{\alpha}}
\end{multline}
where in the last passage we used that $\frac{\alpha +\beta-2}{\alpha}<0$ to pick the starting point $t_0$ in the integral.\newline
Using that $t_0-t=c_0d^{\alpha}(\bm{x},\bm{x}')$, we can conclude that
\[\Bigl{\vert}\int_{t_0}^{T}D_{\bm{x}_1}\tilde{P}^{\tau,\bm{\xi}'}_{t,s}R^{\tau,\bm{\xi}'}(s,\bm{x})
-D_{\bm{x}_1}\tilde{P}^{\tau,\bm{\xi}'}_{t,s}R^{\tau,\bm{\xi}'}(s,\bm{x}') \, ds\Bigl{\vert}
\, \le \, C c^{\frac{\alpha+\beta-2}{\alpha}}_0\Vert \bm{F}\Vert_H\Vert u \Vert_{L^{\infty}(C^{\alpha+\beta}_{b,d})}
d^{\alpha+\beta-1}(\bm{x},\bm{x}').\]

We can conclude this section showing the H\"older control in the degenerate setting when a diagonal regime is assumed.

\paragraph{Proof of Proposition \ref{lemma:Holder_modulus_Deg} in Diagonal Regime.}
We start recalling that in proposition \eqref{lemma:Holder_modulus_Deg} we assumed fixed a time $t$ in $[0,T]$ and two spatial points
$\bm{x},\bm{x}'$ in $\R^{nd}$ such that $\bm{x}_j=\bm{x}'_j$ if $j \neq i$ for some $i$ in $\llbracket 2 ,n \rrbracket$.

Representation \eqref{eq:Change_of_Freez_point_HOlder} and Duhamel-type expansion \eqref{eq:Expansion_along_proxy} allows to control the
Holder modulus of a solution $u$ analyzing separately the
different terms:
\begin{multline*}
u(t,\bm{x}) - u(t,\bm{x}') \, = \, \Bigl(\tilde{P}^{\tau,\tilde{\bm{\xi}}}_{t,T}g(\bm{x}) - \tilde{P}^{\tau'\!,\bm{\xi}'}_{t,T}\!g(\bm{x}')
\Bigr) + \Bigl(\tilde{G}^{\tau,\bm{\xi}}_{t,t_0}f(t,\bm{x})+ \tilde{G}^{\tau,\tilde{\bm{\xi}}}_{t_0,T}f(t,\bm{x}) -
\tilde{G}^{\tau'\!,\bm{\xi}'}_{t,T}\!f(t,\bm{x}')\Bigr) \\
+ \Bigl(\int_{t}^{t_0}  \tilde{P}^{\tau,\bm{\xi}}_{t,s}R^{\tau,\bm{\xi}}(s,\bm{x})\, ds +\int_{t_0}^{T}
\tilde{P}^{\tau,\tilde{\bm{\xi}}}_{t,s}R^{\tau,\tilde{\bm{\xi}}}(s,\bm{x})\, ds -
\int_{t}^{T}  \tilde{P}^{\tau'\!,\bm{\xi}'}\!R^{\tau'\!,\bm{\xi}'}\!(s,\bm{x}') \, ds\Bigr) + \Bigl(\tilde{P}^{\tau,\bm{\xi}}_{t,t_0}u(t_0,\bm{x})
-\tilde{P}^{\tau,\tilde{\bm{\xi}}}_{t,t_0}u(t_0,\bm{x})\Bigr)
\end{multline*}
for some freezing couples $(\tau,\bm{\xi}),(\tau,\tilde{\bm{\xi}}),(\tau,\bm{\xi}')$ fixed but to be chosen later. As done before, we assume
however from this point further that $\tau=\tau'$ and $\tilde{\bm{\xi}}=\bm{\xi}'$.

\emph{Control on the frozen semigroup.} Noticing that we have taken the same freezing couples since $\tilde{\bm{\xi}}=\bm{\xi}'$, the
control on the frozen semigroup $\bigl{\vert}\tilde{P}^{\tau,\bm{\xi}'}_{t,T}g(\bm{x}) -
\tilde{P}^{\tau,\bm{\xi}'}_{t,T}g(\bm{x}')\bigr{\vert}$ can be obtained exploiting the same argument used in the proof of the H\"older
control (Lemma \ref{lemma:Holder_modulus_proxy_Deg}) for the proxy.
%%%%%%%%%%%%%%%%%%%%%%%%%%%%%%%%%%%%%%%%%%%%%%%%%%%%%%%%%%%%%%%%%%%%%%%%%%%%%%%%%%%%%%%%%%%%%%%

\emph{Control on the Green kernel.} The proof of this estimate essentially matches the previous, analogous one in the
non-degenerate setting. Namely, we follow the proof in the global off-diagonal regime of Proposition \ref{lemma:Holder_modulus_Deg} to
control the local off-diagonal regime contribution $\bigl{\vert} \tilde{G}^{\tau,\bm{\xi}}_{t,t_0}f(t,\bm{x})-\tilde{G}^{\tau,\bm{\xi}'
}_{t,t_0} f(t,\bm{x}') \bigl{\vert}$ while in the locally diagonal regime term
\[\bigl{\vert} \tilde{G}^{\tau,\bm{\xi}'}_{t_0,T} f(t,\bm{x}) - \tilde{G}^{\tau,\bm{\xi}'}_{t_0,T}f(t,\bm{x}')\bigr{\vert},\]
the freezing couples coincide and we can thus exploit the same argument used in the proof of the H\"older control (Lemma
\ref{lemma:Holder_modulus_proxy_Deg}) for the proxy.

\emph{Control on the discontinuity term.}
The proof of this result will follow essentially the one about the off-diagonal regime of the frozen semigroup with respect to the
degenerate variables. It holds that
\begin{multline*}
\tilde{P}^{\tau,\bm{\xi}}_{t,t_0}u(t_0,\bm{x}) \, = \, \int_{\R^{nd}}\tilde{p}^{\tau,\bm{\xi}}(t,t_0,\bm{x},\bm{y})u(t_0,\bm{y})
\, d\bm{y} \\
= \, \int_{\R^{nd}}\frac{1}{\det \bigl(\mathbb{M}_{t_0-t}\bigr)}p_S(t_0-t,\mathbb{M}^{-1}_{t_0-t}\bigl(\tilde{\bm{m}}^{\tau,\bm{\xi}}_{t,
t_0}(\bm{x})-\bm{y}\bigr )u(t_0,\bm{y}) \, d\bm{y} \\
= \, \int_{\R^{nd}}\frac{1}{\det\bigl(\mathbb{M}_{t_0-t}\bigr)}p_S(t_0-t,\mathbb{M}^{-1}_{t_0-t}\bm{z}\bigr)u(t_0,\tilde{\bm{m}}^{
\tau,\bm{\xi}}_{t,t_0}(\bm{x})-\bm{z}) \,d\bm{z}
\end{multline*}
where in the last passage we used the change of variable $\bm{z} = \tilde{\bm{m}}^{\tau,\bm{\xi}}_{t,t_0}(\bm{x})-\bm{y}$.
Since a similar argument works also for $\tilde{P}^{\bm{\xi}'}_{t,t_0}u(t_0,\bm{x})$, it then follows that
\begin{multline*}
\bigl{\vert}\tilde{P}^{\tau,\bm{\xi}}_{t,t_0}u(t_0,\bm{x}) - \tilde{P}^{\tau,\bm{\xi}'}_{t,t_0}u(t_0,\bm{x})\bigr{\vert} \\
=\, \Bigl{\vert}\int_{\R^{nd}}\frac{1}{\det\bigl(\mathbb{M}_{t_0-t}\bigr)}p_S\bigl(t_0-t,\mathbb{M}^{-1}_{t_0-t}\bm{z}\bigr)
\bigl[u(t_0,\tilde{\bm{m}}^{\tau,\bm{\xi}}_{t,t_0}(\bm{x})-\bm{z})-u(t_0,\tilde{\bm{m}}^{\tau,\bm{\xi}'}_{t,t_0}(\bm{x})-\bm{z})
\bigr]\,d\bm{z}\Bigr{\vert}
\end{multline*}
Remembering that $u(t_0,\cdot)$ is Lipschitz with respect to the first non-degenerate variable, we can write now that
\begin{multline*}
\bigl{\vert}\tilde{P}^{\tau,\bm{\xi}}_{t,t_0}u(t_0,\bm{x}) - \tilde{P}^{\tau,\bm{\xi}'}_{t,t_0}u(t_0,\bm{x})\bigr{\vert} \\
\le \, C\Vert u \Vert_{L^\infty(C^{\alpha+\beta}_{b,d})}\bigl[d^{\alpha+\beta}(\tilde{\bm{m}}^{\tau,\bm{\xi}}_{t,t_0}(\bm{x}),\tilde{
\bm{m}}^{\tau,\bm{\xi}'}_{t,t_0}(\bm{x}))+\bigl{\vert} (\tilde{\bm{m}}^{\tau,\bm{\xi}}_{t,t_0}(\bm{x})-\tilde{\bm{m}}^{
\tau,\bm{\xi}'}_{t,t_0}(\bm{x}))_1\bigr{\vert}\bigr]\int_{\R^{nd}}p_S\bigl(t_0-t,
\mathbb{M}^{-1}_{t_0-t} \bm{z}\bigr)\,\frac{d\bm{z}}{\det\bigl(\mathbb{M}_{t_0-t}\bigr)} \\
\le \,C\Vert u \Vert_{L^\infty(C^{\alpha+\beta}_{b,d})}\bigl[d^{\alpha+\beta}(\tilde{\bm{m}}^{\tau,\bm{\xi}}_{t,t_0}(\bm{x}),\tilde{
\bm{m}}^{\tau,\bm{\xi}'}_{t,t_0}(\bm{x}))+\bigl{\vert} (\tilde{\bm{m}}^{\tau,\bm{\xi}}_{t,t_0}(\bm{x})-\tilde{\bm{m}}^{
\tau,\bm{\xi}'}_{t,t_0}(\bm{x}))_1\bigr{\vert}\bigr].
\end{multline*}
Taking $(\bm{\xi}=\bm{\xi}'=\bm{x})$, we can then use the sensitivity controls (Lemma \ref{lemma:Controls_on_means} and Equation
\eqref{eq:Sensitivity_on_mean1}) to show that
\[\bigl{\vert}\tilde{P}^{\tau,\bm{\xi}}_{t,t_0}u(t_0,\bm{x}) - \tilde{P}^{\tau,\bm{\xi}'}_{t,t_0}u(t_0,\bm{x})\bigr{\vert} \, \le \,
C\Vert u\Vert_{L^\infty(C^{\alpha+\beta}_{b,d})}\Vert \bm{F} \Vert_Hc_0^{\frac{\alpha+\beta}{1+\alpha(n-1)}}d^{\alpha+\beta}
(\bm{x},\bm{x}).\]
%%%%%%%%%%%%%%%%%%%%%%%%%%%%%%%%%%%%%%%%%%%%%%%%%%%%%%%%%

\emph{Control on the perturbative term.}
The proof of this Estimate essentially matches the previous, analogous one in the non-degenerate setting. Namely, Inequalities
\eqref{Proof:H\"older_Perturbative_term_Id1}, \eqref{Proof:H\"older_Perturbative_term_Idj} and \eqref{Proof:H\"older_Perturbative_term_Id2}
hold again with $(s-t)^{\frac{\beta-2}{\alpha}}$ replaced by $(s-t)^{\frac{\beta}{\alpha}-\frac{1}{\alpha_i}}$.
\setcounter{equation}{0}

\subsubsection{Mollifying Procedure}
We now make the mollifying parameter $m$ appear again using the notations introduced in Section $3.2$ (see
Equation \eqref{eq:Expansion_along_proxy}). Then, Lemmas \ref{lemma:supremum_norm}, \ref{lemma:Holder_modulus_Non-Deg} and
\ref{lemma:Holder_modulus_Deg} rewrite together in the following way. There exists a constant $C>0$ such that for any $m$ in $\N$,
\begin{equation}\label{eq:A_Priori_Estimates_Reg}
 \Vert u_m \Vert_{L^\infty(C^{\alpha+\beta}_{b,d})} \, \le \, Cc_0^{\frac{\beta-\gamma_n}{\alpha}}\bigl[\Vert g_m
\Vert_{C^{\alpha+\beta}_{b,d}}+\Vert f_m \Vert_{L^\infty(C^{\beta}_{b,d})}\bigr]+C\bigl(c_0^{\frac{\beta-\gamma_n}{\alpha}}\Vert
\bm{F}_m\Vert_H + c_0^{\frac{\alpha+\beta-1}{1+\alpha(n-1)}}\bigr)\Vert u_m \Vert_{L^\infty(C^{\alpha+\beta}_{b,d})}
\end{equation}
where $c_0$ is assumed to be fixed but chosen later. Importantly, $c_0$ and $C$ does not depends on the regularizing parameter $m$. Thus,
letting $m$ go to $\infty$ and remembering the definition \ref{definition:mild_sol} of mild solution $u$, the above expression immediately
implies the A priori estimates (Proposition \ref{prop:A_Priori_Estimates}).

\section{Existence Result}
The aim of this section is to show the well-posedness in a mild sense of the original IPDE \eqref{Degenerate_Stable_PDE}.
Recalling Definition \ref{definition:mild_sol} for a mild solution of the IPDE \eqref{Degenerate_Stable_PDE}, let us consider three
sequences $\{f_m\}_{m\in \N}$, $\{g_m\}_{m\in \N}$
and $\{\bm{F}_m\}_{m\in \N}$ of "regularized" coefficients such that
\begin{itemize}
  \item $\{f_m\}_{m\in \N}$ is in $C^\infty_b((0,T)\times\R^{nd})$ and $f_m$ converges to $f$ in
      $L^\infty\bigl(0,T;C^\beta_{b,d}(\R^{nd})\bigr)$;
  \item $\{g_m\}_{m\in \N}$ is in $C^\infty_b(\R^{nd})$ and $g_m$ converges to $g$ in $C^{\alpha+\beta}_{b,d}(\R^{nd})$;
   \item $\{\bm{F}_m\}_{m\in \N}$ is in $C^\infty_b((0,T)\times\R^{nd};\R^{nd})$ and $\Vert\bm{F}_m-\bm{F}\Vert_H$ converges to $0$.
\end{itemize}
It can be derived through stochastic flows techniques (see e.g. \cite{book:Kunita04}) that there exists a solution $u_m$ in $C^\infty_b((0,T)\times\R^{nd})$ of the "regularized" IPDE:
\[
\begin{cases}
   \partial_t u_m(t,\bm{x})+L_\alpha u_m(t,\bm{x}) + \langle A \bm{x} + \bm{F}_m(t,\bm{x}), D_{\bm{x}}
   u_m(t,\bm{x})\rangle  \, = \, -f_m(t,\bm{x}) &\mbox{on }   (0,T)\times \R^{nd}, \\
    u_m(T,\bm{x}) \, = \, g_m(\bm{x}) & \mbox{on }\R^{nd}.
  \end{cases}
\]
In order to pass to the limit in $m$, we notice now the arguments above for the proof of the Schauder estimates (Equation
\eqref{equation:Schauder_Estimates}) can be applied to the above dynamics, too. Namely, there exists a constant $C>0$ such that
\[\Vert u_m \Vert_{L^\infty(C^{\alpha+\beta}_{b,d})} \, \le \, C\bigl[\Vert f_m\Vert_{L^\infty(C^{\beta}_{b,d})}+\Vert g_m
\Vert_{C^{\alpha+\beta}_{b,d}}\bigr] \, \le \,C\bigl[\Vert f\Vert_{L^\infty(C^{\beta}_{b,d})}+\Vert g
\Vert_{C^{\alpha+\beta}_{b,d}}\bigr].\]
Importantly, the above estimates is uniformly in $m$ and thus, the sequence $\{u_m\}_{m\in \N}$ is bounded in the space $L^\infty\bigl(C^{\alpha+\beta}_{b,d}(\R^{nd})\bigr)$.
From Arzel\`a-Ascoli Theorem, we deduce now that there exists $u$ in $L^\infty\bigl(C^{\alpha+\beta}_{b,d}(\R^{nd})\bigr)$ and a sequence
$\{u_{m_k}\}_{k \in \N}$ of smooth and bounded functions converging to $u$ in $L^\infty\bigl(C^{\alpha+\beta}_{b,d}(\R^{nd})\bigr)$ and such
that $u_{m_k}$ is solution of the "regularized" IPDE \eqref{Regularizied_PDE}. It is then clear that $u$ is a mild solution of the original
IPDE \eqref{Degenerate_Stable_PDE}.

\paragraph{From Mild to Weak Solutions}
We conclude showing that any mild solution $u$ of the IPDE \eqref{Degenerate_Stable_PDE} is indeed a weak solution. The proof
of this result will be essentially an application of the arguments presented before, especially the Second Besov Control (Lemma \ref{lemma:Second_Besov_COntrols}).
Let $u$ be a mild solution of the IPDE \eqref{Degenerate_Stable_PDE} in $L^\infty\bigl(0,T;C^{\alpha+\beta}_{b,d}(\R^{nd})\bigr)$. Recalling
the definition of weak solution in \eqref{eq:Def_weak_sol}, we start fixing a test function $\phi$ in $C^\infty_0\bigl((0,T]\times
\R^{nd}\bigr)$ and passing to the "regularized" setting (see Definition \ref{definition:mild_sol}), we then notice that it holds that
\[\int_{0}^{T}\int_{\R^{nd}}\phi(t,\bm{y})\Bigl(\partial_t+\mathcal{L}_{m,\alpha}\Bigr)u_m(t,\bm{y}) \, d\bm{y} \, = \,
-\int_{0}^{T}\int_{\R^{nd}}\phi(t,\bm{y})f_m(t,\bm{y}) \, d\bm{y}\]
where $\mathcal{L}_{m,\alpha}$ is the "complete" operator defined in \eqref{eq:Complete_Operator} but with respect to the regularized
coefficients. An integration by parts allows now to move the operators to the test function. Indeed, remembering that
$u_m(T,\cdot)=g_m(\cdot)$, it holds that
\begin{equation}\label{TOweak:1}
\int_{0}^{T}\int_{\R^{nd}}\Bigl(-\partial_t+\mathcal{L}_{m,\alpha}^*\Bigr)\phi(t,\bm{y})u_m(t,\bm{y}) \, d\bm{y}dt +
\int_{\R^{nd}}\phi(T,\bm{y})g_m(\bm{y}) \, d\bm{y} \, = \, -\int_{0}^{T}\int_{\R^{nd}}\phi(t,\bm{y})f_m(t,\bm{y}) \, d\bm{y}dt
\end{equation}
where $\mathcal{L}^*_{m,\alpha}$ denotes the formal adjoint of $\mathcal{L}_{m,\alpha}$. We would like now to go back to the solution $u$,
letting $m$ go to $\infty$. We start rewriting the right-hand side term  in the following way:
\[\int_{0}^{T}\int_{\R^{nd}}\phi(t,\bm{y})f_m(t,\bm{y}) \, d\bm{y}dt \, = \, \int_{0}^{T}\int_{\R^{nd}}\phi(t,\bm{y})f(t,\bm{y}) \,
d\bm{y}dt + \int_{0}^{T}\int_{\R^{nd}}\phi(t,\bm{y})\bigl[f_m-f\bigr](t,\bm{y}) \, d\bm{y}dt.\]
Exploiting that by assumption, $f_m$ converges to $f$ in $L^\infty(0,T;C^{\beta}_{b,d}(\R^{nd}))$, it is easy to see that the second
contribution above goes to $0$ if we let $m$ go to $\infty$. A similar argument can be used to show that
\[\int_{\R^{nd}}\phi(T,\bm{y})g_m(\bm{y}) \, d\bm{y} \, \overset{m}{\to} \, \int_{\R^{nd}}\phi(T,\bm{y})g(\bm{y}) \, d\bm{y}.\]
On the other hand, we can decompose the first term in the left-hand side of Equation \eqref{TOweak:1} as
\begin{equation}\label{weak:decom}
\int_{0}^{T}\int_{\R^{nd}}\Bigl(-\partial_t+\mathcal{L}_{m,\alpha}^*\Bigr)\phi(t,\bm{y})u_m(t,\bm{y}) \, d\bm{y}dt \, = \, \int_{0}^{T}\int_{\R^{nd}}\Bigl(-\partial_t+\mathcal{L}_\alpha^*\Bigr)\phi(t,\bm{y})u(t,\bm{y}) \, d\bm{y}dt + R^1_m + R^2_m
\end{equation}
where above we have denoted
\begin{align*}
  R^1_m \, & = \, \int_{0}^{T}\int_{\R^{nd}}\Bigl[\mathcal{L}_{\alpha}^*-\mathcal{L}_{m,\alpha}^*\Bigr]\phi(t,\bm{y})u_m(t,\bm{y}) \, d\bm{y}dt\\
  R^2_m \, & = \, \int_{0}^{T}\int_{\R^{nd}}\Bigl(-\partial_t+\mathcal{L}_{\alpha}^*\Bigr)\phi(t,\bm{y})\bigl[u_m(t,\bm{y})-u(t,\bm{y})\bigr] \, d\bm{y}dt
\end{align*}
with $\mathcal{L}_{\alpha}^*$ as the formal adjoint of the complete operator $\mathcal{L}_\alpha$. Noticing that
\[\Bigl[\mathcal{L}_{\alpha}^*-\mathcal{L}_{m,\alpha}^*\Bigr]\phi(t,\bm{y}) \, = \,D_{\bm{y}}\cdot\bigl{\{}\phi(t,\bm{y})[\bm{F}(t,\bm{y})-\bm{F}_m(t,\bm{y})] \bigr{\}},\]
it is clear that the remainder contribution $R^1_m$ can be essentially handled as in the introduction of Section $5.1$, exploiting that $\Vert \bm{F}-\bm{F}_m\Vert_H\to 0$.\newline
To control instead the second contribution $R^2_m$, we start decomposing it as
\begin{multline*}
R^2_m \, = \, -\int_{0}^{T}\int_{\R^{nd}}\partial_t\phi(t,\bm{y})\bigl[u_m-u\bigr](t,\bm{y}) \, d\bm{y}dt + \sum_{j=1}^{n}
\int_{0}^{T}\int_{\R^{nd}}D_{\bm{y}_j}\bigl[\phi \bm{F}_j\bigr](t,\bm{y})\bigl[u_m-u\bigr](t,\bm{y}) \, d\bm{y}dt \\
=: \, R^2_{0,m}+\sum_{j=1}^{n}R^2_{j,m}.
\end{multline*}
We firstly observe that $\vert R^2_{0,m}\vert$ goes to $0$ if we let $m$ go to $\infty$, since $\Vert u-u_m\Vert_{L^\infty(C^{\alpha+\beta}_{b,d})
}\overset{m}{\to} 0$. On the other hand, an integration by parts allows to show that
\[\vert R^2_{1,m} \vert \, = \, \Bigl{\vert}\int_{0}^{T}\int_{\R^{nd}}\bigl[\phi \bm{F}\bigr](t,\bm{y})D_{\bm{y}_j}\bigl[u_m-u\bigr](t,\bm{y}) \,
d\bm{y}dt\Bigr{\vert}\]
which again tends to $0$ when $m$ goes to $\infty$.
To control instead the contributions $R^m_{j,m}$ for $j>1$, the point is to use the Besov duality argument again. Namely, from Equations
\eqref{Besov:duality_in_Besov}, \eqref{Besov:ident_Holder_Besov} and with the notations in \eqref{eq:notation_smallsetminus}, it holds that
\begin{multline*}
\vert R^2_{j,m} \vert \, \le \, \int_{0}^{T}\int_{\R^{d(n-1)}}\bigl{\Vert}D_{\bm{y}_j}\bigl[\phi \bm{F}\bigr](t,\bm{y}_{\smallsetminus
j},\cdot)\bigr{\Vert}_{B^{-(\alpha_j+\beta_j)}_{1,1}}\bigl{\Vert}\bigl[u_m-u\bigr](t,\bm{y}_{\smallsetminus
j},\cdot)\bigr{\Vert}_{B^{\alpha_j+\beta_j}_{\infty,\infty}} \, d\bm{y}_{\smallsetminus j}dt \\
\le \, \int_{0}^{T}\int_{\R^{d(n-1)}}\bigl{\Vert}D_{\bm{y}_j}\bigl[\phi \bm{F}\bigr](t,\bm{y}_{\smallsetminus
j},\cdot)\bigr{\Vert}_{B^{-(\alpha_j+\beta_j)}_{1,1}}\bigl{\Vert}\bigl[u_m-u\bigr](t,\bm{y}_{\smallsetminus
j},\cdot)\bigr{\Vert}_{C^{\alpha_j+\beta_j}_b} \, d\bm{y}_{\smallsetminus j}dt.
\end{multline*}
Following the same arguments in the Proof of the Second Besov Control (Lemma \ref{lemma:Second_Besov_COntrols}), we know that there exists a
constant $C$ such that $\bigl{\Vert}D_{\bm{y}_j}\bigl[\phi \bm{F}\bigr](t,\bm{y}_{\smallsetminus
j},\cdot)\bigr{\Vert}_{B^{\alpha_j+\beta_j}_{1,1}} \le C\psi_j(t,\bm{y}_{\smallsetminus j})$ where $\psi_j$ has compact support on
$\R^{d(n-1)}$. \newline
Since moreover $\Vert u_m -u \Vert$ goes to zero with $m$, we easily deduce that $R^2_{m,j}\overset{m}{\to}0$ for any $j$ in $\llbracket
2,n\rrbracket$. From the above controls, we can deduce now that $R^1_m + R^2_m\overset{m}{\to} 0$. From Equation \eqref{weak:decom}, it then
follows that
\[\int_{0}^{T}\int_{\R^{nd}}\Bigl(-\partial_t+\mathcal{L}_{m,\alpha}^*\Bigr)\phi(t,\bm{y})u_m(t,\bm{y}) \, d\bm{y}dt \, \overset{m}{\to} \,
\int_{0}^{T}\int_{\R^{nd}}\Bigl(-\partial_t+\mathcal{L}_\alpha^*\Bigr)\phi(t,\bm{y})u(t,\bm{y}) \, d\bm{y}dt\]
and we have concluded.

\setcounter{equation}{0}
\section{Extensions}
As already said in the introduction, our assumption of (global) H\"older regularity on the drift $\bar{\bm{F}}$, as well as the choice of
considering a perturbed Ornstein-Uhlenbeck operator instead of a more general non-linear dynamics, was done to
preserve, as possible, the clarity and understandability of the article. In this conclusive section, we would like to explain briefly how it
possible to naturally extend it.
\subsection{General Drift}
We start illustrating how the perturbative method explained above can be easily adapted to work in a
more general setting. In particular, the same results (Schauder-type estimates and well-posedness of the IPDE \eqref{Degenerate_Stable_PDE})
can be shown also for an equation of the form:
\begin{equation}
\label{Ext:Degenerate_Stable_PDE}
\begin{cases}
   \partial_t u(t,\bm{x})+ L_\alpha u(t,\bm{x}) + \langle\bar{\bm{F}}(t,\bm{x}), D_{\bm{x}}u(t,\bm{x})\rangle \, =
   \, -f(t,\bm{x}), & \mbox{on } (0,T)\times \R^{nd} \\
    u(T,\bm{x}) \, = \, g(\bm{x}) & \mbox{on }\R^{nd}.
  \end{cases}
\end{equation}
where $\bar{\bm{F}}(t,\bm{x})=\bigl(\bar{\bm{F}}_1(t,\bm{x}),\dots,\bar{\bm{F}}_n(t,\bm{x})\bigr)$ has the following structure
\[\bar{\bm{F}}_i(t,\bm{x}_{(i-1)\vee1},\dots,\bm{x}_n).\]
We remark in particular that if for any $i$ in
$\llbracket 2,n\rrbracket$, $\bar{\bm{F}}_i$ is linear with respect to $\bm{x}_{i-1}$ and independent from time, the previous analysis
works since we can rewrite $\bar{\bm{F}}(t,\bm{x}) =A\bm{x}+\bm{F}(t,\bm{x})$.

In order to deal with this more general dynamics addressed in the diffusive setting in \cite{Chaudru:Honore:Menozzi18_Sharp}, we will need however to add some additional constraints and to modify slightly the ones
presented in assumption $(\bm{A})$. First of all, the non-degeneracy assumption $(\textbf{H})$ does not make sense in this new framework and
it will be replaced by the following condition:
\begin{description}
  \item[(H')] the matrix $D_{\bm{x}_{i-1}}\bar{\bm{F}}_i(t,\bm{x})$ has full rank $d$ for any $i$ in $\llbracket2,n
      \rrbracket$ and any $(t,\bm{x})$ in $[0,T]\times \R^{nd}$.
\end{description}
In particular, we will say that assumption $(\bar{\textbf{A}})$ is in force when
\begin{description}
   \item[(S')] assumption (\textbf{ND}) and (\textbf{H'}) are satisfied and the drift $\bar{\bm{F}}=(\bar{\bm{F}}_1,\dots,\bar{\bm{F}}_n)$ is such that for any
       $i$ in $\llbracket 2,n\rrbracket$, $\bar{\bm{F}}_i$ depends only on time and on the last $n-(i-2)\vee 0$ components, i.e.\
       $\bar{\bm{F}}_i(t,\bm{x}_{i-1},\dots,\bm{x}_n)$;
  \item[(P')] $\alpha$ is a number in $(0,2)$, $\beta$ is in $(0,1)$ and it holds that
  \[\beta<\alpha, \quad \alpha+\beta\in (1,2) \, \, \text{ and }\,\, \beta < (\alpha-1)(1+\alpha(n-1));\]
\item[(R')] The source $f$ is in $L^\infty(0,T;C^{\beta}_{b,d}(\R^{nd}))$, the terminal condition $g$ is in
      $C^{\alpha+\beta}_{b,d}(\R^{nd})$ and for any $i$ in $\llbracket 1,n\rrbracket$, $\bar{\bm{F}}_i$ belongs
      $L^\infty(0,T;C^{\gamma_i+\beta}_{d}(\R^{nd}))$ where $\gamma_i$ was defined in \eqref{Drift_assumptions}.
    \end{description}

To prove Schauder-type estimates for a solution of equation \eqref{Ext:Degenerate_Stable_PDE}, our idea is to adapt the perturbative
approach to this new dynamics. In particular, we can exploit the differentiability of $\bar{\bm{F}}_i$ with respect to
$\bm{x}_{i-1}$ to "linearize" it along a flow that takes into account the perturbation (cf. Section $3.1$). Namely, we
are interested in:
\begin{equation}\label{Extension:Frozen_PDE}
\begin{cases}
    \partial_t \bar{u}^{\tau,\bm{\xi}}(t,\bm{x})+L_\alpha \bar{u}^{\tau,\bm{\xi}}(t,\bm{x}) + \bigl{\langle}\bar{A}^{\tau,\bm{\xi}}_t
   \bigl(\bm{x}-\bar{\bm{\theta}}_{\tau,t}(\bm{\xi})\bigr) + \bar{\bm{F}}(t,\bar{\bm{\theta}}_{\tau,t}(\bm{\xi})), D_{\bm{x}}
   \bar{u}^{\tau,\bm{\xi}}(t,\bm{x})\bigr{\rangle} \, &= \, -f(t,\bm{x}), \\
    \bar{u}^{\tau,\bm{\xi}}(T,\bm{x}) \, &= \, g(\bm{x})
  \end{cases}
\end{equation}
where the time-dependent matrix $\bar{A}^{\tau,\bm{\xi}}_t$ is defined through
\[\bigl[\bar{A}^{\tau,\bm{\xi}}_t\bigr]_{i,j} \, = \,
\begin{cases}
  D_{\bm{x}_{i-1}}\bar{\bm{F}}_i(t,\bm{\theta}_{\tau,t}(\bm{\xi})), & \mbox{if } j=i-1 \\
  0_{d\times d}, & \mbox{otherwise}
\end{cases}
\]
and $\bar{\bm{\theta}}_{\tau,t}(\bm{\xi}))$ is a fixed flow satisfying the dynamics
\begin{equation}
\bar{\bm{\theta}}_{\tau,t}(\bm{\xi}) \, = \, \bm{\xi} +\int_{\tau}^{t} \bar{\bm{F}}(v,\bar{\bm{\theta}}_{\tau,v}(\bm{\xi})) \, dv.
\end{equation}
A first significant difference with respect to the previous approach consists in handling a time-dependent matrix
$\bar{A}^{\tau,\bm{\xi}}_t$.
Indeed, it is possible to modify slightly the presentation in \cite{Priola:Zabczyk09} (allowing time-dependency on $A$) in order to
show that under assumption (\textbf{S'}), the two parameters semigroup $\bigl(\bar{P}^{\tau,\bm{\xi}}_{t,s}\bigr)_{t\le s}$  associated with
the proxy operator
\[L_\alpha + \langle\bar{A}^{\tau,\bm{\xi}}_t \bigl(\bm{x}-\bar{\bm{\theta}}_{\tau,t}(\bm{\xi})\bigr) +
\bar{\bm{F}}(t,\bar{\bm{\theta}}_{\tau,t}(\bm{\xi})), D_{\bm{x}} \rangle\]
admits a density $\bar{p}^{\tau,\bm{\xi}}$ and that it can be rewritten as
\[ \bar{p}^{\tau,\bm{\xi}}(t,s,x,y)\,  = \, \frac{1}{\det (\mathbb{M}_{s-t})} p_S\bigl( s-t,\mathbb{M}^{-1}_{s-t} (\bm{y} -
\bar{\bm{m}}^{\tau,\bm{\xi}}_{t,s}(\bm{x}))\bigr).\]
Here, the notations for $p_S$ and $\mathbb{M}_{t}$ remain the same of above while this time the shift $\bar{\bm{m}}^{\tau,\bm{\xi}}_{t,s}$
is defined through
\[\bar{\bm{m}}^{\tau,\bm{\xi}}_{t,s}(\bm{x}) \, = \, \mathcal{R}^{\tau,\bm{\xi}}_{t,s}\bm{x} + \int_{t}^{s}\mathcal{R}^{\tau,\bm{\xi}}_{v,s}
\bigl[\bar{\bm{F}}(v, \bar{\bm{\theta}}_{\tau,v}(\bm{\xi}))-\bar{A}^{\tau,\bm{\xi}}_v\bar{\bm{\theta}}_{\tau,v}(\bm{\xi})\bigr] \, dv\]
where $\mathcal{R}^{\tau,\bm{\xi}}_{t,s}$ is the time-ordered resolvent of $\bar{A}^{\tau,\bm{\xi}}_s$ starting at time $t$, i.e.
\[\begin{cases}
    d\mathcal{R}^{\tau,\bm{\xi}}_{t,s} \, = \,  \bar{A}^{\tau,\bm{\xi}}_s\mathcal{R}^{\tau,\bm{\xi}}_{t,s}ds, & \mbox{on } [t,T] \\
    \mathcal{R}^{\tau,\bm{\xi}}_{t,t}\, = \, I.
  \end{cases}\]
We can as well refer to \cite{Huang:Menozzi15} for related issues (see Proposition $3.2$ and Section C about the
linearization, therein).

Following the same reasonings of Propositions \ref{prop:frozen_Duhamel_Formula} and \ref{prop:Expansion_along_proxy}, it is then possible to
state a Duhamel type formula suitable for the IPDE \ref{Ext:Degenerate_Stable_PDE}:
\begin{equation}\label{Ext:Expansion_along_proxy}
u(t,\bm{x}) \, = \, \bar{P}^{\tau,\bm{\xi}}_{t,T}g(\bm{x}) + \int_{t}^{T}\bar{P}^{\tau,\bm{\xi}}_{t,s} \bigl[f(s,\cdot) +\bar{R}^{\tau,
\xi}(s,\cdot)\bigr](\bm{x}) \, ds
\end{equation}
where the remainder term is given now by
\[\bar{R}^{\tau,\xi}(t,\bm{x}) \, = \, \langle\bm{F}(t,\bm{x})-\bm{F}(t,\bar{\bm{\theta}}_{\tau,t}(\bm{\xi}))-
\bar{A}_t^{\tau,\bm{\xi}}\bigl(\bm{x}-\bar{\bm{\theta}}_{\tau,t}(\bm{\xi})\bigr),D_{\bm{x}}u(t,\bm{x})\rangle.\]

Looking back at the first part of the article, it is important to notice that the main steps of proof (cf. Equation
\eqref{eq:Smoothing_effects_of_tilde_p}, Propositions \ref{prop:Schauder_Estimates_for_proxy}, \ref{prop:A_Priori_Estimates} and Section
$3.3$) does not rely on the explicit formulas for $\bar{\bm{m}}^{\tau,\bm{\xi}}_{t,s}(\bm{x})$ and
$\bar{R}^{\tau,\xi}$ but instead, they
exploit only the Besov controls for the remainder $\bar{R}^{\tau,\xi}$ (cf. Section $5.1$ ) and the controls on the shift
$\bar{\bm{m}}^{\tau,\bm{\xi}}_{t,s}(\bm{x})$ (Section A.$2$).
Hence, once we have proven the suitable controls, the proofs of the analogous results for the new dynamics \eqref{Ext:Degenerate_Stable_PDE}
can be obtained easily modifying slightly the notations and following the same reasonings above.\newline
For example, exploiting that
\[\bar{\bm{m}}^{\tau,\bm{\xi}}_{t,s}(\bm{x}) \, = \, \bm{x} + \int_{t}^{s}\mathcal{R}^{\tau,\bm{\xi}}_{t,v}\Bigl(\bar{\bm{m}}^{\tau,
\bm{\xi}}_{t,v}(\bm{x}) - \bm{\theta}_{\tau,v}(\bm{\xi})\Bigr) + \bm{F}(v, \bm{\theta}_{\tau,v}(\bm{\xi})) \, dv,\]
we can follow the same method of proof in the above lemma \ref{lemma:identification_theta_m} to show again that
\[\bar{\bm{m}}^{\tau,\bm{\xi}}_{t,s}(\bm{x}) \, = \, \bar{\bm{\theta}}_{\tau,s}(\bm{\xi})\]
taking $\tau=t$ and $\bm{\xi}=\bm{x}$.

Letting the interested reader look in the appendix for the suggestions on how to extend the controls on the shift
$\bar{\bm{m}}^{\tau,\bm{\xi}}_{t,s}(\bm{x})$ in this more general setting, we will focus now on proving the Besov controls.
First of all, we notice immediately that the proof of the first Besov control \ref{lemma:First_Besov_COntrols} relies essentially only on
the smoothing effect \eqref{eq:Smoothing_effects_of_tilde_p} and thus, it can be obtained following the same reasoning above. The proof of
the second Besov control (Lemma \ref{lemma:Second_Besov_COntrols}) in this framework is a bit more involved and we are going to
explain it below more in details.\newline
We start noticing that the second Besov Lemma \ref{lemma:Second_Besov_COntrols} can be reformulated for the new dynamics in the
following way
\[\int_{\R^{(n-1)d}}\Bigl{\Vert}D_{\bm{y}_j}\cdot\Bigl{\{}D^{\vartheta}_{\bm{x}}\bar{p}^{\tau,\bm{\xi}}(t,s,\bm{x},
\bm{y}_{\smallsetminus j},\cdot) \otimes \bar{\Delta}^{\tau,\bm{\xi}}_j(s,\bm{y}_{\smallsetminus j},\cdot) \Bigr{\}}
\Bigr{\Vert}_{B^{-(\alpha_j+\beta_j)}_{1,1}} \, d\bm{y}_{\smallsetminus j}\, \le \, C\Vert \bar{\bm{F}}\Vert_H
(s-t)^{\frac{\beta}{\alpha} -\sum_{k=1}^{n}\frac{\vartheta_k}{\alpha_k}}\]
taking $(\tau,\bm{\xi})=(t,\bm{x})$, where we have denoted for simplicity
\[\bar{\Delta}^{\tau,\bm{\xi}}_j(s,\bm{y}) \, := \bar{\bm{F}}_j(s,\bm{y})-\bar{\bm{F}}_j(s,\bm{\theta}_{\tau,s}(\bm{\xi}))
-D_{\bm{x}_{j-1}}\bar{\bm{F}}_j(s,\bm{\theta}_{\tau,s}(\bm{\xi}))\bigl(\bm{y}-\bm{\theta}_{\tau,s}(\bm{\xi})\bigr)_{j-1}\]
for any $j$ in $\llbracket 2,n\rrbracket$.
The above control can be obtained mimicking the proof in the second Besov control (Lemma \ref{lemma:Second_Besov_COntrols}), exploiting this
time that
\[\vert \bar{\Delta}^{\tau,\bm{\xi}}_j(s,\bm{y}) \vert \le
C\Vert\bar{\bm{F}}\Vert_Hd^{1+\alpha(j-2)+\beta}_{j-1:n}\bigl(\bm{y},\bar{\bm{\theta}}_{\tau,s}(\bm{\xi})\bigr)\]
and the additional assumption (\textbf{P'}) in order to make the partial smoothing effect (Equation \eqref{eq:Partial_Smoothing_Effect})
work in this framework too.\newline
The main difference in the proof is related to the control of the component $J_2(v,\bm{y}_{\smallsetminus j},z)$ appearing in
Equation \eqref{eq:J_2_component_for_extension}. Namely,
\[\int_{\R^d}D_z \partial_vp_h(v,z-\bm{y}_j)\cdot\Bigl{\{}\bar{\Delta}^{\tau,\bm{\xi}}_j(s,\bm{y}_{\smallsetminus j},z)\otimes
\int_{0}^{1}D_{\bm{y}_j}D^\vartheta_{\bm{x}} \bar{p}^{\tau,\bm{\xi}}(t,s,\bm{x},\bm{y}_{\smallsetminus j},z+\lambda(\bm{y}_j-z)) \cdot
(\bm{y}_j-z)\Bigr{\}} \,d\lambda d\bm{y}_j\]
with our new notations. Indeed, the dependence of $\bar{\bm{F}}$ on $\bm{x}_{j-1}$ pushes us to add a new term in the difference $\vert
\bar{\bm{F}}_j(s,\bm{y}_{\smallsetminus j}, z)-\bar{\bm{F}}_j(s,\bm{\theta}_{\tau,s}(\bm{\xi})) \vert$ (now,
$\vert\bar{\Delta}^{\tau,\bm{\xi}}_j (s,\bm{y}_{\smallsetminus j},z)\vert$) before splitting it up. In particular,
\begin{multline*}
\vert \bar{\Delta}^{\tau,\bm{\xi}}_j (s,\bm{y}_{\smallsetminus j},z) \vert \\
= \, \bigl{\vert} \bar{\bm{F}}_j(s,\bm{y}_{\smallsetminus j},z)-\bar{\bm{F}}_j(s,\bm{\theta}_{\tau,s}(\bm{\xi}))-D_{\bm{x}_{j-1}}
\bar{\bm{F}}_j(s,\bar{\bm{\theta}}_{\tau,s}(\bm{\xi}))\bigl(\bm{y}-\bar{\bm{\theta}}_{\tau,s}(\bm{\xi})\bigr)_{j-1} \pm
\bar{\bm{F}}_j(s,\bm{y}_{1:j-1},\bigl(\bar{\bm{\theta}}_{\tau,s} (\bm{\xi})\bigr)_{j:n}) \bigr{\vert} \\
\le \, C\Vert \bar{\bm{F}}\Vert_H\bigl(\vert z-\bigl(\bar{\bm{\theta}}_{\tau,s}(\bm{\xi})\bigr)_j\vert^{\frac{1+\alpha(j-2)+\beta}{1
+\alpha(j-1)}}+\sum_{k=j+1}^{n}\vert \bigl(\bm{y}-\bar{\bm{\theta}}_{\tau,s}(\bm{\xi})\bigr)_k
\vert^{\frac{1+\alpha(j-2)+\beta}{1+\alpha(k-1)}}+\vert \bigl(\bm{y}-\bar{\bm{\theta}}_{\tau,s}(\bm{\xi})\bigr)_{j-1}\vert^{\frac{1+
\alpha (j-2)+\beta}{1+\alpha (j-2)}}\bigr) \\
\le \, C\Vert \bar{\bm{F}} \Vert_H\bigl(\vert \lambda(z-\bm{y}_j) \vert^{\frac{1+\alpha(j-2)+\beta}{1 +\alpha(j-1)}}+\vert
z+\lambda(\bm{y}_j-z)-\bm{\theta}_{\tau,s}(\bm{\xi})_j\vert^{\frac{1+ \alpha(j-2) +\beta}{1 +\alpha(j-1)}}+\sum_{k=j+1}^{n}\vert
\bm{y}-\bar{\bm{\theta}}_{\tau,s}(\bm{\xi})_k\vert^{\frac{1+\alpha(j-2)+\beta}{1+\alpha(k-1)}}\\
+\vert \bigl(\bm{y} - \bar{\bm{\theta}}_{\tau,s}(\bm{\xi})
\bigr)_{j-1}\vert^{\frac{1+\alpha (j-2)+\beta}{1+\alpha (j-2)}}\bigr)\, \le \, C\Vert \bar{\bm{F}} \Vert_H\bigl(\vert z-\bm{y}_j
\vert^{\frac{1+\alpha(j-2)+\beta}{1 +\alpha(j-1)}} +
d^{1+\alpha(j-2)+\beta}_{j+1:n}((\bm{y}_{\smallsetminus j},z+\lambda(\bm{y}_j-z)),\bar{\bm{\theta}}_{\tau,s}(\bm{\xi}))\bigr).
\end{multline*}
The remaining part of the proof exactly matches the original method in Lemma \ref{lemma:Second_Besov_COntrols}.

Even in this more general framework, it is thus possible to obtain the following:

\begin{theorem}[Well-posedness]
Under ${(\bar{\rm\textbf{A}})}$, there exists a unique mild solution $u$ of \eqref{Ext:Degenerate_Stable_PDE} such that
\[\Vert u \Vert_{L^\infty(C^{\alpha+\beta}_d)} \, \le \, C \bigl[\Vert f \Vert_{L^\infty(C^{\beta}_{b,d})} + \Vert g
\Vert_{C^{\alpha+\beta}_{b,d}}\bigr].\]
\end{theorem}

\subsection{Locally H\"older Drift}
This part is designed to give a brief explanation on how it is possible to deal with the general IPDE \eqref{Ext:Degenerate_Stable_PDE} when
the drift $\bar{\bm{F}}$ is only locally H\"older continuous in space. Namely, we assume with the notations in \eqref{Drift_assumptions} that
\begin{description}
  \item[(LR')] there exists a constant $K_0>0$ such that for any $i$ in $\llbracket1,n\rrbracket$
\[d(\bar{\bm{F}}(t,\bm{x}),\bar{\bm{F}}(t,\bm{x}')) \,\le \, K_0d^{\beta+\gamma_i}(\bm{x},\bm{x}'), \quad t\, \in [0,T], \,
\bm{x},\bm{x}'\in \R^{nd} \, \text{ s.t. } d(\bm{x},\bm{x}')<1.\]
\end{description}
In other words, it is required that $\bar{\bm{F}}_i$ is in $L^\infty(0,T;C^{\beta+\gamma_i}(B(x_0,1/2)))$, uniformly in $x_0\in \R^{nd}$.

Under assumption $(\bar{\textbf{A}})$ (with condition (\textbf{R'}) replaced by (\textbf{LR'})), it is possible to recover the Schauder-type
estimates (Theorem \ref{theorem:Schauder_Estimates}), following the approach developed successfully in \cite{Chaudru:Menozzi:Priola19} for
the non-degenerate, super-critical stable setting.  Roughly speaking, in order to handle the local assumption, as well as the potentially unboundedness of the drift $\bar{\bm{F}}$,
we need to introduce a "localized" version of the Duhamel formulation (cf. Equation \eqref{eq:Expansion_along_proxy}). The key point here is
to multiply a solution $u$ by a suitable bump function $\bar{\eta}^{\tau,\bm{\xi}}$ that "localizes" in space along the deterministic flow
$\bar{\theta}_{\tau,t}(\bm{\xi})$ that characterizes the proxy. Namely, we fix a smooth function $\rho$ that is equal to $1$ on $B(0,1/2)$
and vanishes outside $B(0,1)$) and then define for any $(\tau,\bm{\xi})$ in $[0,T]\times\R^{nd}$,
\[\bar{\eta}^{\tau,\bm{\xi}}(t,\bm{x}) \, := \, \rho(\bm{x}-\bar{\theta}_{\tau,t}(\bm{\xi})).\]

We mention however that in the setting of \cite{Chaudru:Menozzi:Priola19}, the "localization" with the cut-off function
$\bar{\eta}^{\tau,\bm{\xi}}$ is not simply motivated by the local H\"older continuity condition but it is also needed to give a proper
meaning to the Duhamel formulation for a solution (cf. Proposition \ref{prop:Expansion_along_proxy}) when $\alpha<1/2$, because of the low
integrability properties of the underlying stable density. Such a problem does not however appear here since condition (\textbf{P}) forces
us to consider only the case $\alpha>1/2$.

Given a mild solution $u$ of the IPDE \eqref{Ext:Degenerate_Stable_PDE} assuming $\bar{\bm{F}}$ to be only locally H\"older
continuous as in [\textbf{LR'}], it is possible to show, at least formally, that the function
$\bar{v}^{\tau,\bm{\xi}}:=u\bar{\eta}^{\tau,\bm{\xi}}$ solves the following equation
\begin{equation}
\label{Localized_Degenerate_Stable_PDE}
\begin{cases}
 \partial_t \bar{v}^{\tau,\bm{\xi}}(t,\bm{x}) + \langle \bar{\bm{F}}(t,\bm{x}),D_{\bm{x}}\bar{v}^{\tau,\bm{\xi}}(t,\bm{x})\rangle +
 L_\alpha \bar{v}^{\tau,\bm{\xi}}(t,\bm{x}) \, = \,
 -\bigl[\bar{\eta}^{\tau,\bm{\xi}}f+\bar{\mathcal{S}}^{\tau,\bm{\xi}}\bigr](t,\bm{x}) & \mbox{on } [0,T]\times \R^{nd}; \\
 \bar{v}^{\tau,\bm{\xi}}(T,\bm{x}) \, = \, \bar{\eta}^{\tau,\bm{\xi}}(T,\bm{x})g(\bm{x}) & \mbox{on }\R^{nd},
  \end{cases}
\end{equation}
where we have denoted above
\begin{multline*}
\bar{\mathcal{S}}^{\tau,\bm{\xi}}(t,\bm{x}) \, := \,\int_{\R^{d}}\bigl[u(t,\bm{x}+B\bm{y})-u(t,\bm{x}) \bigr] \bigl[ \bar{\eta}^{
\tau,\bm{\xi}}(t,t,\bm{x}+B\bm{y})- \bar{\eta}^{\tau,\bm{\xi}}(t,\bm{x})\bigr] \, \nu_\alpha(dy)\\
-u(t,\bm{x})\langle\bar{\bm{F}}(t,\bm{x})-\bar{\bm{F}}(t,\bar{\theta}_{\tau,t}(\bm{\xi})), D\rho(x-\bar{\theta}_{\tau,t}(\bm{\xi}))\rangle.
\end{multline*}
The IPDE \eqref{Localized_Degenerate_Stable_PDE} can be seen essentially as a "local" version of the original one
\eqref{Ext:Degenerate_Stable_PDE}, depending on the freezing parameter $(\tau,\bm{\xi})$. In particular, it is important to notice that the
difference
\[\bar{\bm{F}}(t,\bm{x})-\bar{\bm{F}}(t,\bar{\theta}_{\tau,t}(\bm{\xi}))\]
appearing in the "localizing" error $\bar{\mathcal{S}}^{\tau,\bm{\xi}}$ can be controlled exactly because it is multiplied by the
derivative of the bump function $\rho$ in the right point $\bm{x}-\bar{\theta}_{\tau,t}(\bm{\xi})$, allowing us to exploit the
\emph{local} H\"older regularity. On the other hand, the first integral term in the r.h.s. can be seen as a commutator which involves only
the non-degenerate variables and thus, that can be handled with interpolation techniques as in \cite{Chaudru:Menozzi:Priola19}.

Even with the additional difficulty in controlling the remainder term, the perturbative approach explained
in Section $3$ can be applied, leading to show Schauder-type estimates as in Theorem \ref{theorem:Schauder_Estimates} and the
well-posedness of the IPDE \eqref{Ext:Degenerate_Stable_PDE} when assuming $\bar{\bm{F}}$ to be only locally H\"older continuous.

Our procedure could be also used in order to establish Schauder-type estimates for the full Ornstein-Uhlenbeck operator as done, for
example, in \cite{Lunardi97} for the diffusive case. Indeed, a general operator of the form $\langle
\bar{A}\bm{x},D_{\bm{x}}\rangle+L_\alpha$ can be treated, decomposing the matrix as $\bar{A}=A+U$ where $A$ is, as before, the sub-diagonal
matrix that makes the Ornstein-Ulhenbeck operator invariant by the dilation operator associated with the distance $d$, while $U$ is an upper
triangular matrix that could be seen as an additional \emph{locally} H\"older term.

\subsection{Diffusion Coefficient}
We conclude the article showing briefly how an additional diffusion coefficient $\sigma\colon [0,T]\times\R^{nd}\to
\R^d\otimes \R^d$ can be handled in the IPDE \eqref{Ext:Degenerate_Stable_PDE} with an operator $L_alpha$ of the form:
\[L_\alpha\phi(t,\bm{x})\, := \, \text{p.v.}\int_{\R^d}\bigl[\phi(t,\bm{x}+B\sigma
(t,\bm{x})y) -\phi(t,\bm{x})\bigr]\nu_\alpha(dy).\]
In this framework, it is quite standard (cf. \cite{Hao:Wu:Zhang19} and \cite{Zhang:Zhao18}) to assume the L\'evy measure
$\nu_\alpha$ to be absolutely continuous with respect to the Lebesgue measure on $\R^d$ i.e.\ $\nu_\alpha(dy)=f(y)dy$, for some Lipschitz
function $f\colon \R^d \to \R$. In particular, since $\nu_\alpha$ is a symmetric, $\alpha$-stable L\'evy measure, it holds passing to polar
coordinates $y=\rho s$ where $(\rho,s)\in [0,\infty)\times \mathbb{S}^{d-1}$ that
\[f(y) \, = \, \frac{g(s)}{\rho^{d+\alpha}}\]
for an even, Lipschitz function $g$ on $\mathbb{S}^{d-1}$ (see also Equation \eqref{eq:decomposition_measure}). Moreover, $\sigma$ is
considered uniformly elliptic and in
$L^\infty(0,T;C^{\beta}(\R^n,\R))$. \newline
Introducing now the "frozen" operator $\bar{L}^{\tau,\bm{\xi}}_\alpha\phi(t,x)= \text{p.v.}\int_{\R^d}\bigl[\phi(t,\bm{x}+B\sigma
(t,\bar{\bm{\theta}}_{\tau,t}(\bm{\xi}))y) -\phi(t,\bm{x})\bigr]\nu_\alpha(dy)$, this would lead to consider for the IPDE an
additional term in the Duhamel formula (cf. Equation \eqref{Ext:Expansion_along_proxy}) that would write:
\begin{equation}\label{Duhamel_formula_for_additional}
u(t,\bm{x}) \, = \, \breve{P}^{\tau,\bm{\xi}}_{t,T}g(\bm{x}) +
\int_{t}^{T}\breve{P}^{\tau,\bm{\xi}}_{t,s}f(s,\bm{x})+\breve{P}^{\tau,\bm{\xi}}_{t,s}\bar{R}^{\tau,\bm{\xi}}(s,\bm{x})+
\breve{P}^{\tau,\bm{\xi}}_{t,s}\bigl[\bigl(L_\alpha-\bar{L}^{\tau,\bm{\xi}}_\alpha\bigr)u(s,\cdot)\bigr](\bm{x}) \, ds.
\end{equation}
Here, $\bigl(\breve{P}^{\tau,\bm{\xi}}_{t,s}\bigr)_{t\le s}$ denotes the two parameter semigroup associated with the proxy operator
\[\bar{L}^{\tau,\bm{\xi}}_\alpha+\langle\bar{A}^{\tau,\bm{\xi}}_t \bigl(\bm{x}-\bar{\bm{\theta}}_{\tau,t}(\bm{\xi})\bigr) +
\bar{\bm{F}}(t,\bm{\bar{\theta}}_{\tau,t}(\bm{\xi})), D_{\bm{x}} \rangle.\]
Let us focus on the last term in the integral of Equation \eqref{Duhamel_formula_for_additional}. Looking back at the proof of the A Priori
Estimates (Proposition \ref{prop:A_Priori_Estimates}), we notice in
particular that we aim to establish the following control:
\begin{equation}\label{zzz1}
\bigl{\vert} \bigl(L_\alpha- \bar{L}^{\tau,\bm{\xi}}_\alpha\bigr)u(t,\bm{x})\bigr{\vert} \, \le \, C\Vert \sigma
\Vert_{L^\infty(C^{\beta}_{b,d})}\Vert u
\Vert_{L^\infty(C^{\alpha+\beta}_{b,d})}d^\beta(\bm{x},\bar{\bm{\theta}}_{\tau,t}(\bm{\xi}))
\end{equation}
in order to apply the same reasoning above in this new framework. To this end, we write that
\begin{multline*}
\bigl(L_\alpha- \bar{L}^{\tau,\bm{\xi}}_\alpha\bigr)u(t,\bm{x}) \\
= \, \text{p.v.}\int_{\R^{d}}\bigl{\{}u(t,\bm{x}+B\sigma(t,\bm{x})y)-u(t,\bm{x})\bigr{\}} \, \nu_\alpha(dy) -
\int_{\R^{d}}\bigl{\{}u(t,\bm{x}+B\sigma(t,\bar{\bm{\theta}}_{\tau,t}(\bm{\xi}))y)-u(t,\bm{x})\bigr{\}} \, \nu_\alpha(dy) \\
= \, \text{p.v.}\int_{\R^{d}}\bigl{\{}u(t,\bm{x}+Bz)-u(t,\bm{x})\bigr{\}}\frac{f\bigl(\sigma^{-1}(t,\bm{x})z\bigr)}{\det \sigma(t,\bm{x})} \, dz -
\int_{\R^{d}}\bigl{\{}u(t,\bm{x}+Bz)-u(t,\bm{x})\bigr{\}}\frac{f\bigl(\sigma^{-1}(t,\bar{\bm{\theta}}_{\tau,t}(\bm{\xi}))z\bigr)}{\det
\sigma(t,\bar{\bm{\theta}}_{\tau,t}(\bm{\xi}))} \, dz \\
= \,\text{p.v.}\int_{0}^\infty \frac{1}{\rho^{1+\alpha}}\int_{\mathbb{S}^{d-1}}\bigl{\{}u(t,\bm{x}+B\rho
s)-u(t,\bm{x})\bigr{\}}\bar{D}^{\tau,\bm{\xi}}(t,\bm{x},s)
\, dsd\rho
\end{multline*}
where we have denoted, for notational convenience
\[\bar{D}^{\tau,\bm{\xi}}(t,\bm{x},s) \, := \, \Bigl{\{}\frac{g\bigl(\frac{\sigma^{-1}(t,\bm{x})s}{\vert
\sigma^{-1}(t,\bm{x})s\vert}\bigr)}{\vert
\sigma^{-1}(t,\bm{x})s\vert^{d+\alpha}\det \sigma(t,\bm{x})} -\frac{g\bigl(\frac{\sigma^{-1}(t,\bar{\bm{\theta}}_{\tau,t}
(\bm{\xi}))s}{\vert\sigma^{-1}(t,\bar{\bm{\theta}}_{\tau,t} (\bm{\xi}))s\vert}\bigr)}{\vert \sigma^{-1}(t,\bar{\bm{\theta}}_{\tau,t}
(\bm{\xi}))s\vert^{d+\alpha}\det\sigma(t,\bar{\bm{\theta}}_{\tau,t}(\bm{\xi}))}\Bigr{\}}.\]
Using now that $g$ is Lipschitz and the assumptions on $\sigma$, we can show that
\begin{equation}\label{eq:COntrol_on_D}
  \vert \bar{D}^{\tau,\bm{\xi}}(t,\bm{x},s) \vert \, \le \,  C\vert \sigma(t,\bm{x}) - \sigma(t,\bar{\bm{\theta}}_{\tau,t} (\bm{\xi}))\vert
\, \le \, C\Vert \sigma \Vert_{L^\infty(C^{\beta}_{b,d})} d^\beta(\bm{x},\bar{\bm{\theta}}_{\tau,t}(\bm{\xi})).
\end{equation}
Finally, Equation \eqref{zzz1} follows from the previous controls using Taylor expansions and the symmetry condition on $\nu_\alpha$.
Namely, considering the case $\alpha\ge 1$, which is the most delicate one for this part and precisely requires the symmetry of $g$, we write that
\begin{multline}\label{eq:Control_additional}
\Bigl{\vert}\bigl(L_\alpha- \bar{L}^{\tau,\bm{\xi}}_\alpha\bigr)u(t,\bm{x})\Bigr{\vert} \, = \, \Bigl{|}\text{p.v.}\int_{0}^\infty
\frac{1}{\rho^{1+\alpha}}\int_{\mathbb{S}^{d-1}}\bigl{\{}u(t,\bm{x}+B\rho s)-u(t,\bm{x})\bigr{\}}\bar{D}^{\tau,\bm{\xi}}(t,\bm{x},s)\,
dsd\rho\Bigr{|} \\
\le \, \Bigl{|}\text{p.v.}\int_{(0,1)}
\frac{1}{\rho^{1+\alpha}}\int_{\mathbb{S}^{d-1}}\bigl{\{}u(t,\bm{x}+B\rho s)-u(t,\bm{x})\bigr{\}}\bar{D}^{\tau,\bm{\xi}}(t,\bm{x},s)\,
dsd\rho \Bigl{|}\\
+ \int_{(1,\infty)} \frac{1}{\rho^{1+\alpha}}\int_{\mathbb{S}^{d-1}}\bigl{|}u(t,\bm{x}+B\rho
s)-u(t,\bm{x})\bigr{|}\,|\bar{D}^{\tau,\bm{\xi}}(t,\bm{x},s)|\, dsd\rho \, =: \bigl[\bar{I}^{\tau,\bm{\xi}}_s + \bar{I}^{\tau,\bm{\xi}}_l\bigr](t,\bm{x}).
\end{multline}
The \emph{large jump} contribution $\bar{I}^{\tau,\bm{\xi}}_l$ is easily handled from Equation \eqref{eq:COntrol_on_D}. We get that
\begin{equation}\label{eq:Control_additional1}
\bar{I}^{\tau,\bm{\xi}}_l(t,\bm{x}) \,\le \, 2C\Vert \sigma \Vert_{L^\infty(C^{\beta}_{b,d})}\Vert u \Vert_{L^\infty(L^\infty)}
d^\beta(\bm{x},\bar{\bm{\theta}}_{\tau,t}(\bm{\xi})) \, \le \, 2C\Vert \sigma \Vert_{L^\infty(C^{\beta}_{b,d})}\Vert u
\Vert_{L^\infty(C^{\alpha+\beta}_{b,d})} d^\beta(\bm{x},\bar{\bm{\theta}}_{\tau,t}(\bm{\xi})).
\end{equation}
On the other hand, from the symmetry assumption on $\nu_\alpha$, which transfers to $g$, we can control the \emph{small jump} contribution
$\bar{I}^{\tau,\bm{\xi}}_s$ through Taylor expansion and a centering
argument. Indeed,
\begin{multline}\label{eq:Control_additional2}
\bar{I}^{\tau,\bm{\xi}}_s(t,\bm{x}) \, = \, \Bigl{|}\text{p.v.}\int_{(0,1)}
\frac{1}{\rho^{1+\alpha}}\int_{\mathbb{S}^{d-1}}\int_{0}^{1}\bigl[D_{\bm{x}_1}u(t,\bm{x}+\lambda B\rho s)-D_{\bm{x}_1}u(t,\bm{x})\bigr]\rho
s\bar{D}^{\tau,\bm{\xi}}(t,\bm{x},s)\, d\lambda dsd\rho  \Bigl{|}\\
\le \, C\Vert \sigma \Vert_{L^\infty(C^{\beta}_{b,d})} d^\beta(\bm{x},\bar{\bm{\theta}}_{\tau,t}(\bm{\xi}))\int_{(0,1)}
\frac{1}{\rho^\alpha}\int_{\mathbb{S}^{d-1}}\int_{0}^{1}\bigl{\vert}D_{\bm{x}_1}u(t,\bm{x}+\lambda B\rho
s)-D_{\bm{x}_1}u(t,\bm{x})\bigr{\vert} \, d\lambda dsd\rho \\
\le \, C\Vert \sigma \Vert_{L^\infty(C^{\beta}_{b,d})} \Vert D_{\bm{x}_1}u
\Vert_{L^\infty(C^{\alpha+\beta-1}_{b,d})}d^\beta(\bm{x},\bar{\bm{\theta}}_{\tau,t}(\bm{\xi}))\int_{(0,1)}
\frac{1}{\rho^{\alpha}} \rho^{\alpha+\beta -1} \, d\rho \\
\le \, C\Vert \sigma \Vert_{L^\infty(C^{\beta}_{b,d})} \Vert u
\Vert_{L^\infty(C^{\alpha+\beta}_{b,d})}d^\beta(\bm{x},\bar{\bm{\theta}}_{\tau,t}(\bm{\xi})).
\end{multline}
Using Controls \eqref{eq:Control_additional1} and \eqref{eq:Control_additional2} in the decomposition \eqref{eq:Control_additional}, we
obtain the expected bound (Equation \eqref{zzz1}). We remark that the case $\alpha<1$ could be handled similarly for
the contribution $\bar{I}^{\tau,\bm{\xi}}_l$ and even more directly for $\bar{I}^{\tau,\bm{\xi}}_s$. Indeed, in that case, the centering argument is not needed since the Taylor
expansion already yields an integrable singularity.

\setcounter{equation}{0}
\appendix

\section{Appendix}

\subsection{Smoothing Effects for Ornstein-Ulhenbeck Operator}

We state and prove here some of the key properties of the Ornstein-Uhlenbeck operator. Namely, we will prove the representation
\eqref{eq:Representation_of_p_ou} and the associated $\alpha$-smoothing effect \eqref{Smoothing_effect_of_S}. We highlight however that
these results are only a slight modification to our purpose of those in  \cite{Huang:Menozzi:Priola19}.

The two lemma below presents a deep connection with stochastic analysis and their proofs relies on tools that are more familiar in the
probabilistic realm. For this reason, we are going to consider the stochastic counterpart of the Ornstein-Ulhenbeck operator $L^{ou}$.
Namely, for a given starting point $\bm{x}$ in $\R^{nd}$, we are interested in the following dynamics
\begin{equation}\label{eq:Stochastic_OU_equation}
\begin{cases}
dX_t \, = \, AX_tdt+BdZ_t, & \mbox{on } [0,T] \\
X_0 \, = \, \bm{x}
\end{cases}
\end{equation}
where $(Z_t)_{t\ge 0}$ is an $\alpha$-stable, $\R^{nd}$-dimensional process with L\'evy measure $\nu_\alpha$, defined on some complete
probability space $(\Omega,\mathcal{F},\mathcal{F},\mathbb{P})$.

\begin{lemma}[Representation]
Under (\textbf{A}), the semigroup $\bigl(P^{ou}_t\bigr)_{t>0}$ generated by the Ornstein-Ulehnbeck operator $L_{ou}$ (defined in \eqref{eq:def_of_OU_operator})  admits for any fixed $t>0$, a density $p^{ou}(t,\cdot)$ which writes for any $t>0$ and any $\bm{x},\bm{y}$
in $\R^{nd}$
\[p^{ou}(t,\bm{x},\bm{y}) \, = \, \frac{1}{\det \mathbb{M}_t}p_S(t,\mathbb{M}^{-1}_t\bigl(e^{At}\bm{x}-\bm{y})\bigr)\]
where $\mathbb{M}_t$ is the matrix defined in \eqref{eq:def_of_Mt} and $p_S$ is the smooth density of an $\R^{nd}$-valued, symmetric and
$\alpha$-stable process $S$ whose L\'evy measure $\mu_S$ satisfies the non-degeneracy assumption (\textbf{ND}) on $\R^{nd}$.
\end{lemma}
\begin{proof}
We start noticing that the above dynamics \eqref{eq:Stochastic_OU_equation} can be explicitly integrated and gives
\[X_t \, = \, e^{tA}\bm{x}+\int_{0}^{t}e^{(t-s)A}B \, dZ_s.\]
It is then readily derived from \cite{Priola:Zabczyk09} that, for any $t>0$, the random variable $X_t$ has a density $p_X(t,\bm{x},\cdot)$
with respect to the Lebesgue measure on $\R^{nd}$ and it is moreover well known (see for example \cite{book:Dynkin65}) that $p_X$ coincides
with the density $p^{ou}$ of the Ornstein-Ulhenbeck operator $L_{ou}$ .\newline
For this reason, we fix $t\ge 0$ and consider, for a given $N$ in $\N$, a uniform partition $\{t_i\}_{i\in\llbracket 0,N\rrbracket}$ of
$[0,t]$. Then, it holds for any $\bm{p}$ in $\R^{nd}$,
\[\mathbb{E}\Bigl[\text{exp}\Bigl(i\langle \bm{p}, \sum_{i=1}^{N}e^{(t-t_{i-1})A}B\bigl(Z_{t_i}-Z_{t_{i-1}}\bigr)\Bigr)\rangle\Bigr] \, = \,
\text{exp}\Bigl(-\frac{1}{N}\sum_{i=1}^{N}\int_{\mathbb{S}^{d-1}}\vert \langle B^\ast e^{(t-t_{i-1})A^\ast}\bm{p},s\rangle\vert^\alpha \,
\mu(ds)\Bigr)\]
where $\mu$ is the spherical measure associated with $\nu_\alpha$ (see Equation \eqref{def:Levy_Symbol_stable}). By dominated convergence
theorem, we let $m$ goes to infinity and show that
\[\mathbb{E}\Bigl[\text{exp}\Bigl(i\langle \bm{p},\int_{0}^{t}e^{(t-s)A}B \, dZ_s\Bigr)\Bigr] \, = \,
\text{exp}\Bigl(-\int_{0}^{t}\int_{\mathbb{S}^{d-1}}\vert \langle e^{uA^\ast}\bm{p},Bs\rangle \, \mu(ds)du\Bigr).\]
Thanks to the above equation, we can rewrite the characteristic function of $X_t$ as:
\begin{multline*}
\psi_{X_t}(\bm{p}) \, = \,\mathbb{E}\Bigl[\text{exp}\Bigl(i\langle \bm{p},e^{tA}\bm{x}+\int_{0}^{t}e^{(t-s)A}B \, dZ_s\Bigr)\Bigr] \,
= \, \text{exp}\Bigl(i\langle \bm{p},e^{tA}\bm{x}\rangle -\int_{0}^{t}\int_{\mathbb{S}^{d-1}}\vert \langle e^{uA^\ast} \bm{p}, Bs \rangle
\vert^\alpha \, \mu(ds)du \Bigr)\\
= \, \text{exp}\Bigl(i\langle \bm{p},e^{tA}\bm{x}\rangle -t\int_{0}^{1}\int_{\mathbb{S}^{d-1}}\vert \langle e^{vtA^\ast} \bm{p}, Bs \rangle
\vert^\alpha \, \mu(ds)dv \Bigr)
\end{multline*}
where in the last passage we used the change of variables $u=vt$. For the next step, we firstly notice that it holds
\[e^{tA} \, = \, \mathbb{M}_te^A\mathbb{M}^{-1}_t,\]
shown using the definition of matrix exponential and the trivial relation $\mathbb{M}_tA\mathbb{M}^{-1}_t = tA$. Exploiting the above
identity, we then find that
\begin{multline*}
  \psi_{X_t}(\bm{p}) \, = \, \text{exp}\Bigl(i\langle \bm{p},e^{tA}\bm{x}\rangle -t\int_{0}^{1}\int_{\mathbb{S}^{d-1}}\vert \langle
  \mathbb{M}_t\bm{p},e^{vA}\mathbb{M}_t^{-1} Bs\rangle\vert^\alpha\, \mu(ds)dv \Bigr)\\
  = \, \text{exp}\Bigl(i\langle \bm{p},e^{tA}\bm{x}\rangle -t\int_{0}^{1}\int_{\mathbb{S}^{d-1}}\vert \langle
  \mathbb{M}_t\bm{p},e^{vA}Bs\rangle\vert^\alpha\, \mu(ds)dv \Bigr)
\end{multline*}
where in the last passage we used the straightforward identity $\mathbb{M}_t^{1}B\bm{y}=B\bm{y}$.
We focus now only on the double integral
\[\int_{0}^{1}\int_{\mathbb{S}^{d-1}}\vert \langle \mathbb{M}_t\bm{p},e^{vA}Bs\rangle\vert^\alpha\, \mu(ds)dv.\]
If we consider the measure $m_\alpha(dv,ds) := \vert e^{vA}Bs \vert^\alpha \mu(ds)dv$ on $[0,1]\times \mathbb{S}^{d-1}$ and the
normalized lift function $l\colon [0,1]\times \mathbb{S}^{d-1}\to \mathbb{S}^{nd-1}$ given by
\[l(v,s) \, := \, \frac{e^{vA}Bs}{\vert e^{vA}Bs \vert},\]
it then follows that
\[\int_{0}^{1}\int_{\mathbb{S}^{d-1}}\vert \langle \mathbb{M}_t\bm{p},e^{vA}Bs\rangle\vert^\alpha\, \mu(ds)dv \, = \,
\int_{0}^{1}\int_{\mathbb{S}^{d-1}}\vert \langle \mathbb{M}_t\bm{p},\frac{e^{vA}Bs}{\vert e^{vA}Bs\vert}\rangle\vert^\alpha\,
m_\alpha(ds,dv)\, = \, \int_{\mathbb{S}^{nd-1}}\vert \langle \mathbb{M}_t\bm{p},\bm{\xi}\rangle\vert^\alpha\, \mu_S(d\bm{\xi})\]
where $\mu_S\, := \, \text{Sym}(l_\ast(m_\alpha))$ is the symmetrized version of the measure $m_\alpha$ push-forwarded through $l$. Noticing
that $\mu_S$ is the L\'evy measure of a symmetric $\alpha$-stable process $(S_t)_{t\ge 0}$ satisfying assumption (\textbf{ND}) on $\R^{nd}$,
we can finally write that
\[\psi_{X_t}(\bm{p}) \, = \, \text{exp}\Bigl(i\langle \bm{p},e^{tA}\bm{x}\rangle -t\Psi_S(\mathbb{M}_t\bm{p}) \Bigr)\]
where $\Psi_S$ is the L\'evy symbol associated with $S_t$ (cf. Equation \eqref{def:Levy_Symbol_stable}).\newline
From Lemma A.$1$ in \cite{Huang:Menozzi:Priola19}, we know that under assumption (\textbf{ND}), the above calculations implies that
\[\int_{0}^{1}\int_{\mathbb{S}^{d-1}}\bigl{\vert}(\mathbb{M}_t\bm{p}) \cdot (e^{Av}Bs)\bigr{\vert}^\alpha \, \mu_S(ds)dv \, \ge \, C\vert
\mathbb{M}_t\bm{p}\vert^\alpha\]
for some constant $C> 0$. It follows in particular that the function $\bm{p}\to \psi_{X_t}(\bm{p})$ is in $L^1(\R^{nd})$. Thus, by inverse
fourier transform and a change of variables we can prove that
\begin{multline*}
\mathcal{F}^{-1}\bigl[\psi_{X_t}\bigr](\bm{y}) \, = \, \frac{1}{(2\pi)^{nd}}\int_{\R^{nd}} e^{-i\langle\bm{p},
\bm{y}\rangle}\text{exp}\Bigl(i\langle \bm{p},e^{tA}\bm{x}\rangle -t\Psi_S(\mathbb{M}_t\bm{p}) \Bigr)\, d\bm{p} \\
=\, \frac{\det(\mathbb{M}^{-1}_t)}{(2\pi)^{nd}}\int_{\R^{nd}} \text{exp}\Bigl(-i\bigl{\langle}\mathbb{M}^{-1}_t\bm{p}, \bm{y}-e^{tA}\bm{x}
\bigr{\rangle}\Bigr)e^{-t\Psi(\bm{p})}\, d\bm{p} \\
= \, \frac{\det (\mathbb{M}^{-1}_t)}{(2\pi)^{nd}}\int_{\R^{nd}} \text{exp}\Bigl(-i\bigl{\langle}\bm{p},
\mathbb{M}^{-1}_t\bigl(\bm{y}-e^{tA}\bm{x}\bigr)\bigr{\rangle}\Bigr) e^{-t\Psi(\bm{p})}\, d\bm{p} \, = \,
\frac{1}{\det(\mathbb{M}_t)}p_S(t,\mathbb{M}^{-1}(y-e^{At}\bm{x}))
\end{multline*}
and we have concluded since $p_S$ is symmetric.
\end{proof}

We can now point out the smoothing effect (Equation \eqref{Smoothing_effect_of_S}) associated with the Ornstein-Uhlenbeck density $p^{ou}$.

\begin{lemma}[Smoothing Effect]
\label{Appendix:Smoothing_effect}
Under $(\textbf{A})$, there exists a family $\{q(t,\cdot)\colon t \in [0,T]\}$ of densities on $\R^{nd}$ such that
\begin{itemize}
  \item  for any $l$ in $\llbracket 0,3 \rrbracket$, there exists a constant $C:=C(l,nd)$ such that $\vert D^l_{y}p_S(t,\bm{y})\vert \, \le \,
      Cq(t,\bm{y})t^{-l/\alpha}$ for any $t$ in $[0,T]$ and any $\bm{y}$ in $\R^{nd}$;
  \item (stable scaling property) $q(t,\bm{y})=t^{-nd/\alpha}q(1,t^{-1/\alpha}\bm{y})$ for any $t$ in $[0,T]$ and any $\bm{y}$ in
      $\R^{nd}$;
  \item (stable smoothing effect) for any $\gamma$ in $[0,\alpha)$, there exists a constant $c:=c(\gamma,nd)$ such that
  \begin{equation}\label{equation:integration_prop_of_q}
    \int_{\R^{nd}} q(t,\bm{y}) \, \vert \bm{y} \vert^\gamma \, d\bm{y} \, \le \, ct^{\gamma/\alpha} \,\, \text{for any }t>0.
  \end{equation}
\end{itemize}
\end{lemma}
\begin{proof}
Fixed a time $t>0$, we start applying the Ito-L\'evy decomposition to $S$ at the associated characteristic stable time
scale, i.e. we choose to truncate at threshold $t^{1/\alpha}$, so that we can write $S_t = M_t+N_t$ for some $M_t,N_t$ independent random
variables corresponding to the small jumps part and the large jumps part, respectively. Namely, we denote for any $s>0$
\[N_s \, := \, \int_{0}^{s}\int_{\vert \bm{x} \vert >t^{1/\alpha}}\bm{x}P(du,d\bm{x}) \,\, \text{ and } \,\, M_s: \,= \, S_s-N_s
\]
where $P$ is the Poisson random measure associated with the process $S$. We can thus rewrite the density $p_S$ in the following way
\[p_S(t,\bm{x}) \, = \, \int_{\R^{nd}}p_M(t,\bm{x}-\bm{y})P_{N_t}(d\bm{y})\]
where $p_M(t,\cdot)$ corresponds to the density of $M_t$ and $P_{N_t}$ is the law of $N_t$.\newline
It is important now to notice that it is precisely our choice of the cutting threshold $t^{1/\alpha}$ that gives $M$ and $N$ the
$\alpha$-similarity property (for any fixed $t$)
\[N_t \, \overset{law}{=} \, t^{1/\alpha}N_1 \, \, \text{ and } M_t \, \overset{law}{=} \, t^{1/\alpha}M_1\]
we will need below. Indeed, to show the assertion for $N$, we can start from the L\'evy-Khintchine formula for the characteristic function of
$N$:
\[\mathbb{E}\bigl[e^{i \langle p,N_t\rangle}\bigr] \, = \, \exp\Bigl[t\int_{\mathbb{S}^{nd-1}}\int_{t^{1/\alpha}}^{\infty}\bigl(\cos(\langle
p,r\bm{\xi} \rangle)-1\bigr)\frac{dr}{r^{1+\alpha}}\overline{\mu}_S(d\bm{\xi})\Bigr]\]
for any $p$ in $\R^{nd}$. We then use the change of variable $rt^{-1/\alpha}=s$ to get that
\[\mathbb{E}\bigl[e^{i \langle p,N_t\rangle}\bigr] \, = \, \mathbb{E}\bigl[e^{i \langle p,t^{1/\alpha}N_1\rangle}].\]
This implies in particular our assertion on $N$. In a similar way, it is possible to get the analogous assertion on $M$.\newline
From lemma A.$2$ in \cite{Huang:Menozzi:Priola19}  with $m =3$, we know that there exist a family $\{p_{\overline{M}}(t,\cdot)\}_{t>0}$ of
densities on $\R^{nd}$ and a constant $C:=C(m,\alpha)$ such that
\[\vert D^l_{\bm{y}}p_M(t,\bm{y}) \vert \, \le \, Cp_{\overline{M}}(t,\bm{y})t^{-l/\alpha}\]
for any $t>0$, any $\bm{x}$ in $\R^{nd}$ and any $l\in\{0,1,2\}$.\newline
Moreover, denoting $\overline{M}_t$ the random variable with density $p_{\overline{M}}(t,\cdot)$ and independent from $N_t$, we can easily
check from $p_{\overline{M}}(t,\bm{y})=t^{-nd/\alpha}p_{\overline{M}}(1,t^{-1/\alpha}\bm{x})$
that $\overline{M}$ is $\alpha$-selfsimilar
\[\overline{M}_t \, \overset{law}{=} \, t^{1/\alpha}\overline{M}_1.\]
We can finally define the family $\{q(t,\cdot)\}_{t>0}$ of densities as
\[q(t,\bm{x}) \,:=\, \int_{\R^{nd}}p_{\overline{M}}(t,\bm{x}-\bm{y})P_{N_t}(d\bm{y})\]
corresponding to the density of the random variable
\[\overline{S}_t \, ; = \, \overline{M}_t +N_t\]
for any fixed $t>0$.
Using Fourier transform and the already proven $\alpha$-selfsimilarity of $\overline{M}$ and $N$, we can show now that
\[\overline{S}_t \, \overset{law}{=} \, t^{1/\alpha}\overline{S}_1\]
or equivalently, that
\[q(t,\bm{y})=t^{-nd/\alpha}q(1,t^{-1/\alpha}\bm{y})\]
for any $t$ in $[0,T]$ and any $\bm{y}$ in $\R^{nd}$. Moreover,
\[\mathbb{E}[\vert \overline{S}_t\vert^\gamma] \, = \, \mathbb{E}[\vert \overline{M}_t+N_t\vert^\gamma] \, = \,
Ct^{\gamma/\alpha}\bigl(\mathbb{E}[\vert \overline{M}_1\vert^\gamma]+\mathbb{E}[\vert N_t\vert^\gamma]\bigr) \, \le \,  Ct^{\gamma/\alpha}.\]
This shows in particular that equation \eqref{equation:integration_prop_of_q} holds.
\end{proof}

We conclude this sub-section showing Control \eqref{eq:translation_inv_for_density} appearing in the proof of Proposition
\ref{prop:A_Priori_Estimates} for the diagonal regime. First of all, we will need the following lemma:

\begin{lemma}
Let $t$ in $[0,T]$, $\bm{x},\bm{b}$ in $\R^{nd}$ such that $\vert \bm{b}\vert \le ct^{1/\alpha}$ for some constant $c>0$.
Under (\textbf{A}), there exists a constant $C:=C(c)$ such that
\[\vert D^l_{\bm{x}}p_S\bigl(t,\bm{x}+\bm{b})\vert \, \le \, \tilde{C} \vert D^l_{\bm{x}}p_S\bigl(t,\bm{x})\vert\]
\end{lemma}
\begin{proof}
Looking back at the proof of the previous lemma \ref{Appendix:Smoothing_effect}, we know that
\[D^l_{\bm{x}}p_S(t,\bm{x}+\bm{b}) \, = \, \int_{\R^{nd}}D^l_{\bm{x}}p_M(t,\bm{x}+\bm{b}-\bm{y})P_{N_t}(d\bm{y})\]
where $p_M(t,\cdot)$ is the density of $M_t$ and $P_{N_t}$ is the law of $N_t$, corresponding to the small and big jumps in the
Ito-L\'evy decomposition.\newline
From lemma $A.2$ in \cite{Huang:Menozzi:Priola19} we know moreover that
\[\vert D^l_{\bm{x}}p_M(t,\bm{x}+\bm{b}-\bm{y})\vert \, \le \, \frac{C}{t^{\frac{l}{\alpha}}}p_{\overline{M}}(t,\bm{x}+\bm{b}-\bm{y}) \,\,
\text{ where } \,\, p_{\overline{M}}(t,\bm{z}) \, = \, \frac{C}{t^{\frac{nd}{\alpha}}}\frac{1}{\Bigl(1+\frac{\vert \bm{z}\vert}{t^
\frac{1}{\alpha}}\Bigr)^3}.\]
It is then enough to show that
\begin{multline*}
p_{\overline{M}}(t,\bm{z}+\bm{b}) \, = \, \frac{C}{t^{\frac{nd}{\alpha}}}\frac{1}{\Bigl(1+\frac{\vert \bm{z}+\bm{b}\vert}{t^
\frac{1}{\alpha}}\Bigr)^3} \, \le \, \frac{\tilde{C}}{t^{\frac{nd}{\alpha}}}\frac{1}{\Bigl(1+c+\frac{\vert \bm{z}+\bm{b}\vert}{t^
\frac{1}{\alpha}}\Bigr)^3} \\
\le \, \frac{C}{t^{\frac{nd}{\alpha}}}\frac{1}{\Bigl(1+c\frac{\vert \bm{z}\vert}{t^\frac{1}{\alpha}}-\frac{\vert \bm{b}\vert}{t^\frac{1}{
\alpha}}\Bigr)^3} \, \le \, \frac{C}{t^{\frac{nd}{\alpha}}}\frac{1}{\Bigl(1+\frac{\vert \bm{z}\vert}{t^\frac{1}{\alpha}}\Bigr)^3} \, \le \,
Cp_{\overline{M}}(t,\bm{z}).
\end{multline*}
to conclude the proof.
\end{proof}

\begin{proof}[Proof of Equation \eqref{eq:translation_inv_for_density}]
We start looking back to the proof of Lemma \ref{lemma:Smoothing_effect_frozen} to find that
\[\bigl{\vert} D^\vartheta_{\bm{x}}\tilde{p}^{\tau,\bm{\xi}'} (t,s,\bm{x}+ \lambda(\bm{x}'- \bm{x}),\bm{y})\bigr{\vert} \,  = \,
C(s-t)^{-\sum_{k=1}^{n}\frac{\vartheta_k}{\alpha_k}}\frac{1}{\det (\mathbb{M}_{s-t})}
\bigl{\vert}D^{\vert \vartheta\vert}_{\bm{z}}p_S\bigl(s-t,\cdot)(\mathbb{M}^{-1}_{s-t} (\tilde{\bm{m}}^{\tau,\bm{\xi}}_{t,s}
(\bm{x})-\bm{y})\bigr)\bigr{\vert}\]
Moreover, we notice that
\[\mathbb{M}^{-1}_{s-t} \bigl(\tilde{\bm{m}}^{\tau,\bm{\xi}}_{t,s} (\bm{x}+\lambda(\bm{x}-\bm{x}'))-\bm{y}\bigr) \, = \,\mathbb{M}^{-1
}_{s-t} \bigl(\tilde{\bm{m}}^{\tau,\bm{\xi}}_{t,s} (\bm{x})-\bm{y}\bigr)+\lambda\mathbb{M}^{-1}_{s-t} e^{A(s-t)}(\bm{x}-\bm{x}').\]
Then, Control \eqref{eq:translation_inv_for_density} follows immediately from the previous lemma once we have shown that
\[\bigl{\vert}\lambda\mathbb{M}^{-1}_{s-t} e^{A(s-t)}(\bm{x}-\bm{x}')\bigr{\vert} \, \le \, C(s-t)^{1/\alpha}\]
for some constant $C:=C(A)$. Indeed, fixed $i$ in $\llbracket 1,n\rrbracket$, we can exploit the structure of $A$ and $\mathbb{M}_{s-t}$
(cf. Equation \eqref{Proof:Scaling_Lemma2} in Scaling Lemma \ref{lemma:Scaling_Lemma}) to write that
\[\bigl[\mathbb{M}^{-1}_{s-t} e^{A(s-t)}(\bm{x}-\bm{x}') \bigr]_i \, = \, \sum_{j=1}^{n}\sum_{k=1}^{n}\bigl[\mathbb{M}^{-1}_{s-t}\bigr]_{i,k}
\bigl[e^{A(s-t)}\bigr]_{k,j}(\bm{x}-\bm{x}')_j \, = \, \sum_{j=i}^{n}(s-t)^{-(i-1)}C_j(s-t)^{i-j}(\bm{x}-\bm{x}')_j.\]
Since moreover we assumed to be in a local diagonal regime, i.e.\ $d^\alpha(\bm{x},\bm{x}')\le (s-t)^{1/\alpha}$, we can conclude that
\[\bigl{\vert}\bigl[\mathbb{M}^{-1}_{s-t} e^{A(s-t)}(\bm{x}-\bm{x}') \bigr]_i\bigr{\vert} \, \le \, C\sum_{j=i}^{n}(s-t)^{-(j-1)}
\vert(\bm{x}-\bm{x}')_j\vert \, \le \, C\sum_{j=i}^{n}(s-t)^{-(j-1)}(s-t)^{\frac{1+\alpha(j-1)}{\alpha}} \, = \,C(s-t)^{1/\alpha}.\qedhere\]
\end{proof}

\subsection{Technical Tools}

In this section, we present the proof of some technical results already used in the article, for the sake of completeness. \newline
We recall moreover that the results below can be proven also for the flow $\bar{\bm{\theta}}_{\tau,s}(\bm{\xi})$ driven by a more general
perturbation $\bm{F}$ under assumption $(\bar{\textbf{A}})$ (cf. Section $7.1$), exploiting that
$\bar{\bm{F}}_i$ is Lipschitz continuous in the $\bm{x}_{i-1}$ variable for any $i$ in $\llbracket 2, n \rrbracket$.

We begin proving Lemma \ref{lemma:Controls_on_Flow1} about the sensitivity of the H\"older flows, appearing in the proof of the
a priori estimates \eqref{eq:A_Priori_Estimates} of Proposition \ref{prop:A_Priori_Estimates}. For this reason, we will assume from this
point further to be under assumption (\textbf{A'}).

\paragraph{Proof of Lemma \ref{lemma:Controls_on_Flow1}.}
We start noticing that our result follows immediately using Young inequality, once we have shown that it holds
\begin{equation}\label{equation:Controls_on_Flow}
\bigl{\vert} (\bm{\theta}_{t,s}(\bm{x})-\bm{\theta}_{t,s}(\bm{x}'))_i\bigr{\vert} \, \le \, C\Bigl[(s-t)^{\frac{1+\alpha(i-1)}{\alpha}}+d^{1+\alpha(i-1)}(
\bm{x},\bm{x}')\Bigr] \quad \text{for any $i$ in }\llbracket 1,n \rrbracket.
\end{equation}
Our proof will rely essentially in iterative applications of the Gr\"onwall lemma. We notice however that under (\textbf{A}), the
perturbation $\bm{F}_i$ is only H\"older continuous with respect to its $i$-th variable. To overcome this problem, we are going to mollify
(but only with respect to the variable of interest) the function $\bm{F}$ in the following way: fixed a mollifier $\rho$ on $\R^d$, i.e. a
compactly supported, non-negative, smooth function such that $\Vert \rho \Vert_{L^1}=1$ and a family $\delta_i$ of positive constants to be
chosen later, the mollified version of the perturbation is given by $\bm{F}^\delta=(\bm{F}_1,\bm{F}^{\delta_2}_2,\dots,\bm{F}^{\delta_n}_n)$
where
\[\bm{F}^{\delta_i}_i(t,\bm{z}_{i:n}) \, := \, \bm{F}_i \ast_i \rho_{\delta_i}(t,\bm{z}_{i:n}) \, = \, \int_{\R^d}
\bm{F}_i(t,\bm{z}_i-\omega,\bm{z}_{i+1},\dots,\bm{z}_n)\frac{1}{\delta_i^d} \rho(\frac{\omega}{\delta_i}) \, d\omega.\]
We remark in particular that we do not need to mollify the first component $\bm{F}_1$ since it is regular enough, say $\beta$-H\"older
continuous in the first $d$-dimensional variable $\bm{x}_1$, by assumption (\textbf{R}).\newline
Then, standard results on mollifier theory and our current assumptions on $\bm{F}$ show us that the following controls hold
\begin{align}
\label{Proof:Controls_on_flows_mollifier} \vert \bm{F}_i(u,\bm{z}) - \bm{F}^\delta_i(u,\bm{z})\vert \, &\le \,\Vert \bm{F}_i \Vert_{L^\infty(C^{\gamma_+\beta}_d)}\delta_i^{\frac{\gamma_i+\beta}{1+\alpha(i-1)}}, \\
\label{Proof:Controls_on_flows_mollifier1} \vert \bm{F}^\delta_i(u,\bm{z}) - \bm{F}^\delta_i(u,\bm{z}')\vert \, &\le \,C\Vert \bm{F}_i \Vert_{L^\infty(C^{\gamma_+\beta}_d)} \bigl[\delta_i^{
\frac{\gamma_i+\beta}{1+\alpha(i-1)}-1}\vert(\bm{z}-\bm{z}')_i\vert +\sum_{j=i+1}^{n}\vert(\bm{z}-\bm{z}')_{j}\vert^{\frac{\gamma_i+\beta}{1+
\alpha(j-1)}}\bigr].
\end{align}
We choose now $\delta_i$ for any $i$ in $\llbracket 2,n\rrbracket$ in order to have any contribution associated with the mollification
appearing in \eqref{Proof:Controls_on_flows_mollifier} at a good current scale time. Namely, we would like $\delta_i$ to satisfy
\[\bigl{\vert} \bigl((s-t)^{\frac{1}{\alpha}}\mathbb{M}_{s-t}\bigr)^{-1}\bigl(\bm{F}(u,\bm{z})-\bm{F}^\delta(u,\bm{z})\bigr) \bigr{\vert} \,
\le \, C(s-t)^{-1}\]
for any $u$ in $[t,s]$ and any $\bm{z}$ in $\R^{nd}$.
Using the mollifier controls \eqref{Proof:Controls_on_flows_mollifier}, it is enough to ask for
\[\sum_{i=2}^n(s-t)^{-\frac{1}{\alpha_i}}\delta_i^{\frac{\gamma_i+\beta}{1+\alpha(i-1)}} \, \le \, C(s-t)^{-1}.\]
Recalling that $\gamma_i:=1+\alpha(i-2)$ by assumption (\textbf{R}), this is true if we fix for example,
\begin{equation}\label{Proof:Controls_on_Flows_Choice_delta}
\delta_i \, = \, (s-t)^{\frac{\gamma_i}{\alpha}\frac{1+\alpha(i-1)}{\gamma_i+\beta}} \quad \text{for $i$ in $\llbracket 2,n\rrbracket$.}
\end{equation}
After this introductive part, we start controlling the last component of the flow. By construction of $\bm{\theta}_{t,s}$, we can write
that
\begin{multline}
\label{Proof:Control_on_Flow_eq1}
\bigl{\vert} (\bm{\theta}_{t,s}(\bm{x})-\bm{\theta}_{t,s}(\bm{x}'))_n\bigr{\vert} \, = \, \Bigl{\vert} (\bm{x}-\bm{x}')_n +
\int_{t}^{s}\bigl{\{}\bigl[A(\bm{\theta}_{t,v}(\bm{x})-\bm{\theta}_{t,v}(\bm{x}'))\bigr]_n + \bm{F}_n(v,\bm{\theta}_{t,v}(\bm{x}))
-\bm{F}_n(v,\bm{\theta}_{t,v}(\bm{x}'))\bigr{\}} \,  dv \Bigr{\vert} \\
\le \, \vert (\bm{x}-\bm{x}')_n\vert + \int_{t}^{s}\bigl{\{}A_{n,n-1}\vert(\bm{\theta}_{t,v}(\bm{x})- \bm{\theta}_{t,v}(\bm{x}'))_{n-1}
\vert + \bigl{\vert}\bm{F}_n(v,\bm{\theta}_{t,v}(\bm{x})) - \bm{F}_n(v,\bm{\theta}_{t,v}(\bm{x}'))\bigr{\vert}\bigr{\}} \,  dv
\end{multline}
where in the last passage we have exploited the sub-diagonal structure of $A$ (cf. Equation \eqref{eq:def_matrix_A}). If we focus only on
the last term involving the difference of the drifts, It holds now that
\begin{multline*}
\bigl{\vert} \bm{F}_n(v,\bm{\theta}_{t,v}(\bm{x})) - \bm{F}_n(v,\bm{\theta}_{t,v}(\bm{x}'))\bigr{\vert} \, \le \, \bigl{\vert} \bm{F}_n(v,\bm{\theta}_{t,v}(\bm{x})) \pm \bm{F}^\delta_n(v,\bm{\theta}_{t,v}(\bm{x})) - \bm{F}_n(v,\bm{\theta}_{t,v}(\bm{x}')) \pm
\bm{F}^\delta_n(v,\bm{\theta}_{t,v}(\bm{x}'))\bigr{\vert} \\
\le \, \bigl{\vert} \bm{F}_n(v,\bm{\theta}_{t,v}(\bm{x})) - \bm{F}^\delta_n(v,\bm{\theta}_{t,v}(\bm{x}))\bigr{\vert} +
\bigl{\vert}\bm{F}_n(v,\bm{\theta}_{t,v}(\bm{x}')) -
\bm{F}^\delta_n(v,\bm{\theta}_{t,v}(\bm{x}') )\bigr{\vert} + \bigl{\vert} \bm{F}^\delta_n( v,\bm{\theta}_{t,v}(\bm{x}))-\bm{F}^\delta_n(
v,\bm{\theta}_{t,v}(\bm{x}'))\bigr{\vert}.
\end{multline*}
Using the controls \eqref{Proof:Controls_on_flows_mollifier}, \eqref{Proof:Controls_on_flows_mollifier1} on the mollified drifts, we then
write from \eqref{Proof:Control_on_Flow_eq1} and the previous equation that
\begin{multline*}
\bigl{\vert} (\bm{\theta}_{t,s}(\bm{x})-\bm{\theta}_{t,s}(\bm{x}'))_n\bigr{\vert} \,\le \\
\vert (\bm{x}-\bm{x}')_n\vert + 2(s-t)\delta_{n}^{\frac{\gamma_n+\beta}{1+\alpha(n-1)}}+C\int_{t}^{s}
\bigl{\{}\bigl{\vert}(\bm{\theta}_{t,v}(\bm{x})- \bm{\theta}_{t,v}(\bm{x}'))_{n-1} \bigr{\vert} +
\delta_n^{\frac{\gamma_n+\beta}{1+\alpha(n-1)}-1} \bigl{\vert}(\bm{\theta}_{t,v}(\bm{x})-
\bm{\theta}_{t,v}(\bm{x}'))_n \bigr{\vert}\bigr{\}} \,  dv.
\end{multline*}
We now apply the Gr\"onwall lemma to show that
\[\bigl{\vert} (\bm{\theta}_{t,s}(\bm{x})-\bm{\theta}_{t,s}(\bm{x}'))_n\bigr{\vert} \, \le \, C\Bigl[\vert(\bm{x}-\bm{x}')_n\vert+(s-t)
\delta_{n}^{\frac{\gamma_n+ \beta}{1+\alpha(n-1)}} + \int_{t}^{s} \bigl{\vert}(\bm{\theta}_{t,v}(\bm{x})- \bm{\theta}_{t,v}(\bm{x}'))_{n-1}
\bigr{\vert} \,  dv\Bigr].\]
From our previous choice for $\delta_n$ (cf. Equation \eqref{Proof:Controls_on_Flows_Choice_delta}), we know that
$(s-t)^{-\frac{1}{\alpha_n}}\delta_n^{\frac{\gamma_n+\beta}{1+\alpha(n-1)}} \le C(s-t)^{-1}$ and thus, we can rewrite the last inequality as
\begin{equation}\label{Proof:Controls_on_flows1}
\bigl{\vert}(\bm{\theta}_{t,s}(\bm{x}) - \bm{\theta}_{t,s}(\bm{x}'))_n\bigr{\vert} \, \le \,C\Bigl[\bigl{\vert}(\bm{x} -
\bm{x}')_n\bigr{\vert} +(s-t)^{\frac{1+\alpha(n-1)}{\alpha}} + \int_{t}^{s}\bigl{\vert}(\bm{\theta}_{t,v}(\bm{x}) -
\bm{\theta}_{t,v}(\bm{x}'))_{n-1}\bigr{\vert} \, dv\Bigr].
\end{equation}
We would like now to obtain a similar control on the $(n-1)$-th term. As already done at the beginning of the proof, we can write that
\begin{multline*}
\bigl{\vert}(\bm{\theta}_{t,s}(\bm{x}) - \bm{\theta}_{t,s}(\bm{x}'))_{n-1}\bigr{\vert} \, \le \, \bigl{\vert}(\bm{x} - \bm{x}')_{n-1}\bigr{\vert} + C\delta_{n-1}^{\frac{\gamma_{n-1} + \beta}{1+\alpha(n-2)}}(s-t)+\int_{t}^{s}
\bigl{\vert}(\bm{\theta}_{t,v}(\bm{x}) - \bm{\theta}_{t,sv}(\bm{x}'))_{n-2}\bigr{\vert} \\
+ \delta_{n-1}^{\frac{\gamma_{n-1} +\beta}{1+\alpha(n-2)}-1} \bigl{\vert}(\bm{\theta}_{t,v}(\bm{x})
-\bm{\theta}_{t,v}(\bm{x}'))_{n-1}\bigr{\vert} +
\bigl{\vert}(\bm{\theta}_{t,v}(\bm{x})-\bm{\theta}_{t,v}(\bm{x}'))_{n}\bigr{\vert}^{\frac{\gamma_{n-1} +\beta}{1+\alpha(n-1)}} \, dv
\end{multline*}
We then apply the Gr\"onwall lemma to find that
\begin{multline*}
\bigl{\vert}(\bm{\theta}_{t,s}(\bm{x}) - \bm{\theta}_{t,s}(\bm{x}'))_{n-1}\bigr{\vert} \, \le \, C\Bigl[\bigl{\vert}(\bm{x} -
\bm{x}')_{n-1}\bigr{\vert} + \delta_{n-1}^{\frac{\gamma_{n-1} + \beta}{1+\alpha(n-2)}}(s-t)\\
+\int_{t}^{s}\bigl{\{}\bigl{\vert}(\bm{\theta}_{t,v}(\bm{x}) - \bm{\theta}_{t,sv}(\bm{x}'))_{n-2}\bigr{\vert} + \bigl{\vert}
(\bm{\theta}_{t,v}(\bm{x}) - \bm{\theta}_{t,v}(\bm{x}'))_{n} \bigr{\vert}^{\frac{\gamma_{n-1} +\beta}{1+\alpha(n-1)}}\bigr{\}} \, dv\Bigr].
\end{multline*}
Remembering our previous choice of $\delta_{n-1}$, it holds now that
\begin{multline}
\label{Proof:Controls_on_flows1/2}
\bigl{\vert}(\bm{\theta}_{t,s}(\bm{x}) - \bm{\theta}_{t,s}(\bm{x}'))_{n-1}\bigr{\vert} \, \le \, C
\Bigl[\vert(\bm{x}-\bm{x}')_{n-1}\vert+(s-t)^{\frac{1+\alpha(n-2)}{\alpha}}+\int_{t}^{s}\bigl{\vert}(\bm{\theta}_{t,v}(\bm{x}) -
\bm{\theta}_{t,v}(\bm{x}'))_{n-2}\bigr{\vert} \\
+ \bigl{\vert}(\bm{\theta}_{t,v}(\bm{x}) - \bm{\theta}_{t,v}(\bm{x}'))_n\bigr{\vert}^{\frac{\gamma_{n-1}+\beta}{1+\alpha(n-1)}} \, dv\Bigr].
\end{multline}
We then use equation \eqref{Proof:Controls_on_flows1} and the Jensen inequality to write
\begin{multline}
\label{Proof:Controls_on_flows1/3}
\bigl{\vert}(\bm{\theta}_{t,s}(\bm{x}) - \bm{\theta}_{t,s}(\bm{x}'))_{n-1}\bigr{\vert} \\
\le \, C \Bigl[\vert(\bm{x}-\bm{x}')_{n-1}\vert + (s-t)^{\frac{1+\alpha(n-2)}{\alpha}} +\int_{t}^{s}\bigl{\{}\bigl{\vert}(\bm{\theta}_{t,v}(\bm{x}) -
\bm{\theta}_{t,v}(\bm{x}'))_{n-2} \bigr{\vert} + \bigl{\vert}(\bm{x} - \bm{x}')_n\bigr{\vert}^{\frac{\gamma_{n-1}+\beta}{1+\alpha(n-1)}} \\
+(v-t)^{\frac{\gamma_{n-1}+\beta}{\alpha}}+ \Bigl(\int_{t}^{v}\bigl{\vert}(\bm{\theta}_{t,\omega}(\bm{x}) - \bm{\theta}_{t,\omega}(\bm{x}'))_{n-1}
\bigr{\vert}\,d\omega\Bigr)^{\frac{\gamma_{n-1}+\beta}{1+\alpha(n-1)}} \bigr{\}}\, dv\Bigr].
\end{multline}
The idea now is to use Gr\"onwall lemma again. To do so, we firstly move the exponent from the last integral term involving the $(n-1)$-th
term using the Young inequality:
\[\Bigl(\int_{t}^{v}\bigl{\vert}(\bm{\theta}_{t,\omega}(\bm{x}) - \bm{\theta}_{t,\omega}(\bm{x}'))_{n-1}\bigr{\vert} \, d\omega
\Bigr)^{\frac{\gamma_{n-1}+\beta}{1+\alpha(n-1)}} \, \le \, B^{-\frac{1+\alpha(n-1)}{\gamma_{n-1}+\beta}} \int_{t}^{v} \bigl{\vert}
(\bm{\theta}_{t,\omega}(\bm{x}) - \bm{\theta}_{t,\omega}(\bm{x}'))_{n-1}\bigr{\vert} \, d\omega + B^{\frac{1+\alpha(n-1)}{2\alpha - \beta}}\]
for a quantity $B$ to be fixed later.\newline
Since we need homogeneity with respect to time in equation \eqref{Proof:Controls_on_flows1/2}, we choose $B$ such that
\[B^{\frac{1+\alpha(n-1)}{2\alpha - \beta}} \, = \, (v-t)^{\frac{\gamma_{n-1}+\beta}{\alpha}} \, \Leftrightarrow \, B=(v
-t)^{\frac{\gamma_{n-1}+\beta}{\alpha}\frac{2\alpha - \beta}{1+\alpha(n-1)}}.\]
Plugging it into the general expression \eqref{Proof:Controls_on_flows1/3}, we find that
\begin{multline*}
\bigl{\vert}(\bm{\theta}_{t,s}(\bm{x}) - \bm{\theta}_{t,s}(\bm{x}'))_{n-1}\bigr{\vert} \, \le \, C \Bigl[\vert(\bm{x}-\bm{x}')_{n-1}\vert +
(s-t)^{\frac{1+\alpha(n-2)}{\alpha}} \\
+\int_{t}^{s}\Bigl{\{}\bigl{\vert}(\bm{\theta}_{t,v}(\bm{x}) - \bm{\theta}_{t,v}(\bm{x}'))_{n-2} \bigr{\vert} + \bigl{\vert}(\bm{x}
-\bm{x}')_n\bigr{\vert}^{\frac{\gamma_{n-1}+\beta}{1+\alpha(n-1)}} +(v - t)^{\frac{\gamma_{n-1}+\beta}{\alpha}} \\
+(v-t)^{\frac{\beta}{\alpha}-2}\int_{t}^{v}\bigl{\vert}(\bm{\theta}_{t,\omega}(\bm{x}) - \bm{\theta}_{t,\omega}(\bm{x}'))_{n-1}
\bigr{\vert} \, d\omega \Bigr{\}}\, dv\Bigr] \\
\le \, C \Bigl[\vert(\bm{x}-\bm{x}')_{n-1}\vert +(s-t)^{\frac{1+\alpha(n-1)}{\alpha}} + (s-t)\bigl{\vert}(\bm{x}-\bm{x}')_n \bigr{\vert}^{ \frac{\gamma_{n-1}+
\beta}{1+\alpha(n-1)}} + (s - t)^{\frac{\gamma_{n-1}+\beta+\alpha}{\alpha}}\\
+ \int_{t}^{s}\Bigl{\{}\bigl{\vert} (\bm{\theta}_{t,v}(\bm{x}) - \bm{\theta}_{t,v}(\bm{x}'))_{n-2} \bigr{\vert} + (v-t)^{\frac{\beta}{\alpha}-1} \sup_{\omega\in[t,v]}
\bigl{\vert}(\bm{\theta}_{t,\omega}(\bm{x}) -\bm{\theta}_{t,\omega}(\bm{x}'))_{n-1} \bigr{\vert}\Bigr{\}} \, dv\Bigr].
\end{multline*}
Since the previous inequality is also true for any $\overline{s}$ in $[t,s]$, it follows that
\begin{multline*}
\sup_{\overline{s}\in[0,s]}\bigl{\vert}(\bm{\theta}_{t,\overline{s}}(\bm{x}) - \bm{\theta}_{t,\overline{s}}(\bm{x}'))_{n-1}\bigr{\vert} \,
\le \, C \Bigl[\vert(\bm{x}-\bm{x}')_{n-1}\vert + (s-t)^{\frac{1+\alpha(n-2)}{\alpha}} + (s-t)\bigl{\vert}(\bm{x} - \bm{x}')_n \bigr{\vert}^{
\frac{\gamma_{n-1}+\beta}{1+\alpha(n-1)}}\\
+(s - t)^{\frac{\gamma_{n-1}+\beta+\alpha}{\alpha}} +\int_{t}^{s} \Bigl{\{}\bigl{\vert} (\bm{\theta}_{t,v}(\bm{x}) -
\bm{\theta}_{t,v}(\bm{x}'))_{n-2}\bigr{\vert} + (v-t)^{\frac{\beta}{\alpha}-1} \sup_{\omega \in[t,v]}\bigl{\vert}(\bm{\theta}_{t,\omega}
(\bm{x}) - \bm{\theta}_{t,\omega}(\bm{x}'))_{n-1} \bigr{\vert}\Bigr{\}} \, dv\Bigr].
\end{multline*}
We can finally apply the Gr\"onwall lemma to show that for any $s$ in $[t,T]$, there exists a constant $C$ such that
\begin{multline*}
\bigl{\vert}(\bm{\theta}_{t,s}(\bm{x}) - \bm{\theta}_{t,s}(\bm{x}'))_{n-1}\bigr{\vert} \\
\le \, C \Bigl[\vert(\bm{x}-\bm{x}')_{n-1}\vert + (s-t)^{\frac{1+\alpha(n-2)}{\alpha}} + (s-t)\vert(\bm{x} - \bm{x}')_n\vert^{\frac{\gamma_{n-1}+\beta}{
1+\alpha(n-1)}} + \int_{t}^{s} \bigl{\vert} (\bm{\theta}_{t,v}(\bm{x}) - \bm{\theta}_{t,v}(\bm{x}'))_{n-2}
\bigr{\vert} \, dv\Bigr].
\end{multline*}
Moreover, thanks to the Young inequality we know that
\[(s-t)\bigl{\vert}(\bm{x} - \bm{x}')_n \bigr{\vert}^{\frac{\gamma_{n-1}+\beta}{1+\alpha(n-1)}} \, \le \,
C\bigl{\{}(s-t)^{\frac{1+\alpha(n-2)}{\alpha}} +\vert(\bm{x} - \bm{x}')_n \vert^{\frac{\gamma_{n-1}+\beta}{1+\alpha(n-1)}
\frac{1+\alpha(n-2)}{1+\alpha(n-3)}}\bigr{\}}\]
and remembering that $d(\bm{x},\bm{x}')\le 1$ by hypothesis,
\[\vert(\bm{x} - \bm{x}')_n \vert^{\frac{\gamma_{n-1}+\beta}{1+\alpha(n-1)}\frac{1+\alpha(n-2)}{1+\alpha(n-3)}} \, \le \, \vert(\bm{x} -
\bm{x}')_n \vert^{\frac{\gamma_{n-1}+\beta}{\gamma_{n-1}}\frac{1+\alpha(n-2)}{1+\alpha(n-1)}} \, \le \, \vert(\bm{x} - \bm{x}')_n
\vert^{\frac{1+\alpha(n-2)}{1+\alpha(n-1)}}.\]
We then use it to write for any $v$ in $[t,T]$,
\begin{multline*}
\bigl{\vert}(\bm{\theta}_{t,v}(\bm{x}) - \bm{\theta}_{t,v}(\bm{x}'))_{n-1}\bigr{\vert} \\
\le \, C \Bigl[\vert(\bm{x}-\bm{x}')_{n-1}\vert + (v-t)^{\frac{1+\alpha(n-2)}{\alpha}} + \vert(\bm{x} - \bm{x}')_n \vert^{\frac{
1+\alpha(n-2)}{1+\alpha(n-1)}} + \int_{t}^{v} \bigl{\vert} (\bm{\theta}_{t,\omega}(\bm{x}) - \bm{\theta}_{t,\omega}(\bm{x}'))_{n-2}
\bigr{\vert} \, d\omega\Bigr].
\end{multline*}
Going back to equation \eqref{Proof:Controls_on_flows1}, we plug in the last bound to find that
\begin{multline*}
\bigl{\vert}(\bm{\theta}_{t,s}(\bm{x}) - \bm{\theta}_{t,s}(\bm{x}'))_n\bigr{\vert} \, \le \, C\Bigl[\vert(\bm{x} - \bm{x}')_n \vert
+(s-t)^{\frac{1+\alpha(n-1)}{\alpha}} + (s-t)\vert(\bm{x} - \bm{x}')_{n-1}\vert \\
+ (s-t)\vert(\bm{x} - \bm{x}')_n \vert^{\frac{1+\alpha(n-2)}{1+\alpha(n-1)}} + \int_{t}^{s}\int_{t}^{v} \bigl{\vert} (\bm{\theta}_{t,
\omega}(\bm{x}) - \bm{\theta}_{t,\omega}(\bm{x}'))_{n-2}\bigr{\vert}\, d\omega dv \Bigr] \\
\le \, C\Bigl[\vert(\bm{x} - \bm{x}')_n\vert + (s-t)^{\frac{1+\alpha(n-1)}{\alpha}} + \vert(\bm{x} - \bm{x}')_{n-1} \vert^{\frac{1+
\alpha(n-1)}{1+\alpha(n-2)}} + \int_{t}^{s}\int_{t}^{v} \bigl{\vert}(\bm{\theta}_{t,\omega}(\bm{x}) - \bm{\theta}_{t,\omega}(\bm{x}'))_{n-2}
\bigr{\vert}\, d\omega dv \Bigr]
\end{multline*}
where in the last passage we used again the Young inequality to show that
\[(s-t)\vert(\bm{x} - \bm{x}')_{n-1}\vert \, \le \, C(s-t)^{\frac{1+\alpha(n-1)}{\alpha}}+ \vert(\bm{x} - \bm{x}')_{n-1}\vert^{\frac{
1+\alpha(n-1)}{1+\alpha( n-2)}}\]
and
\[(s-t)\vert(\bm{x} - \bm{x}')_n\vert^{\frac{1+\alpha(n-2)}{1+\alpha(n-1)}} \, \le \, C(s-t)^{\frac{1+\alpha(n-1)}{\alpha}}+ \vert(\bm{x} -
\bm{x}')_n \vert.\]
This approach may be naturally iterated up to the first term of the chain, so that
\begin{multline*}
\bigl{\vert}(\bm{\theta}_{t,s}(\bm{x}) - \bm{\theta}_{t,s}(\bm{x}'))_n\bigr{\vert} \\
\le \, C \Bigl[\sum_{j=2}^{n}\vert(\bm{x}-\bm{x}')_j\vert^{\frac{1+\alpha(n-1)}{1+\alpha(j-1)}} + (s-t)^{\frac{1+\alpha(n-1)}{\alpha}} +
\int_{t}^{v_n=s} dv_{n-1}\dots\int_{t}^{v=2}dv_1\bigl{\vert} (\bm{\theta}_{t,v_1}(\bm{x}) - \bm{\theta}_{t,v_1}(\bm{x}'))_1 \bigr{\vert}\Bigr].
\end{multline*}
In a similar manner, we can show for any $i$ in $\llbracket 2,n\rrbracket$,
\begin{multline}\label{Proof:Controls_on_flows2}
\bigl{\vert}(\bm{\theta}_{t,s}(\bm{x}) - \bm{\theta}_{t,s}(\bm{x}'))_i\bigr{\vert} \\
\le \, C \Bigl[\sum_{j=2}^{n}\vert(\bm{x}-\bm{x}')_j \vert^{ \frac{1+\alpha(i-1)}{1+\alpha(j-1)}} + (s-t)^{\frac{1+\alpha(i-1)}{\alpha}} +
\int_{t}^{v_i=s}dv_{i-1}\dots\int_{t}^{v=2}dv_1\bigl{\vert} (\bm{\theta}_{t, v_1}(\bm{x}) - \bm{\theta}_{t,v_1}(\bm{x}'))_1 \bigr{\vert}\Bigr].
\end{multline}
Since all the non-integral terms in \eqref{Proof:Controls_on_flows2} are compatible with the statement of the lemma, it remains to find the proper bound for the first component of the flow. As before, let us consider $\overline{s}$ in $[t,s]$. We can write
\[\vert (\bm{\theta}_{t,\overline{s}}(\bm{x})-\bm{\theta}_{t,\overline{s}}(\bm{x}'))_1\vert \, \le \, \vert (\bm{x}-\bm{x}')_1\vert
+C\sum_{j=1}^{n}\int_{t}^{\overline{s}}
\vert (\bm{\theta}_{t,v}(\bm{x})-\bm{\theta}_{t,v}(\bm{x}'))_j\vert^{\frac{\beta}{1+\alpha(j-1)}} \, dv\]
or, passing to the supremum on both sides,
\begin{multline*}
\sup_{\overline{s}\in[t,s]}\vert( \bm{\theta}_{t,\overline{s}}(\bm{x})-\bm{\theta}_{t,\overline{s}}(\bm{x}'))_1\vert \\
\le \, \vert (\bm{x}-\bm{x}')_1\vert +C\Bigl{\{}(s-t)\bigl(\sup_{v\in[t,s]}\vert (\bm{\theta}_{t,v}(\bm{x})-\bm{\theta}_{t,v}(\bm{x}'))_1\vert\bigr)^\beta +
\sum_{j=2}^{n}\int_{t}^{s} \vert (\bm{\theta}_{t,v}(\bm{x})-\bm{\theta}_{t,v}(\bm{x}'))_j\vert^{\frac{\beta}{1+\alpha(j-1)}} \, dv\Bigr{\}}.
\end{multline*}
Using equation \eqref{Proof:Controls_on_flows2}, it holds now that
\begin{multline}\label{Proof:Controls_on_flows3}
\sup_{\overline{s}\in[t,s]}(\vert \bm{\theta}_{t,\overline{s}}(\bm{x})-\bm{\theta}_{t,\overline{s}}(\bm{x}'))_1\vert \, \le \, \vert
(\bm{x}-\bm{x}')_1\vert +C\Bigl{\{}(s-t)\bigl(\sup_{v\in[t,s]}\vert
(\bm{\theta}_{t,v}(\bm{x})-\bm{\theta}_{t,v}(\bm{x}'))_1\vert\bigr)^\beta\\
+ \sum_{j=2}^{n}\Bigl[(s-t)\bigl((s-t)^{\frac{1+\alpha(j-1)}{\alpha}} +\sum_{k=2}^{n} \vert(\bm{x}-\bm{x}')_k \vert^{
\frac{1+\alpha(j-1)}{1+\alpha(k-1)}} + (s-t)^{j-1}\sup_{v\in[t,s]}\vert (\bm{\theta}_{t,v}(\bm{x})-\bm{\theta}_{t,v}(\bm{x}'))_1
\vert\bigr)^{\frac{\beta}{1+\alpha(j-1)}}\Bigr]\Bigr).
\end{multline}
We then apply the Jensen inequality to show that
\begin{multline}\label{Proof:Controls_on_flows4}
\sup_{\overline{s}\in[t,s]}(\vert \bm{\theta}_{t,\overline{s}}(\bm{x})-\bm{\theta}_{t,\overline{s}}(\bm{x}'))_1\vert \, \le \, \vert
(\bm{x}-\bm{x}')_1\vert + C\Bigl{\{}(s-t)\bigl[\sup_{v\in[t,s]}\vert(\bm{\theta}_{t,v}(\bm{x})-\bm{\theta}_{t,v}(\bm{x}'))_1\vert
\bigr]^\beta+\sum_{j=2}^{n}C(s-t)\bigl[(s-t)^{\frac{\beta}{\alpha}}\\
+ \sum_{k=2}^{n} \vert(\bm{x}-\bm{x}')_k \vert^{ \frac{\beta}{1+\alpha(k-1)}} + (s-t)^{\frac{(j-1)\beta}{1+\alpha(j-1)}}\sup_{v\in[t,s]}\vert
(\bm{\theta}_{t,v}(\bm{x})-\bm{\theta}_{t,v}(\bm{x}'))_1 \vert^{\frac{\beta}{1+\alpha(j-1)}} \bigr]\Bigr{\}} \\
\le \,C\Bigl{\{}\vert (\bm{x}-\bm{x}')_1\vert + (s-t)^{\frac{\alpha+\beta}{\alpha}} + (s-t)\sum_{k=2}^{n}
\vert(\bm{x}-\bm{x}')_k\vert^{\frac{\beta}{1+\alpha(k-1)}}\\
+\sum_{j=1}^{n}(s-t)^{1+\frac{(j-1)\beta}{1+\alpha(j-1)}}\sup_{v\in[t,s]}\vert (\bm{\theta}_{t,v}(\bm{x})-\bm{\theta}_{t,v}(\bm{x}'))_1
\vert^{\frac{\beta}{1+\alpha(j-1)}}\Bigr{\}}.
\end{multline}
From Young inequality, we can deduce now that
\[(s-t)\vert(\bm{x}-\bm{x}')_k\vert^{\frac{\beta}{1+\alpha(k-1)}} \, \le \,C\bigl((s-t)^{\frac{1}{1-\beta}}+ \vert(\bm{x}-\bm{x}')_k
\vert^{\frac{1}{1+\alpha(k-1)}}\bigr)\]
and
\[(s-t)^{1+\frac{(j-1)\beta}{1+\alpha(j-1)}}\sup_{v\in[t,s]}\vert
(\bm{\theta}_{t,v}(\bm{x})-\bm{\theta}_{t,v}(\bm{x}'))_1\vert^{\frac{\beta}{1+\alpha(j-1)}}\, \le \, C\Bigl{\{}(s-t)^{\frac{1+(\alpha+\beta)(j-1)}{1+\alpha(j-1)-\beta}}+\sup_{v\in[t,s]}\vert (\bm{\theta}_{t,v}(\bm{x})-\bm{\theta}_{t,v}(\bm{x}'))_1\vert
\Bigr{\}}\]
Plugging these inequalities in the main one \eqref{Proof:Controls_on_flows4}, we find that
\begin{multline*}
\sup_{\overline{s}\in[t,s]}(\vert \bm{\theta}_{t,\overline{s}}(\bm{x})-\bm{\theta}_{t,\overline{s}}(\bm{x}'))_1\vert \, \le \,
C\Bigl{\{}\vert (\bm{x}-\bm{x}')_1\vert +
(s-t)^{\frac{\alpha+\beta}{\alpha}}+\sum_{k=2}^{n} \vert(\bm{x}-\bm{x}')_k\vert^{\frac{1}{1+\alpha(k-1)}}\\
+\sum_{j=1}^{n}(s-t)^{\frac{1+(\alpha+\beta)(j-1)}{1+\alpha(j-1)-\beta}}+\sup_{v\in[t,s]}\vert
(\bm{\theta}_{t,v}(\bm{x})-\bm{\theta}_{t,v}(\bm{x}'))_1\vert\Bigr{\}} \\
\le \,C\Bigl{\{}(s-t)^{\frac{\alpha+\beta}{\alpha}} +(s-t)^{\frac{1}{1-\beta}}+d(\bm{x},\bm{x}')+(s-t)^{\frac{1+(\alpha+\beta)(j-1)}{1+
\alpha(j-1)-\beta}}+\sup_{v\in[t,s]}\vert (\bm{\theta}_{t,v}(\bm{x})-\bm{\theta}_{t,v}(\bm{x}'))_1\vert\Bigr{\}}
\end{multline*}
Remembering that $s-t\le T-t\le 1$, it finally holds that
\[\vert \bm{\theta}_{t,s}(\bm{x})-\bm{\theta}_{t,s}(\bm{x}'))_1\vert \, \le \, C\bigl((s-t)^{1/\alpha} + d(\bm{x},\bm{x}')\bigr)\]
since by assumption (\textbf{P}),
\[\frac{\alpha+\beta}{\alpha}\, > \, \frac{1}{1-\beta} \, > \, \frac{1}{\alpha}\]
and
\[\frac{1+(\alpha+\beta)(j-1)}{1+\alpha(j-1)-\beta} \, = \, 1+\frac{\beta j}{1+\alpha j-(\alpha+\beta)} \, > \, 1+\frac{\beta j}{\alpha j}
\, > \, 1+\Bigl(\frac{1-\alpha}{\alpha}\Bigr)\, = \, \frac{1}{\alpha}.\]
Plugging this control in equation \eqref{Proof:Controls_on_flows2}, we then conclude since
\begin{multline*}
\bigl{\vert}(\bm{\theta}_{t,s}(\bm{x}) - \bm{\theta}_{t,s}(\bm{x}'))_i\bigr{\vert} \\
\le \,C\Bigl(d^{1+\alpha(i-1)}(\bm{x},\bm{x}')+ (s-t)^{\frac{1+\alpha(i-1)}{\alpha}} +(s-t)^{i-1}\sup_{\overline{s}\in[t,s]}(\vert
\bm{\theta}_{t,\overline{s}}(\bm{x})-\bm{\theta}_{t,\overline{s}}(\bm{x}'))_1\vert\Bigr) \\
\le \, C\Bigl(d^{1+\alpha(i-1)}(\bm{x},\bm{x}')+ (s-t)^{\frac{1+\alpha(i-1)}{\alpha}}+(s-t)^{i-1}\bigl((s-t)^{1/\alpha} +
d(\bm{x},\bm{x}')\bigr)\Bigr) \\
\le \, C\Bigl((s-t)^{\frac{1+\alpha(i-1)}{\alpha}} + d^{1+\alpha(i-1)}(\bm{x},\bm{x}')\Bigr),
\end{multline*}
using again the Young inequality in the last passage. The proof is complete.

We can now prove the two results (Lemmas \ref{lemma:Controls_on_means1} and Lemma \ref{lemma:Controls_on_means}) concerning the sensitivity
of the frozen shift $\tilde{\bm{m}}^{\tau,\bm{\xi}}_{t,s}$.

\paragraph{Proof of Lemma \ref{lemma:Controls_on_means1}.}
From the integral representation of $\tilde{\bm{m}}^{t,\bm{x}}_{t,s}(\bm{y})$ (cf. Equation \eqref{eq:def_tilde_m}), we can write that
\begin{multline*}
\bigl{\vert}\bigl(\tilde{\bm{m}}^{t,\bm{x}}_{t,s}(\bm{y})-\tilde{\bm{m}}^{t,\bm{x}'}_{t,s} (\bm{y}')\bigr)_1\bigr{\vert} \, \le \, \int_{t}^{s}\bigl{\vert} \bm{F}_1(v,\bm{\theta}_{t,v}(\bm{x})) - \bm{F}_1(v,\bm{\theta}_{t,v}(\bm{x}')) \bigr{\vert} \, dv \\
 \le \, C\Vert \bm{F}\Vert_H\int_{t}^{s}d^\beta\bigl(\bm{\theta}_{t,v}(\bm{x}),\bm{\theta}_{t,v}(\bm{x}')\bigr) \, dv
\end{multline*}
where in the second passage we used that $\bm{F}_1$ is in $C^{\beta}_{b,d}(\R^{nd})$.
Thanks to the Control on the flows (Lemma \ref{lemma:Controls_on_Flow1}), it then holds that
\[\bigl{\vert} \bigl(\tilde{\bm{m}}^{t,\bm{x}}_{t,s}(\bm{y})-\tilde{\bm{m}}^{t,\bm{x}'}_{t,s} (\bm{y}')\bigr)_1\bigr{\vert} \, \le \,C\Vert
\bm{F}\Vert_H(s-t)\bigl[d^\beta(\bm{x},\bm{x}')+(s-t)^{\frac{\beta}{\alpha}}\bigr]\]
and we have concluded.

\paragraph{Proof of Lemma \ref{lemma:Controls_on_means}.}
We know from Lemma \ref{lemma:identification_theta_m} that $\tilde{\bm{m}}^{t,\bm{x}'}_{t,t_0} (\bm{x}') =
\bm{\theta}_{t,t_0}(\bm{x}')$. Fixed $i$ in $\llbracket 1, n\rrbracket$, we can then write that
\[\bigl( \tilde{\bm{m}}^{t,\bm{x}}_{t,t_0}(\bm{x}') -\tilde{\bm{m}}^{t,\bm{x}'}_{t,t_0} (\bm{x}')\bigr)_i \, = \,
\bigl(\tilde{\bm{m}}^{t,\bm{x}}_{t,t_0}(\bm{x}') -\bm{\theta}_{t,t_0}(\bm{x}')\bigr)_i \,
= \, \bigl(\tilde{\bm{m}}^{t,\bm{x}}_{t,t_0}(\bm{x}')-\bm{\theta}_{t,t_0}(\bm{x})\bigr)_i +  \bigl(\bm{\theta}_{t,t_0}(\bm{x})
-\bm{\theta}_{t,t_0}(\bm{x}')\bigr)_i.\]
We start focusing on the first term of the above expression. From the integral representation of $\tilde{\bm{m}}^{t,\bm{x}}_{t,t_0}(
\bm{x}')$ and $\bm{\theta}_{t,t_0}(\bm{x})$, it holds that
\begin{equation}\label{Proof:Controls_on_means_1}
\tilde{\bm{m}}^{t,\bm{x}}_{t,t_0}(\bm{x}') -\bm{\theta}_{t,t_0}(\bm{x}) \, = \, \bm{x}'-\bm{x}+\int_{t}^{t_0}A\bigl[
\tilde{\bm{m}}^{t,\bm{x}}_{t,v}(\bm{x}') -\bm{\theta}_{t,v}(\bm{x})\bigr] \, dv.
\end{equation}
Remembering from \eqref{eq:def_matrix_A} that $A$ is sub-diagonal, it follows that
\begin{equation}\label{Proof:Controls_on_means_2}
\bigl(\tilde{\bm{m}}^{t,\bm{x}}_{t,t_0}(\bm{x}') -\bm{\theta}_{t,t_0}(\bm{x})\bigr)_i \, = \, (\bm{x}'-\bm{x})_i+A_{i,i-1}\int_{t}^{t_0}
\bigl( \tilde{\bm{m}}^{t,\bm{x}}_{t,v}(\bm{x}')-\bm{\theta}_{t,v}(\bm{x})\bigr)_{i-1} \, dv
\end{equation}
for any $i$ in $\llbracket 2,n\rrbracket$ and
\[\bigl(\tilde{\bm{m}}^{t,\bm{x}}_{t,t_0}(\bm{x}') -\bm{\theta}_{t,t_0}(\bm{x})\bigr)_1 \, = \, (\bm{x}'-\bm{x})_1.\]
Iterating the process, we  can find that
\[\bigl{\vert}\bigl(\tilde{\bm{m}}^{t,\bm{x}}_{t,t_0}(\bm{x}') -\bm{\theta}_{t,t_0}(\bm{x})\bigr)_i \bigr{\vert} \, \le \,
C\sum_{k=1}^{i}\bigl{\vert}(\bm{x}'-\bm{x})_k\bigr{\vert}(t_0-t)^{i-k}.\]
On the other side, the integral representation of $\bm{\theta}_{\tau,s}(\bm{\xi})$ (Equation \eqref{Flow}) allows us to write that
\begin{multline}\label{Proof:Controls_on_means_3}
\bigl(\bm{\theta}_{t,t_0}(\bm{x}) -\bm{\theta}_{t,t_0}(\bm{x}')\bigr)_i \\
= \, (\bm{x}-\bm{x}')_i+A_{i,i-1}\int_{t}^{t_0} \bigl{\{}\bigl(
\bm{\theta}_{t,t_0}(\bm{x}) -
\bm{\theta}_{t,t_0}(\bm{x}')\bigr)_{i-1} + \bm{F}_i(v,\bm{\theta}_{t,v}(\bm{x}))-\bm{F}_i(v,\bm{\theta}_{t,v}(\bm{x}'))\bigr{\}} \, dv
\end{multline}
for any $i$ in $\llbracket 2,n\rrbracket$ and
\begin{equation}\label{Proof:Controls_on_means_4}
\bigl(\bm{\theta}_{t,t_0}(\bm{x}) -\bm{\theta}_{t,t_0}(\bm{x}')\bigr)_1 \, = \, (\bm{x}-\bm{x}')_1 +\int_{t}^{t_0} \bigl{\{}
\bm{F}_1(v,\bm{\theta}_{t,v}(\bm{x}))-\bm{F}_1(v,\bm{\theta}_{t,v}(\bm{x}')) \bigr{\}} \,
dv.
\end{equation}
Fixed $i$ in $\llbracket 2,n\rrbracket$, it then follows from \eqref{Proof:Controls_on_means_1} and \eqref{Proof:Controls_on_means_3} that
\begin{multline*}
\bigl{\vert}\bigl( \tilde{\bm{m}}^{t,\bm{x}}_{t,t_0}(\bm{x}') -\tilde{\bm{m}}^{t,\bm{x}'}_{t,t_0} (\bm{x}')\bigr)_i\bigr{\vert} \, \le \,
C\Vert \bm{F} \Vert_H \Bigl(\sum_{k=1}^{i-1}\vert(\bm{x}'-\bm{x})_k\vert(t_0-t)^{i-k}\\
+\int_{t}^{t_0}\Bigl{\{}\bigl{\vert}\bigl(\bm{\theta}_{t,v}(\bm{x}) -\bm{\theta}_{t,v}(\bm{x}')\bigr)_{i-1}\bigr{\vert}+
\sum_{j=i}^{n}\bigl{\vert} \bigl(\bm{\theta}_{t,v}(\bm{x}) - \bm{\theta}_{t,v}(\bm{x}') \bigr)_{j}
\bigr{\vert}^{\frac{\gamma_i+\beta}{1+\alpha(j-1)}}\Bigr{\}} \, dv \Bigr).
\end{multline*}
Also, from \eqref{Proof:Controls_on_means_2} and \eqref{Proof:Controls_on_means_4}, it holds that
\[\bigl{\vert}\bigl( \tilde{\bm{m}}^{t,\bm{x}}_{t,t_0}(\bm{x}') -\tilde{\bm{m}}^{t,\bm{x}'}_{t,t_0} (\bm{x}')\bigr)_1\bigr{\vert} \, \le
\,C\Vert \bm{F} \Vert_H\int_{t}^{t_0}\sum_{j=1}^{n}\bigl{\vert}\bigl(\bm{\theta}_{t,v}(\bm{x}) - \bm{\theta}_{t,v}(\bm{x}') \bigr)_{j}
\bigr{\vert}^{\frac{\beta}{1+\alpha(j-1)}} \, dv.\]
Using now Lemma \ref{lemma:Controls_on_Flow1}, we can show that
\begin{multline*}
\bigl{\vert}\bigl( \tilde{\bm{m}}^{t,\bm{x}}_{t,t_0}(\bm{x}') -\tilde{\bm{m}}^{t,\bm{x}'}_{t,t_0} (\bm{x}')\bigr)_i\bigr{\vert}\, \le \,
C\Vert \bm{F} \Vert_H\Bigl(\sum_{k=1}^{i-1}\vert(\bm{x}'-\bm{x})_k\vert(t_0-t)^{i-k} +(t_0-t)^{\frac{1+\alpha(i-2)}{\alpha}+1}\\
+(t_0-t)d^{1+\alpha(i-2)}(\bm{x},\bm{x}')+(t_0-t)^{\frac{1+\alpha(i-2)+\beta}{\alpha}+1}+
(t_0-t)d^{1+\alpha(i-2)+\beta}(\bm{x},\bm{x}')\Bigr)
\end{multline*}
for any $i$ in $\llbracket 2,n\rrbracket$ and
\[\bigl{\vert}\bigl( \tilde{\bm{m}}^{t,\bm{x}}_{t,t_0}(\bm{x}') -\tilde{\bm{m}}^{t,\bm{x}'}_{t,t_0} (\bm{x}')\bigr)_1\bigr{\vert} \, \le
\,C\Vert \bm{F} \Vert_H(t_0-t)^{\frac{\beta+\alpha}{\alpha}}+(t_0-t)d^{\beta}(\bm{x},\bm{x}').\]
Since $t_0-t=c_0d^{\alpha}(\bm{x},\bm{x}')$ by Equation \eqref{eq:def_t0}, we can conclude that
\begin{multline*}
\bigl{\vert}\bigl( \tilde{\bm{m}}^{t,\bm{x}}_{t,t_0}(\bm{x}') -\tilde{\bm{m}}^{t,\bm{x}'}_{t,t_0} (\bm{x}')\bigr)_i\bigr{\vert}\, \le \,
C\Vert \bm{F} \Vert_H \Bigl(\sum_{k=1}^{i-1}d^{1+\alpha(k-1)}(\bm{x}',\bm{x})c^{i-k}_0d^{\alpha(i-k)}(\bm{x},\bm{x}')+
c^{\frac{1+\alpha(i-1)}{\alpha}}_0d^{1+\alpha(i-1)}(\bm{x},\bm{x}')\\ +
c_0d^{1+\alpha(i-1)}(\bm{x},\bm{x}')+c_0^{\frac{1+\alpha(i-2)+\beta}{\alpha}+1}d^{1+\alpha(i-1)+\beta}(\bm{x},\bm{x}')+c_0
d^{1+\alpha(i-1)+\beta}(\bm{x},\bm{x}')\Bigr) \\
\le \, C\Vert \bm{F} \Vert_H \Bigl[\bigl(c_0+c^{\frac{1+\alpha(i-1)}{\alpha}}_0\bigr)d^{1+\alpha(i-1)}(\bm{x},\bm{x}')
+\bigl(c_0+ c_0^{\frac{1+\alpha(i-1)+\beta}{\alpha}}\bigr)d^{1+\alpha(i-1)+\beta}(\bm{x},\bm{x}')\Bigr] \\
\le \, Cc_0\Vert \bm{F} \Vert_Hd^{1+\alpha(i-1)}(\bm{x},\bm{x}')
\end{multline*}
for any $i$ in $\llbracket 2,n\rrbracket$ and
\[\bigl{\vert}\bigl( \tilde{\bm{m}}^{t,\bm{x}}_{t,t_0}(\bm{x}') -\tilde{\bm{m}}^{t,\bm{x}'}_{t,t_0} (\bm{x}')\bigr)_1\bigr{\vert} \, \le
\,C\Vert \bm{F}\Vert_H\bigl(c^{\frac{\beta+\alpha}{\alpha}}_0+c_0\bigr)d^{\alpha+\beta}(\bm{x},\bm{x}')\, \le \,Cc_0\Vert
\bm{F} \Vert_Hd^{\alpha+\beta}(\bm{x},\bm{x}')\]
where in the last passage we used that $c_0\le 1$ and $d(\bm{x},\bm{x}')\le 1$.
After summing all the terms together at the right scale, we finally show that
\[d(\tilde{\bm{m}}^{t,\bm{x}}_{t,t_0}(\bm{x}'),\tilde{\bm{m}}^{t,\bm{x}'}_{t,t_0} (\bm{x}')) \, \le \, Cc_0^{\frac{1}{1+\alpha(n-1)}}
\Vert \bm{F}\Vert_Hd(\bm{x},\bm{x}')\]
thanks to convexity inequalities and $c_0\le 1$.

We conclude this section showing the reverse Taylor formula which was used in the proof of Proposition \ref{lemma:Holder_modulus_Non-Deg} in
the diagonal regime to handle the discontinuity term:

\begin{lemma}[Reverse Taylor Expansion]
\label{lemma:Reverse_Taylor_Expansion}
Let $\gamma$ be in $(1,2)$, $\phi$ a function in $C^\gamma_{b,d}(\R^{nd})$ and $\bm{x},\bm{x}'$ two points in $\R^{nd}$. Then, there
exists a constant $C:=C(\gamma)$ such that
\[\vert D_{\bm{x}_1}\phi(\bm{x})-D_{\bm{x}_1}\phi(\bm{x}')\vert \, \le \, C \Vert \phi \Vert_{C^\gamma_{b,d}}
d^{\gamma-1}(\bm{x},\bm{x}').\]
\end{lemma}
\begin{proof}
We start rewriting the left-hand side in the following way
\begin{multline*}
D_{\bm{x}_1}\phi(\bm{x})-D_{\bm{x}_1}\phi(\bm{x}') \\
= \, \Bigl(\int_{0}^{1} D_{\bm{x}_1}\phi(\bm{x})-D_{\bm{x}_1}\phi(\bm{x}_1+\lambda d(\bm{x},\bm{x}'),(\bm{x})_{2:n})
\, d\lambda\Bigr) - \Bigl(\int_{0}^{1} D_{\bm{x}_1}\phi(\bm{x}')-D_{\bm{x}_1}\phi(\bm{x}_1+\lambda
d(\bm{x},\bm{x}'),(\bm{x}')_{2:n}) \, d\lambda\Bigr) \\
- \Bigl(\int_{0}^{1} D_{\bm{x}_1}\phi(\bm{x}_1+\lambda d(\bm{x},\bm{x}'),(\bm{x}')_{2:n}) -
D_{\bm{x}_1}\phi(\bm{x}_1+\lambda d(\bm{x},\bm{x}'),(\bm{x})_{2:n}) \, d\lambda\Bigr) \, =: \, I_1+I_2+I_3.
\end{multline*}
The first two components can be treated directly using that $D_{\bm{x}_1}\phi$ is in $C^{\gamma-1}(\R^d)$ with respect to the first
non-degenerate variable. Indeed,
\begin{multline*}
\vert I_1\vert \, \le \, \int_{0}^{1} \vert D_{\bm{x}_1}\phi(\bm{x})-D_{\bm{x}_1}\phi(\bm{x}_1+\lambda d(\bm{x},\bm{x}'),(\bm{x})_{2:n})
\vert\, d\lambda \\
\le \, C\Vert \phi \Vert_{C^\gamma}\int_{0}^{1} \vert\lambda d(\bm{x},\bm{x}') \vert^{\gamma-1} \, d\lambda \, \le \, C\Vert \phi
\Vert_{C^\gamma}d^{
\gamma-1}(\bm{x},\bm{x}')
\end{multline*}
and
\begin{multline*}
\vert I_2\vert \, \le \, \int_{0}^{1} \vert D_{\bm{x}_1}\phi(\bm{x}')-D_{\bm{x}_1}\phi(\bm{x}_1+\lambda d(\bm{x},\bm{x}'),(\bm{x}')_{2:n})
\vert\, d\lambda \\
\le\,C\Vert\phi\Vert_{C^\gamma}\int_{0}^{1} \vert (\bm{x}'-\bm{x})_1 + \lambda d(\bm{x},\bm{x}')
\vert^{\gamma-1}\,d\lambda\,\le\,C\Vert\phi\Vert_{C^\gamma}d^{
\gamma-1}(\bm{x},\bm{x}')
\end{multline*}
where in the last expression we used Young inequality.\newline
To control the last term, we assume for the sake of brevity to be in the scalar case, i.e.\ $d=1$. In the general setting, the proof below
can be reproduced component-wise. The idea is to use a reverse Taylor expansion to pass from the derivative to the function itself. Namely,
\begin{multline*}
\vert I_3 \vert \, = \, \frac{1}{d(\bm{x},\bm{x}')}\Bigl{\vert}\int_{0}^{1}\bigl[\partial_{\lambda}\phi(\bm{x}_1+\lambda
d(\bm{x},\bm{x}'),(\bm{x}')_{2:n}) - \partial_\lambda\phi(\bm{x}_1+\lambda d(\bm{x},\bm{x}'),(\bm{x})_{2:n})\bigr] \, d\lambda\Bigr{\vert} \\
\le \, \frac{1}{d(\bm{x},\bm{x}')} \bigl{\vert} \phi(\bm{x}_1+ d(\bm{x},\bm{x}'),(\bm{x}')_{2:n}) - \phi(\bm{x}_1,(\bm{x}')_{2:n}) +
\phi(\bm{x}_1+ d(\bm{x},\bm{x}'),(\bm{x})_{2:n}) - \phi(\bm{x})\bigr{\vert} \\
\le \, C \Vert \phi \Vert_{C^\gamma}d^{\gamma-1}(\bm{x},\bm{x}').
\end{multline*}
\end{proof}

\bibliography{bibli}
\bibliographystyle{alpha}

\end{document}

%% file: preamble.tex
\usepackage{lineno} \modulolinenumbers[5]
\usepackage[a4paper]{geometry}
\usepackage{amsfonts}
\usepackage{parskip}
\usepackage{graphicx}
\usepackage{xtab}
\usepackage{setspace}
\usepackage{fancyhdr}
\usepackage{graphicx}
\usepackage[T1]{fontenc}
\usepackage[latin1]{inputenc}
\usepackage[main=english,italian]{babel}
\usepackage{type1ec}
\usepackage[final]{microtype}
\usepackage{amsmath,amssymb,amsthm,eucal,dsfont,bm,mathrsfs,stmaryrd}
\usepackage{lmodern}
\usepackage{textcomp}
\usepackage{pict2e,enumitem}
\usepackage{hyperref}
\usepackage[nodayofweek]{datetime} % change the format of printed dates (no american style!)
\usepackage{comment}

\hypersetup{
    pdfpagemode={UseOutlines},
    bookmarksopen,
    pdfstartview={FitH},
    colorlinks,
    linkcolor={black},
    citecolor={black},
    urlcolor={black}
            }

\usepackage{geometry}

\newdateformat{monthyear}{\monthname[\THEMONTH], \THEYEAR}% new date format

\newcommand{\N}{\ensuremath{\mathbb{N}}}% natural numbers
\newcommand{\R}{\ensuremath{\mathbb{R}}}% real numbers
\theoremstyle{definition}
        \newtheorem{definition}{Definition}
        \newtheorem*{remark}{Remark}

\theoremstyle{plain}
        \newtheorem{theorem}{Theorem}
        \newtheorem{lemma}{Lemma}
        
        \newtheorem{prop}{Proposition}